\documentclass[12pt, a4paper]{amsart}
\usepackage{graphics,color,hyperref,verbatim}
\usepackage{amsmath}
\usepackage{amsfonts}
\usepackage{amssymb}
\usepackage{amsthm}
\usepackage{graphicx}
\usepackage{psfrag}
\usepackage{marvosym}
\usepackage{mathrsfs}
\usepackage{enumitem}

\usepackage[a4paper, includeheadfoot, margin=2 cm, top=0.8 cm, footskip=0.8 cm]{geometry} 

\usepackage{multicol} 
\usepackage[usenames,dvipsnames]{xcolor} % allows you to use color names, call this BEFORE you call TikZ
\usepackage{tikz, tikz-3dplot, pgfplots}
\usepackage{tikz-cd}
\usetikzlibrary[positioning,patterns] % tikz libraries for relative positioning and silly patterns
\usetikzlibrary{matrix,arrows,decorations.pathmorphing}

\definecolor{light-gray}{gray}{0.7}

\usepackage{ytableau}

\input xy
\xyoption{all}

\usepackage{calligra}
\usepackage{mathrsfs}
\DeclareMathAlphabet{\mathcalligra}{T1}{calligra}{m}{n}
\DeclareFontShape{T1}{calligra}{m}{n}{<->s*[1.5]callig15}{}

\newtheorem{theorem}{Theorem}[section]
\newtheorem*{theoremstar}{Theorem}
\newtheorem{lemma}[theorem]{Lemma}
\newtheorem{proposition}[theorem]{Proposition}

\newtheorem{corollary}[theorem]{Corollary}

\theoremstyle{definition}
\newtheorem{definition}[theorem]{Definition}
\newtheorem{construction}[theorem]{Construction}
\newtheorem{example}[theorem]{Example}

\newtheorem{remark}[theorem]{Remark}
\newtheorem{theorem-definition}[theorem]{Theorem-Definition}
\newtheorem{lemma-definition}[theorem]{Lemma-Definition}

\newtheorem{variant}[theorem]{Variant}
\newtheorem{notation}[theorem]{Notation}

\numberwithin{equation}{section}

\renewcommand{\AA} {\mathbb{A}}

\newcommand{\CC} {\mathbb{C}}
\newcommand{\DD} {\mathbb{D}}
\newcommand{\EE} {\mathbb{E}}
\newcommand{\FF} {\mathbb{F}}
\newcommand{\GG} {\mathbb{G}}

\newcommand{\KK} {\mathbb{K}}
\newcommand{\LL} {\mathbb{L}}

\newcommand{\NN} {\mathbb{N}}

\newcommand{\PP} {\mathbb{P}}
\newcommand{\QQ} {\mathbb{Q}}
\newcommand{\RR} {\mathbb{R}}

\newcommand{\VV} {\mathbb{V}}

\newcommand{\ZZ} {\mathbb{Z}}

\newcommand {\shC} {\mathcal{C}}
\newcommand {\shD} {\mathcal{D}}
\newcommand {\shE} {\mathcal{E}}

\newcommand {\shJ} {\mathcal{J}}

\newcommand {\shL} {\mathcal{L}}

\newcommand {\shR} {\mathcal{R}}
\newcommand {\shS} {\mathcal{S}}

\newcommand {\shP} {\mathcal{P}}

\newcommand {\sE} {\mathscr{E}}
\newcommand {\sF} {\mathscr{F}}

\newcommand {\sH} {\mathscr{H}}
\newcommand {\sI} {\mathscr{I}}

\newcommand {\sK} {\mathscr{K}}
\newcommand {\sL} {\mathscr{L}}
\newcommand {\sM} {\mathscr{M}}

\newcommand {\sO} {\mathscr{O}}
\newcommand {\sP} {\mathscr{P}}
\newcommand {\sQ} {\mathscr{Q}}
\newcommand {\sR} {\mathscr{R}}

\newcommand {\sV} {\mathscr{V}}
\newcommand {\sW} {\mathscr{W}}

\newcommand {\foA} {\mathfrak{A}}

\newcommand {\foC} {\mathfrak{C}}

\newcommand {\foS} {\mathfrak{S}}

\newcommand {\bL} {\mathbf{L}}

\newcommand {\bS} {\mathbf{S}}

\newcommand {\bdd} {\mathbf{d}}

\usepackage{calligra}
\usepackage{mathrsfs}
\DeclareMathAlphabet{\mathcalligra}{T1}{calligra}{m}{n}
\DeclareFontShape{T1}{calligra}{m}{n}{<->s*[1.5]callig15}{}

\newcommand{\blank}{\underline{\hphantom{A}}}

\newcommand {\Coker} {\operatorname{Coker}}

\newcommand {\cone} {\operatorname{cone}}

\newcommand {\D} {\operatorname{D}}

\newcommand {\Ext} {\operatorname{Ext}}
\newcommand{\sExt}{\mathscr{E} \kern -1pt xt}

\newcommand {\Gr} {\operatorname{Gr}}

\renewcommand {\H} {\operatorname{H}}
\newcommand {\Hom} {\operatorname{Hom}}
\newcommand {\sHom}{\mathscr{H}\kern-5pt\mathcalligra{om}}

\newcommand {\id} {\operatorname{id}}
\newcommand {\Id} {\operatorname{Id}}

\renewcommand {\Im} {\operatorname{Im}}

\renewcommand {\ker } {\operatorname{Ker}}
\newcommand {\Ker} {\operatorname{Ker}}

\newcommand {\N} {\operatorname{N}}

\newcommand {\Pic} {\operatorname{Pic}}
\newcommand {\Proj} {\operatorname{Proj}}

\newcommand{\pr} {\mathrm{pr}}

\newcommand {\rank} {\operatorname{rank}}

\newcommand {\Spec} {\operatorname{Spec}}

\newcommand {\Sym} {\operatorname{Sym}}

\newcommand {\Tor} {\operatorname{Tor}}
\newcommand{\sTor}{\mathscr{T} \kern -3pt or}

\newcommand {\Vect} {\operatorname{Vect}}

\newcommand{\Perf} {\operatorname{Perf}}

\newcommand{\Set}{\operatorname{Set}} %{\underline{\mathrm{Set}}}
\newcommand{\op}{\mathrm{op}}

\newcommand{\Dqc}{\mathrm{D}_{\mathrm{qc}}}

\newcommand{\bDelta}{\operatorname{{\bf \Delta}}}

\newcommand{\cn}{\mathrm{cn}}

\newcommand{\Cat}{\shC\mathrm{at}}

\newcommand{\Fun}{\mathrm{Fun}}

\newcommand{\Shv}{\shS\mathrm{hv}}

\newcommand{\Map}{\mathrm{Map}}

\newcommand{\CAlg}{\operatorname{CAlg}}
\newcommand{\CAlgDelta}{\operatorname{CAlg}^\Delta}

\newcommand{\Mod}{\operatorname{Mod}}
\newcommand{\Modcn}{\operatorname{Mod}^{\rm cn}}

\newcommand {\perf}{\mathrm{perf}}

\newcommand{\Ch}{\operatorname{Ch}}

\newcommand{\SCRMod}{\operatorname{SCRMod}}
\newcommand{\SCRModcn}{\operatorname{SCRMod}^{\rm cn}}

\newcommand{\cl}{\mathrm{cl}}

\newcommand{\fib}{\operatorname{fib}}
\newcommand{\cofib}{\operatorname{cofib}}

\newcommand{\QCoh}{\operatorname{QCoh}}
\newcommand{\QCohcn}{\operatorname{QCoh}^{\rm cn}}

\newcommand {\T} {\operatorname{T}}

\newcommand {\sign} {\mathrm{sign}}

\newcommand {\bSym} {\mathbf{S}}
\newcommand {\bwedge} {\mathbf{\Lambda}}

\newcommand {\Schur}{\operatorname{S}}%{\mathrm{L}} 
\newcommand{\Weyl}{\operatorname{W}}%{\mathrm{K}}
\newcommand{\bSchur} {\bL}

\newcommand {\dSchur}{\mathbb{S}}%{\mathcal{S}\mathrm{chur}}
\newcommand {\dWeyl}{\mathbb{W}} %{\mathcal{W}\mathrm{eyl}}

\newcommand{\Kos}{\operatorname{Kos}}

\usepackage{graphicx}

\makeatletter
\newcommand*\bigcdot{\mathpalette\bigcdot@{.5}}
\newcommand*\bigcdot@[2]{\mathbin{\vcenter{\hbox{\scalebox{#2}{$\m@th#1\bullet$}}}}}
\makeatother

\makeatletter
\newcommand*\bigdot{\mathpalette\bigdot@{.5}}
\newcommand*\bigdot@[2]{\mathbin{\hbox{\scalebox{#2}{$\m@th#1\bullet$}}}}
\makeatother

\newcommand{\Grass}{\operatorname{Grass}}
\newcommand{\Grasscl}{\operatorname{Grass}^{\rm cl}}

\newcommand{\Flag}{\operatorname{Flag}}

\title[Derived Grassmannians and Schur Functors]{Derived Grassmannians and Derived Schur Functors}
\author[Q.Y.\ JIANG]{Qingyuan Jiang}
\address{School of Mathematics, University of Edinburgh, James Clerk Maxwell Building, Peter Guthrie Tait Road, Edinburgh EH9 3FD, United Kingdom.}
\email{qingyuan.jiang@ed.ac.uk}

\begin{document}

\begin{abstract} 
This paper develops two theories, the geometric theory of derived Grassmannians (and flag schemes) and the algebraic theory of derived Schur (and Weyl) functors, and establishes their connection, a derived generalization of the Borel--Weil--Bott theorem. 
More specifically:
\begin{enumerate}
	\item  
	The theory of derived Grassmannians and flag schemes is the natural extension of the theory of derived projectivizations \cite{J22a} and generalizes Grothendieck's theory of Grassmannians and flag schemes of sheaves to the case of complexes. We establish their fundamental properties and study various natural morphisms among them.
	\item 
	The theory of derived Schur and Weyl functors extends the classical theory of Schur and Weyl module functors studied in $\mathrm{GL}_n(\mathbb{Z})$-representation theory to the case of complexes. We show that these functors have excellent functorial properties and satisfy derived generalizations of classical formulae such as Cauchy's decomposition formula, direct-sum decomposition formula and Littlewood--Richardson rule. We also generalize various results from the case of derived symmetric powers to derived Schur functors, such as Illusie--Lurie's d{\'e}calage isomorphisms.
	\item  
	These two theories are connected by a derived version of the Borel--Weil--Bott theorem, which generalizes the classical Borel--Weil--Bott theorem and calculates the derived pushforwards of tautological perfect complexes on derived flag schemes 
	in terms of derived Schur functors when the complexes have perfect-amplitude $\le 1$ and positive ranks.
	\end{enumerate}
\end{abstract}

\maketitle
\tableofcontents
 
%%% sec: Introduction
\section{Introduction}

Grassmannians are ubiquitous across mathematics and have numerous applications. The first objective of this paper is to develop the counterpart of the theory of Grassmannian and flag varieties in the context of derived algebraic geometry; see \S \ref{sec:intro:dGrass.dflag}.

Schur and Weyl functors play a fundamental role in ${\rm GL}_n$-representation theory. The second objective of this paper is to extend the classical theory of Schur and Weyl functors of modules to a theory of derived Schur and Weyl functors of complexes; see \S \ref{sec:intro:dSchur}. 

The Borel--Weil--Bott theorem provides a fundamental link between algebraic geometry and representation theory. This paper's third objective is to establish a derived generalization of the Borel--Weil--Bott theorem, which connects the geometry of derived Grassmannians and flag schemes and the algebra of derived Schur functors; see \S \ref{sec:intro:Bott}. 
For this generalized version of the Borel--Weil--Bott theorem, it is crucial to consider derived constructions, since the corresponding statement for classical constructions only holds under certain conditions.

% subsec: intro: dGrass
\subsection{Derived Grassmannians and Derived Flag Schemes}
\label{sec:intro:dGrass.dflag}

The classical Grassmannian variety parametrizes $d$-dimensional linear subspaces of an $n$-dimensional vector space $V$ over a field $\kappa$. %, denoted by $\Gr_d(V)$ or $\Gr_d(n)$, parametrizes $d$-dimensional linear subspaces of an $n$-dimensional vector space $V$ over a field $\kappa$. We can easily generalize this concept to a family version: let $X$ be any scheme and $\sV$ a vector bundle on $X$, there is a Grassmannian scheme $\Gr_{X,d}(\sV)$ over $X$ which parametrizes $d$-dimension vector sub-bundles of $\sV$ and whose fiber over each point $x \in X(\kappa)$ is the Grassmannian variety $\Gr_d(n)$. %; the fiber of the natural projection $\pr \colon \Gr_{X,d}(\sV) \to X$ over a point $x \in X(\kappa)$ is the usual Grassmannian variety $\Gr_d(n)$. 
%However, this definition does not work if $\sV$ is replaced with a sheaf $\sE$ that is not a vector bundle.
Grothendieck extends this construction and defines Grassmannian as a moduli functor parametrizing vector bundle quotients of a given sheaf.
Concretely, let $X$ be a scheme and $\sE$ a quasi-coherent sheaf on $X$, then Grothendieck defines Grassmannian as a functor
	$$\Grasscl_{d}(\sE) \colon (\mathrm{Scheme}/X)^{\rm op} \to \shS\mathrm{et}$$
which carries every $X$-scheme $\eta \colon T \to X$ to the set of isomorphism classes of quotients $q \colon \eta_{\cl}^* (\sE) \twoheadrightarrow \sP$, where $\sP$ is a vector bundle on $T$ of rank $d$ (see Definition \ref{def:Grass:classical}). Here, we use the symbol ``$\cl$" to mean ``classical", e.g., ``$\eta_{\cl}^*$" means the classical pullback of sheaves. 

Grothendieck's theory of Grassmannians has the advantages of generalizing to the case of all (quasi-coherent) sheaves over schemes and having good functorial properties. It also plays a crucial role in the development of modern moduli theory (see \cite{FGA, FGAexplained}).

The theory of derived Grassmannians naturally extends Grothendieck's Grassmannian theory to the context of derived algebraic geometry. 

In this subsection, we assume for simplicity that $X$ is a classical scheme; the entire theory is set up in this paper for any (derived) prestack $X$. 
A complex $\sE = [\cdots \to \sE^{i} \to \sE^{i+1} \to \cdots ]$ on $X$ is said to be {\em connective} if its homotopy sheaves satisfy $\pi_i(\sE)=0$ for $i<0$, or equivalently, its cohomology sheaves satisfy $\sH^{j}(\sE) (= \pi_{-j}(\sE))=0$ for $j>0$. 
Let $\sE$ be any connective quasi-coherent complex on $X$ and $d \ge 1$ an integer. The {\em derived Grassmannian} is the functor
	$$\Grass_{d}(\sE) \colon (\mathrm{dScheme}/X)^{\rm op} \to \shS := \{\mathrm{Spaces}\} = \{\text{$\infty$-groupoids}\}$$
which carries every morphism $(\eta \colon T \to X)  \in \mathrm{dScheme}/X$,  where $T$ is a {derived} scheme,  to the space %(i.e., $\infty$-groupoid) 
$\Grass_{d}(\sE)(\eta)$ of morphisms $\eta^* \sE \to \sP$ which are surjective on $\pi_0$, where $\sP$ is a vector bundle on $T$ of rank $d$ (see \S \ref{sec:defn:dGrass}). If $T$ is classical, the condition ``$\eta^* \sE \to \sP$ is surjective on $\pi_0$" means that the induced map on zeroth cohomology sheaves is surjective.

If $d=1$, then $\Grass_1(\sE) = \PP(\sE)$ is the {\em derived projectivization} that we studied in \cite{J22a}.

The main differences between the definitions of the derived Grassmannian functor $\Grass_d(\sE)$ and its classical counterpart $\Grasscl_d(\sE)$ are:
\begin{enumerate}[label=(\alph*), ref=\alph*]
	\item \label{intro:dGrass:def.point-1}
	We test the ``value" of the Grassmannian $\Grass_d(\sE)$ not only on classical schemes $T$ but also on {\em derived} schemes $T$ (i.e., ``schemes" $T$ whose ``structure sheaves" $\sO_T$ are allowed to have nonzero higher homotopy groups $\pi_i(\sO_T)$ for $i>0$).
	\item \label{intro:dGrass:def.point-2}
Given any testing derived scheme $T$ and a map $\eta \colon T\to X$, the ``value" $\Grass_d(\sE)(\eta)$ is a {\em space} (i.e., an $\infty$-groupoid) rather than a set. Informally, this means that $\Grass_d(\sE)(\eta)$ remembers not only the isomorphism classes of quotients $\eta^* \sE \to \sP$ but also homotopies between those isomorphisms, and homotopies between homotopies, etc. 
\end{enumerate}

Roughly speaking, point \eqref{intro:dGrass:def.point-1} contributes to the derived structure of $\Grass_d(\sE)$ while point \eqref{intro:dGrass:def.point-2} to its ``stacky structure". Regarding point \eqref{intro:dGrass:def.point-2}, the representability result below (Proposition \ref{prop:Grass:represent}) implies that the ``stacky structure" is trivial along the natural projection $\pr \colon \Grass_d(\sE) \to X$. Regarding point \eqref{intro:dGrass:def.point-1}, the derived structures of $\Grass_d(\sE)$ are often nontrivial and record important ``higher" information, as illustrated in the following example:

\begin{example}[{Cf. \cite[Example 4.23]{J22a}}]
Let $d=1$, $X = \Spec \ZZ$, $\sE = [\ZZ \xrightarrow{0} \ZZ] = \ZZ[1] \oplus \ZZ$. If we test the value of $\Grass_1(\sE) = \PP(\sE)$ on a classical affine scheme $\eta \colon T = \Spec R \to \Spec \ZZ$, then the space $\Grass_1(\sE)(\eta)$ is homotopy equivalent to a single point $\{R[1] \oplus R \xrightarrow{(0,\id)} R\}$.
 %because any map from $R[1]$ to $R$ is zero ($\Hom(R[1],R)\simeq\Ext^{-1}(R,R) \simeq 0$). 
In other words, the values of $\Grass_1(\sE)$ on classical rings $R$ could not ``see" the contribution of the summand $\ZZ[1] \subseteq \sE$. %. In other words, there is no difference between $\PP(\ZZ[1] \oplus \ZZ)$ and the classical scheme $\PP(\ZZ)$ ($\simeq \Spec \ZZ$) when testing on classical schemes.  
However, if we test the value of $\Grass_1(\sE)$ on ``derived rings" (more precisely, simplicial commutative rings) $A$, %; see \cite[Definition 25.1.1.1]{SAG}
the summand $\ZZ[1] \subseteq \sE$ generally contributes nontrivially. % as $\Hom(A[1],A) \simeq \pi_1(A)$ are generally nonzero. 
For example, let $A = \Sym_\ZZ^*(\ZZ[1]) = \ZZ[\varepsilon]$ be the ``derived ring of derived dual numbers" (where $\varepsilon$ denotes a generator of cohomological degree $-1$ which necessarily satisfy $\varepsilon^2=0$), and let $\eta \colon \Spec A \to \Spec \ZZ$ be the natural map, then $\Grass_1(\sE)(\eta) \simeq \ZZ$ rather than a point. In fact, $\Grass_1(\sE) \simeq \Spec \Sym^*_\ZZ(\ZZ[1]) $; see Example \ref{eg:dGrass:Z[1]oplusZ} for more details.
\end{example}

The theory of derived Grassmannians also generalizes to derived flag schemes (\S \ref{sec:flag}). Concretely, instead of a single integer $d$, we consider any increasing sequence of positive integers $\bdd = (d_1, \ldots, d_k)$. 
The {\em derived flag scheme} of $\sE$ of type $\bdd$ over $X$ is the functor
	$$\Flag_{\bdd}(\sE) = \Flag(\sE;\bdd) \colon (\mathrm{dScheme}/X)^{\rm op} \to \shS$$
which carries every map $\eta \colon T \to X$ of derived schemes to the space of sequence of quotients 
	 $$\sE_T = \eta^*(\sE) \xrightarrow{\phi_{k,k+1}} \sP_{k} \xrightarrow{\phi_{k-1,k}} \sP_{k-1} \to \cdots \xrightarrow{\phi_{1,2}} \sP_{1},$$
where $\sP_{i}$ are vector bundles on $T$ of rank $d_i$; see Definition \ref{def:dflag}. If $\bdd = \underline{n}: = (1,2,\ldots,n)$ for an integer $n \ge 1$, then $\Flag(\sE;\underline{n})$ generalizes the classical complete flag varieties.

Derived Grassmannians and derived flag schemes have excellent functorial properties (\S \ref{sec:dGrass.dFlag}). Here is a (partial) list of their fundamental properties established in this work:

\begin{enumerate}%[leftmargin=*]
	 \setcounter{enumi}{-1}
	\item%[$(0)$] 
	\label{intro:dGrass:property-0}
	({\em Underlying classical schemes}). The underlying classical scheme of $\Grass_d(\sE)$ is equivalent to Grothendieck's  Grassmannian scheme of the classical truncation;
	%, i.e., there is a canonical equivalence 
		%$(\Grass_{X,d}(\sE))_{\cl} \simeq \Grass_{X_{\cl}, d}^{\cl}(\pi_0(\sE));$
	see Proposition \ref{prop:Grass-classical}. A similar result holds for general derived flag schemes Proposition \ref{prop:Flag:classical}.
	
	\item%[$(1)$] 
	\label{intro:dGrass:property-1}
	({\em Representability}). The functor $\Grass_d(\sE)$ (resp. $\Flag_{\bdd}(\sE)$) is representable by a relative derived scheme over $X$; see Proposition  \ref{prop:Grass:represent} (resp. Proposition \ref{prop:Flag:rep}). 
	
	\item%[$(2)$] 
	\label{intro:dGrass:property-2}
	({\em Functoriality}). The formation of derived Grassmannians $\Grass_{d}(\sE)$ (or more generally, of derived flag schemes $\Flag_{d}(\sE)$) commutes with arbitrary derived base change and tensoring with line bundles; see Proposition \ref{prop:Grass-4,5} (or Proposition \ref{prop:flag:functorial}).
	Consequently, the formation of derived categories of derived flag schemes is compatible with any base change, a property that is not usually true for classical constructions.
	
	\item%[$(3)$] 
	\label{intro:dGrass:property-3}
	({\em Finiteness properties}). If $\sE$ is locally of finite type, then 
	$\pr \colon \Grass_d(\sE) \to X$ and $\pr \colon \Flag_{\bdd}(\sE) \to X$ are proper. If $\sE$ furthermore pseudo-coherent to order $n$ for some $n \ge 0$ (resp. almost perfect, perfect), then $\pr \colon \Grass_d(\sE) \to X$ and $\pr \colon \Flag_{\bdd}(\sE) \to X$ are locally of finite generation to order $n$ (resp. locally almost of finite presentation, locally of finite presentation); see Propositions \ref{prop:Grass:finite} and \ref{prop:Flag:rep}, respectively.
		
 	\item%[$(4)$] 
	\label{intro:dGrass:property-4}
	({\em Closed immersions induced by surjective morphisms of complexes}). If $\varphi'' \colon \sE \to \sE''$ is a map of connective complexes that is surjective on $\pi_0$, then there is a closed immersion $\iota_{\varphi''} \colon \Grass_d(\sE'') \hookrightarrow \Grass_d(\sE)$ which is the derived zero locus of a cosection of a complex; see Proposition \ref{prop:Grass:PB} for details. A similar result holds for all derived flag schemes $\Flag_{\bdd}(\sE)$; see Proposition \ref{prop:dflag:immersion}. 
	
	\item%[$(5)$] 
	\label{intro:dGrass:property-5}
	({\em Relative cotangent complexes}). The projection $\pr \colon \Grass_d(\sE) \to X$ admits a connective relative cotangent complex $\LL_{\Grass_d(\sE)/X}$ which can be described by the formula 
			$$\LL_{\Grass_d(\sE)/X}  \simeq \sQ^\vee \otimes \fib\big( \pr^*(\sE) \xrightarrow{\rho} \sQ\big),$$
		where $\rho \colon \pr^*(\sE)  \to \sQ$ is the tautological quotient morphism; see Theorem \ref{thm:Grass:cotangent}. 
	
		In general, the derived flag scheme $\pr \colon \Flag_{\bdd}(\sE) \to X$ admits a connective relative cotangent complex $\LL_{\Flag_{\bdd}(\sE)/X}$  concretely described by Theorem \ref{thm:dflag:cotangent}.
		
	\item%[$(6)$]
	\label{intro:dGrass:property-6}
	 ({\em The Pl{\"u}cker morphisms and Segre morphisms \S \ref{sec:dGrass:Plucker.Segre}}). There is a {\em Pl{\"u}cker closed immersion} (\S \ref{sec:dGrass:Plucker}) from the derived Grassmannian $\Grass_{d}(\sE)$ to the derived projectivization $\PP(\bigwedge\nolimits^d \sE)$; see Proposition \ref{prop:Plucker} for more properties of this morphism. The derived generalization of Segre morphisms is studied in \S \ref{sec:dGrass:Segre}.
	
	\item%[$(7)$] 
	\label{intro:dGrass:property-7}
	({\em Natural morphisms among derived flag schemes}). 
	There are many natural morphisms among various derived flag schemes:
		\begin{itemize}
			\item (Forgetful morphisms among different types $\bdd$; \S \ref{sec:dflag:forget}). If $\bdd'$ is a subsequence of $\bdd$, then there is a natural forgetful morphism $\pi_{\bdd',\bdd} \colon \Flag_{\bdd}(\sE) \to \Flag_{\bdd'}(\sE)$  (\S \ref{sec:dflag:forget}) whose behavior is described by Lemma \ref{lem:dflag:forget} and Corollary \ref{cor:dflag:forget}.
			\item  (Closed immersions into products of Grassmannians; \S \ref{sec:closed.dflag.to.dGrass}) For a given sequence $\bdd$, there is a closed immersion of $\Flag_{\bdd}(\sE)$ into the fiber product $\prod_{i} \Grass_{d_i}(\sE)$ over $X$ whose behavior is described by Proposition \ref{prop:dflag:into.Grass}.
\end{itemize}
\end{enumerate}

The properties listed above, particularly \eqref{intro:dGrass:property-4},\eqref{intro:dGrass:property-5} and \eqref{intro:dGrass:property-7}, demonstrate the advantages of using the derived framework to study Grassmannians and flag schemes.
For instance, the above property \eqref{intro:dGrass:property-5} provides explicit descriptions of the relative cotangent complex for {\em all} derived Grassmannians $\Grass_d(\sE) \to X$ and derived flag schemes $\Flag_{\bdd}(\sE) \to X$, whereas, as far as the author is aware, there is no general description of the relative cotangent complexes for classical Grassmannians and flag schemes of sheaves.

The derived framework of Grassmannians and flag schemes has many applications to classical situations as well as moduli problems; some of these applications are discussed in \S \ref{sec:intro:related.fields}.

% subsec: intro: dSchur
\subsection{Derived Schur and Weyl Functors}
\label{sec:intro:dSchur}
\subsubsection{Overview}
Classical Schur and Weyl functors are generalizations of symmetric, exterior and divided power functors and are of fundamental importance in the representation theory. 
The characteristic-zero theory of Schur and Weyl functors has been developed by
Schur, Young, Frobenius, Weyl and others at the beginning of the twentieth century (see \cite{Weyl}). 
Later, in the 1970s and 1980s, Carter--Lusztig \cite{CL}, Akin--Buchsbaum--Weyman \cite{ABW} and others established the characteristic-free theory (see also Green's book \cite{Green} and Towber's papers \cite{Tow1,Tow2}). 
We refer readers to books \cite{FH, Ful, Wey, Jan, Green, Martin} and papers \cite{CL, AB, ABW} and references therein for relevant background and detailed expositions (see also \cite{Bra, BMT21} for recent perspectives).

This work extends the classical theory of Schur and Weyl functors of finite free modules to a theory of derived Schur and Weyl functors of complexes.  There are two precursors of such a theory: Dold, Puppe, and Quillen's construction \cite{DP, Qui2} via simplicial resolutions (see also \cite{TW}), and Akin, Buchsbaum, and Weyman's constructions of Schur complexes \cite{ABW} (in the case of two-term complexes). 
Both of these approaches have advantages and limitations. 
We will take an approach which uses Lurie's theory of non-abelian derived functors \cite{HTT, SAG} (along with the classical theory on Schur and Weyl functors and Schur complexes mentioned in the preceding paragraphs) and extends our methods in \cite[\S2, \S3]{J22a}. 
This framework generalizes both approaches and benefits from their best features, such as being functorial, applicable to all connective complexes, and computationally effective.

\subsubsection{Classical Schur and Weyl functors}
Regarding classical Schur and Weyl functors, 
this paper will primarily follow Akin, Buchsbaum, and Weyman's approach in  \cite{ABW}, which we will review in \S \ref{sec:Schur}. The characteristic-free and functorial nature of their approach allows us to extend the theory to the derived setting in \S \ref{sec:dSchur}.

Let $R$ be a commutative ring, and let $V$ be a finite free $R$-module. Let $N$ be a positive integer, and let $\lambda=(\lambda_1, \ldots, \lambda_n)$ be a partition of $N$.

 If $R= \CC$ is the field of complex numbers (or any field of characteristic zero), then the Schur and Weyl modules for $V$ coincide and are defined as the image $\Schur^{\lambda}(V)= \Weyl^{\lambda}(V) = \Im (c_{\lambda}|V^{\otimes N})$, where $c_{\lambda} \in \CC \foS_{N}$ is the Young symmetrizer % is the primitive idempotent (up to a scalar multiple) of $\CC \foS_{n}$ 
 associated with $\lambda$ (\cite[\S 4.1]{FH}). Moreover, Schur modules $\Schur^{\lambda}(V)$ are irreducible representations of the symmetric group $\foS_N$ and the general linear group ${\rm GL}_N(\CC)$ (see \cite[Theorem 6.3, Proposition 16.47]{FH}).

If $R$ has positive or mixed characteristic, then the above constructions using Young symmetrizers
 are no longer suitable. Akin, Buchsbaum, and Weyman \cite{ABW} circumvented the need of Young symmetrizers by using the Hopf algebra structures of symmetric and exterior algebras. The Schur module $\Schur_R^{\lambda}(V)$ is defined as the image of the composite map
	$$\bigwedge\nolimits_R^{\lambda_1^t} (V) \otimes \cdots \otimes \bigwedge\nolimits_R^{\lambda_p^t} (V) \xrightarrow{\otimes \Delta_j }  V^{\otimes N} \xrightarrow{\otimes  m_i} \Sym_R^{\lambda_1} (V)  \otimes \cdots \otimes \Sym_R^{\lambda_n}(V),$$
where $\Delta_j$'s are the anti-symmetrization maps and $m_i$'s are the multiplication maps. The Weyl module $\Weyl_R^{\lambda}(V)$ is defined via a dual manner; %using exterior and divided powers; 
see Remark \ref{rem:Schur.Weyl.as.images} for details. 
Then the constructions $V \mapsto \Schur_R^{\lambda}(V)$ and $V \mapsto \Weyl_R^{\lambda}(V)$ define functors 
	$$\Schur^{\lambda} ~\text{and}~ \Weyl^{\lambda} \colon \{\text{Finite free $R$-modules}\} \to  \{\text{Finite free $R$-modules}\}% (=: \Mod_R^{\rm ff})
	$$
which are referred to as the {\em (classical) Schur and Weyl functors} (\cite[Definition II.1.3]{ABW}), respectively.
Since the formation of Schur and Weyl functors commutes with usual base change  %\cite[Theorem II.2.16 \& Theorem II.3.16]{ABW} or 
(see Theorem \ref{thm:Schur:free}), 
they extend to the geometric context and give functors
	$$\Schur^{\lambda} ~\text{and}~ \Weyl^{\lambda} \colon \{\text{Vector bundles on $X$}\} \to  \{\text{Vector bundles on $X$}\}, %(=: {\rm Vect}_X)
	$$
%which carries a vector bundle $\sV$ on a scheme $X$ to vector bundles $\Schur^{\lambda}_X(\sV)$ and $\Weyl^{\lambda}_X(\sV)$ on $X$, respectively.
where $X$ is a scheme. These functors are also called {\em (classical) Schur and Weyl functors}.

\subsubsection{Derived Schur and Weyl functors} 
This paper extends the above theory to a theory of derived Schur and Weyl functors of complexes.
The derived Schur and Weyl functors are generalizations of derived symmetric, exterior and divided power functors studied in \cite[Chapter 25]{SAG}, \cite[\S 2, \S 3]{J22a}, and are the derived functors of classical Schur and Weyl functors in the sense of Lurie's non-abelian derived theory (\cite{HTT}) which we will review in \S \ref{sec:non-abelian}. 

The derived theory extends the classical theory as follows.
Let $X$ be a derived scheme (or more generally, prestack), then the {\em derived Schur and Weyl functors} are functors
 $$\dSchur^{\lambda} \colon \QCoh(X)^\cn \to \QCoh(X)^\cn \quad \text{and}  \quad \dWeyl^{\lambda} \colon \QCoh(X)^\cn \to \QCoh(X)^\cn$$
which carries a connective quasi-coherent complex $\sE$ on $X$ to connective quasi-coherent complexes $\dSchur_X^{\lambda}(\sE)$ and $\dWeyl_X^{\lambda}(\sE)$, respectively (Definition \ref{def:dSchurWeyl}). Here,
	$$\QCoh(X)^\cn := \text{the $\infty$-category of connective quasi-coherent complexes on $X$}$$
is the ``$\infty$-categorical enhancement" of the usual derived category $\Dqc(X)^\cn = \Dqc(X)^{\le 0}$ of connective quasi-coherent complexes on $X$. 
If $X$ is a classical scheme, then objects of $\QCoh(X)^\cn$ are complexes $\sE = [\cdots \to \sE^{i} \to \sE^{i+1} \to \cdots]$ on $X$ whose cohomology sheaves $\sH^i(\sE)(=\pi_{-i}(\sE))$ are quasi-coherent for all $i$ and satisfy $\sH^i(\sE) = 0$ for all $i>0$.

If $\sE = \sV$ is a vector bundle, then $\dSchur^{\lambda}(\sV)$ and $\dWeyl^{\lambda}(\sV)$ coincide with the classical Schur and Weyl functors $\Schur^{\lambda}(\sV)$ and $\Weyl^{\lambda}(\sV)$, respectively. Generally, $\dSchur^{\lambda}(\sE)$ and $\dWeyl^{\lambda}(\sE)$ are complexes rather than sheaves even when $\sE$ is a discrete sheaf on $X$.

This paper also develops the theory of derived Schur and Weyl functors, $\dSchur^{\lambda/\mu}$ and $\dWeyl^{\lambda/\mu}$ respectively, for general {\em skew partitions} $\lambda/\mu$, i.e., pairs of partitions $\mu \subseteq \lambda$; see Definition \ref{def:dSchurWeyl}.

The fundamental properties of derived Schur and Weyl functors are investigated in \S \ref{sec:dSchur}. For readers' convenience, we summarize here many of their important properties that established in this work. For simplicity, in the following list we assume $X = \Spec A$ for a ``derived ring" $A$ and $\sE = M$ is a connective complex over $A$.

\begin{enumerate}
	\item ({\em Base change; Proposition \ref{prop:dSchur:basechange}}). 
	The formation of derived Schur and Weyl functors commutes with arbitrary base change of ``derived rings". 
        %i.e. for any map $A \to B$ of ``derived rings", there are canonical equivalences 
	%$\dSchur_B^{\lambda/\mu}(B \otimes_A M)\simeq B \otimes_A \dSchur_A^{\lambda/\mu}(M)$ and $\dWeyl_B^{\lambda/\mu}(B \otimes_A M) \simeq B \otimes_A \dWeyl_A^{\lambda/\mu}(M)$.
	\item ({\em Freeness; Proposition \ref{prop:dSchur:free}}). 
	If $M$ is finite free (resp. locally finite free, flat) over $A$, then so are $\dSchur_A^{\lambda/\mu}(M)$ and $\dWeyl_A^{\lambda/\mu}(M)$. %are finite free (resp. locally finite free, flat); see 
	\item ({\em Flatness; Proposition \ref{prop:dSchur:flat}}).  If $A=R$ is classical  and $M$ is flat, then $\dSchur_R^{\lambda/\mu}(M)$ and $\dWeyl_R^{\lambda/\mu}(M)$ are equivalent to the classical modules $\Schur_R^{\lambda/\mu}(M)$ and $\Weyl_R^{\lambda/\mu}(M)$, respectively. 
	\item ({\em Classical truncations; Proposition \ref{prop:dSchur:classical}}) 
	The classical truncations of the derived Schur and Weyl functors are equivalent to the (generalized) classical Schur and Weyl functors applied to their classical truncations. %That is, there are canonical equivalences 
%	$\pi_0(\dSchur_A^{\lambda/\mu}(M)) \xrightarrow{\sim} \Schur_{\pi_0 A}^{\lambda/\mu}(\pi_0(M))$ and 
%	$\pi_0( \dWeyl_A^{\lambda/\mu}(M)) \xrightarrow{\sim} \Weyl_{\pi_0 A}^{\lambda/\mu}(\pi_0(M)),$
%where the right-hand sides denote the (generalized) classical Schur and Weyl modules of Definition \ref{def:SchurWeyl}.
	\item ({\em Connectivity; Corollary \ref{cor:dSchur:connective}}). If $M$ is $m$-connective for some integer $m \ge 0$, then $\dWeyl^{\lambda/\mu}_A(M)$ is $m$-connective, and if $m \ge 1$, $\dSchur^{\lambda/\mu}_A(M)$ is $(m-1+ |\lambda|-|\mu|)$-connective. 
	\item ({\em Finiteness; Proposition \ref{prop:dSchur:pc}}). If $M$ is pseudo-coherent to order $m$ for some $m \ge 0$ (resp. almost perfect), then so are $\dSchur_A^{\lambda/\mu}(M)$ and $\dWeyl_A^{\lambda/\mu}(M)$.
	\item ({\em Tor-amplitudes and perfectness; Proposition \ref{prop:dSchur:Tor-amp}}). If $M$ has Tor-amplitude $\le m$ for some integer $m \ge 0$, then $\dSchur_A^{\lambda/\mu}(M)$ and $\dWeyl_A^{\lambda/\mu}(M)$ have Tor-amplitude $\le m (|\lambda| - |\mu|)$. If $M$ is a perfect complex, then $\dSchur_A^{\lambda/\mu}(M)$ and $\dWeyl_A^{\lambda/\mu}(M)$ are perfect complexes. 
\end{enumerate}

The theory of derived Schur and Weyl functors also extends various fundamental formulae and rules (\cite{ABW, Wey, Kou, Bo1, Bo2}) from its classical counterpart. Let $A$ be a ``derived ring", $M, M'$ connective complexes over $A$,  $\lambda/\mu$ a skew partition and $\nu$ any partition.

\begin{enumerate}
	\item ({\em Cauchy Decomposition Formula; Theorem \ref{thm:fil:dsym_otimes}}). 
	The derived $n$th symmetric power $\Sym_A^n(M \otimes_A M')$ %and exterior power $\bigwedge^n(M \otimes M')$ 
	admits a canonical ``filtration" whose ``subquotients" are given by derived tensor products of derived Schur  functors $\dSchur_A^{\lambda^i}(M) \otimes_A \dSchur_A^{\lambda^i}(M')$, where $\lambda^i$ are partitions of $n$.
	A similar formula holds for the derived $n$th exterior power $\bigwedge_A^n(M \otimes M')$. 
	\item ({\em Direct-Sum Decomposition Formula; Theorem \ref{thm:fib:dSchur_oplus}}). 
	$\dSchur^{\lambda/\mu}(M\oplus M')$ canonically splits into a direct sum $\bigoplus_{k=0}^{N} \dSchur_A^{\lambda/\mu}(M,M')_{(k,N-k)}$, where $N = |\lambda| - |\mu|$, and $\dSchur_A^{\lambda/\mu}(M,M')_{(k,N-k)}$ are the bi-degree $(k,N-k)$ ``homogeneous components" of $\dSchur_A^{\lambda/\mu}(M \oplus M')$ (Remark \ref{rmk:fil:dSchur_oplus:(k,N-k)}). Furthermore, each $\dSchur_A^{\lambda/\mu}(M,M')_{(k,N-k)}$ admits a canonical ``filtration" whose ``subquotients" are given by derived tensor products of the form 
	$\dSchur_A^{\gamma_{(k)}^i/\mu}(M) \otimes_A \dSchur_A^{\lambda/\gamma^i_{(k)}}(M')$, where $\mu \subseteq \gamma_{(k)}^i \subseteq \lambda$
are partitions. A similar formula holds for derived Weyl functors. 
	\item ({\em Littlewood--Richardson Rule; Corollary \ref{cor:fil:dSchur_LR} 
}). 
The derived tensor product $\dSchur^{\lambda}(M) \otimes \dSchur^{\nu}(M)$ admits a canonical ``filtration" whose ``subquotients" are of the form $\dSchur_A^{\nu^i}(M)$, where $\nu^i$ are partitions such that $|\nu^i| = |\lambda| + |\nu|$ and the Littlewood--Richardson number $c_{\lambda,\nu}^{\nu^i} \ne 0$. A similar formula holds for $\dWeyl^{\lambda}(M) \otimes \dWeyl^{\nu}(M)$.
	\item ({\em Decomposition Rule for Skew Partitions; Theorem \ref{thm:fil:dSchur_LR}}). The Littlewood--Richardson Rule is a special case of a decomposition rule for skew partitions, which associates each derived Schur functor $\dSchur_A^{\lambda/\mu}(M)$ a canonical ``filtration" whose ``subquotients" are of the form $\dSchur_A^{\tau^i}(M)$, where $\tau^i$ are partitions such that $|\tau^i| = |\lambda| - |\mu|$ and the Littlewood--Richardson number $c_{\mu,\tau^i}^{\lambda} \ne 0$. A similar rule holds for $\dWeyl_A^{\lambda/\mu}(M)$.
	\item ({\em Koszul-type Sequences Associated with Derived Schur Functors; Theorem \ref{thm:fib:dbSchur}}). We also obtain canonical sequence of morphisms of Koszul type associated with derived Schur functors of a morphism $\rho \colon M \to M'$ between connective complexes, which are generalizations of the Koszul complexes for classical symmetric powers.
\end{enumerate}
In the above formulations, by a canonical ``filtration" we mean a canonical sequence $F_m \to \cdots \to F_1 \to F_0$ of objects in the $\infty$-category $\QCoh(X)^\cn$ for some integer $m \ge 0$, which we should regard as a {\em complex} in the prestable $\infty$-category $\QCoh(X)^\cn$ in the sense of Lurie \cite[Definition 1.2.2.2]{HA} (see Remark \ref{rmk:fib:dSchur_oplus}). By ``subquotients" we mean the cofibers $\cofib(F_{i+1} \to F_{i})$ of the morphisms $F_{i+1} \to F_{i}$ in the above sequence, which correspond to the mapping cones $\cone(F_{i+1} \to F_{i})$ in the homotopy category $\Dqc(X)^{\le 0}$. 

\medskip
Additionally, various results for derived symmetric and exterior powers generalize to the case of derived Schur and Weyl functors:

\begin{enumerate}
	\item ({\em Generalized Illusie--Lurie's d{\'e}calage isomorphisms; Theorem \ref{thm:Illusie--Lurie} and Corollary \ref{cor:dSchur.decalage}}). There is a canonical equivalence $\dSchur^{\lambda/\mu}_{A}(M[1]) \simeq \dWeyl_A^{\lambda^t/\mu^t}(M)[|\lambda|- |\mu|].$ Consequently, in characteristic zero, we obtain ``periodicity" equivalences: for an integer $m \ge 0$,
	$\dSchur_{A}^{\lambda/\mu} (M[m]) \simeq \dSchur_{A}^{\lambda^t/\mu^t} (M) [m \, (|\lambda| - |\mu|)]$ if $m$ is odd, and 
	 $\dSchur_{A}^{\lambda/\mu} (M[m]) \simeq \dSchur_{A}^{\lambda/\mu} (M) [m \, (|\lambda| - |\mu|)]$ if $m$ is even.
	This theorem generalizes Illusie--Lurie's isomorphism for derived symmetric and exterior powers \cite[Proposition 25.2.4.2]{SAG}, and is a categorification of the d{\'e}calage isomorphisms of Quillen \cite{Qui} and Bousfield \cite{Bous} for symmetric and exterior functors, and Touz{\'e} \cite{Tou} for Schur and Weyl functors.
	\item ({\em Generalized Illusie's equivalences; Proposition \ref{prop:dSchur_vs_bSchur}}). In the case where $A=R$ is an ordinary commutative ring, and $M$ is represented by a morphism $P_1 \to P_0$ of finite projective $R$ modules, then the Schur complex $\bSchur_R^{\lambda/\mu}(\rho)$ (\S \ref{sec:bSchur}) canonically represents the derived Schur functor $\dSchur_R^{\lambda/\mu}(M)$. This result generalizes Illusie's equivalences for derived symmetric and exterior powers \cite{Ill}
	(see also \cite[Corollary 2.30 \& Variant 2.31]{J22a}). 
\end{enumerate}

All the above-mentioned results hold for general prestacks $X$ and connective quasi-coherent complexes $\sE$, even though we only mention the derived affine case in this introduction; see \S \ref{sec:dSchurdWeyl.prestacks}. 

\subsection{Derived Generalization of the Borel--Weil--Bott Theorem}
\label{sec:intro:Bott}
The classical Borel--Weil--Bott theorem (\cite{Bott, Dem}) computes cohomologies of tautological bundles on flag varieties in terms of Schur modules studied in the $\mathrm{GL}_n$-representation theory. 

In this introduction, we shall only address the version for line bundles on complete flag manifolds, as other classical versions of the Borel--Weil--Bott theorem are consequences of this. 

Let $n \ge 1$ be an integer, $R$ a commutative ring, and let $\Flag_R(R^n; \underline{n}) = {\rm GL}_n(R)/ B_R$ denote the complete flag manifold over $R$ (where $B_R$ is the Borel subgroup of the Chevalley group scheme ${\rm GL}_n(R)$ over $R$; see \S \ref{sec:tautological.complexes}). Let $\lambda = (\lambda_1, \lambda_2, \ldots, \lambda_n)$ be a sequence of non-negative integers; we can identify $\lambda$ as an element of the character group $X(T_R)$ (where $T_R \subseteq B_R$ is a maximal $R$-torus). Let $\sL(\lambda)$ denote the associated line bundle on $\Flag_R(R^n; \underline{n})$. 

The classical Borel--Weil--Bott theorem \cite{Bott, Dem1, Dem, Lurie, Wey} together with Kempf's vanishing theorem \cite{Kempf} computes the cohomology modules of the line bundles $\sL(\lambda)$:

\begin{theoremstar}[{\bf Borel--Weil, Bott, Kempf}; see {\cite[Theorems 4.1.4 \&  4.1.10]{Wey}}] In the above situation, we have:
\begin{enumerate}
	\item (Characteristic-free). If $\lambda$ is dominant (i.e., a partition), then $\H^i\big(\Flag_R(R^n; \underline{n}); \sL(\lambda)\big) \simeq 0$ for $i >0$ and $\H^0\big(\Flag_R(R^n; \underline{n}); \sL(\lambda)\big) \simeq \Schur_R^{\lambda}(R^n)$.
	\item (Characteristic-zero). If $R$ is a $\QQ$-algebra, then for any $\lambda \in \ZZ_{\ge 0}^n$, one of the following two mutually exclusive cases occurs:
	\begin{enumerate}
	\item There exists a pair of integers $1 \le i < j \le n-1$ such that $\lambda_i  - \lambda_j = i - j$. In this case,  
		$\H^i\big(\Flag_R(R^n; \underline{n}); \sL(\lambda)\big) \simeq 0$ for all $i$.
	\item There exists a unique permutation $w \in \foS_n$ such that $w \bigdot \lambda$ is non-increasing. In this case,
		$\H^i\big(\Flag_R(R^n; \underline{n}); \sL(\lambda)\big) = 0$ for all $i \neq \ell(w)$, and there is an isomorphism $\H^{\ell(w)}\big(\Flag_R(R^n; \underline{n}); \sL(\lambda)\big) \simeq \Schur_R^{w \bigdot \lambda}(R^n)$. (Here, $\ell(w)$ denotes the length of $w$.)
\end{enumerate}
\end{enumerate}
\end{theoremstar}

In the literature, the vanishing result of assertion $(1)$ is known as Kempf's vanishing theorem, and assertion $(2)$ is commonly referred to as Bott's theorem or the Borel--Weil--Bott theorem. The explicit presentation of cohomology modules using Schur functors is established in Weyman's book {\cite[\S 4]{Wey}}.
{\em To keep language simple, this paper will collectively refer to all the assertions of the above theorem as ``the Borel--Weil--Bott theorem".}
However, this terminology is misleading because it fails to mention Kempf's vanishing result \cite{Kempf} and Weyman's presentation via Schur functors \cite{Wey}, both of which are crucial components of its content.

This paper establishes a derived generalization of the above Borel--Weil--Bott theorem. Concretely, let $X$ be a derived scheme (or more generally, a prestack), $\sE$ a connective quasi-coherent complex on $X$, and $\pr \colon \Flag_X(\sE; \underline{n}) \to X$ the derived complete flag scheme of $\sE$ considered in \S \ref{sec:intro:dGrass.dflag}. There is a tautological sequence of quotients 
	$$\pr^*(\sE) \xrightarrow{\varphi_{n}} \sQ_{n} \xrightarrow{\phi_{n-1,n}} \sQ_{n-1} \to \cdots \xrightarrow{\phi_{1,2}} \sQ_{1}  $$
over $\Flag_X(\sE; \underline{n})$, where $\sQ_i$ are vector bundles of rank $i$, $1 \le i \le n$,  and we agree by convention that $\sQ_{0} = 0$. Let $\sL_i = \Ker(\phi_{i-1,i} \colon \sQ_{i} \to \sQ_{i-1})$. Then for each ``weight" $\lambda = (\lambda_1 , \ldots, \lambda_n) \in \ZZ_{\ge 0}^n$, there is an associated line bundle
	$$\sL(\lambda) : = \sL_1^{\otimes \lambda_1} \otimes \sL_2^{\otimes \lambda_2} \otimes \cdots \otimes \sL_n^{\otimes \lambda_n} \in \Pic(\Flag_X(\sE; \underline{n}))$$
whose construction generalizes the above classical one; see Construction \ref{constr:dflag}. % \eqref{eqn:flag:linebundle}.

\begin{theoremstar}[{{\bf Borel--Weil--Bott Theorem for Derived Complete Flag Schemes}; Theorems \ref{thm:Bott:dflag} and \ref{thm:BBW:dflag}}]
In the situation of the preceding paragraph, assume furthermore that $\sE$ is a perfect complex of rank $n \ge 1$ and Tor-amplitude in $[0,1]$ over $X$. Then the following are true:
\begin{enumerate}
	\item (Characteristic-free). If $\lambda$ is dominant (i.e., a partition),  $\pr_*(\sL(\lambda)) \simeq \dSchur^{\lambda}(\sE)$.
	\item (Characteristic-zero). If $X$ is defined over the field $\QQ$ of rational numbers, then for any $\lambda \in \ZZ_{\ge 0}^n$, one of the following two mutually exclusive cases occurs:
	\begin{enumerate}
	\item There exists a pair of integers $1 \le i < j \le n-1$ such that $\lambda_i  - \lambda_j = i - j$. In this case, $\pr_* (\sL(\lambda)) \simeq 0.$
	\item There exists a unique permutation $w \in \foS_n$ such that $w \bigdot \lambda$ is non-increasing. In this case,  there is a canonical equivalence
		$\pr_* (\sL(\lambda)) \simeq \dSchur^{w \bigdot \lambda}(\sE) [- \ell(w)].$
	\end{enumerate}
\end{enumerate}
\end{theoremstar}

The derived version of Borel--Weil--Bott theorem reflects Kontsevich's philosophy of ``hidden smoothness", and extends its classical counterpart in the following ways:

\begin{itemize}%[leftmargin=*]
	\item  It extends from the case of a classical scheme $X=\Spec R$ to any prestack $X$.
	\item It extends from the case where $\sE=R^n$ is a vector bundle (on $X=\Spec R$) to any perfect complex $\sE$ of Tor-amplitude in $[0,1]$ of rank $n$ (on a prestack $X$).
	\item The cohomology $\H^*(\sL(\lambda))$ is replaced by the derived pushforward $\pr_*(\sL(\lambda))$.
	\item The classical Schur functor $\Schur_R^{\lambda}(R^n)$ is replaced by the derived Schur functor $\dSchur_X^{\lambda}(\sE)$. 
\end{itemize}

We emphasize that the functors in this paper are {\em derived}, so that assertion $(1)$ already includes {\em Kempf's vanishing theorem} as part of its formulation, i.e., the the equivalence $\pr_*(\sL(\lambda)) \simeq \dSchur^{\lambda}(\sE)$ of assertion (1) implies that $\RR^i \pr_*(\sL(\lambda)) \simeq \pi_{-i}(\dSchur^{\lambda}(\sE)) \simeq 0$ for all $i >0$. However, if $\sE$ is not a vector bundle, $\pr_*(\sL(\lambda))$ could have nonzero negative cohomology sheaves even when $X$ is classical and $\lambda$ is dominant. In this case, these cohomology sheaves can be computed by combining the equivalence $\pr_*(\sL(\lambda)) \simeq \dSchur^{\lambda}(\sE)$ of assertion $(1)$, the generalized Illusie's equivalence (Proposition \ref{prop:dSchur_vs_bSchur}) and the theory of Schur complexes \cite{ABW, Wey, BHL+}. 

\begin{remark}[Characteristics] 
If $R$ has characteristic-zero, the Schur modules $\Schur^{\lambda}_R(R^n)$ are irreducible representations of ${\rm SL}_n$. 
In this case, the classical Borel--Weil--Bott theorem consists of two parts: first, the line bundles $\sL(\lambda)$ have nonzero cohomologies precisely in a single degree, and second, Bott's algorithm (assertion $(2ii)$) completely determines these cohomology modules. These statements are no longer true in positive characteristics. In general,  determining the vanishing behaviors and computing the cohomology modules of the line bundles $\sL(\lambda)$ in positive characteristics are difficult and important questions; we refer readers to \cite{Jan} for more information, \cite{AMRW, RW} for recent advances, and \cite{And} for a nice survey.

The situation is similar for the derived version of Borel--Weil--Bott theorem. Since this paper uses a similar strategy as the classical one (\cite{Dem, Wey, Lurie}) to reduce the non-dominant weight cases to dominant weight cases, the characteristic issues for the derived Borel--Weil--Bott theorem are the same as for the classical one. 
\end{remark}

This paper also obtains derived generalizations of the Borel--Weil--Bott theorem for Grassmannians and partial flag schemes.
For example, in the above situation, let $\pr \colon \Grass(\sE;d) \to X$ denote the derived Grassmannian,  $\sR \to \pr^*(\sE) \to \sQ$ the universal fiber sequence, and let $\alpha = (\alpha_1, \ldots , \alpha_{d})$ and $\beta = (\beta_1, \ldots, \beta_{n-d})$ be two partitions. %There is a canonically associated perfect complex \eqref{eqn:Grass:Vlambda} on $\Grass(\sE;d)$,$$\sV(\alpha,\beta) = \dSchur^{\alpha}(\sQ) \otimes \dSchur^{\beta}(\sR).$$
The derived version of Borel--Weil--Bott theorem computes the derived pushforward of the perfect complex 
	$$\sV(\alpha,\beta) = \dSchur^{\alpha}(\sQ) \otimes \dSchur^{\beta}(\sR)$$
via a formula similar to the above Borel--Weil--Bott theorem for derived flags schemes, with the role of $\lambda$ replaced by $(\alpha,\beta)$ and $\pr_*(\sL(\lambda))$ by $\pr_*(\sV(\alpha,\beta))$; see Corollaries \ref{cor:Bott:dGrass:dominant} and \ref{cor:BBW:dGrass}. 

Generally, the Borel--Weil--Bott theorem for derived (partial) flag schemes computes  pushforwards of perfect complexes $\sV(\lambda)$  \eqref{eqn:flag:Vlambda} on $\Flag_{\bdd}(\sE)$; see Corollaries \ref{cor:BBW:dpflag} and \ref{cor:Bott:dpflag}. 

In addition to the results mentioned above, we also obtain a variant of the above  Borel--Weil--Bott theorem for positive characteristics (Variant \ref{variant:BBW:dflag}), a characteristic-free vanishing result (Proposition \ref{prop:dflag:vanishing}), and a version of derived Borel--Weil--Bott theorem for {\em skew} partitions (Corollary \ref{cor:Bott:dGrass}). The last result appears to have escaped the literature even in the classical situation for vector bundles.

%subsec: intro: DAG
\subsection{The Framework of Derived Algebraic Geometry}
This work is based on the framework of derived algebraic geometry (shortened as DAG) developed by Lurie \cite{DAG, HTT, DAGV, HA, SAG, kerodon}, To{\"e}n and Vezzosi \cite{HAGI, HAGII, ToenDAG} and many others (see also \cite{BFN10, GR, TVa07, PTVV} and references therein). 
%For definiteness, this paper uses Lurie's framework of $\infty$-categories and his series of works as the primary references for results of DAG. However, we expect that our arguments will hold in other DAG contexts.

The derived framework is essential in this work on both the geometric and algebraic side. Geometrically, the framework of DAG is crucial for extending Grothendieck's theory \cite{EGAI} to any base (prestack) and any connective complex (\S \ref{sec:dGrass.dFlag}). Algebraically, Lurie's non-abelian derived theory is crucial for our study of derived Schur and Weyl functors (\S \ref{sec:dSchur}). 

The use of the derived framework has both practical and theoretical advantages. From a practical perspective, it requires no conditions on base spaces or  complexes, greatly broadening the application range of Grothendieck's already powerful framework (see also \S \ref{sec:intro:related.fields}).

From a theoretical perspective, the framework of DAG has the benefit 
of overcoming classical limitations of non-functoriality of exact triangles, possessing excellent base change properties, converting algebraic geometry problems into algebraic topology problems (see, e.g., Theorem \ref{thm:Grass:cotangent}), and revealing phenomena that are not readily visible in the classical context. For a more detailed explanation, we direct readers to \cite[\S 1.5]{J22a}.

%subsec: intro:Further
\subsection{Related Works and Further Directions}
\label{sec:intro:related.fields}
As derived extensions of Grothendieck's theory of Grassmannians \cite{EGAI} and the classical theory of Schur and Weyl functors, we anticipate a wide range of applications for the theories presented in this paper.

%There are several specific scenarios that we have in mind:

One of the main applications of this paper's framework is the proof of the Quot formula conjecture studied in  \cite{JL18, J20, J21}. Recently, Yukinobu Toda \cite{Tod6} provided an elegant proof of the conjecture using categorified Hall products, in the case where the base space $X$ is a smooth quasi-projective variety over $\CC$. In \cite{J23}, we offer an alternative approach based on the framework presented in this paper. Our method of using derived flag schemes and derived Schur functors has the advantage of not requiring any conditions on the base space $X$ and providing explicit descriptions of Fourier--Mukai kernels in terms of derived Schur functors.

We expect the framework and approach presented in this paper to be useful in the homological study of algebraic geometry in general. Recent decades have witnessed the development of many important tools for producing derived equivalences and semiorthogonal decompositions, such as the framework of homological projective duality \cite{Kuz07,Pe19, KP21, JLX17, JL18, JL18b}, the GIT method and beyond \cite{BFK, HL15, HLS20, HL20}, the method of categorical actions \cite{Ca,CKL, H21}, and categorical wall-crossing \cite{Tod1, Tod2, Tod3, Tod4, Tod5, Tod6}.
In general, it is a difficult question to provide explicit descriptions of Fourier--Mukai kernels beyond vector bundle cases and study their properties.
Derived Schur functors (and their variants) of this paper greatly enlarge the class of objects that can be used as Fourier-Mukai kernels for which we now have good knowledge of their homological properties.

Another situation in which this framework will be useful is the study of Hecke correspondences and Hall algebras \cite{NegW, NegHecke, NegShuffle, Z21, PS}; see \cite[\S 8]{J22a} for applications of the framework of derived projectivizations. This paper's framework provides effective tools for dealing with the typically highly singular moduli spaces that arise when studying Hecke correspondences.

In a forthcoming work \cite{J22b}, we will study Abel maps for integral curves from their derived Hilbert schemes to their compactified Jacobians, using the framework of derived projectivizations \cite{J22a} and the framework of this paper. 
%This paper's framework useful in further advancing the study of Hilbert schemes and Quot schemes of objects on singular curves.

The geometry that we study in the derived version of the Borel--Weil--Bott theorem is closely related to the situation studied in Donaldson--Thomas theory \cite{RT, PT, MT, JT, BBBBJ} (see also the above-mentioned papers of Toda). For example, if $\sE$ has perfect amplitude in $[0,1]$ and the stack $X$ has a perfect obstruction theory in the sense of \cite{LT,BF}, then the derived flag schemes $\Flag_{\bdd}(\sE)$ carries a natural perfect obstruction theory (Corollary \ref{cor:dpflag:perfect.amp<=1}).

There are many further interesting questions about the structures of derived Grassmannians, such as relative exceptional sequences on derived Grassmannians (\cite{Kap85, Kap88, BLV, Ef}),  their morphisms algebras and mutation theory (see \cite[Proposition 6.19]{J22a}), analogues of Beilinson's relations (\cite{Be}, \cite[Corollary 6.21]{J22a}, and analogues of the Borel--Weil--Bott theorem for negative ranks and ``negative" line bundles (see \cite[Theorem 5.6 (2)]{J22a} for the cases of projectivizations).

We show in \cite{J21} that the underived Grassmannian framework unifies many birational geometry constructions, including blowups along local complete intersection ideals and determinantal ideals, standard flips and flops, and Grassmannian-type flips and flops; it also provides an effective approach to the degeneracy loci theory. The framework of this paper will enable us to extend these theories to the derived setting. We expect these study to be closely related to the derived birational geometry studied in \cite{KR18, He21}. 

Akin--Buchbaum--Weyman's theory of Schur complexes \cite{ABW, Wey} plays an important role in this paper and provides a chain-complex model of derived Schur functors (under certain conditions in the classical situation). The results of this paper, in turn, are beneficial for studying Schur complexes. For example, this paper provides two alternative descriptions of Schur complexes -- as complexes obtained from direct-sum decomposition formula for derived Schur functors (Remarks \ref{rmk:fib:dSchur_oplus} and \ref{rem:bSchur_k=Schur_oplus_k}) and as Koszul-type complexes obtained from geometry (Remark \ref{rmk:Koszul.v.s.Schur.complexes}) -- as well as poses many interesting questions on Schur complexes. 

The study of strict polynomial functors \cite{CL, Green, Martin} is closely related to the characteristic-free theory of Schur functors. This paper's strategy could be used to investigate the derived functors of strict polynomial functors in general.
It would be interesting to look into how the structure of the category of strict polynomial functors \cite{EH, Tou, Kr} interacts with the theory of derived Schur and Weyl functors studied in this paper.

%subsec: Notations
\subsection{Notations and Conventions}  
\label{sec:Notations} 
In order to eliminate possible confusion among different readers, we begin by listing a number of notations and conventions that are sometimes used differently in the literature:
\begin{enumerate}
	\item ({\bf Notations for Schur and Weyl functors}).  This paper uses the notations $\Schur^{\lambda}(E)$ and $\Weyl^{\lambda}(E)$ to denote the classical Schur and Weyl modules associated with a partition $\lambda$ (Definition \ref{def:SchurWeyl}), respectively, where $E$ is a finite free $R$-module. This convention agrees with \cite{FH} but is {\em different} from \cite{ABW, Wey}. Specifically, our convention is such that 
	\begin{itemize}
		\item $\Schur^{(n)}(E) = \Sym^n(E)$ is the symmetric power, $\Weyl^{(n)}(E) = \Gamma^n(E)$ is the divided power, and $\Schur^{(1^n)} (E)= \Weyl^{(1^n)}(E) = \bigwedge\nolimits^n (E)$ is the exterior power. 
			\item Our notation ``$\Schur^{\lambda}(E) = \Schur^{\lambda}_R(E)$" for Schur functors corresponds to 
			\begin{itemize}
				\item ``$L_{\lambda^t}(E)$" of Akin--Buchsbaum--Weyman \cite{ABW, Wey}, where $\lambda^t$ (also denoted $\lambda'$ or $\widetilde{\lambda})$ is the {\em conjugate} (or {\em transpose}) partition of $\lambda$.
				\item  ``$\dSchur_{\lambda}(E)$" of Fulton--Harris \cite{FH} (in the case where $R =\CC$).
				\item ``$H^0(\lambda)$" of Jantzen \cite{Jan}, ``$D_{\lambda, R}$"  of Green \cite{Green}, and ``$M(\lambda)$" of Martin \cite{Martin} (in the case where $R= \kappa$ is a field and $E=\kappa^n$).
				\end{itemize} 
			 % ``$L^{\lambda}$" of \cite{BLV}, and ``$S_{\lambda}$" of \cite{Ef}.		
			 \item Our notation ``$\Weyl^{\lambda}(E) = \Weyl^{\lambda}_R(E)$" for Weyl functors corresponds to 
			 \begin{itemize}
				\item ``$K_{\lambda}(E)$" of Akin--Buchsbaum--Weyman \cite{ABW, Wey}.
				\item  ``$\dSchur_{\lambda}(E)$" of Fulton--Harris \cite{FH} (if $R =\CC$; note that in this case $\Schur^{\lambda} = \Weyl^{\lambda}$).

				\item  ``$V(\lambda)$" of Jantzen \cite{Jan} and Martin \cite{Martin}, and ``$V_{\lambda, R}$"  of Green \cite{Green}  (in the case where $R= \kappa$ is a field and $E=\kappa^n$).
\end{itemize}
		\end{itemize}
		This paper uses $\bSchur_R^{\lambda/\mu}(\rho \colon M' \to M) = \bSchur_R^{\lambda/\mu}(\rho)$ to denote {\em Schur complex} (\cite{ABW}; see also \S \ref{sec:bSchur}). Our convention also {\em differs} from \cite{ABW, Wey}: ``$\bSchur_R^{\lambda/\mu}(\rho)$" of this paper corresponds to ``$L_{\lambda^t/\mu^t}(\rho)$" in \textit{loc. cit.}, where $\lambda^t/\mu^t$ is the conjugate of $\lambda/\mu$.
	\item ({\bf Grothendieck's convention for Grassmannian and flag schemes}). This paper follows Grothendieck's convention, so that the (classical or derived) projectivization $\PP(\sE)$, Grassmannian $\Grass_d(\sE)$ or flag scheme $\Flag_{\bdd}(\sE)$ parametrizes {\em quotients} rather than sub-objects. For example, if $X= \Spec \kappa$, where $\kappa$ is a field, and $\sV$ is a $\kappa$-vector space of dimension $n$, then $\Grass_d(\sV)$ parametrizes $d$-dimensional quotients of $\sV$, and is isomorphic to the classical Grassmannian variety $\Gr_{n-d}(n)=\Grass_{n-d, \, \mathrm{sub}}(\sV^\vee)$ which parametrizes $(n-d)$-dimensional subspaces of the {\em dual} vector space $\sV^\vee$.
	\item ({\bf Derived convention}). All the functors are assumed to be {\em derived}. If we want to use the classical functor, we will use the notations for zeroth derived functors or indicate with a subscript ``$\cl$".  For example, if $f \colon X \to Y$ is a map between schemes, $\sE$ is a sheaf on $X$ and $\sF$ is a sheaf on $Y$, then $f_*(\sE)$ and $f^* (\sF)$ denote the {\em derived} pushforward $\RR f_*(\sE)$ and the {\em derived} pullback $\LL f^*(\sF)$, respectively. We will use $\RR^0 f_*(\sE)$ to denote the classical pushforward, and $\LL^0 f^*(\sE)$ or $f_{\cl}^*(\sE)$ to denote the classical pullback.
	 Similarly, if $R$ is a commutative ring, and $M$, $N$ are (discrete) $R$-modules, then we use $M \otimes_R N$ to denote the derived tensor product $M \otimes^\LL_R N$ in the classical sense, and use $\Tor_0^R(M,N)$ or $M \otimes_R^\cl N$ to denote the classical tensor product of $R$-modules. 
	\item ({\bf Homological grading convention}).
	In the introduction, we used the cohomological grading convention for complexes. In the main body of the paper, we will use {\em homological grading convention}, i.e., we write a complex as $\sE_* = [\cdots \to \sE_{i+1} \to \sE_{i} \to \cdots]$. Setting $\sE^i = \sE_{-i}$ allows one to easily regard a homologically graded complex $\sE_*$ as a cohomologically graded complex $\sE^* = [\cdots \to \sE^{j} \to \sE^{j+1} \to \cdots]$, and vice versa. If $R$ is an ordinary commutative ring and $\sE_*$ is a complex of $R$-module, then the homotopy group $\pi_i(\sE_*)$ is canonically equivalent to the (co)homology group $H_i(\sE_*) = H^{-i}(\sE^*)$. In particular, we use $\Mod_R^\cn = (\Mod_R)_{\ge 0} $ to denote the $\infty$-category of connective complexes (i.e., complexes $\sE_*$ such that $\pi_i(\sE_*)=0$ for $i<0$), which is sometimes written as $(\Mod_R)^{\le 0}$ in the literature if cohomological grading convention is used. 	
    %This convention is different from most algebraic geometry books, but agrees with Lurie's and with topological convention. 
\end{enumerate}

We refer readers to Notation \ref{notation:Young.diagrams} for conventions about partitions. We mention that a {\em partition} means a non-increasing sequence of non-negative integers $\lambda = (\lambda_1, \ldots, \lambda_d)$. 
A {\em skew partition} $\lambda/\mu$ is a pair of partitions $\lambda$ and $\mu$ such that $\mu \subseteq \lambda$, that is, $\mu_i \le \lambda_i$ for all $i$. 

We will use the framework of $\infty$-categories developed by Lurie in \cite{HTT}; our notations and terminologies will mostly follow Lurie's in \cite{DAG, HTT, HA, SAG, kerodon}. The major difference is that this paper will refer to modules (or sheaves) that are not necessarily discrete as ``{\em complexes}" rather than ``modules" (or ``sheaves"), following the classical convention. 
 
We list some of the notations and terminologies used frequently in this paper, and refer readers to \cite[\S 1.7]{J22a} for more detailed account:

\begin{itemize}
	\item We let $\bDelta$ denote the {\em simplex category}. For each $n \geq 0$, we let $\Delta^n \in \Set_\Delta$ denote the simplicial set that represents the functor $\Hom_{\bDelta}(\blank, [n])$.
	\item We let $\shS$ denote the {\em $\infty$-category of spaces} (or equivalently, the $\infty$-category of $\infty$-groupoids), i.e., the homotopy-coherent nerve of simplicial category of Kan complexes; see \cite{HTT}.
	\item For an $\infty$-category $\shC$ and objects $C,D \in \shC$, we let $\Map_\shC(C,D) \in \shS$ denote their {\em mapping space}. For a pair of $\infty$-categories $\shC$ and $\shD$, we let $\Fun(\shC,\shD)$ denote the $\infty$-category of functors from $\shC$ to $\shD$.  We let $\shC^{\simeq}$ denote {\em core} of $\shC$, that is, the $\infty$-category obtained from $\shC$ by discarding all non-invertible morphisms.
	\item We let $\CAlgDelta$ denote the $\infty$-category of ``derived rings", that is, {\em simplicial commutative rings} (\cite[Definition 25.1.1.1]{SAG}). These are also referred to as {\em animated commutative rings}; see \cite[\S 5.1]{CS}, \cite[Appendix A]{Mao}, \cite[Appendix A]{BL22}.
	 A simplicial commutative ring $A \in \CAlgDelta$ is called {\em discrete} if $\pi_i(A) = 0$ for all $i \neq 0$.
	The (nerve of the) category of ordinary commutative rings is canonically equivalent to the full subcategory $\CAlg^{\heartsuit}$ of $\CAlgDelta$ spanned by discrete simplicial commutative rings.
	\item For a simplicial commutative ring $A \in \CAlgDelta$, we will let $\Mod_A$ denote the $\infty$-category of complexes over its underlying $\EE_{\infty}$-ring spectrum $A^{\circ}$ (\cite[Notation 7.1.1.1]{HA}). We say that a complex $M$ is {\em connective} if %its underlying spectrum is connective; that is, if 
	$\pi_i(M) = 0$ for all $i<0$, and that $M$ is {\em discrete} if $\pi_i(M) = 0$ for all $i \neq 0$.
	 We let $\Mod_A^\cn$ denote the full subcategory of $\Mod_A$ spanned by connective $A$-complexes and let  $\Mod_A^\heartsuit$ denote the full subcategory of discrete $A$-modules.
	\item By a {\em derived scheme} we mean a pair $X = (|X|, \sO_X)$ where $|X|$ is a topological space, and $\sO_X$ is an $\CAlgDelta$-valued sheaf on $|X|$, such that the following conditions are satisfied: (i) The underlying ringed space $X_\cl = (|X|, \pi_0 \sO_X)$ is a (classical) scheme, which is called the {\em underlying classical scheme of $X$}. (ii) Each of the sheaves $\pi_n(\sO_X)$ is a quasi-coherent $\pi_0(\sO_X)$-module on $X_\cl$. (iii) The structure sheaf $\sO_X$ is hypercomplete when regarded as an object of the $\infty$-topos $\Shv_\shS(X)$ in the sense of \cite[\S 6.5.2]{HTT}. 
	\item By a {\em prestack} we mean an element of the $\infty$-category $\Fun(\CAlgDelta, \shS)$, that is, a functor $X \colon \CAlgDelta \to \shS$. A map $f \colon X \to Y$ between prestacks is a natural transformation of the functors $X, Y \colon  \CAlgDelta \to \shS$. 
	\item For a prestack $X$, we let $\QCoh(X)$ denote the {\em $\infty$-category of (derived category) of quasi-coherent complexes on $X$}, let $\QCoh(X)^\cn$ and $\Perf(X)$ denote the full subcategories of $\QCoh(X)$ spanned by {\em connective complexes} and, respectively, {\em perfect complexes}. In the case where $X$ is a quasi-compact, quasi-separated scheme, the homotopy categories of $\QCoh(X)$, $\QCoh(X)^{\cn}$ and $\Perf(X)$ are the usual triangulated derived category $\Dqc(X)$ of complexes with quasi-coherent cohomology sheaves, the full subcategory $\Dqc(X)^{\le 0}$ of connective quasi-coherent complexes (that is, those complexes $\sE^*$ whose cohomology sheaves $\sH^{i}(\sE^*)=0$ for $i >0$), and the full triangulated subcategory $\D^{\perf}(X)$ of perfect complexes, respectively.
\end{itemize}

\subsection{Acknowledgment}
The author would like to thank Arend Bayer for numerous helpful discussions and suggestions throughout this project, as well as Jerzy Weyman, Richard Thomas,  Yukinobu Toda and Andrei Negu{\c{t}} for useful communications and comments on an earlier draft of this paper, and Dougal Davis and Kostya Tolmachov for helpful representation theory related discussions. The author was supported by the Engineering and Physical Sciences Research Council [EP/R034826/1], and by the ERC Consolidator grant WallCrossAG, no. 819864.

%%% sec: Classical Schur functors and Schur complexes
\newpage
\section{Classical Schur Functors and Schur Complexes}
\label{sec:Schur}
This section reviews the characteristic-free theory of Schur functors and Schur complexes developed by Akin, Buchsbaum, and Weyman \cite{ABW}, but presents it in a more functional manner that prepares us to extend the theory to the derived setting in \S \ref{sec:dSchur}.

\subsection{Classical Symmetric, Exterior and Divided Power (Co)Algebras}\label{sec:classical_sym}  
This subsection briefly reviews the classical theory of symmetric, exterior, and divided power algebras, including their coalgebra structures, which we will need later. Throughout this subsection, we assume that $R$ is a(n ordinary) commutative ring and $M$ a discrete $R$-module. 

For two discrete $R$-modules $M$ and $M'$, we let $M \otimes_R^{\cl} M' = \Tor_0^R(M,M')$ denote the classical tensor product of $M$ and $M'$ over $R$.
 
 \begin{enumerate}[leftmargin=*]
	\item (Tensor algebras) We let $\T^n(M)$ denote $n$-fold classical tensor product $M \otimes_R^\cl M \otimes_R^\cl \cdots \otimes _R^\cl M$. The classical tensor algebra of $M$ over $R$ (\cite[Chapter III, \S 5]{Bou}),  denoted by $\T^*(M)$, is the (ordinary) graded associative $R$-algebra $\T^*(M) = \bigoplus _{n \ge 0} \T^n(M)$, with the multiplication map $m_{p,q} \colon \T^p(M) \times \T^q(M) \to \T^{p+q}(M)$, $p,q \ge 0$, given by the formula:
	$$m_{p,q} (x_1 \otimes \cdots \otimes x_p, x_{p+1} \otimes \cdots \otimes x_{p+q}) = x_1 \otimes \cdots \otimes x_p \otimes x_{p+1} \otimes \cdots \otimes x_{p+q}.$$
There exists a unique $R$-linear map $\Delta \colon \T^*(M) \to \T^*(M) \otimes \T^*(M)$, called {\em comultiplication}, such that for all $x_i \in M$,
	$$\Delta(x_1 \otimes \cdots \otimes x_{n}) = \sum_{0 \le p \le n} (x_1 \otimes \cdots \otimes x_p) \otimes (x_{p+1} \otimes \cdots \otimes x_{n}).$$
%The coalgebra structure on $(\T^*(M),\Delta)$ is coassociative and has a canonical counit; see \cite[Chapter III, \S 11.1, \S 11,2]{Bou}. 

	\item (Classical symmetric algebras) The classical symmetric algebra of $M$ over $R$ (\cite[Chapter III, \S 6]{Bou}), denoted by $S^*(M)$, $S_{R}^*(M)$, or $\Sym_\cl^*(M)$, is the (ordinary) commutative $R$-algebra obtained as the quotient of the tensor algebra $\T_R(M)$ by the two-sided ideal $\foC$ generated by the elements of the form $x \otimes y - y \otimes x$ for $x, y \in M$. The algebra is naturally graded $S_R^*(M) = \bigoplus _{n \ge 0} S_R^n(M)$, where $S_R^n(M) = \T^n(M)/ (\foC \cap \T^n(M))$, where $S_R^n(M)$ is called the classical $n$th symmetric power of $M$ over $R$. The diagonal map $(x \in M) \mapsto ((x,x) \in M\oplus M)$ induces an $R$-algebra homomorphism $\Delta \colon S_R^*(M) \to S_R^*(M \oplus M) \simeq S_R^*(M) \otimes_R^\cl S_R^*(M)$, called the comultiplication map, which is the unique $R$-algebra homomorphism satisfying $\Delta(x) = x \otimes 1 + 1 \otimes x$ for all $x \in M$. Concretely, for all $x_{1}, \ldots, x_{n} \in M$, 
 	\begin{align*}
		 \Delta(x_1  \cdots  x_n)= \sum_{0 \le p \le n} \sum_{\sigma \in \foS_{p, n-p}}  (x_{\sigma(1)}  \cdots  x_{\sigma(p)}) \otimes ( x_{\sigma(p+1)}  \cdots  x_{\sigma(n)}).
 	\end{align*}
(Here, $\foS_{p,q}$ denote the subgroup $\{ \sigma \in \foS_{p+q} \mid \sigma(1) <  \cdots <\sigma(p),  \sigma(p+1) < \cdots < \sigma(p+q)\}$ of the permutation group $\foS_{p+q}$.) By abuse of notations, we also denote the induced map on graded components by $\Delta \colon S^{p+q}_R(M) \to S_R^{p}(M) \otimes_R^\cl S_R^{q}(M)$.
	\item (Classical exterior algebras) The classical exterior algebra of $M$ over $R$ (\cite[Chapter III, \S 7]{Bou}), denoted by $\bigwedge_\cl^*(M)$, or $\bigwedge_{\cl, R}^*(M)$, is the ordinary $R$-algebra defined as the quotient of the tensor algebra $\T_R(M)$ by the two-sided ideal $\foA$ generated by elements of the form $x \otimes x$ for $x \in M$. The algebra $\bigwedge_\cl^*(M) = \bigoplus _{n \ge 0} \bigwedge_\cl^n(M)$ is graded , where $\bigwedge_\cl^n(M) = \T^n(M)/ ( \foA \cap \T^n(M))$. We refer to $\bigwedge_\cl^n(M)$ as the classical $n$th exterior power of $M$ over $R$. 
	The diagonal map $(x \in M) \mapsto ((x,x) \in M\oplus M)$ induces an $R$-algebra homomorphism $\Delta' \colon \bigwedge_\cl^*(M) \to \bigwedge_\cl^*(M \oplus M) \simeq \bigwedge_\cl^*(M) \otimes_R^\cl \bigwedge_\cl^*(M)$, comultiplication map, which is the unique $R$-algebra homomorphism satisfying $\Delta'(x) = x \otimes 1 + 1 \otimes x$ for all $x \in M$. Concretely, 
 	\begin{align*}
		 \Delta'(x_1 \wedge \cdots \wedge x_n)= \sum_{0 \le p \le n} \sum_{\sigma \in \foS_{p, n-p}} \sign(\sigma) \cdot (x_{\sigma(1)} \wedge \cdots \wedge x_{\sigma(p)}) \otimes ( x_{\sigma(p+1)} \wedge \cdots \wedge x_{\sigma(n)})
 	\end{align*}
	for all $x_{1}, \ldots, x_{n} \in M$. By abuse of notations, we also denote the induced map on graded components by $\Delta' \colon \bigwedge_\cl^{p+q}(M) \to \bigwedge_\cl^p (M) \otimes_R^\cl \bigwedge_\cl^q(M)$.

	\item (Classical divided algebras) The classical divided power algebra of $M$ over $R$ (\cite[Chapter III, \S 1, p. 248]{Ro}), denoted by $\Gamma_\cl^*(M)$, $\Gamma_{\cl,R}^*(M)$, or $\DD_{\cl}^*(M)$, is quotient $R$-algebra of the free polynomial algebra $R[\{x^{(n)}\}_{n\ge 0, x \in M}]$ on symbols $\{x^{(n)}\}_{n \ge 0, x \in M}$ by the ideal generated by relations for all $x, y \in M$, $r \in R$ and all integers $m, n \ge 0$:
	\begin{align*}
		& x^{(0)} = 1 \quad (x \ne 0) & (r x)^{(n)} = r^n x^{(n)} \quad (n \ge 1)\\
		& x^{(m)} \cdot x^{(n)} = \binom{m+n}{m} x^{(m+n)} & (x+y)^{(n)} = \sum_{i=0}^n x^{(i)} y^{(n-i)}  \quad (n \ge 1).
	\end{align*}
The algebra $\Gamma_\cl^*(M) = \bigoplus_{n \ge 0} \Gamma_\cl^n(M)$ is commutative and graded; we will refer to the $R$-module $\Gamma_\cl^n (M) = \Gamma_{\cl, R}^n(M)$ as the classical $n$th  divided power of $M$ over $R$. The diagonal map $(x \in M) \mapsto ((x,x) \in M\oplus M)$ induces an $R$-algebra homomorphism $\Delta'' \colon \Gamma_\cl^*(M) \to \Gamma_\cl^*(M \oplus M) \simeq \Gamma_\cl^*(M) \otimes_R^\cl \Gamma_\cl^*(M)$, comultiplication map, which is the unique $R$-algebra homomorphism satisfying $\Delta''(x) = x \otimes 1 + 1 \otimes x$ for all $x \in M$. Concretely,
\begin{align*}
		\Delta''(x_1^{(\nu_1)} \cdots x_m^{(\nu_m)})  =\sum_{0 \le \alpha_i \le \nu_i}  (x_1^{(\alpha_1)} \cdots x_m^{(\alpha_m)}) \otimes (x_1^{(\nu_1 - \alpha_1)} \cdots x_m^{(\nu_m - \alpha_m)})
 	\end{align*}
	for all $x_1, \ldots, x_m \in M$ and $\nu_i \ge 0$, $\nu_1 + \ldots + \nu_k = n$.
By abuse of notations, we also denote the induced map on graded components by $\Delta'' \colon \Gamma_\cl^{p+q}(M) \to \Gamma_\cl^p (M) \otimes_R^\cl \Gamma_\cl^q(M)$.
\end{enumerate}

The algebra and coalgebras structures of $\T^*(M)$, $S^*(M)$, $\bigwedge\nolimits_\cl^*(M)$ and $\Gamma_\cl^*(M)$ satisfy a sequence of compatibility conditions, making them {\em Hopf algebras}. As these details will not be used in this paper, we omit them and refer readers to \cite[I.1-I.4]{ABW}.

The following constructions will be used in constructing Schur complexes (\S \ref{sec:bSchur}):

\begin{enumerate}[leftmargin=*]
	\item (Symmetric complexes) Let $\rho \colon M' \to M$ be a morphism between finite projective $R$-modules and $n \ge 0$ an integer. Then the $n$th symmetric complex of $\rho$ is the complex % %$\bS^n_R(\rho \colon M' \to M) = \bS^n_R(\rho)$ is the complex
	\begin{equation*}
	\bS^n_R(\rho \colon M' \to M) \colon \quad \bigwedge\nolimits^n_R M'  \xrightarrow{d_n}  \cdots \xrightarrow{d_2} \bigwedge\nolimits_R^1 M'  \otimes_R (S_R^{n-1} M) \xrightarrow{d_1} S_R^n M,
	\end{equation*}
	where the differentials $d_i$ are given by the formula: for all $x_1, \ldots, x_i \in M'$ and $y \in S_R^{d-i} M$,
	\begin{align*}
		d_i (x_{1} \wedge \cdots \wedge x_{i} \otimes y) = \sum_{j=1}^i (-1)^{j-1} (x_{1} \wedge \cdots  \widehat{x_{j}} \cdots \wedge x_{i} ) \otimes (\rho(x_{j}) \cdot y) \in \bigwedge\nolimits_R^{i-1} M' \otimes_R S_R^{d-i+1} M. 
	\end{align*}
	\item (Exterior complexes)
	Let $\rho \colon M' \to M$ be a morphism between finite projective $R$-module and $n \ge 0$ an integer. Then the $n$th exterior complex of $\rho$ is the complex%$\bwedge^n_R(\rho \colon M' \to M) = \bwedge^n_R(\rho)$ is the complex
	\begin{equation*}
		 \bwedge^n_R(\rho \colon M' \to M) \colon \quad \Gamma_R^n(M') \xrightarrow{d_n'}  \cdots  \xrightarrow{d_2'} \Gamma_R^1(M') \otimes_R \bigwedge\nolimits_R^{n-1}(M)   \xrightarrow{d_1'} \bigwedge\nolimits_R^{n} M,
	\end{equation*}
	where the differentials $d_j'$ are given by the formula: for all $g \in \bigwedge^{d-j} M$ and $\varepsilon_1^{(\nu_1)} \cdots \varepsilon_{m'}^{(\nu_{m'})} \in \Gamma^{j}(M')$ (where $\varepsilon_1, \ldots, \varepsilon_{m'}$ is a basis of $M'$, $\nu_i \ge 0$ and $\sum \nu_i = j$),
	$$d_j'(g \otimes \varepsilon_1^{(\nu_1)} \cdots \varepsilon_{m'}^{(\nu_{m'})}) = \sum_{i=1}^{m'}  (\rho(\varepsilon_i) \wedge g ) \otimes \varepsilon_1^{(\nu_1)}  \cdots \varepsilon_{i}^{(\nu_{i} - 1)} \cdots \varepsilon_{m'}^{(\nu_{m'})} \in  \bigwedge\nolimits_R^{d-j+1} M \otimes_R \Gamma_R^{j-1} M'.$$
\end{enumerate}

Similar to the cases of usual symmetric and exterior powers of modules, there are multiplication maps $m \colon \bS^{p}_R(\rho) \otimes_R \bS^{q} (\rho) \to \bS^{p+q}_R(\rho)$ and $m' \colon \bwedge^{p}_R(\rho) \otimes_R \bwedge^{q} (\rho) \to \bwedge^{p+q}_R(\rho)$, and comultiplication maps $\Delta \colon \bS^{p+q}_R(\rho) \to \bS^{p}_R(\rho) \otimes_R \bS^{q} (\rho)$ and $\Delta' \colon \bwedge^{p+q}_R(\rho) \to \bwedge^{p}_R(\rho) \otimes_R \bwedge^{q} (\rho)$ for symmetric and exterior complexes, which satisfy a differential graded version of  Hopf algebra conditions. The details of these constructions and properties are omitted, as they will not be used in the current paper, and readers are referred to \cite[\S V.1]{ABW} for further details.

% subsec: Classical Schur Weyl
\subsection{Classical Schur Functors and Weyl Functors}
\label{sec:CSchurWeyl}
In order to present the theory of Schur and Weyl functors in a more functorial way, the following notations will be used:

\begin{notation}
\label{notation:CRingMod}
\begin{enumerate}[leftmargin=*]
	\item We let ${\rm CRing}$ denote the (ordinary) category of (ordinary) commutative rings. For any commutative ring $R \in {\rm CRing}$, we let $\Mod_R^\heartsuit$ denote the abelian category of (discrete) $R$-modules. 
	\item  We let ${\rm CRingMod}^\heartsuit$ the (ordinary) category of pairs $(R, M)$, where $R$ is a commutative ring and $M$ is a (discrete) $R$-module.  A morphism $(R,M) \to (R',M')$ in ${\rm CRingMod}$ is given by a  pair $(R \to R', M \to M')$, where $R \to R'$ is a morphism of commutative rings, and $M \to M'$ is a morphism of $R$-modules. Then the natural forgetful functor
		$$U \colon {\rm CRingMod}^\heartsuit \to {\rm CRing}, \qquad (R,M) \mapsto R$$
	is a {Grothendieck opfibration} (or called a coCartesian fibration; see \cite[\href{https://kerodon.net/tag/01RN}{Tag 01RN}]{kerodon}) which classifies the functor $\Mod_{\bullet} \colon (R \in {\rm CRing}) \mapsto (\Mod_R^\heartsuit \in \Cat)$. A morphism $(R, M) \to (R',M')$ is $U$-coCartesian if and only if the natural map $R' \otimes_R M \to M'$ is an isomorphism of $R'$-modules. 
	\item If $R$ is a commutative ring, we consider the following sequence of subcategories:
	$$\Mod_R^{\rm ff} \subseteq \Mod_R^{\rm fproj} %\subseteq \Mod_R^{\rm proj} 
	\subseteq \Mod_R^{\flat} \subseteq \Mod_R^\heartsuit,$$
where $\Mod_R^{\rm ff}$ (resp. $\Mod_R^{\rm fproj}$, resp. $\Mod_R^{\flat}$) denote the full subcategory of $\Mod_R^\heartsuit$ spanned by {\em finite free} (resp. {\em finite projective}, resp. {\em flat}) $R$-modules. 
	\item For any $? \in \{{\rm ff}, {\rm fproj}, \flat\}$, we let ${\rm CRingMod}^{?}$ denote the full subcategory of ${\rm CRingMod}^\heartsuit$ spanned by those $(R, M)$ where $R \in {\rm CRing}$ and $M \in \Mod_R^?$. Since the subcategory $\Mod_R^?$ stable under base change, the restriction of the forgetful functor
		$$U^? \colon {\rm CRingMod}^{?} \to {\rm CRing}, \qquad (R,M) \mapsto R$$
	is also a {Grothendieck opfibration}, and a morphism $(R, M) \to (R',M')$ is $U^?$-coCartesian if and only if the natural map $R' \otimes_R M \to M'$ is an isomorphism of $R'$-modules. 
\end{enumerate}
\end{notation}

\begin{notation} 
\label{notation:Young.diagrams}
A {\em partition} is a sequence of integers $\lambda = (\lambda_1, \lambda_2, \ldots, \lambda_\ell)$ such that $\lambda_1 \ge \lambda_2 \ge \cdots \ge \lambda_\ell \ge 0$. We denote $|\lambda| = \sum \lambda_i$. If $n = |\lambda|$, we say that $\lambda$ is a {\em partition of $n$}. For a partition $\lambda$, we let $\lambda^{t} = (\lambda_1^t, \lambda_2^t, \ldots, \lambda_s^t)$ dente the partition defined in the following way: for any $i$, $\lambda^t_i$ is the number of $j$s such that $\lambda_j \ge i$, and refer to $\lambda^t$ as the {\em conjugate partition} (or {\em transpose}) of $\lambda$. For two partitions $\mu$ and $\lambda$, we write $\mu \subseteq \lambda$ if $\mu_i \le \lambda_i$ for any $i \ge 0$. A {\em skew partition} $\lambda/\mu$ is a pair of partitions $\lambda$ and $\mu$ such that $\mu \subseteq \lambda$. By convention, we regard a usual partition $\lambda$ as a skew partition $\lambda/(0)$.

We let $\mathrm{Y}(\lambda)\subset \ZZ_{\ge 0}^2$ denote the {\em Young diagram} of $\lambda$, that is, $\mathrm{Y}(\lambda) = \{ (i,j) \in \ZZ_{\ge 0}^2 \mid 1 \le j \le \lambda_i \text{~~for all $i$} \}$. Similarly, for skew partitions $\lambda/\mu$, its corresponding Young diagram is $\mathrm{Y}(\lambda/\mu) = \{(i,j) \in \NN^2 \mid \mu_i +1 \le j \le \lambda_i \text{~~for all $i \in [1, \lambda_1^t]$} \}.$ Notice that $(i,j) \in \mathrm{Y}(\lambda/\mu)$ if and only if $(j,i) \in \mathrm{Y}(\lambda^t/\mu^t)$, if and only if $\mu_j^t +1 \le i \le \lambda_j^t$ for all $j \in [1, \lambda_1]$. 

There are different (but equivalent) conventions to depict the Young diagram of a skew partition via a picture consisting of ``boxes". The convention that the author has in mind is the following: for a skew partition $\lambda/\mu$, we let $A(\lambda/\mu)= (a_{ij})$ denote the matrix which satisfies $a_{ij} = 1$ if $(i,j) \in \mathrm{Y}(\lambda/\mu)$ and $a_{ij}=0$ otherwise. Notice that the matrix $A(\lambda^t/\mu^t)$ for $\lambda^t/\mu^t$ is precisely the {\em transpose} of the matrix $A(\lambda/\mu)$ for $\lambda/\mu$. Next, instead of drawing the matrix $A(\lambda/\mu)$, we draw a box $\square$ at the place where $a_{ij}=1$ and leave the places where $a_{ij}=0$ empty.  For example, if $\mu=(1,1)$, $\lambda=(2,2,1)$, then $\mu^t=(2)$, $\lambda^t = (3,2)$, and the corresponding matrices and Young diagrams for $\lambda/\mu$ and $\lambda^t/\mu^t$ are depicted as:
$$
A(\lambda/\mu) = 
\begin{pmatrix}
0 & 1 \\
0 & 1 \\
1 & 0  
\end{pmatrix}
\quad \mathrm{Y}(\lambda/\mu) = 
\begin{ytableau}
 \none &   \\
 \none & \\
~ & \none
\end{ytableau} 
\qquad
A(\lambda^t/\mu^t) = 
\begin{pmatrix}
0 & 0 & 1\\
1& 1 & 0
\end{pmatrix}
\quad
\mathrm{Y}(\lambda^t/\mu^t)= \begin{ytableau}
\none &\none  & \\
~& 
\end{ytableau}
\, \,.
$$
\end{notation}

\begin{notation} [{Slight Variant of \cite[\S V.2]{ABW}}]
\label{notation:square}
Let $\lambda/\mu$ be a skew partition, $R \in {\rm CRing}$ a commutative ring, and $M \in \Mod_R^\heartsuit$ a discrete $R$-module. 
\begin{enumerate}[leftmargin=*]
	\item (See \cite[beginning of \S 2.1 \& Proposition (2.1.9)]{Wey}) We let $p_j: = \lambda_j^t - \mu_j^t$, $s: = \max\{j \mid p_j \ne 0\}$, and let 
%$\bigwedge^{\lambda^t/\mu^t}$ to be the functor:
	$$\bigwedge\nolimits_{\cl}^{\lambda^t/\mu^t} (M) := \bigwedge\nolimits_{\cl}^{p_1} M \otimes_R^{\cl}\cdots \otimes_R^{\cl}\bigwedge\nolimits_{\cl}^{p_j} M \otimes_R^{\cl}\cdots \otimes_R^{\cl}\bigwedge\nolimits_{\cl}^{p_s} M,$$
	\begin{align*}
		\widetilde{\bigwedge\nolimits}_{\cl}^{\lambda^t/\mu^t}(M) := \bigoplus_{j=1}^{q-1} 
		\underbrace{\big(
		\bigwedge\nolimits_{\cl}^{p_1} M \otimes_R^{\cl}\cdots \otimes_R^{\cl}\bigwedge\nolimits_{\cl}^{p_{j-1}} M\big)}_{j-1}
		 \otimes_R^{\cl}\, R^{\lambda^t/\mu^t}_{j}(M)
			 \otimes_R^{\cl}
			 \underbrace{\big(\bigwedge\nolimits_{\cl}^{p_{j+2}} M \otimes_R^{\cl}\cdots \otimes_R^{\cl}\bigwedge\nolimits_{\cl}^{p_{s}}M\big)}_{s-j}. 
	\end{align*}
Here, $R^{\lambda^t/\mu^t}_{j}(M) : = \bigoplus_{\substack{u, v \ge 0,  \\ u+v < \lambda_{j+1}^t - \mu^t_{j}}}  R^{\lambda^t/\mu^t}_{u,v;j}(M)$ is an auxiliary module, where 
	\begin{align*}
	R^{\lambda^t/\mu^t}_{u,v;j}(M) : = \bigwedge\nolimits_{\cl}^{u} M \otimes_R^{\cl}\bigwedge\nolimits_{\cl}^{p_j+p_{j+1}-u-v} M  \otimes_R^{\cl}\bigwedge\nolimits_{\cl}^{v} M .
	\end{align*}			 
Next, we define a natural transformation $\square^{\lambda^t/\mu^t} \colon \widetilde{\bigwedge\nolimits}_\cl^{\lambda^t/\mu^t} \to  \bigwedge_\cl^{\lambda^t/\mu^t}$ as follows. If $s<2$, we set $\widetilde{\bigwedge\nolimits}_{\cl}^{\lambda^t/\mu^t}(M) =0$. Now we assume $s \ge 2$. 
For any $j \in [1, s-1]$, and any pair of integers $(u, v)$ such that $u, v \ge 0$ and $u+v < \lambda_{j+1}^t - \mu_j^t$, we let $\square_{\lambda^t/\mu^t;u,v;j}$ denote the composition
	$$R^{\lambda^t/\mu^t}_{u,v;j}(M)  \xrightarrow{1 \otimes \Delta' \otimes 1}  \bigwedge\nolimits_{\cl}^{u} M  \otimes_R^{\cl}\bigwedge\nolimits_{\cl}^{p_{j}-u} M  \otimes_R^{\cl}\bigwedge\nolimits_{\cl}^{p_{j+1}-v} M  \otimes_R^{\cl}\bigwedge\nolimits_{\cl}^{v} M  \xrightarrow{m' \otimes m'} \bigwedge\nolimits_{\cl}^{p_j} M  \otimes_R^{\cl}\bigwedge\nolimits_{\cl}^{p_{j+1}} M,$$
where $\Delta'$ and $m'$ are the comultiplication and multiplication map for classical exterior algebras. We then let $\square_{\lambda^t/\mu^t; j}$ denote the sum of maps
	$$\square_{\lambda^t/\mu^t; j} := \sum_{\substack{u, v \ge 0,  \\ u+v < \lambda_{j+1}^t - \mu^t_{j}}}  \square_{\lambda^t/\mu^t;u,v;j}\colon   \quad
	R^{\lambda^t/\mu^t}_{j}(M) \to \bigwedge\nolimits_{\cl}^{p_j} M  \otimes_R^{\cl}\bigwedge\nolimits_{\cl}^{p_{j+1}} M.$$
Finally, we define $\square_{\lambda^t/\mu^t} = \square_{\lambda^t/\mu^t}(R;M) \colon \widetilde{\bigwedge}^{\lambda^t/\mu^t}_\cl(M) \to \bigwedge^{\lambda^t/\mu^t}_\cl(M)$ to be the sum of maps:
	$$\square_{\lambda^t/\mu^t}(R;M) := \sum_{j=1}^{s-1}~ \underbrace{1 \otimes_R^{\cl}\cdots \otimes_R^{\cl}1}_{j-1} \otimes_R^{\cl}\, \square_{\lambda^t/\mu^t;j} \otimes_R^{\cl}\underbrace{1 \otimes_R^{\cl}\cdots \otimes_R^{\cl}1}_{s-j} \colon \widetilde{\bigwedge\nolimits}_\cl^{\lambda^t/\mu^t} (M) \to \bigwedge\nolimits_\cl^{\lambda^t/\mu^t}(M).$$
	\item (See \cite[Proposition 2.1.15]{Wey}) Similarly, we let $q_i : = \lambda_i - \mu_i$ and $\ell : \max{i \mid q_i \ne 0}$.
	We let 
		$$\Gamma_{\cl}^{\lambda/\mu} (M) := \Gamma_{\cl}^{q_1} (M) \otimes_R^{\cl}\cdots \otimes_R^{\cl}\Gamma_{\cl}^{q_j} (M) \otimes_R^{\cl}\cdots \otimes_R^{\cl}\Gamma_{\cl}^{q_\ell} (M),$$
	\begin{align*}
		\widetilde{\Gamma}_{\cl}^{\lambda/\mu}(M) := \bigoplus_{i=1}^{\ell-1} 
		\underbrace{\big(
		\Gamma_{\cl}^{q_1} (M) \otimes_R^{\cl}\cdots \otimes_R^{\cl}\Gamma_{\cl}^{q_{i-1}} (M)\big)}_{i-1}
		 \otimes_R^{\cl}\, U^{\lambda/\mu}_{i}(M)
			 \otimes_R^{\cl}
			 \underbrace{\big(\Gamma_{\cl}^{q_{i+2}} (M) \otimes_R^{\cl}\cdots \otimes_R^{\cl}\Gamma_{\cl}^{q_{\ell}}(M)\big)}_{s-i}. 
	\end{align*}
Here, $U^{\lambda/\mu}_{i}(M) : = \bigoplus_{\substack{u, v \ge 0,  \\ u+v < \lambda_{i+1} - \mu^t_{i}}}  U^{\lambda/\mu}_{u,v;i}(M)$ is an auxiliary module, where 
	\begin{align*}
	U^{\lambda/\mu}_{u,v;i}(M) : = \Gamma_{\cl}^{u} (M) \otimes_R^{\cl}\Gamma_{\cl}^{q_i+q_{i+1}-u-v} (M)  \otimes_R^{\cl}\Gamma_{\cl}^{v} (M) .
	\end{align*}			 
Next, we define a natural transformation $\square_{\lambda/\mu}' \colon \widetilde{\Gamma}_\cl^{\lambda/\mu} \to  \Gamma_\cl^{\lambda/\mu}$ as follows. If $\ell <2$, we set $\widetilde{\Gamma}_{\cl}^{\lambda/\mu}(M) =0$. Now we  assume $q \ge 2$. 
For any $i \in [1, \ell]$, and any pair of integers $(u, v)$ such that $u, v \ge 0$ and $u+v < \lambda_{i+1} - \mu_i$, we let $\square_{\lambda/\mu; u,v;i}'$ denote the composition
	$$U^{\lambda/\mu}_{u,v;i}(M)  \xrightarrow{1 \otimes \Delta'' \otimes 1}  \Gamma_{\cl}^{u} (M)  \otimes_R^{\cl}\Gamma_{\cl}^{q_{i}-u} (M)  \otimes_R^{\cl}\Gamma_{\cl}^{q_{i+1}-v} (M)  \otimes_R^{\cl}\Gamma_{\cl}^{v} (M)  \xrightarrow{m'' \otimes m''} \Gamma_{\cl}^{q_i} (M)  \otimes_R^{\cl}\Gamma_{\cl}^{q_{i+1}} (M),$$
where $\Delta''$ and $m''$ are the comultiplication and multiplication map for classical divided power algebras, respectively. We then let $\square'_{\lambda/\mu;i}$ denote the sum of maps
	$$\square'_{\lambda/\mu;i} := \sum_{\substack{u, v \ge 0,  \\ u+v < \lambda_{i+1} - \mu_{i}}}  \square'_{\lambda/\mu; u,v;i} \colon   \quad
	U^{\lambda/\mu}_{i}(M) \to \Gamma_{\cl}^{q_i} (M)  \otimes_R^{\cl}\Gamma_{\cl}^{q_{i+1}} (M).$$
Finally, we define $\square'_{\lambda/\mu} = \square'_{\lambda/\mu}(R;M) \colon \widetilde{\Gamma}^{\lambda/\mu}_\cl(M) \to \Gamma^{\lambda/\mu}_\cl(M)$ to be the sum of maps:
	$$\square'_{\lambda/\mu}(R;M) := \sum_{i=1}^{\ell-1}~ \underbrace{1 \otimes_R^{\cl}\cdots \otimes_R^{\cl}1}_{i-1} \otimes_R^{\cl}\, \square'_{\lambda/\mu;i} \otimes_R^{\cl}\underbrace{1 \otimes_R^{\cl}\cdots \otimes_R^{\cl}1}_{\ell-i} \colon \widetilde{\Gamma}_\cl^{\lambda/\mu} (M) \to \Gamma_\cl^{\lambda/\mu}(M).$$
	\end{enumerate}
All the above expressions can be simplified by dropping the suffix ``cl" (= ``classical") if $M$ is {\em flat} over $R$; we will almost only use the above definitions in this case.
\end{notation}

% Definition: Classical Schur and Weyl
\begin{definition}[{Classical Schur and Weyl Functors; see \cite[V.2]{ABW}}]
\label{def:SchurWeyl}
Let $\lambda/\mu$ be a skew partition. Consider the following functors:
	$$\Schur^{\lambda/\mu} \colon {\rm CRingMod}^\heartsuit \to {\rm CRingMod}^\heartsuit
	\qquad 
	(R, M) \mapsto \big(R, \Schur^{\lambda/\mu}_R(M) : = \Coker (\square_{\lambda^t/\mu^t}(R; M)) \big).$$
	$$\Weyl^{\lambda/\mu} \colon {\rm CRingMod}^\heartsuit \to {\rm CRingMod}^\heartsuit
	\qquad 
	(R, M) \mapsto \big(R, \Weyl^{\lambda/\mu}_R(M) : = \Coker(\square'_{\lambda/\mu}(R; M)) \big).$$
Here, the maps $\square_{\lambda^t/\mu^t}(R; M)$ and $\square'_{\lambda/\mu}(R; M)$ are defined in Notation \ref{notation:square}.  We refer to $\Schur^{\lambda/\mu}$ and $\Weyl^{\lambda/\mu}$ as the (classical) {\em Schur functor} and the {\em Weyl functor} (associated with $\lambda/\mu$), respectively. 
For any $(R, M) \in {\rm CRingMod}^\heartsuit$, $\Schur^{\lambda/\mu}_R(M)$ and $\Weyl^{\lambda/\mu}_R(M)$ are called the (classical) {\em Schur module} and {\em Weyl module} (associated with $\lambda/\mu$), respectively. By convention, we set $\Schur_R^{\lambda/\mu}(M) = \Weyl_R^{\lambda/\mu}(M) =R$ if $\lambda=\mu$ and $M \neq 0$, and $\Schur_R^{\lambda/\mu}(M) =\Weyl_R^{\lambda/\mu}(M)=0$ if $M=0$. 
\end{definition}

\begin{remark}
In the preceding definition, we set up the framework of classical Schur and Weyl functors for general (discrete) modules $M$, but this is purely for theoretical purposes; in practice, the classical Schur and Weyl functors only have a satisfactory theory for flat modules. When the modules $M$ are not flat, the derived Schur and Weyl functors of \S \ref{sec:dSchurdWeyl} should be used, along with classical criteria Proposition \ref{prop:dSchur:classical} and Remark \ref{remark:classical.criteria}, instead of the classical ones.
\end{remark}

% Examples of classical Schur
\begin{example} Let $n$ be a positive integer.
\begin{enumerate}[label=(\roman*), leftmargin=*]
	\item 
        If $\lambda = (n)$, $\mu=(0)$, then $\Schur_R^{(n)}(M)=\Sym_{\cl, R}^n (M)$ is the {\em $n$th (classical) symmetric power} functor, and $\Weyl_R^{(n)}(M)=\Gamma_{\cl, R}^n (M))$ is the {\em $n$th (classical) divided power} functor. 
	\item 
        If $\lambda = (1^n) := \underbrace{(1, \ldots, 1)}_{n \,\text{terms}}$, $\mu=(0)$, then $\Schur_R^{(1^n)}(M) = \Weyl_R^{(1^n)}(M) =\bigwedge_{\cl, R}^n (M)$ are both equal to the {\em $n$th (classical) exterior power} functor.
	\item 
        If $\lambda= (n,n-1, \ldots, 1)$ and $\mu= (n-1, n-2, \ldots, 1)$, then $\Schur_R^{\lambda/\mu}(M) = \Weyl_R^{\lambda/\mu}(M) =\bigotimes_{\cl, R}^n (M)$ are both equal to the {\em $n$th (classical) tensor product} functor.
	\item 
        If $\lambda = (2,1)$ and $\mu=(0)$, then for any commutative ring $R$ and any (discrete) $R$-module $M$, the Schur module $\Schur_R^{(2,1)} (M)$ is defined by the exact sequence
		\begin{align*}
\bigwedge\nolimits_\cl^3(M) \xrightarrow{\psi} \bigwedge\nolimits_\cl^2(M) \otimes_R^{\cl}M \to \Schur^{(2,1)}_R(M) \to 0,
		\end{align*}
where $\psi=\square_{(2,1)^t} = \Delta \colon \bigwedge\nolimits_\cl^3(M) \to \bigwedge\nolimits_\cl^2(M) \otimes_R^{\cl}M$ is the comultiplication map of the exterior power algebra, given by the formula: for any $x, y, z \in M$,
		$$\psi(x \wedge y \wedge z) = (x \wedge y) \otimes z - (x \wedge z) \otimes y + (y \wedge z) \otimes x.$$
	Similarly, the Weyl module $\Weyl^{(2,1)} (M)$ is defined by the exact sequence
		\begin{align*}
			\Gamma_\cl^3(M) \xrightarrow{\psi' } \Gamma_\cl^2(M) \otimes_R^{\cl}M \to \Weyl^{(2,1)}_R(M) \to 0,
	\end{align*}
where $\psi'=\square'_{(2,1)} = \Delta'' \colon \Gamma_\cl^3(M) \to \Gamma_\cl^2(M) \otimes_R^{\cl}M$ is the comultiplication map of the divided power algebra, given by the formula: for any $x, y, z \in M$,
		\begin{align*}
			\psi'(x^{(3)}) = x^{(2)} \otimes x, \quad
			\psi'(x^{(2)}y) = x^{(2)} \otimes y + xy \otimes x, \quad
			\psi'(x y z) = x y \otimes z + xz \otimes y + yz \otimes x.
		\end{align*}
	\end{enumerate}
\end{example}

% remark: images
\begin{remark}[Schur and Weyl Functors as Images]
\label{rem:Schur.Weyl.as.images}
Let $\lambda/\mu$ be a skew partition, and $R \in {\rm CRing}$. Assume that $M$ is a {\em finite projective} $R$-module. We introduce notations
	\begin{align*}
	&\bigwedge\nolimits_R^{\lambda/\mu} (M): = \bigotimes_{i \in [1, \lambda_1]} \bigwedge\nolimits_R^{\lambda_i - \mu_i} (M) 
	&  
	\Sym_R^{\lambda/\mu} (M): = \bigotimes_{i \in [1, \lambda_1]} \Sym_R^{\lambda_i - \mu_i} (M) \\
	&\Gamma_R^{\lambda/\mu} (M): = \bigotimes_{i \in [1, \lambda_1]} \Gamma_R^{\lambda_i - \mu_i} (M) 
	&
	\bigotimes\nolimits_R^{\lambda/\mu} (M): = \bigotimes_{i \in [1, \lambda_1]} M^{\otimes (\lambda_i - \mu_i)}.
	\end{align*}
Then in this case, one can also express $\Schur_R^{\lambda/\mu}(M)$ and $\Weyl_R^{\lambda/\mu}(M)$ via the formulae:
$$\Schur_R^{\lambda/\mu}(M) = {\rm image} \Big(\bigwedge\nolimits_R^{\lambda^t/\mu^t} (M) \xrightarrow{\otimes_{j \in [1,\lambda_1]} \Delta_j'} \bigotimes\nolimits^{\lambda^t/\mu^t}_R(M) = \bigotimes\nolimits_R^{\lambda/\mu} (M) \xrightarrow{\otimes_{i \in [1, \lambda_1^t]} m_i} \Sym_R^{\lambda/\mu} (M) \Big).$$
 $$\Weyl_R^{\lambda/\mu}(M) = {\rm image} \Big(\Gamma_R^{\lambda/\mu} (M) \xrightarrow{\otimes_{i\in[1,\lambda_1^t]} \Delta_i''} 
 \bigotimes\nolimits_R^{\lambda/\mu} (M)  = \bigotimes\nolimits^{\lambda^t/\mu^t}_R(M) 
 \xrightarrow{\otimes_{j \in [1, \lambda_1]} m_j'} \bigwedge\nolimits_R^{\lambda^t/\mu^t} (M).$$
Here, $\Delta_j', \Delta_i''$ are the comultiplications of exterior and divided power algebras, and $m_i, m_j'$ are the canonical multiplication maps of symmetric and exterior algebras; see \cite{ABW} or \cite{Wey}.
\end{remark}

% Theorem: universal freeness
\begin{theorem}[Universal Freeness; {\cite{ABW}}]
\label{thm:Schur:free}
\begin{enumerate}
	\item 
	\label{thm:Schur:free-1}
	The Schur functor $\Schur^{\lambda/\mu}$ and Weyl functor $\Weyl^{\lambda/\mu}$ of Definition \ref{def:SchurWeyl} restrict to functors
		$$\Schur^{\lambda/\mu} \colon {\rm CRingMod}^? \to {\rm CRingMod}^? \quad \text{and} \quad \Weyl^{\lambda/\mu} \colon {\rm CRingMod}^? \to {\rm CRingMod}^?
 ,$$  
 	respectively, for any $? \in \{{\rm ff}, {\rm fproj}, \flat\}$. In other words, for any $R \in {\rm CRing}$ and $M \in \Mod_R^\heartsuit$, if $M$ is a finite free (resp. finite projective, resp. flat) $R$-module, then $\Schur^{\lambda/\mu}_R(M)$ and $\Weyl^{\lambda/\mu}_R(M)$ are finite free (resp. finite projective, resp. flat) $R$-modules. 
	\item 
	\label{thm:Schur:free-2}
	For any $? \in \{{\rm ff}, {\rm fproj}, \flat\}$, the following diagrams commute:
		$$
\begin{tikzcd}
		{\rm CRingMod}^? \ar{r}{\Schur^{\lambda/\mu}} \ar{d}{U^?} & {\rm CRingMod}^? \ar{d}{U^?} \\
		{\rm CRing} \ar{r}{\id} & {\rm CRing}
	\end{tikzcd}
	\quad
	\begin{tikzcd}
		{\rm CRingMod}^? \ar{r}{\Weyl^{\lambda/\mu}} \ar{d}{U^?} & {\rm CRingMod}^? \ar{d}{U^?} \\
		{\rm CRing} \ar{r}{\id} & {\rm CRing}.
	\end{tikzcd}
	$$
Moreover, $\Schur^{\lambda/\mu}$ and $\Weyl^{\lambda/\mu}$ carry $U^?$-coCartesian morphisms to $U^?$-coCartesian morphisms. In other words, for any base change of commutative rings $R \to R'$ and any finite free (resp. finite projective, resp. flat) $R$-module $M$, the canonical maps
		$$\alpha_{M} \colon R' \otimes_R \Schur_R^{\lambda/\mu}(M) \to \Schur_{R'}^{\lambda/\mu}(R' \otimes_R M) \qquad \alpha_{M}' \colon R' \otimes_R \Weyl_{R}^{\lambda/\mu}(M) \to \Weyl_{R'}^{\lambda/\mu}(R' \otimes_R M)$$
are isomorphisms of finite free (resp. finite projective, resp. flat) $R'$-modules.
\end{enumerate}
\end{theorem}

In what follows, we will refer to the properties of assertion \eqref{thm:Schur:free-2}
 as ``the formation of Schur and Weyl functors commutes with base change of commutative rings".

\begin{proof}
In the case where $M$ is a finite free $R$-module, this theorem is a reformulation of \cite[Theorem II.2.16 \& Theorem II.3.16]{ABW}, which also implies the theorem in the case where $M$ is a finite projective $R$-module. %, and the assertions follow from the case where $M$ is a finite free module. 
If $M$ is a flat $R$-module, then by Lazard's theorem (see, for example, \cite[Theorem 7.2.2.15]{HA}), we could write $M$ as a filtered colimit of ﬁnite free $R$-modules. Observe that, for any fixed $R$, the formation of the functors
	$$\square_{\lambda^t/\mu^t}(R; \blank) \colon \widetilde{\bigwedge\nolimits}_\cl^{\lambda^t/\mu^t} (\blank) \to \bigwedge\nolimits_\cl^{\lambda^t/\mu^t}(\blank) \quad \text{and} \quad  \square'_{\lambda/\mu}(R;\blank)  \colon \widetilde{\Gamma}_\cl^{\lambda/\mu} (\blank) \to \Gamma_\cl^{\lambda/\mu}(\blank)$$
 of Notation \ref{notation:square} commutes with filtered colimits of $R$-modules. Since cokernel preserves small colimits, for any fixed $R$, the functors $\Schur_R^{\lambda/\mu}(\blank)$ and $\Weyl_R^{\lambda/\mu}(\blank)$ defined in Definition \ref{def:SchurWeyl} preserve filtered colimits of $R$-modules. Therefore, %from the assertions for finite free modules, we obtain that  
$\Schur_R^{\lambda/\mu}(M)$ and $\Weyl_R^{\lambda/\mu}(M)$ are filtered colimits of finite free $R$-modules, hence flat by Lazard's theorem.
Moreover, since the formation of $\alpha_M$ and $\alpha_M'$ commutes with  filtered colimits of the $R$-module $M$, the assertions that they are isomorphisms in the case where $M$ is flat follow from the case where $M$ is finite free. 
\end{proof}

% remark: Duality
\begin{remark}[Duality] Let $\lambda/\mu$ be a skew partition, and $R \in {\rm CRing}$. If $M$ is a {finite projective} $R$-module, then there is a canonical isomorphism $\Schur^{\lambda/\mu}(M^\vee) \simeq \Weyl^{\lambda/\mu}(M)^\vee$ (\cite[Proposition II.4.1]{ABW}), where $(\blank)^\vee = \Hom_R(\blank, R) \colon (\Mod_R^{\rm fproj})^{\rm op} \xrightarrow{\sim} \Mod_R^{\rm fproj}$ denotes the functor of taking dual of a finite projective $R$-module.
\end{remark}

% sec: filtrations
\subsection{Filtrations Associated with Schur and Weyl Functors}
\label{sec:univ.fil:Schur}
This subsection reviews the classical theory of canonical filtrations associated with Schur and Weyl functors \cite{ABW, Wey, Kou, Bo1, Bo2, BB88, Wang}, presented in a manner that allows direct generalization to the derived setting in \S \ref{sec:univ.fib:dSchur}. Our exposition is mostly based on Kouwenhoven's \cite{Kou}.

Recall that the lexicographic order $<$ for partitions is defined as follows: if $\lambda$ and $\mu$ are two partitions of $n$, let $i$ be the smallest integer such that $\lambda_i \ne \mu_i$, then $\lambda < \mu$ in the lexicographic order if and only if $\lambda_i < \mu_i$. For example, $(1,1,1) < (2,1) < (3)$.

%Recall that the (graded) lexicographic order $<$ for partitions is defined as follows: given two partitions $\lambda$ and $\mu$, we set $\lambda < \mu$ if $|\lambda| < |\mu|$. If $|\lambda| = |\mu|$, we let $i$ be the smallest integer such that $\lambda_i \ne \mu_i$, then $\lambda < \mu$ in the lexicographic order if and only if $\lambda_i < \mu_i$. For example, $(1) < (1,1) < (2) < (1,1,1) < (2,1) < (3).$

% Theorem Cauchy decomposition
\begin{theorem}[{Cauchy Decomposition Formula; \cite[Theorems III.1.4 \& III.2.4]{ABW}, \cite[Propositions 2.3 \& 2.4]{Kou}}] 
\label{thm:fil:sym_otimes}
Let $n$ be a positive integer, and let 
	$$\lambda^{0} = (1^n) < \lambda^{1} \cdots < \lambda^{p-2} < \lambda^{p-1}=(n)$$
be all partitions of $n$ in lexicographic order, where $p = p(n)$ is number of partitions of $n$. 
\begin{enumerate}
	\item	 
	\label{thm:fil:sym_otimes-1}
	For any $? \in \{\mathrm{ff}, \mathrm{fproj}\}$, the functor 
		$$\Sym_R^n(\blank \otimes_R \blank) \colon {\rm Mod}_R^? \times {\rm Mod}_R^? \to {\rm Mod}_R^?$$
	admits a canonical filtration by subfunctors
		$$0 = F_R^{-1} \subset F_R^{0} \subset F_R^{1} \subset \cdots \subset F_R^{p-2} \subset F_R^{p-1} =  \Sym_R^n(\blank \otimes_R \blank)$$
	for which there is a canonical equivalence for each $1 \le i \le p$:
		$$a_R^i \colon \Schur^{\lambda^i} (\blank)\otimes \Schur^{\lambda^i} (\blank) \xrightarrow{\sim} F_R^{i} / F_R^{i-1} .$$
	Moreover, the formations of the filtrations $F_R^i$ and equivalences $a_R^i$ commute with base change of commutative rings $R$. 
	\item  
	\label{thm:fil:sym_otimes-2}
	For any $? \in \{\mathrm{ff}, \mathrm{fproj}\}$, the functor 		$$\bigwedge\nolimits_R^n(\blank \otimes_R \blank) \colon {\rm Mod}_R^?  \times {\rm Mod}_R^? \to {\rm Mod}_R^?$$
	admits a canonical filtration by subfunctors
		$$0 = G_R^{-1} \subset G_R^{0} \subset G_R^{1} \subset \cdots \subset G_R^{p-2} \subset G_R^{1} =  \bigwedge\nolimits_R^n(\blank \otimes_R \blank)$$
	for which there is an canonical equivalence for each $1 \le i\le p$:
		$$b_R^i \colon \Schur^{\lambda^i} (\blank)\otimes \Weyl^{(\lambda^i)^t} (\blank) \xrightarrow{\sim} G_R^{i} / G_R^{i-1} .$$
	Moreover, the formations of the filtrations $G_R^i$ and equivalences $b_R^i$ commute with base change of commutative rings $R$. 
\end{enumerate}
\end{theorem}

\begin{example}
Let $n=3$, then we have $p=3$, and
	$$\lambda^{0}=(1,1,1) < \lambda^{1}=(2,1) < \lambda^{2}=(3)$$
are all the partitions of $3$. Then we have canonical filtrations with subquotients: 
			\begin{equation*}
				\begin{tikzcd} [back line/.style={dashed}, row sep=1 em, column sep= -.2 em]
	0 \ar{rr} 	& 	& F_R^{0} \ar[hook]{rr} \ar[equal]{ld}		&		&   F_R^{1} \ar[two heads]{ld} \ar[hook]{rr} 	&	&F_R^{2} = \Sym_R^3(\blank \otimes \blank)  \ar[two heads]{ld}. \\
							& \bigwedge\nolimits_R^3(\blank) \otimes \bigwedge\nolimits_R^3(\blank) \ar[dashed]{lu} &	& \Schur_R^{(2,1)}(\blank) \otimes \Schur_R^{(2,1)}(\blank)  \ar[dashed]{lu}  & & \Sym_R^3(\blank) \otimes \Sym_R^3(\blank)  \ar[dashed]{lu}
				\end{tikzcd}
			\end{equation*}		
			\begin{equation*}
				\begin{tikzcd} [back line/.style={dashed}, row sep=1 em, column sep=-.2 em]
	0 \ar{rr} 	& 	& G_R^{0} \ar[hook]{rr} \ar[equal]{ld}		&		&   G_R^{1} \ar[two heads]{ld} \ar[hook]{rr} 	&	&G_R^{2} =\bigwedge\nolimits_R^3(\blank \otimes \blank)  \ar[two heads]{ld}. \\
							& \bigwedge\nolimits_R^3(\blank) \otimes \Gamma_R^3(\blank) \ar[dashed]{lu} &	& \Schur_R^{(2,1)}(\blank) \otimes \Weyl_R^{(2,1)}(\blank)  \ar[dashed]{lu}  & & \Sym_R^3(\blank) \otimes \bigwedge\nolimits_R^3(\blank)  \ar[dashed]{lu}
				\end{tikzcd}
			\end{equation*}
In the above diagrams, each triangle represents an exact triangle (in this case, a short exact sequence), and the doted arows represent the connecting morphisms that have degree $+1$. 
\end{example}

% Theorem: direct sum decomposition.
\begin{theorem}[Direct-Sum Decomposition Formula {\cite[Theorem II.4.11]{ABW}, \cite[Proposition 2.3.1]{Wey}, \cite[Theorems 1.4 \& 1.5]{Kou}}]
\label{thm:fil:Schur_oplus} 
Let $\lambda/\mu$ be a skew partition and set $N = |\lambda|-|\mu|$. For each integer $0 \le k \le N$, we let
	$$\gamma^{0}_{(k)} <   \gamma^{1}_{(k)} < \cdots < \gamma^{\ell_k-1}_{(k)}$$
denote all the partitions in $I_k(\lambda/\mu) = \{\gamma \mid \mu \subseteq \gamma \subseteq  \lambda, |\gamma| - |\mu|=k\}$ listed in lexicographic order (so  our $\gamma^i$ is $(\lambda^{\ell_k-i})^t$ in \cite{Kou}), where $\ell_k = |I_k(\lambda/\mu)|$ is the cardinality of $I_k(\lambda/\mu)$ Then:
\begin{enumerate}[leftmargin=*]
	\item 
	\label{thm:fil:Schur_oplus-1} 
	\begin{enumerate}
		\item 
		\label{thm:fil:Schur_oplus-1i} 
		For any $R \in {\rm CRing}$, the functor  
		$$\Schur_R^{\lambda/\mu} (\blank \oplus \blank) \colon {\rm Mod}_R^? \times {\rm Mod}_R^? \to {\rm Mod}_R^?,$$
	where $? \in \{ \mathrm{ff}, \mathrm{fproj}\}$, admits a canonical decomposition by subfunctors
		$$\Schur_R^{\lambda/\mu}  (\blank \oplus \blank) = \bigoplus\nolimits_{k=0}^{N} \Schur_R^{\lambda/\mu}(\blank, \blank)_{(k,N-k)},$$
	 and for each $k$, $\Schur_R^{\lambda/\mu}(\blank, \blank)_{(k,N-k)}$ admits a canonical filtration by subfunctors
		$$0 = F_R^{-1, (k)} \subset F_R^{0, (k)}\subset F_R^{1, (k)} \subset \cdots  \subset F_R^{\ell_k-1, (k)} = \Schur_R^{\lambda/\mu}(\blank, \blank)_{(k,N-k)}$$
	for which there is an canonical equivalence for each $1 \le i\le \ell_k$:
		$$a_R^{i, (k)} \colon \Schur_R^{\gamma_{(k)}^i/\mu} (\blank)\otimes \Schur_R^{\lambda/\gamma_{(k)}^i} (\blank) \xrightarrow{\sim} F_R^{i, (k)} /F_R^{i-1, (k)}.$$
	Moreover, the formations of the decomposition of $\Schur_R^{\lambda/\mu}  (\blank \oplus \blank) $, the subfunctors $F_R^{i, (k)}$, and the equivalences $a_R^{i, (k)}$ commute with base change of commutative rings $R$.
		\item 
		\label{thm:fil:Schur_oplus-1ii} 
		For any $R \in {\rm CRing}$, the functor  
		$$\Weyl_R^{\lambda/\mu} (\blank \oplus \blank) \colon {\rm Mod}_R^? \times {\rm Mod}_R^? \to {\rm Mod}_R^?,$$
	where $? \in \{ \mathrm{ff}, \mathrm{fproj}\}$, admits a canonical decomposition by subfunctors
		$$\Weyl_R^{\lambda/\mu}  (\blank \oplus \blank) = \bigoplus_{k=0}^{|\lambda|-|\mu|} \Weyl_R^{\lambda/\mu}(\blank, \blank)_{(k,N-k)},$$
	 and for each $k$, $\Weyl_R^{\lambda/\mu}(\blank, \blank)_{(k,N-k)}$ admits a canonical filtration by subfunctors
		$$0 = G^R_{\ell_k,(k)} \subset G^R_{\ell_k-1,(k)}  \subset \cdots \subset G^R_{1,(k)}  \subset G^R_{0,(k)} = \Weyl_R^{\lambda/\mu}(\blank, \blank)_{(k,N-k)}$$
	for which there is an canonical equivalence for each $1 \le i\le \ell_k$:
		$$b^R_{i, (k)} \colon \Weyl_R^{\gamma_{(k)}^i/\mu} (\blank)\otimes \Weyl_R^{\lambda/ \gamma_{(k)}^i} (\blank) \xrightarrow{\sim} G^R_{i, ,(k)} / G^R_{i+1, (k)}.$$
	Moreover, the formations of the decomposition of $\Weyl_R^{\lambda/\mu}  (\blank \oplus \blank)$, the subfunctors $G^R_{i,(k)}$, and the equivalences $b^R_{i, (k)} $ commute with base change of commutative rings $R$.	\end{enumerate}

	\item 
	\label{thm:fil:Schur_oplus-2} 
	\begin{enumerate}
		\item
		\label{thm:fil:Schur_oplus-2i} 
		Let $R \in {\rm CRing}$ and let $\alpha$ denote a short exact sequence $0 \to M' \xrightarrow{u} M \xrightarrow{v} M'' \to 0$ of finite free $R$-modules, then there is a canonical filtration of $\Schur_R^{\lambda/\mu}(M)$ by submodules
			$$0 \subset F_{|\lambda|-|\mu|}(R, \alpha) \subset  \cdots \subset F_{1}(R,\alpha) \subset F_{0}(R, \alpha)= \Schur_R^{\lambda/\mu}(M)$$
		which is functorial on $\alpha$. Moreover, for each $k$, there is an functorial equivalence
			$$c^{R, \alpha}_k \colon \Schur_{R}^{\lambda/\mu}(M', M'')_{(k,N-k)} \xrightarrow{\sim} F_{k}(R, \alpha)/ F_{k+1}(R, \alpha),$$
		where $\Schur_{R}^{\lambda/\mu}(M', M'')_{(k,N-k)}$ is defined in \eqref{thm:fil:Schur_oplus-1i}. The formations of the filtrations $F_{i}(R, \alpha)$ and equivalences $c^{R, \alpha}_k$ commute with base change of commutative rings $R$.
			\item
			\label{thm:fil:Schur_oplus-2ii}
		  	 Let $R \in {\rm CRing}$ and let $\alpha$ denote a short exact sequence $0 \to  M' \xrightarrow{u} M \xrightarrow{v} M'' \to 0$ of finite free $R$-modules, there is a canonical filtration of $\Weyl_R^{\lambda/\mu}(M)$ by submodules
					$$0 \subset G_{|\lambda|-|\mu|}(R, \alpha) \subset  \cdots \subset G_{1}(R,\alpha) \subset G_{0}(R, \alpha) =\Weyl_R^{\lambda/\mu}(M)$$
		which is functorial on $\alpha$. Moreover, for each $k$, there is an functorial equivalence
			$$d^{R, \alpha}_k \colon \Weyl_{R}^{\lambda/\mu}(M', M'')_{(k,N-k)} \xrightarrow{\sim} G_{k}(R,\alpha)/G_{k+1}(R,\alpha)$$
		where $ \Weyl_{R}^{\lambda/\mu}(M', M'')_{(k,N-k)}$ is defined in \eqref{thm:fil:Schur_oplus-1ii}. The formations of the filtrations $G_{i}(R,\alpha)$ and equivalences $d^{R, \alpha}_k$ commute with base change of commutative rings $R$.  
		\end{enumerate}
\end{enumerate}
 \end{theorem}

 \begin{proof}
These statements are proved in \cite[Theorem 1.4, Theorem 1.5]{Kou} except that we add assertions about base-change properties and that we claim that the equivalences $c^{R, \alpha}_k$ and $d^{R, \alpha}_k$ of the second part \eqref{thm:fil:Schur_oplus-2} can be constructed functorially. Regarding the base-change properties, it suffices to observe that in the proof of \cite[Theorem 1.4 (a)]{Kou}, the formation of the functor $db_{\alpha}(N,L)$, of the image submodules  $K_k=\sum_{|\tau|=k}{\rm Im}(db_{\tau})$, $\sum_{|\tau|=k, \tau \ge \alpha}{\rm Im}(db_{\tau})$ and $\sum_{|\tau|=k, \tau > \alpha}{\rm Im}(db_{\tau})$, and of the equivalence $c_{\lambda}$ there commutes with base change. 

Now we consider assertion \eqref{thm:fil:Schur_oplus-2i}. The argument of \cite[Theorem 1.4]{Kou} already implies the existence the equivalences of the form $c_{R, \alpha}^k$; it remains to provide functorial constructions of these equivalences.
For each exact sequence $0 \to M' \xrightarrow{u} M \xrightarrow{v} M'' \to 0$ of finite free $R$-modules, the canonical map $(u, \id_M) \colon M' \oplus M \to M$ induces a canonical map $\varphi \colon \Schur_R^{\lambda/\mu}(M' \oplus M) \to \Schur^{\lambda/\mu}(M)$. For each $k$, we let $F_{k}(R, \alpha)$ denote the image $\varphi(\Schur_R^{\lambda/\mu}(M' \oplus M)_{\ge k})$ inside $ \Schur_R^{\lambda/\mu}(M)$, where $\Schur_R^{\lambda/\mu}(M' \oplus M)_{\ge k}$ denotes $\bigoplus_{\ell \ge k} \Schur_R^{\lambda/\mu}(M', M)_{(\ell,N-\ell)}$; we wish to show that $F_{k}(R, \alpha)$, for $0 \le k \le |\lambda|-|\mu|$, make up the desired filtration. Consider the commutative diagram
	$$
	\begin{tikzcd} 
		0 \ar{r} & M' \ar{d}{\id_{M'}} \ar{r}{(\id_{M'},0)^t} & M'\oplus M \ar{d}{(u,\id_M)} \ar{r}{(0,\id_{M})}  &M \ar{r} \ar{d}{v} & 0  \\
		0 \ar{r} & M'  \ar{r}{u} & M  \ar{r}{v}  &M'' \ar{r} & 0.
	\end{tikzcd}
	$$
For any element $(x,y) \in \bigwedge^{\alpha/\mu^t} M' \otimes \bigwedge^{\lambda^t/\alpha} M$, then the equation 
	$$\varphi (db_{\alpha}(x,y)) = db_{\alpha} (x, \wedge^{\lambda^t/\alpha}(v)(y))+ \big(\text{elements in ~} \sum\nolimits_{m > k} \sum\nolimits_{\tau \in P(m)} {\rm Im} (db_{\tau})\big)$$
in the proof of \cite[Theorem 1.4 (b), page 92]{Kou} implies that the our submodules $F_{k}(R, \alpha)=\varphi(\Schur_R^{\lambda/\mu}(M' \oplus M)_{\ge k})$ coincide with the submodules $\sum\nolimits_{m \ge k} \sum\nolimits_{\tau \in P(m)} {\rm Im} (db_{\tau}) \subseteq \Schur^{\lambda}_R(M)$ considered in {\em loc. cit.} (which are constructed by choosing a splitting of the map $v \colon M \to M''$, but shown to be independent of the choice a splitting in \cite{Kou}). In particular, by considering the split case, we have $F_{k}( R, \alpha) / F_{k+1}(R, \alpha) \simeq \Schur_R^{\lambda/\mu}(M', M'')_{(k,N-k)}$. It remains to show that the natural map 
	$\Schur_R^{\lambda/\mu}(M', M)_{(k,N-k)} \to F_{k}(R, \alpha) / F_{k+1}(R, \alpha)$
induced by $\varphi$ factorizes as 
%$\Schur^{\lambda/\mu}(M' \oplus M)_k \twoheadrightarrow \Schur^{\lambda/\mu}(M' \oplus M'')_k $ and induces 
	\begin{equation*} %\label{eqn:thm:fil:Schur_oplus:factorization}
\Schur_R^{\lambda/\mu}(M', M)_{(k,N-k)} \xrightarrow{can.} 
	\Schur_R^{\lambda/\mu}(M', M'')_{(k,N-k)} \to F_{k}( R, \alpha) / F_{k+1}(R, \alpha)
	\end{equation*}
where the first map is the canonical projection, and the latter morphism is a functorial isomorphism of $R$-modules. To prove this assertion, it suffices to choose a splitting, where the desired result is again a consequence of the above equation. Assertion \eqref{thm:fil:Schur_oplus-2ii} is proved similarly. 
\end{proof}

\begin{remark} 
\label{rmk:fil:Schur_oplus:(k,N-k)}
The notation $\Schur_R^{\lambda/\mu}(M_1, M_2)_{(k,N-k)}$ (resp. $\Weyl_R^{\lambda/\mu}(M_1, M_2)_{(k,N-k)}$) in \eqref{thm:fil:Schur_oplus-1i} (resp. \eqref{thm:fil:Schur_oplus-1ii}) is intended to indicate that it is the component of $\Schur_R^{\lambda/\mu}(M_1 \oplus M_2)$  (resp. $\Weyl_R^{\lambda/\mu}(M_1 \oplus M_2)$) which has homogeneous degree $k$ on $M_1$ and $N-k$ on $M_2$; 
we expect these components are homogeneous in the sense of polynomial functor theory \cite{Bous}.
%We expect to make We should be able to in the polynomial functor theory.
It is obvious from construction that the canonical equivalence $\Schur_R^{\lambda/\mu}(M_1 \oplus M_2) \simeq \Schur_R^{\lambda/\mu}(M_2 \oplus M_1)$ (resp. $\Weyl_R^{\lambda/\mu}(M_1 \oplus M_2) \simeq \Weyl_R^{\lambda/\mu}(M_2 \oplus M_1)$) induces canonical equivalences of summands $\Schur_R^{\lambda/\mu}(M_1, M_2)_{(k,N-k)} \simeq \Schur_R^{\lambda/\mu}(M_2, M_1)_{(N-k,k)}$ (resp. $\Weyl_R^{\lambda/\mu}(M_1, M_2)_{(k,N-k)} \simeq \Weyl_R^{\lambda/\mu}(M_2, M_1)_{(N-k,k)}$).
\end{remark}

The following result shows that the Schur and Weyl functors associated with skew partitions can be functorially built up from those associated with usual partitions.

\begin{theorem}[{\cite[Theorem 1.3]{Bo2}, \cite[Theorem 1.5 \& Theorem 2.6]{Kou}}]
\label{thm:fil:Schur_LR} 
Let $\lambda/\mu$ be a skew partition. For any partition $\gamma$, we let $c_{\mu, \gamma}^{\lambda}$ denote the {\em Littlewood--Richardson} number (see, for example, \cite[page 62, \S 5.1]{Ful}). Then $c_{\mu, \gamma}^{\lambda} \ne 0$ implies that $|\mu| + |\gamma| = |\lambda|$ and $\gamma \subseteq \lambda$. We let 
	$$\alpha^{0} < \alpha^{1} < \cdots < \alpha^{s-1}$$
denote the list of all the partitions $\alpha$ of size $|\lambda| - |\mu|$ such that $c_{\mu,\alpha}^{\lambda} \ne 0$ in the lexicographical order, and consider the sequence of partitions with each $\alpha^i$ repeated $c_{\mu,\alpha^i}^{\alpha}$-many times:
	$$(\tau^0, \tau^1, \ldots, \tau^{\ell-1}) : =(\underbrace{\alpha^0, \cdots, \alpha^0}_{c_{\mu,\alpha^0}^{\lambda} \,\text{terms}}, \underbrace{\alpha^1, \cdots, \alpha^1}_{c_{\mu,\alpha^1}^{\lambda} \,\text{terms}}, \cdots, \underbrace{\alpha^{s-1}, \cdots, \alpha^{s-1}}_{c_{\mu,\alpha^{s-1}}^{\lambda} \,\text{terms}}).$$
Here, $\ell: = \ell(\lambda/\mu) = \sum_{\gamma \subseteq \lambda} c_{\mu, \gamma}^{\lambda} = \sum_{i=1}^{s} c_{\mu,\alpha^i}^{\lambda}.$
The following statements are true:
\begin{enumerate}
	\item \label{thm:fil:Schur_LR-1} 
	For any $R \in {\rm CRing}$, there is a functorial filtration of the Schur functor $\Schur_R^{\lambda/\mu}(\blank) \colon \Mod_R^{?} \to \Mod_R^{?}$, where $? \in \{ \mathrm{ff}, \mathrm{fproj}\}$, by subfunctors
	$$0 = F_R^{-1} \subset F_R^{0} \subset  F_R^{1} \subset \cdots \subset F_R^{\ell-2} \subset F_R^{\ell-1} = \Schur_R^{\lambda/\mu}$$
and a functorial isomorphism for each $0 \le i \le \ell-1$,
	$$a^i_R \colon \Schur_R^{\tau^i} \xrightarrow{\sim} F_R^{i}/F_R^{i-1}.$$
	Moreover, the formations of the subfunctors $F_R^{i}$ and the isomorphisms $a^i_R$ commute with base change of commutative rings $R \to R'$. 
	\item \label{thm:fil:Schur_LR-2} 
	For any $R \in {\rm CRing}$, there is a functorial filtration of the Weyl functor $\Weyl_R^{\lambda/\mu}(\blank) \colon \Mod_R^{?} \to \Mod_R^{?}$, where $? \in \{ \mathrm{ff}, \mathrm{fproj}\}$, by subfunctors
	$$0 = G^R_{\ell} \subset G^R_{\ell-1} \subset \cdots \subset G^R_{2} \subset G^R_{1} \subset G^R_{0} = \Weyl_R^{\lambda/\mu}$$
and a functorial isomorphism for each $1 \le i \le \ell$,
	$$b_i^R \colon \Weyl_R^{\tau^i} \xrightarrow{\sim} G^R_{i}/G^R_{i+1}.$$
	Moreover, the formations of the subfunctors $G^R_{i}$ and the isomorphisms $b_i^R$ commute with base change of commutative rings $R \to R'$. 

\end{enumerate}
\end{theorem}

\begin{example}
\begin{enumerate}
	% Example 1.
	\item 
	Let $\lambda/\mu = (3,2,1)/(1) = \ytableausetup{smalltableaux}\ydiagram{1+2,2,1}$, then we have $\ell=3$ and 
	$$\tau^{0} =\ydiagram{2,2,1}=(2,2,1) < \tau^{1}=\ydiagram{3,1,1} = (3,1,1) < \tau^{2}= \ydiagram{3,2} = (3,2).$$
Then we obtain functorial filtrations with subquotients: 
	\begin{equation*}
				\begin{tikzcd} [back line/.style={dashed}, row sep=1 em, column sep=0.2 em]
	0 \ar{rr} 	& 	& F^{0} \ar[hook]{rr} \ar[equal]{ld}		&		&   F^{1} \ar[two heads]{ld} \ar[hook]{rr} 	&	&F^{2} = \Schur_R^{(3,2,1)/(1)}(\blank)  \ar[two heads]{ld}. \\
							& \Schur_R^{(2,2,1)}(\blank) \ar[dashed]{lu} &	& \Schur_R^{(3,1,1)}(\blank)  \ar[dashed]{lu}  & & \Schur_R^{(3,2)}(\blank) \ar[dashed]{lu}
				\end{tikzcd}
			\end{equation*}		
\begin{equation*}
				\begin{tikzcd} [back line/.style={dashed}, row sep=1 em, column sep=0.2 em]
	0 \ar{rr} 	& 	& G_{2} \ar[hook]{rr} \ar[equal]{ld}		&		&   G_{1}  \ar[two heads]{ld} \ar[hook]{rr} 	&	&G_{0}= \Weyl_R^{(3,2,1)/(1)}(\blank)  \ar[two heads]{ld}. \\
							& \Weyl_R^{(3,2)}(\blank) \ar[dashed]{lu} &	& \Weyl_R^{(3,1,1)}(\blank)  \ar[dashed]{lu}  & & \Weyl_R^{(2,2,1)}(\blank) \ar[dashed]{lu}
				\end{tikzcd}
			\end{equation*}	
		%Example 2
		\item 
		Let $\lambda/\mu = (2,2,1,1)/(1,1) = \ydiagram{1+1,1+1,1,1}$, then we have $\ell =3$, and 
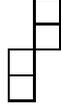
$$
	\tau^0 =\ydiagram{1,1,1,1} = (1^4) <\tau^1 = \ydiagram{2,1,1} = (2,1,1) < \tau^2 = \ydiagram{2,2} = (2,2).
	$$
	In this case, $\Schur_R^{(2,2,1,1)/(1,1)} (M)= \Weyl_R^{(2,2,1,1)/(1,1)} (M)=\bigwedge\nolimits_R^2 (M) \otimes \bigwedge\nolimits_R^2 (M)$ for all $M \in \Mod_R^?$, and the theorem implies there are functorial filtrations with subquotients: 
	\begin{equation*}
				\begin{tikzcd} [back line/.style={dashed}, row sep=1 em, column sep=0.2 em]
	0 \ar{rr} 	& 	& F^{0} \ar[hook]{rr} \ar[equal]{ld}		&		&   F^{1} \ar[two heads]{ld} \ar[hook]{rr} 	&	&F^{2} =(\bigwedge\nolimits_R^2 \otimes \bigwedge\nolimits_R^2 )(\blank) \ar[two heads]{ld}. \\
							&\Schur_R^{(1^4)}= \bigwedge\nolimits_R^4(\blank) \ar[dashed]{lu} &	& \Schur_R^{(2,1,1)} (\blank)\ar[dashed]{lu}  & & \Schur_R^{(2,2)}(\blank) \ar[dashed]{lu}
				\end{tikzcd}
			\end{equation*}		
	\begin{equation*}
				\begin{tikzcd} [back line/.style={dashed}, row sep=1 em, column sep=0.2 em]
	0 \ar{rr} 	& 	& G_{2} \ar[hook]{rr} \ar[equal]{ld}		&		&   G_{1} \ar[two heads]{ld} \ar[hook]{rr} 	&	&G_{0} = (\bigwedge\nolimits_R^2 \otimes \bigwedge\nolimits_R^2)(\blank)  \ar[two heads]{ld}. \\
							&\Weyl_R^{(2,2)}(\blank) \ar[dashed]{lu} &	& \Weyl_R^{(2,1,1)}(\blank)  \ar[dashed]{lu}  & & \Weyl_R^{(1^4)}=\bigwedge\nolimits_R^{4}(\blank) \ar[dashed]{lu}
				\end{tikzcd}
			\end{equation*}		
\end{enumerate}
\end{example}

% Proof 
\begin{proof}[Proof of Theorem \ref{thm:fil:Schur_LR}]
It suffices to establish a canonical order-preserving bijection between the sequence consists of the transposes of the underlying Young diagrams of the Young tableaux $T \in L$ constructed in \cite[Theorem 1.5]{Kou} and the above sequence of Young diagrams $\tau^i$. 

To construct the desired bijection, we first review the constructions in \cite[Theorem 1.5]{Kou}. Let $\mu,\nu \subseteq \lambda$ be partitions such that $|\mu| + |\nu| = |\lambda|$. Let us consider the following bijection:
	$$\varphi \colon {\rm CST}^{\nu}(\lambda/\mu) \xrightarrow{\sim} {\rm CST}^{\lambda^t/\mu^t}(\nu^t).$$
Here, ${\rm CST}^{\nu}(\lambda/\mu)$ denotes the set of column-standard tableaux $T$ with shape $\nu$ and content $\lambda/\mu$, that is, each column of $T$ is non-decreasing from top to bottom, $T$ has underlying Young diagram $\nu$, and the entries of $T$ contain exactly $(\lambda_1-\mu_1)$ $1$'s, $(\lambda_2-\mu_2)$ $2$'s, etc. The set ${\rm CST}^{\lambda^t/\mu^t}(\nu^t)$ is defined similarly. Then the map $\varphi$ and its inverse $\varphi^{-1}$ are defined as follows:
\begin{itemize}
	\item Given $T \in {\rm CST}^{\nu}(\lambda/\mu)$. Then $Y = \varphi(T)$ is the tableau such that, for each $j$, $1 \le j \le \lambda_1^t$, the $j$th column of $Y$, $Y(*,j): = (Y(\mu_j+1, j) \le  Y(\mu_j+2, j) \le \cdots \le Y(\lambda_j, j))$, is precisely the sequence of the column-indices (with multiplicities) for which the number ``$j$" occurs as an entry of $T$.
	\item Conversely, for each $Y \in {\rm CST}^{\lambda^t/\mu^t}(\nu^t)$, the tableau $T = \varphi^{-1}(Y)$ is such that its $j$th column, $(T(1,j) \le T(2,j) \le \cdots \le T(\nu_j^t, j))$ (where $1 \le j \le \nu_1$), is precisely the sequence of the column-indices (with multiplicities) for which the number ``$j$" occurs as an entry of $Y$.
\end{itemize}
Next, we let $\gamma = \nu^t$, and consider the following set of tableaux
	$$\shL\shR(\lambda/\nu; \gamma): = {\rm ST}^{\gamma^t}(\lambda/\mu) \cap \varphi^{-1}({\rm ST}^{\lambda^t/\mu^t}(\gamma)).$$
The bijection $\varphi$ induces a bijection
	$\shL\shR(\lambda/\mu; \gamma) \simeq \varphi(\shL\shR(\lambda/\mu;  \gamma)).$
Unwinding the definitions, we see that the set $\shL\shR(\lambda/\mu; \gamma)$ is the set $L$ of tableaux considered in \cite[Theorem 1.5]{Kou} (after change of notations), and the transposes of the tableaux in $\varphi(\shL\shR(\lambda/\mu; \gamma))$ are precisely the {\em Littlewood--Richardson skew tableaux on the skew shape $\lambda/\mu$ with content $\gamma$} defined in \cite[page 64, Chapter 5, Proposition 3]{Ful}. In particular, \textit{loc. cit.} implies that 
	$$\# \shL\shR(\lambda/\mu; \gamma) = c_{\mu, \gamma}^{\lambda}$$
is the Littlewood--Richardson number. Hence assertion \eqref{thm:fil:Schur_LR-1} is proved. Assertion \eqref{thm:fil:Schur_LR-2} is proved similarly.
\end{proof}

% sec: Schur complexes
\subsection{Schur Complexes}
 \label{sec:bSchur}
This subsection reviews the theory of Schur complexes \cite{ABW}. Additional to Notation \ref{notation:CRingMod}, we introduce the following notations:

% Notations for Schur complexes
\begin{notation}
\label{notation:CRingCh}
\begin{enumerate}[leftmargin=*]
	\item For any commutative ring $R$, we let $\Ch^{\heartsuit}(R) = \Ch(\Mod_R^\heartsuit)$ denote the abelian category of chain complexes of (discrete) $R$-modules $F_* = (\cdots \to F_{i} \xrightarrow{d_i} F_{i-1} \to \cdots)$, and for any pair integers $a \le b$, we let $\Ch_{[a,b]}^{\heartsuit}(R) \subseteq \Ch^{\heartsuit}(R)$ denote the abelian subcategory spanned by the chain complexes $F_*$ which satisfy $F_i =0$ for $i \notin [a,b]$. 
	 We also consider the following sequence of subcategories:
	$$\Ch_{[a,b]}^{\rm ff}(R) \subseteq \Ch_{[a,b]}^{\rm fproj}(R) %\subseteq \Ch_{[a,b]}^{\rm proj}(R)
	\subseteq \Ch_{[a,b]}^{\flat}(R) \subseteq \Ch_{[a,b]}^\heartsuit(R),$$
where $\Ch_{[a,b]}^{\rm ff}(R)$ (resp. $\Ch_{[a,b]}^{\rm fproj}(R)$, resp. $\Ch_{[a,b]}^{\flat}(R)$) denote the full subcategory of $\Ch_{[a,b]}^\heartsuit(R)$ spanned by those chain complexes $F_*$ whose $i$th terms $F_i$ are {\em finite free} (resp. {\em finite projective}, resp. {\em flat}) $R$-modules for all $i$. 

	\item For any $? \in \{{\rm ff}, {\rm fproj}, \flat, \heartsuit\}$, we can canonically identify the category $\Ch_{[0,1]}^?(R)$ of two-term chain complexes $F_* = (F_1 \xrightarrow{d_1} F_0)$, where $F_i \in \Mod_R^?$, with the category $\Fun(\Delta^{1}, \Mod_R^?)$ of morphisms $\rho \colon M' \to M$ of the category $\Mod_R^?$, by setting $F_1= M'$, $F_0 = M$, and $d_1 = \rho$. In what follows, we will always make this identification. 
	\item For any $? \in \{{\rm ff}, {\rm fproj}, \flat, \heartsuit\}$ and any pair of integers $a \le b$, we let ${\rm CRingCh}_{[a,b]}^{?}$ denote the (ordinary) category of pairs $(R, F_*)$, where $R \in {\rm CRing}$ and $F_* \in \Ch_{[a,b]}^{?}(R)$. A morphism $(R,F_*) \to (R',F_*')$ in ${\rm CRingCh}_{[a,b]}^{?}$ is given by a  pair $(R \to R', F_* \to F_*')$, where $R \to R'$ is a morphism of commutative rings, and $F_* \to F_*'$ is a morphism of chain complexes of $R$-modules, where $F_*'$ is regarded as a $R$-module chain complex via restrictions of scalars along $R \to R'$. Then the natural forgetful functor
		$$V \colon {\rm CRingCh}_{[a,b]}^{?} \to {\rm CRing}, \qquad (R,F_*) \mapsto R$$
	is a {Grothendieck opfibration} which classifies the functor $\Ch_{[a,b]}^?(\blank) \colon {\rm CRing} \to \Cat, R \mapsto \Ch_{[a,b]}^?(R)$. A morphism $(R, F_*) \to (R',F_*')$ is $V$-coCartesian if and only if the natural chain map of $R'$-modules $R' \otimes_R F_* \to F_*'$ is an isomorphism. 
\end{enumerate}
\end{notation}

% Definition: Schur complexes
\begin{definition}[{Schur Complexes; \cite{ABW}}] 
\label{def:bSchur}
Let $\lambda/\mu$ be a skew partition, we define a functor 
	$$\bSchur^{\lambda/\mu}(\blank) \colon {\rm CRingCh}_{[0,1]}^{\rm fproj} \to {\rm CRingCh}_{[0, \infty]}^\heartsuit,  \qquad
	(R, \rho \colon M' \to M) \mapsto (R, \bSchur_R^{\lambda/\mu}(\rho \colon M' \to M))$$
as follows. For any $R \in {\rm CRing}$ and a two-term complex $\rho \colon M' \to M$ of finite projective $R$-modules, we consider the following composite map
	$$d_{\lambda^t/\mu^t}(R; \rho)  \colon \bigotimes_{j \in [1, \lambda_1]} \bwedge^ {\lambda_j^t - \mu_j^t} (\rho) \xrightarrow{\otimes_j \Delta'_j } \bigotimes_{(i,j) \in \lambda/\mu} \rho(i,j) \xrightarrow{\otimes_i m_i } \bigotimes_{i \in [1, \lambda_1^t]} \bSym^{\lambda_i-\mu_i} (\rho) $$
where $\rho(i,j)$ denotes a copy of the complex $M' \xrightarrow{\rho} M$ labeled by $(i,j) \in \mathrm{Y}(\lambda/\mu)$, $\Delta'_j$ is the $(\lambda_j^t-\mu_j^t)$-fold comultiplication map of the exterior complexes, and $m_i$ is the $(\lambda_i-\mu_i)$-fold multiplication map of the symmetric complexes. We let $\bSchur_R^{\lambda/\mu}(\rho \colon M' \to M)$ denote the image of the map $d_{\lambda^t/\mu^t}(R; \rho)$ in the abelian category ${\rm CRingCh}_{[0, \infty]}^\heartsuit$, and refer to $\bSchur_R^{\lambda/\mu}(\rho \colon M' \to M)$ as the {\em Schur complex} associated with $\rho \colon M' \to M$ and the skew partition $\lambda/\mu$. 
\end{definition}

% rmk: cokernels
\begin{remark}[Schur Complexes as Cokernels]
\label{rmk:bSchur:coker}
Similar to Definition \ref{def:SchurWeyl}, the Schur complexes can be defined as the cokernel of a natural map $\square_{\lambda^t/\mu^t} \colon \widetilde{\bwedge}_\cl^{\lambda^t/\mu^t}(\rho) \to \bwedge_\cl^{\lambda^t/\mu^t}(\rho)$ between tensor products of exterior complexes; see \cite{ABW} for details. In particular, this definition allows us to extend the definition of Schur complex to the whole category ${\rm CRingCh}_{[0,1]}^\heartsuit$:
	$$\bSchur^{\lambda/\mu}(\blank) \colon {\rm CRingCh}_{[0,1]}^\heartsuit \to {\rm CRingCh}_{[0, \infty]}^\heartsuit$$
	$$(R, \rho \colon M' \to M) \mapsto (R, \bSchur_R^{\lambda/\mu}(\rho \colon M' \to M): = \Coker(\square_{\lambda^t/\mu^t})).$$
However, since we will almost only use Schur complexes in the case where $\rho$ is a morphism between finite free modules, we leave the details of this construction to readers. 
\end{remark}

% Examples of Schur complexes
\begin{example} Let $R \in {\rm CRing}$, let $\rho \colon M' \to M$ be a morphism of finite projective $R$-modules, and let $\lambda/\mu$ be a skew partition.
\begin{enumerate}[label=(\roman*), leftmargin=2 em]
	\item If $M'=0$, then $\bSchur_R^{\lambda/\mu}(\rho \colon 0 \to M) = \Schur_R^{\lambda/\mu}(M)$ consists of a single Schur module $\Schur_R^{\lambda/\mu}(M)$ of $M$ placed at homological degree $0$.
	\item If $M= 0$, then $\bSchur_R^{\lambda/\mu}(\rho \colon M' \to 0) = \Weyl_R^{\lambda^t/\mu^t}(M')[|\lambda| - |\mu|]$ consist of a single Weyl module $\Weyl_R^{\lambda^t/\mu^t}(M')$ of $M'$ placed at homological degree $|\lambda| - |\mu|$.
	\item If $\lambda = (1^n)$, then $\bSchur_R^{(1^n)}(\rho \colon M' \to M) = \bwedge_R^n(\rho)$ is the exterior complex.
	\item $\lambda = (n)$, then $\bSchur^{(n)}(\rho \colon M' \to M) = \bSym_R^n(\rho)$ is the symmetric complex. 
	\end{enumerate}
\end{example}

% Theorem: Universal freeness of bSchur
\begin{theorem}[Universal Freeness of Schur complexes; {\cite[Theorem V.1.10]{ABW}}]
\label{thm:bSchur:free}
\begin{enumerate}
	\item The Schur complex functor $\bSchur^{\lambda/\mu}(\blank)$ defines a  functor 
		$$\bSchur^{\lambda/\mu}(\blank) \colon {\rm CRingCh}_{[0,1]}^{?} \to {\rm CRingCh}_{[0, |\lambda| - |\mu|]}^{?}$$ 	
	for any $? \in \{{\rm ff}, {\rm fproj}\}$. In other words, for any $R \in {\rm CRing}$ if $\rho \colon M' \to M$ is a morphism between finite free (resp. finite projective) $R$-modules, then $\bSchur^{\lambda/\mu}_R(\rho \colon M' \to M)$ is a finite chain complex of length $\le |\lambda| - |\mu|$ whose terms are given by 
	%, and for any integer $0 \le k \le |\lambda|- |\mu|$, the $k$th term $\bSchur^{\lambda/\mu}_R(\rho \colon M' \to M)_k$ of the Schur complex is a 
	finite free (resp. finite projective) $R$-modules.
	\item For any $? \in \{{\rm ff}, {\rm fproj}\}$, the following diagram commutes:
		$$
	\begin{tikzcd}[column sep=3.5 em]
		{\rm CRingCh}_{[0,1]}^{?} \ar{r}{\bSchur^{\lambda/\mu}(\blank)} \ar{d}{V} & {\rm CRingCh}_{[0,|\lambda| - |\mu|]}^{?} \ar{d}{V} \\
		{\rm CRing} \ar{r}{\id} & {\rm CRing}
	\end{tikzcd}
	$$
Moreover, $\bSchur^{\lambda/\mu}(\blank)$ carries $V$-coCartesian morphisms to $V$-coCartesian morphisms. In other words, for any base change of commutative rings $R \to R'$ and any morphism $\rho \colon M' \to M$ between finite free (resp. finite projective) $R$-modules, the canonical map
		$$\alpha_{\rho} \colon R' \otimes_R \bSchur_R^{\lambda/\mu}(\rho \colon M' \to M) \to \bSchur_{R'}^{\lambda/\mu}(\id_{R'} \otimes \rho \colon R' \otimes_R M' \to R' \otimes_R M)$$
is an isomorphism of chain complexes of finite free (resp. finite projective) $R'$-modules. 
\end{enumerate}
\end{theorem}

In the situation of Definition \ref{def:bSchur}, for any $R \in {\rm CRing}$ and a morphism $\rho \colon M' \to M$ in $\Mod_R^{?}$, where $? \in \{\mathrm{ff}, \mathrm{fproj}\}$, any $k \ge 0$, we denote the $k$th term of the Schur complex $\bSchur_R^{\lambda/\mu} (\rho \colon M' \to M)$ by $\bSchur_R^{\lambda/\mu} (\rho \colon M' \to M)_k$. Then $\bSchur_R^{\lambda/\mu} (\rho \colon M' \to M)_k \in \Mod_R^{?}$ and 
$\bSchur_R^{\lambda/\mu} (\rho \colon M' \to M)_k = 0$ if $k \notin [0, |\lambda| - |\mu|]$. Moreover, for any $0 \le k \le |\lambda| - |\mu|$, the module $\bSchur^{\lambda/\mu} (R;M' \xrightarrow{\rho} M)_k$ does not depend on the morphism $\rho \in \Hom_R(M',M)$. Hence we obtain:

% Corollary: kth component of a Schur complex
\begin{corollary}
\label{cor:bSchurk}
For $? \in \{\mathrm{ff}, \mathrm{fproj}\}$, the $k$th component of Schur complex defines a functor	
	$$\bSchur^{\lambda/\mu}(\blank, \blank)_k \colon {\rm CRingMod}^? \times_{\rm CRing} {\rm CRingMod}^?  \to {\rm CRingMod}^? $$
	$$ (R, M, M') \mapsto  (R, \bSchur_R^{\lambda/\mu}(M, M')_k)$$
where, for any fixed $R$, the module $\bSchur_R^{\lambda/\mu}(M, M')_k$ can be  defined as the $k$th component $\bSchur_R^{\lambda/\mu} (\rho \colon M' \to M)_k \in \Mod_R^?$ of the Schur complex $\bSchur_R^{\lambda/\mu}(\rho)$ for any choice of a morphism $\rho \in \Hom_R(M',M)$; a canonical choice is $\rho=0$. Moreover, there is a commutative diagram
	$$
	\begin{tikzcd}[column sep = 5 em]
		{\rm CRingMod}^? \times_{\rm CRing} {\rm CRingMod}^? \ar{r}{\bSchur^{\lambda/\mu}(\blank, \blank)_k} \ar{d}[swap]{W^? = U^? \times_{\Id_{\rm CRing}} U^?} & {\rm CRingMod}^? \ar{d}{U^?} \\
		{\rm CRing} \ar{r}{\id} & {\rm CRing}
	\end{tikzcd}
	$$
in which the functor $\bSchur^{\lambda/\mu}(\blank, \blank)_k$ carries $W^?$-coCartesian morphisms to $U^?$-coCartesian morphisms. In other words, for any base change of commutative rings $R \to R'$ and any pair of finite free (resp. finite projective) $R$-modules $(M',M)$, the canonical map
		$$R' \otimes_R \bSchur_R^{\lambda/\mu}(M',M)_k \to \bSchur_{R'}^{\lambda/\mu}(R' \otimes_R M', R' \otimes_R M)_k$$
is an isomorphism of finite free (resp. finite projective) $R'$-modules. 
\end{corollary}

The next theorem shows that the $k$th components of Schur complexes can be constructed from Schur and Weyl modules through iterated extensions. 

% Theorem: filtrations of kth component 
\begin{theorem}[{\cite[Corollary V.1.14 ]{ABW}, \cite[Theorem 2.4.10]{Wey}}]
\label{thm:fil:bSchur_k}
Let $\lambda/\mu$ be a skew partition, for any fixed commutative ring $R$, let 
	$$ \bSchur^{\lambda/\mu}(\blank, \blank)_k \colon \Mod_R^{\rm ff} \times \Mod_R^{\rm ff}  \to \Mod_R^{\rm ff}, \qquad (M, M') \mapsto  \bSchur^{\lambda/\mu} (R;M' \xrightarrow{0} M)_k$$
denote the functor of taking $k$th components of Schur complexes considered in Corollary \ref{cor:bSchurk}.
\begin{enumerate}
	\item
	\label{thm:fil:bSchur_k-1}
	Let $\gamma^{0}_{(k)} <   \gamma^{1}_{(k)} < \cdots < \gamma^{\ell_k-1}_{(k)}$
be all the partitions in $I_k (\lambda/\mu) = \{\gamma \mid \mu \subseteq \gamma \subseteq  \lambda, |\gamma| - |\mu|=k\}$ listed in lexicographic order, where $\ell_k = |I_k(\lambda/\mu)|$ is the cardinality. Then $\bSchur^{\lambda/\mu}(\blank, \blank)_k$ admits a canonical filtration by subfunctors
			$$0 = P_A^{-1,(k)}\subset P_R^{0,(k)} \subset P_R^{1,(k)} \subset \cdots \subset P_R^{\ell_k-2,(k)} \subset P_R^{\ell_k-1,(k)} = \bSchur_R^{\lambda/\mu}(\blank, \blank)_k,$$
		and for each $0 \le i \le \ell_k-1$, there is a canonical natural isomorphism of functors
			$$\Schur_R^{\lambda/\gamma^i} (\blank) \otimes \Weyl^{(\gamma^i)^t/\mu^t}(\blank) \xrightarrow{\sim} P_R^{i,(k)}/ P_R^{i-1,(k)} \colon \Mod_R^{\rm ff} \times \Mod_R^{\rm ff}  \to \Mod_R^{\rm ff}.$$
		Furthermore, the formations of the functors $P_R^{i,(k)}$ and the above equivalences commute with base change of ordinary commutative rings $R \to R'$.
	\item 
	\label{thm:fil:bSchur_k-2}
	Let 
			$\nu_{(k)}^{0} < \nu_{(k)}^{1} < \cdots < \nu_{(k)}^{m_k-1}$
		be all the partitions in $J_k (\lambda/\mu) = \{\nu \mid \mu \subseteq \nu \subseteq \lambda, |\lambda| - |\nu| = k\}$ listed in lexicographic order, where $m_k = |J_k(\lambda/\mu)|$. Then $\bSchur_R^{\lambda/\mu}(\blank, \blank)_k$ admits a canonical filtration by subfunctors
			$$0 = Q_R^{-1,(k)} \subset Q_R^{1,(k)}  \subset Q_R^{2,(k)} \subset \cdots \subset Q_R^{m_k-2,(k)}  \subset Q_R^{m_k-1,(k)}= \bSchur^{\lambda/\mu}(\blank, \blank)_k,$$
		and for each $0 \le i \le m_k-1$, there is a functorial equivalence of functors
			$$\Schur_R^{\nu^i/\mu} (\blank) \otimes \Weyl_R^{\lambda^t/(\nu^i)^t}(\blank) \xrightarrow{\sim} Q_R^{i,(k)} /Q_R^{i-1,(k)} \colon \Mod_R^{\rm ff} \times \Mod_R^{\rm ff}  \to \Mod_R^{\rm ff}.$$
		Furthermore, the formations of the functors $Q_R^{i,(k)}$ and the above equivalences commute with base change of ordinary commutative rings $R \to R'$.
		\end{enumerate}
\end{theorem}	

\begin{proof}
For any fixed $R$ and $M', M$, assertion \eqref{thm:fil:bSchur_k-2} is proved in \cite[Corollary V.1.14]{ABW}, and assertion \eqref{thm:fil:bSchur_k-1} can be proved in a similar way by setting $\phi_1 = M' \to 0$ and $\phi_2 = 0 \to M$ in the proof of  \cite[Corollary V.1.14]{ABW} (see also \cite[Theorem 2.4.10 (a) \& (b)]{Wey}). To finish the proof of assertion \eqref{thm:fil:bSchur_k-2}, it suffices to observe that the formation of the subcomplexes $M_\gamma(\bSchur^{\lambda/\mu}(\rho \colon M' \to M))$ and $\dot{M}_\gamma(\bSchur^{\lambda/\mu}(\rho \colon M' \to M))$ of $\bSchur^{\lambda/\mu}(\rho)$ defined in \cite[Definition V.1.11]{ABW} (in the case $\phi_1 = 0 \to M$ and $\phi_2 = M' \to 0$), and the formation of the isomorphism 
	$$\Schur_R^{\gamma/\mu}(M) \otimes \Weyl_R^{\lambda^t/\gamma^t}(M') \to M_\gamma(\bSchur^{\lambda/\mu}(\rho))/\dot{M}_\gamma(\bSchur^{\lambda/\mu}(\rho))$$
 of \cite[Theorem V.1.13]{ABW} (in the case $\phi_1 = 0 \to M$ and $\phi_2 = M' \to 0$) are functorial on $M$ and $M'$ and commute with base change of rings $R$. Assertion \eqref{thm:fil:bSchur_k-1} is proved similarly. 
\end{proof}

\begin{example}
Let $\lambda/\mu$ be a skew partition such that $\mu \neq \lambda$.
\begin{enumerate}[leftmargin=*]
	\item As special cases of Theorem \ref{thm:fil:bSchur_k}, we have $\bSchur_R^{\lambda/\mu}  (M',M)_{0} = \Schur_R^{\lambda/\mu}(M)$ and $\bSchur_R^{\lambda/\mu}  (M',M)_{|\lambda|-|\mu|} = \Weyl_R^{\lambda^t/\mu^t}(M')$. Consequently, the Schur complex $\bSchur_R^{\lambda/\mu} (\rho \colon M' \to M)$ has the form
	\begin{align*} 
	 0 \to \Weyl_R^{\lambda^t/\mu^t}(M') \xrightarrow{d_{|\lambda|-|\mu|}} \bSchur_R^{\lambda/\mu}  (M',M)_{|\lambda|-|\mu|-1} \to \cdots 
	\xrightarrow{d_2} & \bSchur_R^{\lambda/\mu}  (M',M)_1 \xrightarrow{d_1}   \Schur_R^{\lambda/\mu}(M) \to 0.
	\end{align*}
	\item If $\mu=0$, then Theorem \ref{thm:fil:bSchur_k} implies $\bSchur_R^{\lambda}(M', M)_1 = \Schur_{R}^{\lambda/(1)}(M) \otimes_R M'$ and $\bSchur_R^{\lambda}(M', M)_{|\lambda|-1} =M \otimes_R  \Weyl_R^{\lambda^t/(1)}(M')$. Consequently, the Schur complex $\bSchur_R^{\lambda} (\rho \colon M' \to M)$ has the form
	\begin{align*}
	 0 \to \Weyl_R^{\lambda^t}(M') \xrightarrow{d_{|\lambda|}} M \otimes_R  \Weyl_R^{\lambda^t/(1)}(M') \xrightarrow{d_{|\lambda|-1}} \cdots 
	\xrightarrow{d_2} & \Schur_{R}^{\lambda/(1)}(M) \otimes_R M' \xrightarrow{d_1} \Schur_R^{\lambda}(M) \to 0.
	\end{align*}
\end{enumerate}
\end{example}

Finally, there are systematic results about the acyclic properties of Schur complexes; see \cite[Theorem V.1.17]{ABW}, \cite[Theorem 6.1]{AT19}. Here, we will only review two special cases and refer readers to the above-mentioned references for the general situation.

\begin{theorem}[Acyclicity; {\cite[Theorem V.1.17]{ABW}}] Let $R$ be a commutative ring.
\label{thm:acyclic}
\begin{enumerate}
	\item \label{thm:acyclic-1}
	(Split-injective cases) Let $0 \to M' \xrightarrow{\rho} M \xrightarrow{q} M'' \to 0$ be a short exact sequence of finite projective $R$-modules, and let $\lambda/\mu$ be a skew partition such that $\lambda \neq \mu$. In this case, the Schur complex $\bSchur_{R}^{\lambda/\mu} (\rho \colon M' \to M) = [0 \to \bSchur_{R}^{\lambda/\mu} (M',M)_{|\lambda|-|\mu|} \to \cdots \to \bSchur_{R}^{\lambda/\mu} (M',M)_{0}]$ adjoined by the sequence $[\bSchur_{R}^{\lambda/\mu} (M',M)_{0} = \Schur_R^{\lambda/\mu}(M) \xrightarrow{\Schur_R^{\lambda/\mu}(q)} \Schur_R^{\lambda/\mu}(M'') \to 0]$ to its right constitutes a long exact sequence of finite projective $R$-modules.
 	\item \label{thm:acyclic-2}
	(Universal local cases) Let $\lambda \neq (0)$ be a partition, let $M_0 = R^m$, $N_0 = R^n$ for some integers $n \ge m \ge 0$, let $A = \Sym_R^*(M_1^\vee \otimes M_0) = R[X_{ij}]_{1 \le i \le m, 1 \le j \le n}$ be the polynomial algebra,  let $M = A \otimes_R M_0 = A^m$ and $N = A \otimes_R N_0 = A^n$, and let $\rho \colon M = A^m \to N = A^n$ denote the tautological $A$-module morphism given by the tautological matrix $(X_{ij})_{1 \le i \le m, 1 \le j \le n}$. Assume that $1 \le \lambda_1^t \le n- m +1$, then the Schur complex $\bSchur_{A}^{\lambda} (\rho \colon M \to N)$ together with the canonical quotient map $\bSchur_{A}^{\lambda} (M,N)_{0} = \Schur_A^{\lambda}(N) \twoheadrightarrow \Schur_A^{\lambda}({\rm Coker}(\rho))$ forms a finite projective resolution of the classical Schur module $\Schur_A^{\lambda/\mu}({\rm Coker}(\rho))$.
\end{enumerate}
\end{theorem}

% sec: Derived Schur functors	
%\newpage	
\section{Derived Schur and Weyl Functors}
\label{sec:dSchur}

This section studies derived Schur and Weyl functors, which are the derived functors of the classical Schur and Weyl functors in the sense of non-abelian derived theory (\cite[\S 5.5]{HTT}).

\S \ref{sec:non-abelian} briefly reviews Lurie's non-abelian derived theory (\cite[\S 5.5]{HTT}) and introduces several constructions that will be used in the following subsections. 

\S \ref{sec:dSchurdWeyl} introduces derived Schur and Weyl functors (Definition \ref{def:dSchurWeyl}) and discusses their basic properties, including base change properties (Proposition \ref{prop:dSchur:basechange}), flatness and freeness (Proposition \ref{prop:dSchur:free}), and their classical truncations (Proposition \ref{prop:dSchur:classical}).

\S \ref{sec:univ.fib:dSchur} generalizes the classical results of \S \ref{sec:univ.fil:Schur} to the derived setting. We obtained derived versions of Cauchy decomposition formula (Theorem \ref{thm:fil:dsym_otimes}), the decomposition formula for direct sums of complexes (Theorem \ref{thm:fib:dSchur_oplus}), Littlewood--Richardson rules for derived Schur and Weyl functors (Theorem \ref{thm:fil:dSchur_LR} and Corollary \ref{cor:fil:dSchur_LR}), and Koszul-type sequences associated with derived Schur functors (Theorem \ref{thm:fib:dbSchur}).

\S \ref{sec:dSchurdWeyl.properties} examines the further properties of derived Schur and Weyl functors. It includes results on the d{\'e}calage isomorphisms (Theorem \ref{thm:Illusie--Lurie}, Corollary \ref{cor:dSchur.decalage}), connectivity (Corollary \ref{cor:dSchur:connective}), finiteness (Proposition \ref{prop:dSchur:pc}) and Tor-amplitudes and perfectness (Proposition \ref{prop:dSchur:Tor-amp}), generalizations of Illusie's equivalences (Proposition \ref{prop:dSchur_vs_bSchur}) and the relationship between the $k$th component of Schur complexes and derived Schur functors (Proposition \ref{prop:bSchur_k=Schur_oplus_k}). 

All of the constructions and results from subsections \S \ref{sec:non-abelian} through \S \ref{sec:dSchurdWeyl.properties} are ``globalized" in \S \ref{sec:dSchurdWeyl.prestacks}  from the affine situation to the larger framework of derived schemes and stacks (and more generally, prestacks).

\subsection{Non-abelian Derived Categories}
\label{sec:non-abelian}
This subsection briefly reviews Lurie's theory of non-abelian derived categories \cite[\S 5.5]{HTT} and discusses several key examples that we will need in the later parts of this section.
  
Lurie's non-abelian derived category $\shP_{\Sigma}(\shC)$ of an $\infty$-category $\shC$ can be thought of the $\infty$-category obtained from $\shC$ by freely adjoining {\em sifted colimits}.  Recall that a simplicial set $K$ is called {\em sifted} if it is non-empty and the diagonal map $\delta \colon K \to K \times K$ is cofinal (see \cite[Definition 5.5.8.1]{HTT}, \cite[\href{https://kerodon.net/tag/02QE}{Tag 02QE}]{kerodon}). Let $\shD$ be an $\infty$-category, then sifted colimits in $\shD$ are colimits of the diagrams of the form $K \to \shD$, where $K$ is a sifted simplicial set. The two most important subclasses of sifted colimits are {\em filtered colimits} (\cite[\S 5.4.1]{HTT}) and {\em geometric realizations} of simplicial objects (\cite[Notation 6.1.2.12]{HTT}).

Let $\shC$ be a small $\infty$-category that admits finite coproducts, then its {\em non-abelian derived category} $\shP_{\Sigma}(\shC)$ is defined to be the $\infty$-category $\Fun^{\pi}(\shC^\op, \shS)$ spanned by those functors which preserve finite products. The $\infty$-category $\shP_{\Sigma}(\shC)$ can be characterized by the following  properties (\cite[Propositions 5.5.8.10, 5.5.8.22]{HTT}): $\shP_{\Sigma}(\shC)$ is a presentable $\infty$-category (in particular, it admits all small colimits), and there exists a coproduct-preserving fully faithful functor $j \colon \shC \hookrightarrow \shP_{\Sigma}(\shC)$, called {\em Yoneda embedding}, such that the essential image of $j$ consists of compact projective objects of $\shP_{\Sigma}(\shC)$ which generate $\shP_{\Sigma}(\shC)$ under sifted colimits.

\begin{proposition}[Lurie {\cite[\S 5.5.8]{HTT}}] \label{prop:nonab:derived} 
Let $\shC$ be a small $\infty$-category that admits finite coproducts and $\shD$  an $\infty$-category that admits sifted colimits. Then composition with the Yoneda embedding $j$ induces an equivalence of $\infty$-categories
				$$\Fun_\Sigma(\shP_\Sigma(\shC), \shD) \to \Fun(\shC, \shD),$$
where $\Fun_\Sigma(\shP_\Sigma(\shC), \shD)$ denotes the $\infty$-subcategory of $\Fun(\shP_{\Sigma}(\shC), \shD)$ spanned by those functors which preserve sifted colimits. Moreover, a functor $F\in \Fun(\shP_{\Sigma}(\shC), \shD)$ preserves sifted colimits if and only if $F$ is a left Kan extension of $f = F \circ j$ along $j$. Furthermore, if $\shD$ admits all small colimits, then a functor $F \in \Fun_\Sigma(\shP_\Sigma(\shC), \shD)$ preserves all small colimits if and only if $f = F \circ j \colon \shC \to \shD$ preserves finite coproducts.
\end{proposition}
 
% \begin{proof} This is a reformulation of \cite[Propositions 5.5.8.10, 5.5.8.15]{HTT} in view of the adjunction provided by \cite[Corollary 5.3.6.10]{HTT}. \end{proof}

% Definition: nonabelian derived functors
\begin{definition}\label{def:nonab:derived} 
Let $\shC$ be a small $\infty$-category that admits finite coproducts, and $\shD$  an $\infty$-category that admits sifted colimits. For any functor $f \colon \shC \to \shD$ be any functor, we will refer to the essentially uniquely functor $F \in \Fun_{\Sigma}(\shP_{\Sigma}(\shC), \shD)$ of Proposition \ref{prop:nonab:derived} satisfying $f = F \circ j$ as the {\em derived functor} of $f$. By construction, the derived functor $F$ of $f$ is a left Kan extension of the functor $f = F| \shC$ and preserves all sifted colimits (hence \textit{a fortiori} preserves all filtered colimits and geometric realizations of simplicial objects).
\end{definition}

% Notations for dSchur
\begin{notation}[{See \cite[\S 25.1.1, \S 25.2.1]{SAG}}] 
\label{notation:SCRMod}
\begin{enumerate}[leftmargin=*]
	\item 
	\label{notation:SCRMod-1}
	We let ${\rm Poly} \subseteq {\rm CRing}$ denote the full subcategory spanned by (finite generated) polynomial rings $R = \ZZ[x_1, \ldots, x_m]$, where $m \ge 0$ is an integer. We let ${\rm PolyMod}^{\rm ff} \subseteq {\rm CRingMod}^{\rm ff}$ denote the full subcategory spanned by elements $(R, M)$, where $R = \ZZ[x_1, \ldots, x_m]$ is a polynomial ring, and $M= R^n$ is a finite free $R$-module, for some integers $m,n \ge 0$.

        \item
        \label{notation:SCRMod-2}
	 We let $\CAlgDelta = \shP_{\Sigma}({\rm Poly})$ and refer to its elements as the {\em simplicial commutative rings}. We let $\SCRModcn: = \CAlgDelta \times_{\CAlg} \Mod^\cn$ denote the $\infty$-category of pairs $(A, M)$, where $A$ is a simplicial commutative ring and $M$ is a connective $A^\circ$-module, where $A^\circ$ is the underlying $\EE_\infty$-ring spectrum of $A$. Then the inclusion ${\rm PolyMod}^{\rm ff} \hookrightarrow \SCRModcn$ extends to an equivalence of $\infty$-categories $\shP_\Sigma({\rm PolyMod}^{\rm ff}) \simeq \SCRModcn$ (\cite[Proposition 25.2.1.2]{SAG}).
        		 The natural forgetful functor $q \colon \SCRModcn \to \CAlgDelta$, $(A,M) \mapsto A$ is a {\em coCartesian fibration} (\cite[\href{https://kerodon.net/tag/01UA}{Tag 01UA}]{kerodon}) which classifies the functor $(A \in \CAlgDelta) \mapsto (\Modcn_A \in \widehat{\Cat}_\infty)$; a morphism $(A, M) \to (B, N)$ is $q$-coCartesian if and only if the natural induced morphism $B \otimes_A M \to N$ is an equivalence in $\Modcn_B$. 

	\item 
	\label{notation:SCRMod-3}
	Let $\shE$ denote the $\infty$-category $\Fun(\Delta^1, \SCRModcn) \times_{\Fun(\Delta^1, \CAlgDelta)} \CAlgDelta$ whose objects are pairs $(A, \rho \colon M' \to M)$, where $A$ is a simplicial commutative ring and $\rho$ is a morphism of connective $A$-modules. Let $\shE_0 \subseteq \shE$ be the full subcategory spanned by those pairs $(R, \rho \colon M' \to M)$, where $R$ is a polynomial ring $\ZZ[x_1, \ldots, x_k]$ and $\rho$ fits into a short exact sequence $0 \to M' \xrightarrow{\rho} M \to M'' \to 0$ of finitely generated free $R$-modules. Then the proof of \cite[Proposition 25.2.4.1]{SAG} shows that the inclusion $\shE_0 \hookrightarrow \shE$ induce an equivalence of $\infty$-categories $\shP_\Sigma(\shE_0) \simeq \shE$.
 	
	\item 
	\label{notation:SCRMod-4}
More generally, for an integer $n \ge 0$, we define $\shE^{[n]}$ by the pullback diagram:
		$$
	\begin{tikzcd}[column sep = 2 em]
		\shE^{[n]}: =\Fun(\Delta^n, \SCRModcn) \times_{\Fun(\Delta^n, \CAlgDelta)} \CAlgDelta  \ar{r} \ar{d}{q^{[n]}} & \Fun(\Delta^n, \SCRModcn) \ar{d}{q'} \\
		\CAlgDelta \ar{r}{\delta} & \Fun(\Delta^n, \CAlgDelta),
	\end{tikzcd}
	$$
where $q'$ is given by composition with the forgetful functor $q \colon \SCRModcn \to \CAlgDelta$, and $\delta$ is the diagonal morphism which carries $A$ to the identity sequence $A \xrightarrow{\id} A \xrightarrow{\id} \cdots \xrightarrow{\id} A$. Since the natural forgetful map
	$$\Fun(\Delta^n, \SCRModcn) \to \Fun({\rm Spine}[n], \SCRModcn)$$
is a trivial Kan fibration (where the ``spine" ${\rm Spine}[n]$ of $\Delta^n$ denotes the directed graph $0 \to 1 \to 2 \to \cdots \to n$, and the inclusion ${\rm Spine}[n] \subseteq \Delta^n$ is inner anodyne), we can represent elements of the category $\shE^{[n]}$ by pairs $(A, M^*)$, where $A \in \CAlgDelta$ is a simplicial commutative ring, and $M^* = (M^0 \to M^1 \to \cdots \to M^n) \in \shE_A^{[n]}$ is a sequence of morphisms in $\Modcn_A$. By definition, $\shE^{[1]} = \shE$ is the category of pairs $(A, \rho\colon M' \to M)$ defined in $(3)$. By virtue of 
\cite[{\href{https://kerodon.net/tag/01UF}{Tag 01UF}} \& {\href{https://kerodon.net/tag/01VG}{Tag 01VG}}]{kerodon},
%\cite[Remark 5.1.4.6 \& Theorem 5.2.1.1]{kerodon}, 
the forgetful functor $q^{[n]} \colon \shE^{[n]} \to \CAlgDelta$ is a {\em coCartesian fibration}. A morphism $(A, M^*) \to (B, N^*)$ in $\shE^{[n]} $ is coCartesian if and only if the natural morphism $B \otimes_A M^* \to N^*$ in $\Fun(\Delta^n, \Modcn_B)$ is an equivalence (that is, it induces equivalences $B \otimes_A M^i \to N^i$ for all $0 \le i \le n$).
\end{enumerate}
\end{notation}

\begin{remark}[Base-change properties] 
\label{rem:base-change:CAlgDelta}
%In what follows, we will use the following terminology. 
Let $q \colon \shC \to \CAlgDelta$ and $q' \colon \shC' \to \CAlgDelta$ be coCartesian fibrations, and $F \colon \shC \to \shC'$ be a functor such that the diagram
	$$
	\begin{tikzcd} 
		\shC \ar{d} \ar{r}{F}& \shC' \ar{d} \\
		\CAlgDelta \ar{r}{\id} & \CAlgDelta
	\end{tikzcd}
	$$
 commutes up to canonical equivalences. In practice, we will usually represent the functor $F$ by functorial assignments $A  \mapsto F_A$ for all $A \in \CAlgDelta$, where $F_A \colon \shC_A = \{A\} \times_{\CAlgDelta} \shC \to \shC'_A = \{A\} \times_{\CAlgDelta} \shC'$ are the induced functors on fibers. We say that the functor $F$ (or the formation of the functors $F_A \colon \shC_A \to \shC'_A$, $A \in \CAlgDelta$) {\em commutes with base change of simplicial commutative rings}, if $F$ carries $q$-coCartesian morphisms to $q'$-coCartesian morphisms. For example, if $\shC = \shC' = \SCRModcn$, then a functor $F \colon \shC \to \shC'$ commutes with base change of simplicial commutative rings if and only if for any morphism $A \to B$ in $\CAlgDelta$ and any $M \in \Modcn_A$, the natural map $B \otimes_A F(M) \to F(B \otimes_A M)$ is an equivalence. 
\end{remark}

% sec: derived Schur
 \subsection{Derived Schur and Weyl Functors}
 \label{sec:dSchurdWeyl}
This subsection defines and studies the derived Schur functors $\dSchur^{\lambda/\mu}$ and derived Weyl functors $\dWeyl^{\lambda/\mu}$, which are derived functors of the classical Schur functors $\Schur^{\lambda/\mu}$ and Weyl functors $\Weyl^{\lambda/\mu}$ of \S \ref{sec:CSchurWeyl}, respectively, in the sense of Lurie's non-abelian derived theory (reviewed in \S \ref{sec:non-abelian}). 

Additionally, we will also study the auxiliary functors $\LL^{\lambda/\mu}(\blank,\blank)_k$, which are derived functors of the $k$th component functors $\bSchur^{\lambda/\mu}(\blank, \blank)_k$ of Schur complexes. However, it is safe to skip all the statements regarding these auxiliary functors $\LL^{\lambda/\mu}(\blank,\blank)_k$ if readers aren't interested in the details of the proofs for properties of derived Schur and Weyl functors. In fact, as we will prove in Proposition \ref{prop:bSchur_k=Schur_oplus_k}, the functor $(A,M,M') \mapsto \LL_A^{\lambda/\mu}(M,M')_k$ is nothing else than a degree shift of the bidegree-$(N-k,k)$ homogeneous component of the functor $(A, M,M') \mapsto \dSchur_A^{\lambda/\mu}(M \oplus M'[1])$.

% Definition: dSchur and dWeyl
\begin{definition}[Derived Schur functors and Derived Weyl functors]
\label{def:dSchurWeyl}
In the situation of Proposition \ref{prop:nonab:derived}, we let $\shC = {\rm PolyMod}^{\rm ff}$ (Notation \ref{notation:SCRMod}; so that $\shP_{\Sigma}(\shC) = \SCRModcn$), and let $\shD = \SCRModcn$. Let $\lambda/\mu$ be a skew partition, then the classical Schur and Weyl functors (Definition \ref{def:SchurWeyl}) restrict to functors $\Schur^{\lambda/\mu}|{\rm PolyMod}^{\rm ff} \colon  {\rm PolyMod}^{\rm ff} \to {\rm PolyMod}^{\rm ff} \subseteq \SCRModcn$ and $\Weyl^{\lambda/\mu}|{\rm PolyMod}^{\rm ff} \colon  {\rm PolyMod}^{\rm ff} \to {\rm PolyMod}^{\rm ff} \subseteq \SCRModcn$, respectively, which commute with the natural forgetful functors to ${\rm Poly}$ and carry coCartesian morphisms to coCartesian morphisms (Theorem \ref{thm:Schur:free}). We let 
	$$\dSchur^{\lambda/\mu} \colon \SCRModcn \to \SCRModcn \quad \text{resp.} \quad \dWeyl^{\lambda/\mu} \colon \SCRModcn \to \SCRModcn$$
	$$(A, M) \mapsto (A, \dSchur^{\lambda/\mu}_A(M)) \quad \text{resp.} \quad (A, M) \mapsto (A, \dWeyl^{\lambda/\mu}_A(M))$$
denote the derived functors of $\Schur^{\lambda/\mu}|{\rm PolyMod}^{\rm ff}$, and respectively $\Weyl^{\lambda/\mu}|{\rm PolyMod}^{\rm ff}$, in the sense of Definition \ref{def:nonab:derived}. We will refer to $\dSchur^{\lambda/\mu}$ (resp. $\dWeyl^{\lambda/\mu}$) as the {\em derived Schur functor} (resp. {\em derived Weyl functor}) associated with the skew partition $\lambda/\mu$, and $\dSchur^{\lambda/\mu}_A(M)$ (resp. $\dWeyl^{\lambda/\mu}_A(M)$) as the {\em derived Schur power}  (resp. {\em derived Weyl power}) {\em of $M$}. By construction, $\dSchur^{\lambda/\mu}$ and $\dWeyl^{\lambda/\mu}$ are left Kan extensions of $\dSchur^{\lambda/\mu}|{\rm PolyMod}^{\rm ff} = \Schur^{\lambda/\mu}|{\rm PolyMod}^{\rm ff}$ and $\dWeyl^{\lambda/\mu}|{\rm PolyMod}^{\rm ff}=\Weyl^{\lambda/\mu}|{\rm PolyMod}^{\rm ff}$, respectively. Furthermore, $\dSchur^{\lambda/\mu}$ and $\dWeyl^{\lambda/\mu}$ preserve sifted colimits, and the diagrams 
	$$
	\begin{tikzcd}
		\SCRModcn \ar{r}{\dSchur^{\lambda/\mu}} \ar{d}{q} & \SCRModcn \ar{d}{q} \\
		\CAlgDelta \ar{r}{\id} & \CAlgDelta
	\end{tikzcd}
	\quad
	\begin{tikzcd}
		\SCRModcn \ar{r}{\dWeyl^{\lambda/\mu}} \ar{d}{q} & \SCRModcn \ar{d}{q} \\
		\CAlgDelta \ar{r}{\id} & \CAlgDelta
	\end{tikzcd}
	$$
commute up to canonical equivalences (here, $q$ is the canonical forgetful functor which carries $(A,M)$ to $A$). If the simplicial commutative ring $A$ is clear from the context, we will drop the subscript $A$ in the notations $\dSchur^{\lambda/\mu}_A(M)$ and $\dWeyl^{\lambda/\mu}_A(M)$, and write $\dSchur^{\lambda/\mu}(M)$ and $\dWeyl^{\lambda/\mu}(M)$ instead.
Following the same convention as Definition \ref{def:SchurWeyl} , we have $\dSchur_A^{\lambda/\mu}(M) = \dWeyl_A^{\lambda/\mu}(M) =A$ if $\lambda=\mu$ and $M \neq 0$, and $\dSchur_A^{\lambda/\mu}(M) =\dWeyl_A^{\lambda/\mu}(M)=0$ whenever $M=0$. 
\end{definition}

% Examples of dSchur and dWeyl
\begin{example} Let $n$ be a positive integer, $(A, M) \in \SCRModcn$.
\begin{enumerate}[label=(\roman*), leftmargin=*]
	\item If $\lambda = (n)$, $\mu=(0)$, then $\dSchur_A^{(n)}(M)=\Sym_{A}^n (M)$ is the {\em $n$th derived symmetric power} of $M$, and $\dWeyl_A^{(n)}(M)=\Gamma_{A}^n (M)$ is the {\em $n$th derived divided power} of $M$. 
	\item If $\lambda = (1^n) := \underbrace{(1, \ldots, 1)}_{n \,\text{terms}}$, $\mu=(0)$, then $\dSchur_A^{(1^n)}(M) = \dWeyl_A^{(1^n)}(M) =\bigwedge_{A}^n (M)$ are both equal to the {\em $n$th derived exterior power} of $M$.
	\item If $\lambda= (n,n-1, \ldots, 1)$ and $\mu= (n-1, n-2, \ldots, 1)$, then $\dSchur_A^{\lambda/\mu}(M) = \dWeyl_A^{\lambda/\mu}(M) =\bigotimes_{A}^n (M)$ are both equal to the {\em $n$th derived tensor product} of $M$.
	\item If $\lambda = (2,1)$ and $\mu=(0)$, then $\dSchur_A^{(2,1)} (M)$ fits into a canonical equivalence 
		\begin{align*}
\cofib \Big( \bigwedge\nolimits_A^3(M) \xrightarrow{\psi} \bigwedge\nolimits_A^2(M) \otimes_A M \Big) \xrightarrow{\sim} \dSchur^{(2,1)}_A(M)
		\end{align*}
%where $\psi=\square_{(2,1)^t} = \Delta \colon \bigwedge\nolimits_A^3(M) \to \bigwedge\nolimits_A^2(M) \otimes_A M$ is the comultiplication map of the derived exterior power algebra, extending the classical ones. 
	and $\dWeyl_A^{(2,1)} (M)$ fits into a canonical equivalence 
				\begin{align*}
\cofib \Big( \Gamma_A^3(M) \xrightarrow{\psi'} \Gamma_A^{2}(M) \otimes_A M \Big) \xrightarrow{\sim} \dWeyl^{(2,1)}_A(M).
		\end{align*}
%where $\psi'=\square'_{(2,1)} = \Delta'' \colon \Gamma_A^3(M) \to \Gamma_A^2(M) \otimes_A M$ is the comultiplication map of the derived divided power algebras, extending the classical ones. 
	\end{enumerate}
\end{example}

\begin{remark}[Characteristic-zero cases] Let $(R,M) \in {\rm CRingMod}^{\rm fproj}$, if $R$ is a $\QQ$-algebra, then the canonical map $\Weyl^{\lambda/\mu}_R(M) \to \Schur^{\lambda/\mu}_R(M)$ is an isomorphism. Hence we obtain a canonical equivalence $\dWeyl_A^{\lambda/\mu}  (M) \xrightarrow{\sim} \dSchur_A^{\lambda/\mu} (M)$ for all $A \in \CAlgDelta_\QQ$ and $M \in \Modcn_A$.
\end{remark}

% Definition: derived kth component
\begin{definition}[Derived functors of $k$th component functor of Schur complexes] 
Let $\lambda/\mu$ be a skew partition, and let $k \ge 0$ be an integer. We let $\shC = {\rm PolyMod}^{\rm ff} \times_{\rm Poly} {\rm PolyMod}^{\rm ff}$ and $\shD = \SCRModcn$ (Notation \ref{notation:SCRMod}) in the situation of Proposition \ref{prop:nonab:derived}. Then the $k$th component of Schur complex functor $\bSchur^{\lambda/\mu}$ defines a functor $\bSchur^{\lambda/\mu}(\blank, \blank)_k|\shC  \colon  \shC \to {\rm PolyMod}^{\rm ff} \subseteq \SCRModcn$ (Corollary \ref{cor:bSchurk}) which commute with the natural forgetful functors to ${\rm Poly}$ and carry coCartesian morphisms to coCartesian morphisms. We let 
	$$\LL^{\lambda/\mu}(\blank,\blank)_k \colon \SCRModcn \times_{\CAlgDelta} \SCRModcn \to \SCRModcn$$
	$$(A, M, M') \mapsto (A, \LL_A^{\lambda/\mu}(M, M')_k)$$
denote the derived functor of $\bSchur^{\lambda/\mu}(\blank, \blank)_k|\shC$ in the sense of Definition \ref{def:nonab:derived}. Then $\LL^{\lambda/\mu}(\blank,\blank)_k$ is a left Kan extension of $\bSchur^{\lambda/\mu}(\blank, \blank)_k|\shC$, preserves sifted colimits, and the diagram
	$$
	\begin{tikzcd}[column sep = 5 em]
		\SCRModcn \times_{\CAlgDelta} \SCRModcn \ar{r}{\LL^{\lambda/\mu}(\blank,\blank)_k} \ar{d}{v} & \SCRModcn \ar{d}{q} \\
		\CAlgDelta \ar{r}{\id} & \CAlgDelta
	\end{tikzcd}
	$$
commute up to canonical equivalences; here, $v$ is the canonical forgetful functor which carries $(A,M, M')$ to $A$. Notice that $\LL_A^{\lambda/\mu}(M, M')_{k} =0$ whenever $k \notin [0, |\lambda|-|\mu|]$.
\end{definition}

\begin{remark} In the above situation, we have $\LL_A^{\lambda/\mu}(M,0)_k = 0$ for any $k \ne 0$, and $\LL_A^{\lambda/\mu}(M,0)_0 = \dSchur^{\lambda/\mu}_A(M)$. Similarly, $\LL_A^{\lambda/\mu}(0,M')_k = 0$ for any $k \ne |\lambda| - |\mu|$, and $\LL_A^{\lambda/\mu}(0,M')_{|\lambda|-|\mu|} = \dWeyl^{\lambda^t/\mu^t}_A(M')$. 
\end{remark}

% Base change
 \begin{proposition}[Base change] 
 \label{prop:dSchur:basechange}
 For any skew partition $\lambda/\mu$, the derived Schur functor $\dSchur^{\lambda/\mu}$ (resp. the derived Weyl functor $\dWeyl^{\lambda/\mu}$) carries $q$-coCartesian morphisms to $q$-coCartesian morphisms. In other words, for any map $A \to B$ of simplicial commutative rings and any connective complex $M \in \Modcn_A$, there are canonical equivalences
 	$$B \otimes_A \dSchur_A^{\lambda/\mu}(M) \xrightarrow{\sim} \dSchur_B^{\lambda/\mu}(B \otimes_A M) \qquad B \otimes_A \dWeyl_A^{\lambda/\mu}(M) \xrightarrow{\sim} \dWeyl_B^{\lambda/\mu}(B \otimes_A M).$$	
Similarly, for any skew partition $\lambda/\mu$ and any integer $k \ge 0$, the derived functor $\LL^{\lambda/\mu}(\blank, \blank)_k$ carries $v$-coCartesian morphisms to $q$-coCartesian morphisms. In other words, for any map $A \to B$ of simplicial commutative rings and any pair of connective complexes $M, M' \in \Modcn_A$, there are canonical equivalences
 	$$B \otimes_A \LL_A^{\lambda/\mu}(M, M')_k \xrightarrow{\sim}  \LL_B^{\lambda/\mu}(B \otimes_A M, B \otimes_A M')_k.$$
\end{proposition}
 \begin{proof}
We use the same strategy as the proof of \cite[Proposition 25.2.3.1]{SAG}. We prove the assertion for $\LL^{\lambda/\mu}(\blank, \blank)_k$; and the other cases are similar and simpler. Let $\alpha_{B, M,M'} \colon B \otimes_A \LL_A^{\lambda/\mu}(M, M')_k \to \LL_B^{\lambda/\mu}(B \otimes_A M, B \otimes_A M')_k$ denote the canonical morphism, and we wish to show that $\alpha_{B, M,M'}$ is an equivalence. Since the functor $(M,M') \mapsto \alpha_{B,M,M'}$ commutes with sifted colimits, we are reduced to the case where $M= A \otimes_\ZZ \ZZ^m$ and $M' = A \otimes_\ZZ \ZZ^n$ for some integers $m,n \ge 0$. From the commutative diagram
		$$
	\begin{tikzcd}[column sep = -0.8 em, row sep = 1 em]
		 & B\otimes_A  \LL_{A}^{\lambda/\mu}(A \otimes_\ZZ \ZZ^m, A \otimes_\ZZ \ZZ^n) \ar{dr}{\alpha_{B,M,M'}}& \\
		B\otimes_A A \otimes_{\ZZ} \LL_{\ZZ}^{\lambda/\mu}(\ZZ^m, \ZZ^n) \ar{rr} \ar{ru} & & \LL_{B}^{\lambda/\mu}(B \otimes_A A \otimes_\ZZ \ZZ^m, B \otimes_A A  \otimes_\ZZ \ZZ^n),
	\end{tikzcd}
	$$
to prove that $\alpha_{B,M,M'}$ is an equivalence, it suffices to prove in the case where $A = \ZZ$. Furthermore, when fixing $A=\ZZ$, $M=\ZZ^m$ and $M'=\ZZ^n$, the functor $ B \mapsto \alpha_{B, M, M'}$ commutes with sifted colimits, hence we are reduced to the case where $B = \ZZ[x_1, \ldots, x_s]$ for some integer $s \ge 0$. Then the desired equivalence follows from Corollary \ref{cor:bSchurk}. 

The cases for derived Schur and Weyl functors are proved similarly, by reducing to the case where $A=\ZZ$, $B = \ZZ[x_1, \ldots, x_s]$ and $M=\ZZ^n$, for which the desired equivalences follow from the universal freeness of the classical Schur and Weyl functors (Theorem \ref{thm:Schur:free}).
\end{proof}

\begin{proposition}[Freeness and Flatness]
 \label{prop:dSchur:free} 
 For any skew partition $\lambda/\mu$ and any simplicial commutative ring $A$, if $M \in \Modcn_A$ is finite free (resp. locally finite free, resp. flat), then $\dSchur_A^{\lambda/\mu}(M)$ and $\dWeyl_A^{\lambda/\mu}(M)$ are both  finite free (resp. locally finite free, resp. flat). Similarly, for any integer $k \ge 0$, if $M, M' \in \Modcn_A$ are both finite free (resp. locally finite free, resp. flat), then $\LL_A^{\lambda/\mu}(M,M')_k$ is finite free (resp. locally finite free, resp. flat).
\end{proposition}

\begin{proof}
To prove the assertions about (locally) finite freeness for the derived Schur and Weyl functors (resp. the functor $\LL^{\lambda/\mu}(M,M')$), by virtue of the base change property Proposition \ref{prop:dSchur:basechange}, we are reduced to the cases $A=\ZZ$, $M= \ZZ^m$ (resp. $A=\ZZ$, $M= \ZZ^m$, $M'= \ZZ^n$). In the case, the assertion follows from Theorem \ref{thm:Schur:free} (resp. Corollary \ref{cor:bSchurk}). Finally, the assertions about flatness follow from the above cases of finite free modules by using Lazard’s Theorem (\cite[Theorem 7.2.2.15]{HA}) and the fact that these functors preserve filtered colimits.
\end{proof}

 \begin{remark}[Vanishing criteria] Let $A$ be a simplicial commutative ring.
\begin{enumerate}[leftmargin=*]
	\item If $M$ is a (locally) free $A$-module of finite $\rank M \ge 1$, then the proposition implies that $\dSchur^{\lambda/\mu}(M) \ne 0$ if and only if $\dWeyl^{\lambda/\mu}(M) \ne 0$, if and only if $0 \le \lambda_j^t - \mu_j^t \le \rank M$ for all $j$.
	\item Similarly, suppose $M$ and $M'$ are (locally) free $A$-modules of finite positive ranks $\rank M$ and $\rank M'$, respectively. We define two partitions $\gamma_{\min}$ and $\gamma_{\max}$ by the following formulae:
	$$(\gamma_{\min})_i : = \max\{\mu_i, \lambda_i - \rank M'\} \qquad (\gamma_{\max}^t)_j = \max \{\lambda_j^t, \mu_j^t + \rank M\}.$$
In particular, the partition $\gamma_{\min}$ depends only on $\lambda/\mu$ and $\rank M'$, and $\gamma_{\max} $ on $\lambda/\mu$ and $\rank M$.
Then for any integer $k \in [0,  |\lambda| - |\mu|]$, we have $\bSchur_A^{\lambda/\mu}(M,M')_k \ne 0$ if and only if there exists a partition $\gamma$ such that $\gamma_{\min} \subseteq \gamma \subseteq \gamma_{\max}$ and $|\lambda| - |\gamma| = k$.
\end{enumerate}
\end{remark}

\begin{proposition}
\label{prop:dSchur:flat}
Let $R$ be an ordinary commutative ring and $\lambda/\mu$ a skew partition. Then for any flat $R$-module $M$, there are canonical equivalences
	$$\dSchur_R^{\lambda/\mu}(M) \xrightarrow{\sim} \Schur_R^{\lambda/\mu}(M) \quad \text{and} \quad \dWeyl_R^{\lambda/\mu}(M) \xrightarrow{\sim} \Weyl_R^{\lambda/\mu}(M),$$
where $\Schur_R^{\lambda/\mu}(M)$ and $\Weyl_R^{\lambda/\mu}(M)$ are the classical Schur and Weyl modules of Definition \ref{def:SchurWeyl}. 
\end{proposition}
 
\begin{proof}
By construction, the two functors
	${\rm CRingMod}^\heartsuit \to \SCRModcn$
	$$(R,M) \mapsto (R, \dSchur_R^{\lambda/\mu}(M)) \quad \text{and} \quad (R,M) \mapsto (R, \Schur_R^{\lambda/\mu}(M))$$
agree on the subcategory ${\rm CRingMod}^{\rm ff}$. Since $\dSchur^{\lambda/\mu}$ is a Left Kan extension from its restriction $\dSchur^{\lambda/\mu}|{\rm CRingMod}^{\rm ff}$, there is an essentially unique natural transformation $\alpha$ between these two functors which induces functorial morphisms $\alpha_{R,M} \colon \dSchur_{R}^{\lambda/\mu}(M) \to \Schur_{R}^{\lambda/\mu}(M)$ for all $(R,M) \in {\rm CRingMod}^\heartsuit$ which satisfy $\alpha_{R,M} = \id$ if $(R,M) \in {\rm CRingMod}^{\rm ff}$. We wish to show that $\alpha_{R,M}$ is an equivalence if $M$ is a flat $R$-module. Since the functor $M \mapsto \alpha_{R,M}$ commutes with filtered colimits, by Lazard’s Theorem (\cite[Theorem 7.2.2.15]{HA}) we are reduced to the cases where $M = R \otimes_\ZZ \ZZ^n$ for some $n \ge 0$. In this case, the desired equivalence follows from Proposition \ref{prop:dSchur:basechange}. The proof for Weyl functors is similar. 
\end{proof}

\begin{remark} Similarly, if $R$ is an ordinary commutative ring, for any pair of flat $R$-modules $M, M'$ and any integer $k \ge 0$, there is a canonical equivalence
	$$\LL_{R}^{\lambda/\mu}(M,M')_k \xrightarrow{\sim} \bSchur_R^{\lambda/\mu}(M,M')_k,$$
where, if $M,M'$ are not finitely generated projective modules, the $k$th component $\bSchur_R^{\lambda/\mu}(M,M')_k$ of Schur complexes is defined via the cokernel construction of Remark \ref{rmk:bSchur:coker}.
\end{remark}

The following result shows that we are able to recover classical Schur and Weyl functors from classical truncations of the derived Schur and Weyl functors.

% Prop: classical truncations
\begin{proposition}[Classical Truncation]
\label{prop:dSchur:classical}
Let $\lambda/\mu$ be a skew partition,  then for any $(A,M) \in \SCRModcn$, there are canonical equivalences 
	$$ \pi_0(\dSchur_A^{\lambda/\mu}(M)) \xrightarrow{\sim} \Schur_{\pi_0 A}^{\lambda/\mu}(\pi_0(M)) \quad \text{and} \quad \pi_0( \dWeyl_A^{\lambda/\mu}(M)) \xrightarrow{\sim} \Weyl_{\pi_0 A}^{\lambda/\mu}(\pi_0(M)),$$
where the right-hand sides denote the classical Schur and Weyl modules of Definition \ref{def:SchurWeyl}.
\end{proposition}
 
\begin{proof} We prove for derived Weyl functors; the proof for derived Schur functors are similar. By applying Proposition \ref{prop:dSchur:basechange} to the canonical map $A \to \pi_0 A$, we obtain equivalences 
	$$\pi_0(A) \otimes_A \dWeyl_A^{\lambda/\mu}(M) \xrightarrow{\sim} \dWeyl_{\pi_0(A)}^{\lambda/\mu}(\pi_0 A \otimes_A M).$$
Since $\pi_0 (\dWeyl_A^{\lambda/\mu}(M) ) \simeq \pi_0(\dWeyl_{\pi_0(A)}^{\lambda/\mu}(\pi_0 A \otimes_A M))$ and $\pi_0(M) \simeq \pi_0(\pi_0 A \otimes_A M)$, in order to prove the proposition, we may assume $A= R$ is an ordinary commutative ring and $M \in \Modcn_R$ is any connective complex. We let $P_* = \cdots \to P_2 \to P_1 \xrightarrow{\rho} P_0$ be a projective resolution of $M$, where $P_i$ are projective $R$-modules, and let $Q_\bullet :={\rm DK}(P_*)$ denote the simplicial projective resolution of $M$ obtained by applying the Dold--Kan correspondence functor to $P_*$ (\cite[Theorem 1.2.3.7]{HA}). Consequently, the discrete $R$-module $\pi_0 M$ is a coequalizer of the diagram 
	$$
	\begin{tikzcd}
		 P_0 \oplus P_1 \ar[shift left]{r}{d^0} \ar[shift right]{r}[swap]{d^1}  & P_0
	\end{tikzcd} 
	\quad \text{where} \quad d^0 (x,y)= x, d^1(x,y) =x + \rho(y) \quad\text{for}\quad x \in P_0, y \in P_1.
	$$
	% The reflexive diagram
	\begin{comment}
	$$
	\begin{tikzcd}
		 P_0 \oplus P_1 \ar[shift left=2]{r}{d^0} \ar[shift right=2]{r}[swap]{d^1}  & P_0 \ar{l}[description]{s}
	\end{tikzcd} 
	\quad \text{where} \quad d^0 (x,y)= x, d^1(x,y) =x + \rho(y), s(x)=x \text{~for~} x \in P_0, y \in P_1.
	$$
	\end{comment}
By virtue of %Since the classical divided power functor preserves coequalizers (
\cite[IV.10, p. 284]{Ro}, we obtain for each $d \ge 0$ a coequalizer sequence
	$$
	\begin{tikzcd}[column sep=4em]
		 \Gamma_R^{d}(P_0 \oplus P_1) \ar[shift left]{r}{\Gamma_R^d(d^0)} \ar[shift right]{r}[swap]{\Gamma_R^d(d^1)}  &  \Gamma_R^{d}(P_0) \ar{r}{\Gamma_R^d(q)}& \Gamma_{\cl, R}^d(\pi_0(M)),
	\end{tikzcd} 
	$$
where $q \colon P_0 \to \pi_0(M)$ denotes the canonical quotient map. Inductively, we obtain exact sequences for all skew partitions $\lambda/\mu$:
	$$\Gamma_R^{\lambda/\mu}(P_0 \oplus P_1) \xrightarrow{\Gamma_R^{\lambda/\mu}(d^1) - \Gamma_R^{\lambda/\mu}(d^0)} \Gamma_R^{\lambda/\mu}(P_0) \xrightarrow{\Gamma_R^{\lambda/\mu}(q)} \Gamma_{\cl, R}^{\lambda/\mu}(\pi_0(M)) \to 0$$
	$$\widetilde{\Gamma}_R^{\lambda/\mu}(P_0 \oplus P_1) \xrightarrow{\widetilde{\Gamma}_R^{\lambda/\mu}(d^1) - \widetilde{\Gamma}_R^{\lambda/\mu}(d^0)} \widetilde{\Gamma}_R^{\lambda/\mu}(P_0) \xrightarrow{\widetilde{\Gamma}_R^{\lambda/\mu}(q)} \widetilde{\Gamma}_{\cl, R}^{\lambda/\mu}(\pi_0(M)) \to 0$$
where the functors $\Gamma^{\lambda/\mu}$ and $\widetilde{\Gamma}^{\lambda/\mu}$ are defined in Notation \ref{notation:square}. Consider the following commutative diagram	
	\begin{equation}
	\label{diag:pizerodWeyl=Weylpizero}
	\begin{tikzcd}[column sep= 5em]
		\widetilde{\Gamma}_R^{\lambda/\mu}(P_0 \oplus P_1) \ar{r}{\square'_{\lambda/\mu}(P_0 \oplus P_1)} \ar{d}{\widetilde{\Gamma}_R^{\lambda/\mu}(d^0) - \widetilde{\Gamma}_R^{\lambda/\mu}(d^1)} & 	\Gamma_R^{\lambda/\mu}(P_0 \oplus P_1) \ar{r} \ar{d}{\Gamma_R^{\lambda/\mu}(d^0) - \Gamma_R^{\lambda/\mu}(d^1)}  &  \Weyl_R^{\lambda/\mu}(P_0 \oplus P_1)  \ar{d}{\Weyl_R^{\lambda/\mu}(d^0) - \Weyl_R^{\lambda/\mu}(d^1)} \ar{r}& 0 \\
		\widetilde{\Gamma}_R^{\lambda/\mu}(P_0)  \ar{r}{\square'_{\lambda/\mu}(P_0)}  \ar{d}{\widetilde{\Gamma}_R^{\lambda/\mu}(q)} & {\Gamma}_R^{\lambda/\mu}(P_0)  \ar{r} \ar{d}{{\Gamma}_R^{\lambda/\mu}(q)}   & \Weyl_R^{\lambda/\mu}(P_0)  \ar{d}{\Weyl_R^{\lambda/\mu}(q)} \ar{r}& 0 \\
		\widetilde{\Gamma}_{\cl, R}^{\lambda/\mu}(\pi_0 M)	\ar{r}{\square'_{\lambda/\mu}(\pi_0 M)} \ar{d}  &{\Gamma}_{\cl,R}^{\lambda/\mu}(\pi_0 M) \ar{r} \ar{d}	 & \Weyl_R^{\lambda/\mu}(\pi_0 M) \ar{r} \ar{d}& 0 \\
		0 & 0 & 0,
	\end{tikzcd}
	\end{equation}
where first two columns are exact by the above arguments, and all three rows are  exact by the definition of classical Weyl modules (Definition \ref{def:SchurWeyl}). Consequently, the last column is also exact. On the other hand, since $\dWeyl_R^{\lambda/\mu}$ preserves geometric realizations of simplicial objects, $\pi_0 (\dWeyl_R^{\lambda/\mu}(M))$ is canonically isomorphic to the coequalizer of the diagram
	$$
	\begin{tikzcd}[column sep= 5em]
		 \Weyl_R^{\lambda/\mu}(P_0 \oplus P_1) \ar[shift left]{r}{\Weyl_R^{\lambda/\mu}(d^1)} \ar[shift right]{r}[swap]{\Weyl_R^{\lambda/\mu}(d^0)}  & \Weyl_R^{\lambda/\mu}(P_0).
	\end{tikzcd}
	$$
Comparing above with the last column of \eqref{diag:pizerodWeyl=Weylpizero}, we obtain the desired equivalence. 
\end{proof}

\begin{remark}
An alternative proof of the equivalence $ \pi_0(\dSchur_R^{\lambda/\mu}(M)) \simeq \Schur_R^{\lambda/\mu}(M)$ can be given by combining Proposition \ref{prop:dSchur_vs_bSchur} and \cite[Proposition V.2.2]{ABW} when $R$ is an ordinary commutative ring and $M$ is a discrete coherent $R$-module. This strategy, however, doesn't work in the case of Weyl functors.
\end{remark}

% sec: Universal fiber sequences
\subsection{Universal Sequences Associated with Derived Schur and Weyl Functors}
 \label{sec:univ.fib:dSchur}
This subsection generalizes the classical theory of canonical filtrations associated with Schur and Weyl functors (\S \ref{sec:univ.fil:Schur}) to the derived setting. We obtain results such as derived versions of Cauchy Decomposition Formula for derived symmetric and exterior powers (Theorem \ref{thm:fil:dsym_otimes}), Direct-Sum Decomposition Formula for derived Schur and Weyl powers of direct sums of complexes (Theorem \ref{thm:fib:dSchur_oplus}), and Littlewood--Richardson Rule for tensor products of derived Schur and Weyl functors (Theorem \ref{thm:fil:dSchur_LR} and Corolloary \ref{cor:fil:dSchur_LR}). Additionally, there are canonical sequence of morphisms of Koszul type associated with derived Schur functors (Theorem \ref{thm:fib:dbSchur}) which are vast generalizations of the Koszul complexes for classical symmetric powers.

It is worth noting that, in the following results, the assertion about the existence of certain universal sequences exist is sufficient for all of this paper's subsequent applications. However, we specified the orders of these sequences because they may be useful in future computations.

The following generalizes the classical Cauchy Decomposition Formula {Theorem \ref{thm:fil:sym_otimes}

\begin{theorem}[Cauchy Decomposition Formula] 
 \label{thm:fil:dsym_otimes}
Let $n$ be a positive integer, and let 
	$$\lambda^{0} = (1^n) < \lambda^{1} \cdots < \lambda^{p-2} < \lambda^{p-1}=(n)$$
be all partitions of $n$ in lexicographic order, where $p = p(n)$ is the number of partitions of $n$. 

\begin{enumerate}
	\item 
	 \label{thm:fil:dsym_otimes-1}
	There exists a canonical functor
			$$\SCRModcn \times_{\CAlgDelta} \SCRModcn \to \shE^{[p-1]}$$
			$$(A, M, M') \mapsto (A, F_A^{0,n}(M,M') \to  F_A^{1,n}(M,M') \to \cdots \to F_A^{p-1,n}(M,M'))$$
		which preserves sifted colimits and commutes with base change of simplicial commutative rings, such that there are canonical equivalences
			$$F_A^{p-1,n}(M,M') = \Sym_A^n(M \otimes_A M')  \qquad F_A^{0,n}(M,M') \simeq \bigwedge\nolimits_A^n M \otimes_A \bigwedge\nolimits_A^n M'$$
			$$\cofib\big(F_A^{i-1,n}(M,M') \to F_A^{i,n}(M,M')\big) \simeq \dSchur_A^{\lambda^i}(M) \otimes_A \dSchur_A^{\lambda^i}(M') \quad \text{for} \quad 1 \le i \le p-1.$$
		%(Here, we let $F_A^{-1,n}(M,M')=0$ be the zero functor.) 
		Moreover, the formations of the above equivalences are functorial with respect to $(A, M,M')$ and commute with base change of simplicial commutative rings. 
	\item 
	\label{thm:fil:dsym_otimes-2}
	There exists a canonical functor
			$$\SCRModcn \times_{\CAlgDelta} \SCRModcn \to \shE^{[p-1]}$$
			$$(A, M, M') \mapsto (A, G_A^{0,n}(M,M') \to  G_A^{1,n}(M,M') \to \cdots \to G_A^{p-1,n}(M,M'))$$
		which preserves sifted colimits and commutes with base change of simplicial commutative rings, such that there are canonical equivalences
			$$G_A^{p-1,n}(M,M') = \bigwedge\nolimits_A^n(M \otimes_A M')  \qquad G_A^{0,n}(M,M') \simeq \bigwedge\nolimits_A^n M \otimes_A \Gamma_A^n M'$$
			$$\cofib\big(G_A^{i-1,n}(M,M') \to G_A^{i,n}(M,M')\big) \simeq \dSchur_A^{\lambda^i}(M) \otimes_A \dWeyl_A^{(\lambda^i)^t}(M') \quad \text{for} \quad 1 \le i \le p-1.$$
		%(Here, we let $F_A^{-1,n}(M,M')=0$ be the zero functor.) 
		Moreover, the formations of the above equivalences are functorial with respect to $(A, M,M')$ and commute with base change of simplicial commutative rings. 
\end{enumerate}
\end{theorem} 

\begin{proof}
We apply Proposition \ref{prop:nonab:derived} to the situation where $\shC = {\rm PolyMod}^{\rm ff} \times_{\rm Poly} {\rm PolyMod}^{\rm ff}$, $\shD = \shE^{[p-1]}$ (Notation \ref{notation:SCRMod}), and $f \colon \shC \to \shD$ is the functor
	$$(R, M, M') \mapsto (R, F_R^{0}(M,M') \to  F_R^{1}(M,M') \to \cdots \to F_R^{p-1}(M,M')),$$
where $F_R^{0}(M,M') \to \cdots \to F_R^{p-1}(M,M')$ is the sequence of inclusion morphisms constructed in Theorem \ref{thm:fil:sym_otimes} \eqref{thm:fil:sym_otimes-1}. Then Proposition \ref{prop:nonab:derived} implies there is an essentially unique functor $\shP_{\Sigma}(\shC) \to \shD$ extending $f$ with the desired properties as described in assertion \eqref{thm:fil:dsym_otimes-1} except the base-change properties; The assertion that this functor commutes with base change follows from the same argument of the proof of Proposition \ref{prop:dSchur:basechange}. Assertion \eqref{thm:fil:dsym_otimes-2} is proved similarly. 
\end{proof}

\begin{example}
Let $A \in \CAlgDelta$ and $M, M' \in \Modcn_A$. 
\begin{enumerate}[leftmargin=*]
	\item If $n=2$, we obtain canonical cofiber sequences
	$$\bigwedge\nolimits_A^2 (M) \otimes_A \bigwedge\nolimits_A^2 (M') \to \Sym_A^2(M \otimes_A M') \to \Sym_A^2(M) \otimes_A \Sym_A^2(M').$$
	$$\bigwedge\nolimits_A^2 (M )\otimes_A \Gamma_A^2 (M') \to \bigwedge\nolimits_A^2(M \otimes_A M') \to \Sym_A^2(M) \otimes_A \bigwedge\nolimits_A^2(M').$$
	\item If $n=3$, we obtain canonical sequences of morphisms in $\Modcn_A$,
	$$(F^0_A(M,M') \to F^1_A(M,M') \to F^2_A(M,M')) \in \Fun(\Delta^2, \Modcn_A),$$
	$$(G^0_A(M,M') \to G^1_A(M,M') \to G^2_A(M,M')) \in \Fun(\Delta^2, \Modcn_A),$$
for which there are canonical equivalences:
	$$F^0_A(M,M') =  \bigwedge\nolimits_A^3(M) \otimes_A \bigwedge\nolimits_A^3(M')  \qquad F^2_A(M,M') = \Sym_A^3(M \otimes_A M'),$$
	$$\cofib(F^0_A(M,M') \to F^1_A(M,M')) \simeq \dSchur_A^{(2,1)}(M) \otimes_A \dSchur_A^{(2,1)}(M'),$$
	$$\cofib(F^1_A(M,M') \to F^2_A(M,M')) \simeq \Sym_A^3(M) \otimes_A \Sym_A^3(M').$$
	$$G^0_A(M,M') =  \bigwedge\nolimits_A^3(M) \otimes_A \Gamma_A^3(M')  \qquad G^2_A(M,M') = \bigwedge\nolimits_A^3(M \otimes_A M'),$$
	$$\cofib(G^0_A(M,M') \to G^1_A(M,M')) \simeq \dSchur_A^{(2,1)}(M) \otimes_A \dWeyl_A^{(2,1)}(M'),$$
	$$\cofib(G^1_A(M,M') \to G^2_A(M,M')) \simeq \Sym_A^3(M) \otimes_A \bigwedge\nolimits_A^3(M').$$
\end{enumerate}
\end{example}

%Derived version of Theorem \ref{thm:fil:Schur_oplus}:
\begin{theorem}[Direct-Sum Decomposition Formula; compare with Theorem \ref{thm:fil:Schur_oplus}]
\label{thm:fib:dSchur_oplus} 
Let $\lambda/\mu$ be a skew partition such that $\mu \neq \lambda$. For each integer $0 \le k \le |\lambda| -  |\mu|$, we let
	$$\gamma^{0}_{(k)} <   \gamma^{1}_{(k)} < \cdots < \gamma^{\ell_k-1}_{(k)}$$
denote all the partitions in $I_k(\lambda/\mu) = \{\gamma \mid \mu \subseteq \gamma \subseteq  \lambda, |\gamma| - |\mu|=k\}$ listed in lexicographic order, where $\ell_k = |I_k(\lambda/\mu)| \ge 1$ is the cardinality of $I_k(\lambda/\mu)$. We set $N = |\lambda|- |\mu|$. Then:
\begin{enumerate}[leftmargin=*]
	\item 
	\label{thm:fib:dSchur_oplus-1} 
	\begin{enumerate}
		\item 
		\label{thm:fib:dSchur_oplus-1i} 
		There are canonical decomposition of functors 
			$$\SCRModcn \times_{\CAlgDelta} \SCRModcn \to \SCRModcn$$
			$$(A, M_1, M_2) \mapsto (A, \dSchur_A^{\lambda/\mu}(M_1 \oplus M_2) = \bigoplus_{k=0}^{N} \dSchur_A^{\lambda/\mu}(M_1,M_2)_{(k,N-k)})$$
		which preserves sifted colimits and commutes with base change of simplicial commutative rings, and for each $k$, the functor $(A, M_1, M_2) \mapsto (A, \dSchur_A^{\lambda/\mu}(M_1,M_2)_{(k,N-k)})$ preserves sifted colimits and commutes with base change of simplicial commutative rings. Furthermore, for each $0 \le k \le N$, there is a canonical functor
			$$\SCRModcn \times_{\CAlgDelta} \SCRModcn \to \shE^{[\ell_k-1]}$$
			$$(A, M_1, M_2) \mapsto (A, F_A^{0,(k)}(M_1,M_2) \to  F_A^{1,(k)}(M_1,M_2) \to \cdots \to F_A^{\ell_k-1,(k)}(M_1,M_2))$$
		which preserves sifted colimits and commutes with base change of simplicial commutative rings, such that
			$$F_A^{\ell_k-1,(k)}(M_1,M_2) = \dSchur_A^{\lambda/\mu}(M_1, M_2)_{(k,N-k)},$$
		and for each $0 \le i \le \ell_k-1$, there is a canonical equivalence
			$$\cofib\big(F_A^{i-1,(k)}(M_1,M_2) \to F_A^{i,(k)}(M_1,M_2)\big) \simeq \dSchur_A^{\gamma_{(k)}^i/\mu}(M_1) \otimes_A \dSchur_A^{\lambda/\gamma^i_{(k)}}(M_2)$$
		(Here, by convention, we let $F_A^{\, -1,(k)}(M_1,M_2)=0$.) 
		Moreover, the formations of the above equivalences are functorial with respect to $(A,M_1,M_2)$ and commute with base change of simplicial commutative rings. 

		\item 
		\label{thm:fib:dSchur_oplus-1ii} 
		There are canonical decomposition of functors 
			$$\SCRModcn \times_{\CAlgDelta} \SCRModcn \to \SCRModcn$$
			$$(A, M_1, M_2) \mapsto (A, \dWeyl_A^{\lambda/\mu}(M_1 \oplus M_2) = \bigoplus_{k=0}^{N} \dWeyl_A^{\lambda/\mu}(M_1,M_2)_{(k,N-k)}),$$
		where for each $k$, the functor $(A, M_1, M_2) \mapsto (A, \dWeyl_A^{\lambda/\mu}(M_1,M_2)_{(k,N-k)})$ preserves sifted colimits and commutes with base change of simplicial commutative rings. 
		Furthermore, for each $0 \le k \le N$, there is a canonical functor
			$$\SCRModcn \times_{\CAlgDelta} \SCRModcn \to \shE^{[\ell_k-1]}$$
			$$(A, M_1, M_2) \mapsto (A, G^A_{\ell_k-1,(k)}(M_1,M_2) \to\cdots \to G^A_{1,(k)}(M_1,M_2) \to G^A_{0,(k)}(M_1,M_2))$$
		which preserves sifted colimits and commutes with base change of simplicial commutative rings, such that 
			$$G^A_{0,(k)}(M_1,M_2) = \dWeyl_A^{\lambda/\mu}(M_1, M_2)_{(k,N-k)},$$
		and for each $0 \le i \le \ell_k-1$, there is a canonical equivalence
			$$\cofib\big(G^A_{i+1,(k)}(M_1,M_2) \to G^A_{i,(k)}(M_1,M_2)\big) \simeq \dWeyl_A^{\gamma_{(k)}^{i}/\mu}(M_1) \otimes_A \dWeyl_A^{\lambda/\gamma^{i}_{(k)}}(M_2).$$
		(Here, by convention, we let $G^A_{\ell_k,(k)}(M_1,M_2)=0$.) 
		Moreover, the formations of the above equivalences are functorial with respect to $(A, M_1, M_2)$ and commute with base change of simplicial commutative rings. 
			\end{enumerate}

	\item 
	\label{thm:fib:dSchur_oplus-2} 
	\begin{enumerate}
		\item
		\label{thm:fib:dSchur_oplus-2i} 
		There is a canonical functor 
			$$\shE=\shE^{[1]} \to \shE^{[N]},$$
		$$(A, \rho \colon M' \to M) \mapsto (A, F_{N}(A, \rho) \to \cdots \to F_{1}(A,\rho) \to F_{0}(A, \rho))$$
		(where $N = |\lambda| - |\mu|$) which preserves sifted colimits and commutes with base change of simplicial commutative rings, such that there are canonical equivalencs
			$$F_{0}(A, \rho) = \dSchur_A^{\lambda/\mu}(M) \qquad F_{N}(A, \rho) = \dSchur_A^{\lambda/\mu}(M')$$
			$$\cofib(F_1(A,\rho) \to F_0(A,\rho)) \simeq \dSchur_A^{\lambda/\mu}(\cofib(\rho))$$
			$$\cofib(F_{k+1}(A, \rho) \to F_{k}(A, \rho)) \simeq \dSchur^{\lambda/\mu}_A(M',\cofib(\rho))_{(k,N-k)} \quad \text{for} \quad 0 \le k \le N-1.$$
		Moreover, the formations of the above equivalences are functorial with respect to $(A, \rho \colon M' \to M)$ and commute with base change of simplicial commutative rings. 

			\item
			\label{thm:fib:dSchur_oplus-2ii}
			There is a canonical functor 
			$$\shE=\shE^{[1]} \to \shE^{[N]},$$
		$$(A, \rho \colon M' \to M) \mapsto (A, G_{N}(A, \rho) \to \cdots \to G_{1}(A,\rho) \to G_{0}(A, \rho))$$
		which preserves sifted colimits and commutes with base change of simplicial commutative rings, such that there are canonical equivalencs
			$$G_{0}(A, \rho) = \dWeyl_A^{\lambda/\mu}(M) \qquad G_{N}(A, \rho) = \dWeyl_A^{\lambda/\mu}(M')$$
			$$\cofib(G_1(A,\rho) \to G_0(A,\rho)) \simeq \dWeyl_A^{\lambda/\mu}(\cofib(\rho))$$
			$$\cofib(G_{k+1}(A, \rho) \to G_{k}(A, \rho)) \simeq \dWeyl^{\lambda/\mu}_A(M',\cofib(\rho))_{(k,N-k)} \quad \text{for} \quad 0 \le k \le N-1.$$
		Moreover, the formations of the above equivalences are functorial with respect to $(A, \rho \colon M' \to M)$ and commute with base change of simplicial commutative rings. 
		\end{enumerate}
\end{enumerate}
 \end{theorem}

\begin{proof}
We first prove assertion \eqref{thm:fib:dSchur_oplus-1i}. For each $k$, $0 \le k \le N$, similarly as with Theorem \ref{thm:fil:dsym_otimes}, we apply Proposition \ref{prop:nonab:derived} to the situation where $\shC = {\rm PolyMod}^{\rm ff} \times_{\rm Poly} {\rm PolyMod}^{\rm ff}$, $\shD = \shE^{[\ell_k-1]}$ (Notation \ref{notation:SCRMod}), and $f \colon \shC \to \shD$ is the functor
	$$(R, M, M') \mapsto (R, F_R^{0,(k)}(M,M') \to  F_R^{1,(k)}(M,M') \to \cdots \to F_R^{\ell_k-1, (k)}(M,M')),$$
where $F_R^{0, (k)}(M,M') \to \cdots \to F_R^{p-1, (k)}(M,M')$ is the sequence of inclusion morphisms constructed in Theorem \ref{thm:fil:Schur_oplus} \eqref{thm:fil:Schur_oplus-1i}. Then Proposition \ref{prop:nonab:derived} implies that there is an essentially unique functor extending $f$ as described in assertion \eqref{thm:fib:dSchur_oplus-1i} and base-change properties are similarly proved as we did in Proposition \ref{prop:dSchur:basechange}. Assertion \eqref{thm:fil:Schur_oplus-1ii} is proved similarly. 

Assertion \eqref{thm:fib:dSchur_oplus-2i} is proved similarly by applying Proposition \ref{prop:nonab:derived} to the situation where $\shC = \shE_0$, $\shD = \shE^{[|\lambda| - |\mu|]}$, and $f$ is the functor 
	$$(R, \rho \colon M' \to M) \mapsto (A, F_{|\lambda|-|\mu|}(R, \alpha) \to \cdots \to F_{1}(R,\alpha) \to F_{0}(R, \alpha))$$
where $\alpha$ denote the short exact sequence $M' \xrightarrow{\rho} M \to M'':={\rm Coker}(\rho)$ determined by $\rho$, and $F_{|\lambda|-|\mu|}(R, \alpha) \to \cdots  \to F_{0}(R, \alpha)$ is the sequence of inclusions constructed in Theorem \ref{thm:fil:Schur_oplus} \eqref{thm:fil:Schur_oplus-2i}. The proof of assertion \eqref{thm:fib:dSchur_oplus-2ii} is similar.  
\end{proof}	

\begin{remark}
\label{rmk:fil:dSchur_oplus:(k,N-k)}
Similar to Remark \ref{rmk:fil:Schur_oplus:(k,N-k)}, we should regard the summand $\dSchur_A^{\lambda/\mu}(M_1,M_2)_{(k,N-k)}$ (resp. $\dWeyl_A^{\lambda/\mu}(M_1, M_2)_{(k,N-k)}$) as the {\em bi-degree $(k,N-k)$ homogeneous component of $\dSchur_A^{\lambda/\mu}(M_1 \oplus M_2)$  (resp. {\em $\dWeyl_A^{\lambda/\mu}(M_1 \oplus M_2)$})} in the sense of polynomial functor theory (see \cite{BGMN21}). Moreover, as with Remark \ref{rmk:fil:Schur_oplus:(k,N-k)}, the canonical equivalence $\dSchur_A^{\lambda/\mu}(M_1 \oplus M_2) \simeq \dSchur_A^{\lambda/\mu}(M_2 \oplus M_1)$ (resp. $\dWeyl_A^{\lambda/\mu}(M_1 \oplus M_2) \simeq \dWeyl_A^{\lambda/\mu}(M_2 \oplus M_1)$) induces canonical equivalences of summands $\dSchur_A^{\lambda/\mu}(M_1, M_2)_{(k,N-k)} \simeq \dSchur_A^{\lambda/\mu}(M_2, M_1)_{(N-k,k)}$ (resp. $\Weyl_A^{\lambda/\mu}(M_1, M_2)_{(k,N-k)} \simeq \Weyl_A^{\lambda/\mu}(M_2, M_1)_{(N-k,k)}$).
\end{remark}

\begin{remark}[Complexes Associated with Derived Schur and Weyl Functors]
\label{rmk:fib:dSchur_oplus}
Let $A \in \CAlgDelta$ and let $\gamma \colon P_1 \to P_0$ denote a morphism of objects in $\Mod_A^{\cn}$. Let $M' = P_0$, $M = \cofib(\gamma)$ and let $\rho \colon M' \to M$ denote the canonical morphism. Then $\cofib(\rho) \simeq P_1[1]$. In view of \cite[Remark 1.2.2.3, Lemma 1.2.2.4]{HA}, we should regard the canonical sequence
 	$$F_{N}(A, \rho) \to \cdots \to F_{1}(A,\rho) \to F_{0}(A, \rho)$$
of Theorem \ref{thm:fib:dSchur_oplus} \eqref{thm:fib:dSchur_oplus-2i} as a $\shJ$-{complex} in the sense of \cite[Definition 1.2.2.2]{HA} (where $\shJ = [n] \cup \{-\infty\}$) which resolves $\dSchur_A^{\lambda/\mu}(\cofib(\gamma))$. Concretely, by virtue of \cite[Remark 1.2.2.3]{HA}, the above sequence induces a {\em complex} of objects
	$$0 \to \dSchur_A^{\lambda/\mu}(P_1[1])[-N] \to %\dSchur_A^{\lambda/\mu}(P_0, P_1[1])_{(1,N-1)}[1-N] \to 
	\cdots 
	\to \dSchur_A^{\lambda/\mu}(P_0, P_1[1])_{(N-k,k)}[-k] \to \cdots \to  %\dSchur_A^{\lambda/\mu}(P_0, P_1[1])_{(N,0)} = 
	\dSchur_A^{\lambda/\mu}(P_0)$$
 in the homotopy category $\D_{\ge 0}(A)={\rm Ho}(\Mod_A^\cn)$. In the case where $A=R$ is an ordinary commutative ring and $P_i$ are finite projective modules, the above complex is a complex of discrete finite projective $R$-modules. We will see in Proposition \ref{prop:bSchur_k=Schur_oplus_k} (see Remark \ref{rem:bSchur_k=Schur_oplus_k}) that in this case, the above sequence is equivalent to the Schur complex $\bSchur_R^{\lambda/\mu}(\gamma \colon P_1 \to P_0)$. Therefore, we could regard the above sequence $F_{N}(A, \rho) \to \cdots \to F_{0}(A, \rho)$ as a derived version of Schur complex for any morphism $\rho \colon P_0 \to P_1$ in $\Mod_A^{\cn}$.

Similarly, in the case where $\rho \colon P_0 \to \cofib(P_1 \xrightarrow{\gamma} P_0)$, we should regard the sequence
 	$$G_{N}(A, \rho) \to \cdots \to G_{1}(A,\rho) \to G_{0}(A, \rho)$$
of Theorem \ref{thm:fib:dSchur_oplus} \eqref{thm:fib:dSchur_oplus-2ii} as a derived version of {\em ``Weyl complex"} which resolves $\dWeyl_A^{\lambda/\mu}(\cofib(\gamma))$. (As above, it is a $\shJ$-complex in the sense of \cite[Definition 1.2.2.2]{HA}, where $\shJ = [n] \cup \{-\infty\}$). By virtue of \cite[Remark 1.2.2.3]{HA}, the above sequence induces a {\em complex} of objects
	$$0 \to \dWeyl_A^{\lambda/\mu}(P_1[1])[-N] \to %\dSchur_A^{\lambda/\mu}(P_0, P_1[1])_{(1,N-1)}[1-N] \to 
	\cdots 
	\to \dWeyl_A^{\lambda/\mu}(P_0, P_1[1])_{(N-k,k)}[-k] \to \cdots \to  %\dSchur_A^{\lambda/\mu}(P_0, P_1[1])_{(N,0)} = 
	\dWeyl_A^{\lambda/\mu}(P_0)$$
 in the homotopy category $\D_{\ge 0}(A)={\rm Ho}(\Mod_A^\cn)$. 
  
 Be aware, however, that unlike Schur complexes, the components $\dWeyl_A^{\lambda/\mu}(P_0, P_1[1])_{(N-k,k)}[-k]$ of the above complex are generally {\em not} discrete modules, even in the case where $A=R$ is an ordinary commutative ring and $P_i$ are finite free modules; see, for example, \cite{BMT} for computations in the cases of divided powers. As a result, we should not generally expect a classical ``Weyl complex" (as a connective complex of finite free modules) to exist.
\end{remark}

The next result is a derived generalization of Theorem \ref{thm:fil:Schur_LR}, which asserts that derived Schur (resp. Weyl) functors associated with skew partitions can be expressed as iterated extensions of the derived Schur (resp. Weyl) functors associated with usual partitions.

\begin{theorem}[Decomposition Rule for Skew Partitions]
\label{thm:fil:dSchur_LR} 
Let $\lambda/\mu$ be a skew partition such that  $\mu \subsetneq \lambda$. 
We let 
	$$\alpha^{0} < \alpha^{1} < \cdots < \alpha^{s-1}$$
denote the list of all the partitions $\alpha$ of size $|\lambda| - |\mu|$ such that $c_{\mu,\alpha}^{\lambda} \ne 0$ in the lexicographical order, and consider the sequence of partitions with each $\alpha^i$ repeated $c_{\mu,\alpha^i}^{\alpha}$-many times:
	$$(\tau^0, \tau^1, \ldots, \tau^{\ell-1}) : =(\underbrace{\alpha^0, \cdots, \alpha^0}_{c_{\mu,\alpha^0}^{\lambda}\,\text{terms}}, \underbrace{\alpha^1, \cdots, \alpha^1}_{c_{\mu,\alpha^1}^{\lambda} \,\text{terms}}, \cdots, \underbrace{\alpha^{s-1}, \cdots, \alpha^{s-1}}_{c_{\mu,\alpha^{s-1}}^{\lambda}\,\text{terms}}).$$
Here, $c_{\mu,\alpha^i}^{\lambda}$ are the Littlewood-Richardson numbers, and $ \ell: = \ell(\lambda/\mu) = \sum_{\gamma \subseteq \lambda} c_{\mu, \gamma}^{\lambda} = \sum_{i=0}^{s-1} c_{\mu,\alpha^i}^{\lambda}.$

\begin{enumerate}
	\item \label{thm:fil:dSchur_LR-1} 
	There is a canonical functor
			$$\SCRModcn \to \shE^{[\ell-1]}$$
			$$(A, M) \mapsto (A, F^{0}(A, M) \to  F^{1}(A, M) \to \cdots \to F^{\ell-1}(A, M))$$
		which preserves sifted colimits and commutes with base change of simplicial commutative rings, for which there are canonical equivalences
			$$F^{\ell-1}(A, M) = \dSchur_A^{\lambda/\mu}(M) \qquad F^{0}(A, M) \simeq \dSchur^{\alpha^0}_A(M)$$
			$$\cofib\big(F^{i-1}(A, M) \to F^{i}(A, M)\big) \simeq \dSchur_A^{\tau^i}(M) \quad \text{for} \quad 1 \le i \le \ell-1.$$
	Moreover, the formations of the above equivalences are functorial with respect to $(A,M)$ and commute with base change of simplicial commutative rings. 
	\item \label{thm:fil:dSchur_LR-2} 
	There is a canonical functor
			$$\SCRModcn \to \shE^{[\ell-1]}$$
			$$(A, M) \mapsto (A, G_{\ell-1}(A, M) \to  \cdots \to G_{1}(A,M) \to G_{0}(A, M))$$
		which preserves sifted colimits and commutes with base change of simplicial commutative rings, for which there are canonical equivalences
			$$G_0(A, M) = \dWeyl_A^{\lambda/\mu}(M) \qquad G_{\ell-1}(A, M) \simeq \dWeyl^{\alpha^{s-1}}_A(M)$$
			$$\cofib\big(G_{i+1}(A, M) \to G_{i}(A,M)\big) \simeq \dWeyl_A^{\tau^i}(M) \quad \text{for} \quad 1 \le i \le \ell-1.$$
	Moreover, the formations of the above equivalences are functorial with respect to $(A,M)$ and commute with base change of simplicial commutative rings. 
\end{enumerate}
\end{theorem}

\begin{proof} Similar to the proof of the preceding two theorems, assertion \ref{thm:fil:dSchur_LR-1} follows from applying Proposition \ref{prop:nonab:derived} to the situation where $\shC = {\rm PolyMod}^{\rm ff}$, $\shD = \shE^{[\ell-1]}$ (Notation \ref{notation:SCRMod}), and $f \colon \shC \to \shD$ is the functor 
	$$(R, M) \mapsto  (R, F_R^{0} \to  F_R^{1} \to \cdots \to F_R^{\ell-1})$$
where $F_R^{0} \to  F_R^{1} \to \cdots \to F_R^{\ell-1}$ is the sequence of inclusions constructed in Theorem \ref{thm:fil:Schur_LR} \eqref{thm:fil:Schur_LR-1}. Assertion \ref{thm:fil:dSchur_LR-2} is proved similarly.
\end{proof}

\begin{example}
	Let $\lambda/\mu = (3,2,1)/(1) 
	% = \ytableausetup{smalltableaux}\ydiagram{1+2,2,1}
	$, then we have $\ell=3$ and 
		$$\tau^0 = (2,2,1) < \tau^1= (3,1,1) < \tau^2 = (3,2).$$
%	$$\tau^0 =\ydiagram{2,2,1}=(2,2,1) < \tau^1=\ydiagram{3,1,1} = (3,1,1) < \tau^2= \ydiagram{3,2} = (3,2).$$
For any $(A,M) \in \SCRModcn$, we obtain canonical sequences of morphisms in $\Modcn_A$
	$$F^0(A,M)  = \dSchur_A^{(2,2,1)}(M)  \to F^1(A,M) \to F^2(A,M) =  \dSchur_A^{(3,2,1)/(1)}(M) $$
	$$G_2(A,M) = \dWeyl_A^{(3,2)}(M)  \to G_1(A,M) \to G_0(A,M) = \dWeyl_A^{(3,2,1)/(1)}(M),$$
for which there are canonical equivalences:		
	$$\cofib(F^0(A,M) \to F^1(A,M)) \simeq \dSchur_A^{(3,1,1)}(M) \qquad
	\cofib(F^1(A,M)  \to F^2(A,M)) \simeq \dSchur_A^{(3,2)}(M) $$
	$$\cofib(G_2(A,M) \to G_1(A,M)) \simeq \dWeyl_A^{(3,1,1)}(M) \qquad
	\cofib(G_1(A,M)  \to G_0(A,M)) \simeq \dWeyl_A^{(2,2,1)}(M).$$
\end{example}

We have the following derived generalizations of the classical Littlewood--Richardson rule (\cite[Theorem IV.2.1]{ABW}, \cite[Theorem 2.3]{Bo1}, \cite{BB88}, \cite[Corollary 1.5, 2.6]{Kou}):

% Cor: derived Littlewood--Richardson rule 
\begin{corollary}[Littlewood--Richardson Rule]
\label{cor:fil:dSchur_LR} 
Let $\lambda$ and $\mu$ be two partitions, and let
	$$\gamma^{0} < \gamma^{1} < \cdots < \gamma^{s-1}$$
denote the list of all partitions in the set $\{ \gamma \mid c_{\lambda,\mu}^{\gamma} \ne 0 \}$ in lexicographical order. We then consider the induced sequence of partitions with each $\gamma^i$ repeated $c_{\lambda,\mu}^{\gamma^i}$-many times:
	$$(\nu^0, \nu^1, \ldots, \nu^{\ell-1}) : =(\underbrace{\gamma^0, \cdots, \gamma^0}_{c_{\lambda,\mu}^{\gamma^0} \,\text{terms}}, \underbrace{\gamma^1, \cdots, \gamma^1}_{c_{\lambda,\mu}^{\gamma^1}\,\text{terms}}, \cdots, \underbrace{\gamma^{s-1}, \cdots, \gamma^{s-1}}_{c_{\lambda, \mu}^{\gamma^{s-1}}\,\text{terms}}).$$
Here, $c_{\lambda,\mu}^{\gamma^i}$ are the Littlewood-Richardson numbers, and $ \ell: = \ell(\lambda, \mu) = \sum_{\gamma} c_{\lambda,\mu}^{\gamma} = \sum_{i=0}^{s-1} c_{\lambda, \mu}^{\gamma^i}.$

\begin{enumerate}
	\item There is a canonical functor
			$$\SCRModcn \to \shE^{[\ell-1]}$$
			$$(A, M) \mapsto (A, F^{0}(A, M) \to  F^{1}(A, M) \to \cdots \to F^{\ell-1}(A, M))$$
		which preserves sifted colimits and commutes with base change of simplicial commutative rings, for which there are canonical equivalences
			$$F^{\ell-1}(A, M) = \dSchur_A^{\lambda}(M) \otimes_A \dSchur_A^{\mu}(M) \qquad F^{0}(A, M) \simeq \dSchur^{\gamma^0}_A(M)$$
			$$\cofib\big(F^{i-1}(A, M) \to F^{i}(A, M)\big) \simeq \dSchur_A^{\nu^i}(M) \quad \text{for} \quad 1 \le i \le \ell-1.$$
	Moreover, the formations of the above equivalences are functorial with respect to $(A,M)$ and commute with base change of simplicial commutative rings. 
	\item There is a canonical functor
			$$\SCRModcn \to \shE^{[\ell-1]}$$
			$$(A, M) \mapsto (A, G_{\ell-1}(A, M) \to  \cdots \to G_{1}(A,M) \to G_{0}(A, M))$$
		which preserves sifted colimits and commutes with base change of simplicial commutative rings, for which there are canonical equivalences
			$$G_0(A, M) = \dWeyl_A^{\lambda}(M) \otimes_A \dWeyl_A^{\mu} (M) \qquad G_{\ell-1}(A, M) \simeq \dWeyl^{\gamma^{s-1}}_A(M),$$
			$$\cofib\big(G_{i+1}(A, M) \to G_{i}(A,M)\big) \simeq \dWeyl_A^{\nu^i}(M) \quad \text{for} \quad 1 \le i \le \ell-1.$$
	Moreover, the formations of the above equivalences are functorial with respect to $(A,M)$ and commute with base change of simplicial commutative rings. 
\end{enumerate}
\end{corollary}

\begin{proof} We consider the skew partition $\omega/\rho$ of the form,
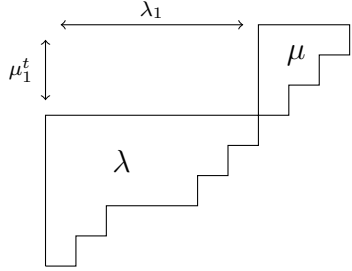
\begin{center}
\begin{tikzpicture}[scale=0.4]
    \draw %[fill=light-gray]
     (0,0) -- (0,5) -- (7,5) -- (7,4) -- (6,4) -- (6,3) -- (5,3) -- (5,2) -- (4,2) -- (2,2) -- (2,1) -- (1,1) -- (1,0) -- (0,0);
    \node at (2.5,3.5) {$\lambda$};
    
       \draw %[fill=light-gray] 
       (7,5) -- (7,8) -- (10,8) -- (10,7) -- (9,7) -- (9,6) -- (8,6) -- (8,5) -- (7,5);
    \node at (8.25,7) {$\mu$};
     %\draw [fill=light-gray] (0,5) -- (0,8) -- (7,8) -- (7,5) -- (0,5);
     
     \draw[to-to]	(0.5,8) -- (6.5,8);
     \node at (3.5, 8.5) {\text{\tiny $\lambda_1$}};
     
     \draw[to-to]  (0,5.5) -- (0,7.5);
         \node at (-0.8, 6.5) {\text{\tiny $\mu_1^t$}};   
 \end{tikzpicture}
\end{center}
(which is denoted by $\lambda * \mu$ in \cite[\S 5.1]{Ful}), where $\rho$ is the partition corresponding to the rectangle region $(\lambda_1^{\mu_1^t})$.
Then we have canonical equivalences $\dSchur_A^{\omega/\rho} (\blank) = \dSchur_A^{\lambda}(\blank) \otimes \dSchur_A^{\mu}(\blank)$ and $\dWeyl_A^{\omega/\rho} (\blank) = \dWeyl_A^{\lambda}(\blank) \otimes \dWeyl_A^{\mu}(\blank)$, and the desired results follow from applying Theorem \ref{thm:fil:dSchur_LR} to the skew partition $\omega/\rho$ and the equality of Littlewood--Richardson numbers $c_{\lambda,\mu}^{\nu}=c_{\nu, \rho}^{\omega} $ for all partitions $\nu$ (which is a consequence of \cite[\S 5.1, page 62, Corollary 2 (v)]{Ful}, where $\omega/\rho$ is denoted by $\lambda * \mu$ and $\shS(\omega/\rho; V_0)$ has cardinality $c_{\nu, \rho}^{\omega}$ by definition).
\end{proof}

\begin{example}
Let $\lambda = \mu = (1,1)$, then we have $\ell =3$, and 
	$$\nu^0= (1^4) < \nu^1 = (2,1,1) < \nu^2 =(2,2).$$
For any $(A,M) \in \SCRModcn$, we obtain canonical sequences of morphisms in $\Modcn_A$
	$$F^0(A,M)  = \bigwedge\nolimits_A^{4}(M)  \to F^1(A,M) \to F^2(A,M) =   \bigwedge\nolimits_A^{2}(M) \otimes_A  \bigwedge\nolimits_A^{2}(M) $$
	$$G_2(A,M) = \dWeyl_A^{(2,2)}(M)  \to G_1(A,M) \to G_0(A,M) = \bigwedge\nolimits_A^{2}(M) \otimes_A  \bigwedge\nolimits_A^{2}(M),$$
for which there are canonical equivalences:		
	$$\cofib(F^0(A,M) \to F^1(A,M)) \simeq \dSchur_A^{(2,1,1)}(M) \qquad
	\cofib(F^1(A,M)  \to F^2(A,M)) \simeq \dSchur_A^{(2,2)}(M) $$
	$$\cofib(G_2(A,M) \to G_1(A,M)) \simeq \dWeyl_A^{(2,1,1)}(M) \qquad
	\cofib(G_1(A,M)  \to G_0(A,M)) \simeq \bigwedge\nolimits_A^{4}(M).$$
\end{example}

The following theorem, which is divided into two parts, concerns the derived functors of $k$th component functors of Schur complexes. The first part is a derived generalization of Theorem \ref{thm:fil:bSchur_k}, and the second part is an extension of the Schur complex construction to the case of connective complexes.

\begin{theorem}
\label{thm:fib:dbSchur}
Let $\lambda/\mu$ be a skew partition.
\begin{enumerate}[leftmargin=*]
	
	\item \label{thm:fib:dbSchur-1}
	For any integer $0 \le k \le |\lambda| - |\mu|$, we let $\LL_A(M,M')_k$ denote derived functor of the $k$th component of the Schur complex. 
	\begin{enumerate}
		\item \label{thm:fib:dbSchur-1i}
		Let $\gamma^{0}_{(k)} <   \gamma^{1}_{(k)} < \cdots < \gamma^{\ell_k-1}_{(k)}$
be all the partitions in $I_k (\lambda/\mu) = \{\gamma \mid \mu \subseteq \gamma \subseteq  \lambda, |\gamma| - |\mu|=k\}$ listed in lexicographic order, where $\ell_k = |I_k(\lambda/\mu)|$ is the cardinality.
	Then there is a canonical functor
			$$\SCRModcn \times_{\CAlgDelta} \SCRModcn \to \shE^{[\ell_k-1]}$$
			$$(A, M, M') \mapsto (A, P_A^{0,(k)}(M,M') \to  P_A^{1,(k)}(M,M') \to \cdots \to P_A^{\ell_k-1,(k)}(M,M'))$$
		which preserves sifted colimits and commutes with base change of simplicial commutative rings, such that
			$$P_A^{\ell_k-1,(k)}(M,M') = \LL_A^{\lambda/\mu}(M, M')_k,$$
		and for each $0 \le i \le m_k-1$, there is a canonical equivalence
			$$\cofib\big(P_A^{i-1,(k)}(M,M') \to P_A^{i,(k)}(M,M')\big) \simeq \dSchur_A^{\lambda/\gamma_{(k)}^i}(M) \otimes_A \dWeyl_A^{(\gamma_{(k)}^i)^t/\mu^t}(M').$$
		(Here, by convention we let $P_A^{\, -1,(k)}(M,M')=0$ denote the zero functor.) 
		Moreover, the formations of the above equivalences are functorial with respect to $(A,M,M')$ and commute with base change of simplicial commutative rings. 
		\item\label{thm:fib:dbSchur-1ii}
		 Let 
			$\nu_{(k)}^{0} < \nu_{(k)}^{1} < \cdots < \nu_{(k)}^{m_k-1}$
		be all the partitions in $J_k (\lambda/\mu) = \{\nu \mid \mu \subseteq \nu \subseteq \lambda, |\lambda| - |\nu| = k\}$ listed in lexicographic order, where $m_k = |J_k(\lambda/\mu)|$. Then there is a canonical functor
			$$\SCRModcn \times_{\CAlgDelta} \SCRModcn \to \shE^{[m_k-1]}$$
			$$(A, M, M') \mapsto (A, Q_A^{0,(k)}(M,M') \to  Q_A^{1,(k)}(M,M') \to \cdots \to Q_A^{m_k-1,(k)}(M,M'))$$
		which preserves sifted colimits and commutes with base change of simplicial commutative rings, such that
			$$Q_A^{m_k-1,(k)}(M,M') = \LL_A^{\lambda/\mu}(M, M')_k,$$
		and for each $0 \le i \le m_k-1$, there is a canonical equivalence
			$$\cofib\big(Q_A^{i-1,(k)}(M,M') \to Q_A^{i,(k)}(M,M')\big) \simeq \dSchur_A^{\nu_{(k)}^i/\mu}(M) \otimes_A \dWeyl_A^{\lambda^t/(\nu^i_{(k)})^t}(M').$$
		(Here, by convention we let $Q_A^{\, -1,(k)}(M,M')=0$ denote the zero functor.) 
		Moreover, the formations of the above equivalences are functorial with respect to $(A,M,M')$ and commute with base change of simplicial commutative rings. 
		\end{enumerate}

		\item \label{thm:fib:dbSchur-2}	
		\begin{enumerate}
			\item \label{thm:fib:dbSchur-2i}
			There are canonical functors 
			$$\shE=\shE^{[1]} \to \shE^{[|\lambda|-|\mu|]}$$
		$$(A, \rho \colon M' \to M) \mapsto (A, \LL^{\lambda/\mu}_A(M,M')_{|\lambda|-|\mu|} \xrightarrow{d_{|\lambda|-|\mu|}} \cdots \xrightarrow{d_2} \LL^{\lambda/\mu}_A(M,M')_{1}  \xrightarrow{d_1} \LL^{\lambda/\mu}_A(M,M')_{0}  )$$	
				which preserve sifted colimits and commute with base change of simplicial commutative rings. Moreover, in the case where $A=R$ is an ordinary commutative ring and $M,M'$ are locally free $R$-modules, the morphism $d_i$ is canonically equivalent to the differential morphism $d_i$ of the Schur complex $\bSchur_{R}^{\lambda/\mu}(\rho \colon M' \to M)$, and in the case where $\rho \simeq 0$, the morphisms $d_i \simeq 0$ for all $i$. 
			\item \label{thm:fib:dbSchur-2ii}
			There is a canonical functor 
			$$\shE=\shE^{[1]} \to \shE^{[|\lambda|-|\mu|]}$$		
		$$(A, \rho \colon M' \to M) \mapsto (A, P^0(A, \rho) \to P^1(A, \rho) \cdots \to P^{|\lambda| - |\mu|}(A,\rho) )$$
		which preserve sifted colimits and commute with base change of simplicial commutative rings, such that there are canonical equivalences
			$$P^{0}(A, \rho) \simeq \dSchur_A^{\lambda/\mu}(M) \qquad P^{|\lambda|-|\mu|}(A, \rho) = \dSchur_A^{\lambda/\mu}(\cofib(\rho))$$
			$$\cofib(P^{k-1}(A, \rho) \to P^{k}(A, \rho)) \simeq \LL_A^{\lambda/\mu}(M, M')_k [k] \quad \text{for} \quad 1 \le k \le |\lambda|-|\mu|.$$
			In particular, we have $\cofib(P^{|\lambda|-|\mu|-1}(A,\rho) \to P^{|\lambda|-|\mu|}(A,\rho)) \simeq \dWeyl_A^{\lambda/\mu}(M')[|\lambda|-|\mu|]$. 
		Moreover, the formations of the above equivalences are functorial with respect to $(A, \rho \colon M' \to M)$ and commute with base change of simplicial commutative rings. 
		Furthermore, for each $1 \le k \le |\lambda|- |\mu|$, the composition of the functors
			$$\LL^{\lambda/\mu}_A(M,M')_{k}[k] \to P^{k-1}(A, \rho)[1] \to \LL^{\lambda/\mu}_A(M,M')_{k-1}[k]$$
		is canonically equivalence to the functor $d_{k}[k]$.
		\end{enumerate}
\end{enumerate}
\end{theorem}	

\begin{proof}
The assertions \eqref{thm:fib:dbSchur-1i} and \eqref{thm:fib:dbSchur-1ii} are proved by the same method as the proof of the preceding theorems of this subsection, namely, by applying Proposition \ref{prop:nonab:derived} to the situation where $\shC = {\rm PolyMod}^{\rm ff} \times_{\rm Poly} {\rm PolyMod}^{\rm ff}$, $\shD = \shE^{[\ell_k-1]}$ (or $\shD = \shE^{[m_k-1]}$) are as defined in Notation \ref{notation:SCRMod}, and $f \colon \shC \to \shD$ is given by the sequences of filtrations constructed in Theorem \ref{thm:fil:bSchur_k}.

Similarly, for assertion \eqref{thm:fib:dbSchur-2i}, Proposition \ref{prop:nonab:derived} implies that the functor $f \colon \shE_0 \to \shE^{[|\lambda|-|\mu|]}$ induced form the Schur complex construction (Definition \ref{def:bSchur}),
	$$(R, \rho \colon M' \to M) \mapsto (R, \bSchur^{\lambda/\mu}_R(M,M')_{|\lambda|-|\mu|} \xrightarrow{d_{|\lambda|-|\mu|}} \cdots \xrightarrow{d_2} \bSchur^{\lambda/\mu}_A(M,M')_{1}  \xrightarrow{d_1} \bSchur^{\lambda/\mu}_A(M,M')_{0}),$$
has an essentially unique extension to a functor $F \colon \shE \to \shE^{[|\lambda|-|\mu|]}$ which has the desired properties described as in \eqref{thm:fib:dbSchur-2i} before the ``moreover" part. For the assertions after ``moreover", let $\overline{f}$ denote the functor $\big((R, \rho) \in {\rm CRingCh}_{[0,1]}^{\rm fproj}\big) \mapsto \big((R, \bSchur_R^{\lambda/\mu}(\rho)) \in \shE\big)$ defined by the Schur complex construction Definition \ref{def:bSchur}, so that $f = \overline{f}|\shE_0$ is the restriction of $\overline{f}$. Since $F|{\rm CRingCh}_{[0,1]}^{\rm fproj}$ is a left Kan extension of $F|\shE_0 = \overline{f}|\shE_0$, there is an essentially unique morphism $\theta \colon F|{\rm CRingCh}_{[0,1]}^{\rm fproj} \to \overline{f}$ extending the identity $F|\shE_0 = \overline{f}|\shE_0$, which induces, for each $(R, \rho) \in {\rm CRingCh}_{[0,1]}^{\rm fproj}$, a functorial morphism $\theta(R, \rho)$ %between the two sequences of morphisms which is 
of the form:
		$$
	\begin{tikzcd}[column sep = 2 em, row sep = 1em]
		\cdots \to \LL_R^{\lambda/\mu}(M,M')_{k} \ar{d} \ar{r}&  \LL_R^{\lambda/\mu}(M,M')_{k-1} \ar{d} \to \cdots \\
		\cdots \to  \bSchur_R^{\lambda/\mu}(M,M')_{k} \ar{r} & \bSchur_R^{\lambda/\mu}(M,M')_{k-1} \to \cdots.
	\end{tikzcd}
	$$
As $M,M'$ are finite projective $R$-modules, the morphism $\theta(R, \rho)$ is an  equivalence by virtue of Proposition \ref{prop:dSchur:basechange}. This finishes the proof of assertion \eqref{thm:fib:dbSchur-2i}.

To prove \eqref{thm:fib:dbSchur-2ii}, consider any element $(R, \rho \colon M' \to  M) \in \shE_0$, where $R$ is a commutative ring and $0 \to M' \xrightarrow{\rho} M \to M'' \to 0$ is a short exact sequence of finite free $R$-modules. For any $0 \le k \le |\lambda|-|\mu|$, we let $P^{k}(R,\rho)$ denote the class of the ``brutal" truncation 
	$$\bSchur_R^{\lambda/\mu}(M,M')_k \xrightarrow{d_k} \cdots \xrightarrow{d_2}  \bSchur_R^{\lambda/\mu}(M,M')_1 \xrightarrow{d_1} \bSchur_R^{\lambda/\mu}(M,M')_0$$
 of the Schur complex $\bSchur_R^{\lambda/\mu}(\rho)$ in $\Modcn_R$. By Theorem \ref{thm:acyclic} \eqref{thm:acyclic-1}, there is a canonical equivalence $P^{|\lambda|-|\mu|}(R, \rho) = \Schur_R^{\lambda/\mu}(M'')$. For each $1 \le k \le |\lambda|-|\mu|$, there are canonical short exact sequences
 	$$0 \to P^{k-1}(R,\rho) \to P^{k}(R,\rho) \to \bSchur_R^{\lambda/\mu}(M,M')_k[k] \to 0$$
in the category ${\rm Ch}^{\rm ff}_{[0, k]}(R)$ of chain complexes of finite free modules. In particular, the functor
	$$\shE_0 \to \shE^{[|\lambda|-|\mu|]}, \quad
	(R, \rho \colon M' \to M) \mapsto (R, P^0(R, \rho) \to P^1(R, \rho) \cdots \to P^{|\lambda| - |\mu|}(R,\rho) )$$
extends essentially unique to a functor $\shE = \shP_\Sigma(\shE_0) \to \shE^{[|\lambda|-|\mu|]}$, as described in assertion \eqref{thm:fib:dbSchur-2ii}, with the desired properties on the cofibers of morphisms between consecutive terms, by virtue of Proposition \ref{prop:nonab:derived}. It only remains to prove the last assertion about commutative relations of \eqref{thm:fib:dbSchur-2ii}. For this purpose, we observe that the functor 
	$$\shE_0 \to \Fun(\Delta^2, \SCRModcn) \times_{\Fun(\Delta^2, \CAlgDelta)} \CAlgDelta \quad (R, \rho \colon M' \to M) \mapsto (R, \sigma(R, \rho)),$$
where $\sigma(R, \rho) \colon \Delta^2 \to \Modcn_R$ denotes the canonical commutative diagram
	$$
	\begin{tikzcd}[column sep = 3 em, row sep = .8 em]
		 & P^{k-1}(R, \rho)[1] \ar{dr}& \\
		\bL_R^{\lambda/\mu}(M, M')_{k}[k] \ar{rr}{d_{k}[k]} \ar{ru} & & \bL_R^{\lambda/\mu}(M, M')_{k-1}[k] 
	\end{tikzcd},
	$$
extends essentially uniquely to a functor 
	$$\shE \to \Fun(\Delta^2, \SCRModcn) \times_{\Fun(\Delta^2, \CAlgDelta)} \CAlgDelta \quad (A, \rho \colon M' \to M) \mapsto (A, \sigma(A, \rho))$$
which preserves sifted colimits and commutes with the natural forgetful morphisms which carry each $\sigma(A, \rho) \in \Fun(\Delta^2, \Modcn_A)$ to its edges $\sigma(A,\rho)|_{\N(\{i<j\})} \in \Fun(\Delta^1, \Modcn_A)$, where $i < j$, $i,j \in \{0,1,2\}$, up to canonical equivalences, where $\sigma(A, \rho) \colon \Delta^2 \to \Modcn_A$ is a diagram of the form
	$$
	\begin{tikzcd}[column sep = 3 em, row sep = .8 em]
		 & P^{k-1}(A, \rho)[1] \ar{dr}& \\
		\LL_A^{\lambda/\mu}(M, M')_{k}[k] \ar{rr}{d_k[k]} \ar{ru} & & \LL_A^{\lambda/\mu}(M, M')_{k-1}[k].
	\end{tikzcd}
	$$
The diagram $\sigma(A, \rho)$ induces the desired commutative relations. 
\end{proof}

% sec: Further properties
\subsection{Further Properties of Derived Schur Functors}
\label{sec:dSchurdWeyl.properties}
This subsection studies further properties of derived Schur and Weyl functors. The main results include the d{\'e}calage isomorphisms (Theorem \ref{thm:Illusie--Lurie}, Corollary \ref{cor:dSchur.decalage}), connective properties (Corollary \ref{cor:dSchur:connective}), preservation of pseudo-coherence (Proposition \ref{prop:dSchur:pc}) and Tor-amplitudes and perfectness (Proposition \ref{prop:dSchur:Tor-amp}), generalizations of Illusie's equivalences (Proposition \ref{prop:dSchur_vs_bSchur}) which connect derived Schur functors and Schur complexes, and the relationship between $k$th component of Schur complexes and derived Schur functors (Proposition \ref{prop:bSchur_k=Schur_oplus_k}).

We start with the following d{\'e}calage isomorphism theorem which generalizes Illusie--Lurie's results for derived symmetric and exterior powers \cite[Proposition 25.2.4.2]{SAG} and categorifies the d{\'e}calage isomorphisms of Quillen \cite{Qui}, Bousfield \cite{Bous}, and Touz{\'e} \cite{Tou}:

\begin{theorem}[D{\'e}calage Isomorphisms]
\label{thm:Illusie--Lurie}
For any pair of skew partition $\lambda/\mu$, any $A \in \CAlgDelta$, and any connective complex $M$ over $A$, there is a canonical equivalence in $\Modcn_A$:
	$$\dSchur^{\lambda/\mu}_{A}(M[1]) \simeq \dWeyl_A^{\lambda^t/\mu^t}(M)[|\lambda|- |\mu|].$$
\end{theorem}
\begin{proof}
Applying Theorem \ref{thm:fib:dbSchur} \eqref{thm:fib:dbSchur-2i} to the cofiber sequence
	$M \xrightarrow{\rho} 0 \to M[1],$
we obtain a canonical sequence of morphisms 
	$$P^0(A, \rho) \to P^1(A, \rho) \cdots \to P^{|\lambda| - |\mu|-1}(A,\rho) \to P^{|\lambda| - |\mu|}(A,\rho) $$
in $\Modcn_A$ such that 
	$$P^{0}(A, \rho) \simeq \dSchur_A^{\lambda/\mu}(0)=0  \qquad P^{|\lambda|-|\mu|}(A, \rho) = \dSchur_A^{\lambda/\mu}(\cofib(\rho))= \dSchur_A^{\lambda/\mu}(M)$$
	$$\cofib(P^{k-1}(A, \rho) \to P^{k}(A, \rho)) \simeq \LL_A^{\lambda/\mu}(0,M[1])_k [k] \quad \text{for} \quad 1 \le k \le |\lambda|-|\mu|.$$
Theorem \ref{thm:fib:dbSchur}  \eqref{thm:fib:dbSchur-1} implies that $\LL_A^{\lambda/\mu}(0,M[1])_k \simeq 0$ for all $k < |\lambda| - |\mu|$ and $\LL_A^{\lambda/\mu}(0,M[1])_{|\lambda| - |\mu|} \simeq  \dWeyl_A^{\lambda^t/\mu^t}(M[1])[|\lambda|-|\mu|]$. Therefore, we obtain $P^{k}(A,\rho) \simeq 0$ for all $0 \le k < |\lambda| - |\mu|$. Hence, the canonical cofiber sequence
	$$0 \simeq P^{|\lambda|-|\mu|-1}(A,\rho) \to P^{|\lambda|-|\mu|}(A,\rho) \to \dWeyl_A^{\lambda^t/\mu^t}(M[1])[|\lambda|-|\mu|]$$
induces the desired canonical equivalence.
\end{proof}

In characteristic zero, from the canonical equivalence $\dWeyl_A^{\lambda/\mu}(M) \simeq \dSchur_A^{\lambda/\mu}(M)$ we obtain:

\begin{corollary}[{D{\'e}calage Isomorphisms in Characteristic Zero}]
\label{cor:dSchur.decalage}
If $A \in \CAlgDelta_\QQ$, then for any connective $A$-complex $M \in \Modcn_A$,
%In characteristic zero, 
we have canonical equivalences for all integers $m \ge 0$:
	$$
	\dSchur_{A}^{\lambda/\mu} (M[m]) \simeq
	\begin{cases}
	 \dSchur_{A}^{\lambda^t/\mu^t} (M) [m \, (|\lambda| - |\mu|)] & \text{if $m$ is odd.} \\
	  \dSchur_{A}^{\lambda/\mu} (M) [m \, (|\lambda| - |\mu|)] & \text{if $m$ is even.} 
	\end{cases}
	$$
\end{corollary}

We also obtain connective properties of derived Schur and Weyl functors:

% Cor: Connectiveness
\begin{corollary}
\label{cor:dSchur:connective}
Let $\lambda/\mu$ be a skew partition such that $\mu \neq \lambda$,  let $A$ be a simplicial commutative ring and let $M$ be a $m$-connective $A$-complex for some integer $m \ge 0$. Then:
\begin{enumerate}
	\item \label{cor:dSchur:connective-1}
	The complex $\dWeyl^{\lambda/\mu}_A(M)$ is $m$-connective.
	\item \label{cor:dSchur:connective-2}
	%If $m=0$, then $\dSchur^{\lambda/\mu}_A(M)$ is connective. 
	If $m \ge 1$, then the complex $\dSchur^{\lambda/\mu}_A(M)$ is $(m-1+ |\lambda|-|\mu|)$-connective. 
\end{enumerate}
 \end{corollary}
 \begin{proof}
 We use the same strategy as Lurie's proof of \cite[Proposition 25.2.4.1]{SAG}.
We prove assertion \eqref{cor:dSchur:connective-1} by induction on $m \ge 0$. If $m=0$, then the assertion is trivial. Now we assume $m \ge 1$ and the assertion \eqref{cor:dSchur:connective-1} holds for $m-1$. We let $U_\bullet^+ \colon \N(\bDelta_+)^\op \to \shC$ be the {\v C}ech nerve of the morphism $0 \to M$ in the sense of \cite[\S 6.1.2]{HTT}, that is, $U_\bullet^+ \colon \N(\bDelta_+)^\op \to \shC$ is a right Kan extension of the diagram
 	$$
	U_\bullet^+ | \N(\bDelta_+^{\le 0})^\op \colon
	\begin{tikzcd}
	U_0=0 \ar[r,shift left]\ar[r,shift right] & U_{-1}=M
	\end{tikzcd}.$$
Then $M \simeq |U_\bullet|$ is the geometric realization, and by virtue of \cite[Proposition 6.1.2.11]{HTT}, we have % $U_1 \simeq M[-1]$ and 
for each $n \ge 1$, $U_n \simeq (M[-1])^{n}$. Hence $U_n$ is $(m-1)$-connective for all $n \ge 0$, and therefore $\dWeyl_{A}^{\lambda/\mu}(U_n)$ is $(m-1)$-connective for all $n \ge 0$ by induction hypothesis. Since the derived Weyl functor preserves geometric realization of simplicial objects and $\tau_{\ge m-1} (\Modcn_A)$ is closed under geometric realization of simplicial objects, $\dWeyl_{A}^{\lambda/\mu}(M) \simeq |\dWeyl_{A}^{\lambda/\mu}(U_\bullet)|$ is also $(m-1)$-connective. Since $\pi_{m-1}(\dWeyl_{A}^{\lambda/\mu}(M))$ is a quotient of $\pi_{m-1} (\dWeyl_{A}^{\lambda/\mu}(U_0))=0$, we obtain that $\dWeyl_{A}^{\lambda/\mu}(M)$ is $m$-connective. This proves  assertion \eqref{cor:dSchur:connective-1}.

Assertion \eqref{cor:dSchur:connective-2} follows from the equivalence $\dSchur^{\lambda/\mu}_A(M) = \dWeyl^{\lambda^t/\mu^t}_A(M[-1]) [|\lambda| - |\mu|]$ of  Theorem \ref{thm:Illusie--Lurie} and the fact that $\dWeyl^{\lambda^t/\mu^t}_A(M[-1])$ is $(m-1)$-connective (assertion \eqref{cor:dSchur:connective-1}). 
\end{proof}
 
 % Prop: Pseudo-coherence
\begin{proposition}[{Pseudo-coherence}] 
\label{prop:dSchur:pc}
For any pair of partitions $\mu \subseteq \lambda$ and any simplicial commutative ring $A$, if a connective complex $M \in \Modcn_A$ is pseudo-coherent to order $m$ for some $m \ge 0$ (resp. almost perfect), then $\dSchur_A^{\lambda/\mu}(M)$ and $\dWeyl_A^{\lambda/\mu}(M)$ are pseudo-coherent to order $m$  (resp. almost perfect).
 \end{proposition}
 
\begin{proof}
If $M \in \Modcn_A$ is pseudo-coherent to order $m$, then by virtue of \cite[Corollary 2.7.2.4]{SAG}, we can write $M$ as the geometric realization of a simplicial object $M_\bullet$ of $\Modcn_A$ such that each $M_k$ is a finite free $A$-module for $0 \le k \le m$. Since derived Schur and Weyl functors preserve geometric realization of simplicial objects, $\dSchur_A^{\lambda/\mu}(M)$ and $\dWeyl_A^{\lambda/\mu}(M)$ are the geometric realizations of $\dSchur_A^{\lambda/\mu}(M_\bullet)$ and $\dWeyl_A^{\lambda/\mu}(M_\bullet)$, respectively, where $\dSchur_A^{\lambda/\mu}(M_k)$ and $\dWeyl_A^{\lambda/\mu}(M_k)$ are finite free $A$-modules for all $0 \le k \le m$, by virtue of Proposition \ref{prop:dSchur:free}. Hence we obtain $\dSchur_A^{\lambda/\mu}(M)$ and $\dWeyl_A^{\lambda/\mu}(M)$ are pseudo-coherent to order $m$, by virtue of \cite[Corollary 2.7.2.4]{SAG} again. Consequently, if $M$ is almost perfect, $\dSchur_A^{\lambda/\mu}(M)$ and $\dWeyl_A^{\lambda/\mu}(M)$ are also almost perfect.
\end{proof}

% Prop: Perfectness
\begin{proposition}[{Tor-amplitudes and Perfectness}] 
\label{prop:dSchur:Tor-amp}
For any pair of partitions $\mu \subseteq \lambda$ and any simplicial commutative ring $A$, if a connective complex $M \in \Modcn_A$ has Tor-amplitude $\le m$ for some integer $m \ge 0$, then $\dSchur_A^{\lambda/\mu}(M)$ and $\dWeyl_A^{\lambda/\mu}(M)$ have Tor-amplitude $\le m (|\lambda| - |\mu|)$. Consequently, if $M$ is a perfect complex, then $\dSchur_A^{\lambda/\mu}(M)$ and $\dWeyl_A^{\lambda/\mu}(M)$ are perfect complexes. 
\end{proposition}

\begin{proof}
We will only prove the proposition for $\dSchur_A^{\lambda/\mu}(M)$; the proof for $\dWeyl_A^{\lambda/\mu}(M)$ is similar. For a fixed simplicial commutative ring $A$, we let $P(m,N)$ denote the following statement:
\begin{itemize}
	\item%[$P(m,N) \colon$] 
	$P(m,N) \colon \quad$ For any skew partition $\lambda/\mu$ and any $M \in \Modcn_A$, if $|\lambda|-|\mu| \le N$ and $M$ has Tor-amplitude $\le m$, then $\dSchur_A^{\lambda/\mu}(M)$ has Tor-amplitude $\le m \cdot N$. 
\end{itemize}
The statement $P(m,1)$ is trivially true for any $m \ge 0$. If $M$ has Tor-amplitude $\le 0$, then $M$ is a flat $A$-module, and the statement $P(0,N)$ holds for any $N \ge 0$ by virtue of Proposition \ref{prop:dSchur:free}. 

Now we assume $m \ge 1$, $N \ge 2$, and the statements $P(m-1,N)$ and $P(m,N-1)$ are true; we wish to prove that $P(m,N)$ is true. Let $M$ be a connective complex with Tor-amplitude $\le m$, where $m \ge 1$, and $\lambda/\mu$ a skew partition such that $|\lambda|-|\mu|=N$. Let $I$ be a generating set of $\pi_0(M)$ as a $\pi_0(R)$-module, then the natural map $A^{I} \xrightarrow{\rm ev} M$ is surjective on $\pi_0$. We form a fiber sequence $M' \xrightarrow{\rho} A^{I} \xrightarrow{\rm ev} M$, then $M'$ is a connective complex with Tor-amplitude $\le m-1$. By Theorem \ref{thm:fib:dSchur_oplus} \eqref{thm:fib:dSchur_oplus-1i} \& \eqref{thm:fib:dSchur_oplus-2i}, we have a canonical sequence of morphisms
	$$\dSchur^{\lambda/\mu}_A(M') = F_{\ell-1} \to F_{\ell -2} \to \cdots \to F_{1} \to F_{0} \simeq \dSchur^{\lambda/\mu}_A(A^{I})$$
(here, $\ell = \sum_{k=0}^{N} \ell_k \ge 3$, where $\ell_k \ge 1$ is defined in Theorem \ref{thm:fib:dSchur_oplus} \eqref{thm:fib:dSchur_oplus-1i}) such that
	$$\cofib(F_{i+1} \to F_{i}) \simeq \dSchur_A^{\gamma^i/\mu}(M') \otimes \dSchur_A^{\lambda/\gamma^i}(M) \quad \text{for} \quad 0 \le i \le \ell-2,$$
where $\gamma^i$ are partitions such that $\mu \subseteq \gamma^i \subseteq \lambda$, and $\gamma^i \neq \mu$ if $i \ge 1$. By induction hypothesis, $\dSchur^{\lambda/\mu}_A(M')$ has Tor-amplitude $\le N (m-1)$, and furthermore, for any $1 \le i \le \ell-2$, 
	$${\rm Tor.amp}_A \big(\dSchur_A^{\gamma^i/\mu}(M') \otimes \dSchur_A^{\lambda/\gamma^i}(M)\big) \le (m-1)(|\gamma^i| - |\mu|) + m (|\lambda| - |\gamma^i|) \le m \cdot N - 1$$
(see \cite[Lemma 6.1.1.6]{SAG}). Consequently, we obtain that $F_{i}$ have Tor-amplitude $\le m \cdot N - 1$ for all $1 \le i \le \ell-1$ (\cite[Proposition 7.2.4.23 (2)]{HA}). Moreover, from the equivalence 
	$$\cofib \big(F_{1} \to F_{0}\big) \simeq \dSchur_A^{\lambda/\mu}(M),$$
and the fact that $F_0\simeq \dSchur_A^{\lambda/\mu}(A^{I})$ is flat (Proposition \ref{prop:dSchur:free}), we obtain that $\dSchur_A^{\lambda/\mu}(M)$ has Tor-amplitude $\le m\cdot N -1 + 1 = m \cdot N$. 
%In other words, we have shown that for any $m >0$ and $N>0$, $P(m-1,N) \land P(m, N-1) \implies P(m,N).$
By induction, the proposition is proved. 	
\end{proof}

\begin{corollary}
\label{cor:dSchur:perf-amp}
Let $\lambda/\mu$ be a skew partition with $d = |\lambda|- |\mu| \ge 1$, for any simplicial commutative ring $A$, if $M \in \Modcn_A$ is a perfect complex of Tor-amplitude in $[m,n]$, where $n \ge m \ge 0$, then $\dSchur_A^{\lambda/\mu}(M)$ and $\dWeyl_A^{\lambda/\mu}(M)$ are perfect complexes with Tor-amplitude contained in $[m, n d]$. If $m \ge 1$, then we furthermore obtain that $\dSchur_A^{\lambda/\mu}(M)$ is a perfect complex with Tor-amplitude contained in $[m-1+d, n d]$.
\end{corollary}

\begin{proof}
Combine Propositions \ref{prop:dSchur:pc} and \ref{prop:dSchur:Tor-amp}.
\end{proof}

\begin{example} Let $A \in \CAlgDelta$, $M \in \Modcn_A$ a connective  complex of Tor-amplitude in $[m,n]$, where $n \ge m \ge 0$, and $d \ge 1$ an integer. Corollary \ref{cor:dSchur:perf-amp} implies that
	$${\rm Tor.amp}_A(\Gamma^d_A(M)) \subseteq [m, nd]
	\quad \text{and} \quad
	{\rm Tor.amp}_A(\bigwedge\nolimits^d_A(M)) \subseteq 
	\begin{cases}
	[0, nd], & \text{if $m=0$}; \\
	[m-1+d, nd], & \text{if $m \ge 1$}.
	\end{cases}
	$$
For derived symmetric powers, combining with the equivalences $\Sym^d_A(M) \simeq \bigwedge\nolimits^d_A(M[-1])[d]$ if $m \ge 1$, and $\Sym^d_A(M)  \simeq \Gamma^d_A(M[-2])[2d]$ if $m \ge 2$, we obtain from Corollary \ref{cor:dSchur:perf-amp} that
	$$
	{\rm Tor.amp}_A(\Sym^d_A(M)) \subseteq  
	\begin{cases}
		 [0, nd], & \text{if $m=0$}; \\
		 [d, nd], & \text{if $m=1$}; \\
		 [m-2+2d, nd], & \text{if $m \ge 2$}.
  	\end{cases}
	$$
\end{example}

In classical situation, our next result shows that the Schur complex $\bSchur^{\lambda/\mu}_R(\rho \colon M' \to M)$ canonically represents the derived Schur power $\dSchur^{\lambda/\mu}_R(\cofib (\rho \colon M' \to M))$. This is a generalization of Illusie's results for derived symmetric and exterior powers (\cite[Corollary 2.30 \& Variant 2.31]{J22a}).

\begin{proposition}[Generalizations of Illusie's Equivalences]
\label{prop:dSchur_vs_bSchur}
Let $R$ be an ordinary commutative ring and $\rho \colon M' \to M$ a morphism between finite projective $R$-modules. Then for any skew partition $\lambda/\mu$, there is a canonical equivalence
	$$[\bSchur_R^{\lambda/\mu}(\rho \colon M' \to M)] \xrightarrow{\sim} \dSchur_R^{\lambda/\mu} ( \cofib(M' \xrightarrow{\rho} M)) \in  \Mod_R^\cn,$$
where $[\bSchur_R^{\lambda/\mu}(\rho)] \in  \Mod_R^\cn$ denotes the image of the class of the Schur complex $\bSchur_R^{\lambda/\mu}(\rho)$ under the canonical equivalence $\shD_{\ge 0}(\Ch(\Mod_R^\heartsuit)) \simeq \Mod_R^\cn$ of \cite[Proposition 7.1.1.15]{HA}.
\end{proposition}

\begin{proof}
By virtue of Theorem \ref{thm:fib:dbSchur} \eqref{thm:fib:dbSchur-1} \&\eqref{thm:fib:dbSchur-2ii} and Proposition \ref{prop:dSchur:free}, there is a canonical sequence 
	$$P^0(R, \rho) \to P^1(R, \rho) \to \cdots \to P^{|\lambda| - |\mu|}(R,\rho)$$
of objects in $\Modcn_R$ such that
	$$P^{0}(R, \rho) \simeq \Schur_R^{\lambda/\mu}(M) \qquad P^{|\lambda|-|\mu|}(R, \rho) = \dSchur_R^{\lambda/\mu}(\cofib(\rho))$$
			$$\cofib(P^{k-1}(R, \rho) \to P^{k}(R, \rho)) \simeq \bSchur_R^{\lambda/\mu}(M, M')_k [k] \quad \text{for} \quad 1 \le k \le |\lambda|-|\mu|$$
and that the composition of canonical morphisms
	$$\bSchur^{\lambda/\mu}_R(M,M')_{k} \to P^{k-1}(R, \rho)[1-k] \to \bSchur^{\lambda/\mu}_R(M,M')_{k-1}$$
		is canonically equivalent to the differential morphism $d_{k}$ of the Schur complex $\bSchur_{R}^{\lambda/\mu}(\rho)$. Since 
		$$\Ext^s_R\left( \bSchur^{\lambda/\mu}_R(M,M')_{i}, \bSchur^{\lambda/\mu}_R(M,M')_{j} \right) = 0$$
for all $s \ne 0$ and $0 \le i,j \le |\lambda|-|\mu|-1$, from \cite[Lemma 2.28]{J22a} we obtain that the above sequence of morphisms $P^*(R, \rho)$ is canonically equivalent to the sequence $P'^{*}(R,\rho)$ in the proof of Theorem \ref{thm:fib:dbSchur} \eqref{thm:fib:dbSchur-2ii}, which is obtained by taking ``brutal" truncations 
	$$P'^{k}(R,\rho) : = [\bSchur_R^{\lambda/\mu}(M,M')_k \xrightarrow{d_k} \cdots \xrightarrow{d_2}  \bSchur_R^{\lambda/\mu}(M,M')_1 \xrightarrow{d_1} \bSchur_R^{\lambda/\mu}(M,M')_0]$$
of the Schur complex $\bSchur_{R}^{\lambda/\mu}(\rho)$. In particular, we obtain canonical equivalences 
	$[\bSchur_R^{\lambda/\mu}(\rho)]  \simeq P'^{|\lambda| - |\mu|}(R,\rho) \simeq P^{|\lambda| - |\mu|}(R,\rho) \simeq  \dSchur_R^{\lambda/\mu}(\cofib(\rho))$ as desired.
\end{proof}

\begin{remark}[The Classical Criteria]
\label{remark:classical.criteria}
Let $R$ be a commutative ring, let $\rho \colon M' \to M$ be a morphism between finite projective $R$-modules, and let $\sE = \cofib(\rho)$ denote the cofiber (which canonically represents the total complex $[M ' \xrightarrow{\rho} M]$). By combining the equivalence $\dSchur_R^{\lambda/\mu}(\sE) = [\bSchur_R^{\lambda}(\rho)]$ of Proposition \ref{prop:dSchur_vs_bSchur} and the acyclic criteria for the Schur complex $\bSchur_{R}^{\lambda/\mu}(\rho)$ of Akin-- Buchsbaum--Weyman \cite[Theorem V.1.17]{ABW} and Allahverdi--Tchernev \cite[Theorem 6.3]{AT19}, we obtain a necessary and sufficient condition for the derived Schur functor $\dSchur_R^{\lambda/\mu}(\sE)$ to be classical (or equivalently, for the canonical composite map $\dSchur_R^{\lambda/\mu}(\sE) \to \pi_0(\dSchur_R^{\lambda/\mu}(\sE)) \xrightarrow{\sim} \Schur_R^{\lambda/\mu}(\pi_0(\sE))$ to be an equivalence, where the last equivalence is from Proposition \ref{prop:dSchur:classical}).
\end{remark}
 
Next, we show that the derived functor of the $k$th component of Schur complexes can be expressed in terms of derived Schur functors. 

\begin{proposition}
\label{prop:bSchur_k=Schur_oplus_k}
For any skew partition $\lambda/\mu$, any $A \in \CAlgDelta$, any $M, M' \in \Mod_A^\cn$ and any $0 \le k \le N: =|\lambda| - |\mu| \ge 0$, there is a functorial equivalence of connective complexes:
	$$\LL_A^{\lambda/\mu}(M, M')_k  \xrightarrow{\sim} \dSchur_{A}^{\lambda/\mu}(M, M'[1])_{(N-k,k)} [-k].$$
\end{proposition}

\begin{proof}
By applying Theorem \ref{thm:fib:dbSchur} \eqref{thm:fib:dbSchur-2} to the split cofiber sequence $M' \xrightarrow{0} M \to M\oplus M'[1]$ and using the decomposition of Theorem \ref{thm:fib:dSchur_oplus} \eqref{thm:fib:dSchur_oplus-1i}, we obtain a canonical equivalence
	$$\Phi_{(A, M, M')} \colon \bigoplus_k \, \LL_A^{\lambda/\mu}(M,M')_k [k] \xrightarrow{\sim}  \dSchur_A^{\lambda/\mu}(M \oplus M'[1]) = \bigoplus_k \, \dSchur_A^{\lambda/\mu}(M, M'[1])_{(N-k,k)}.$$
We wish to show that $\Phi_{(A, M, M')} $ induces an equivalence on the $k$th summands. Since the functor $(A, M, M') \mapsto \Phi_{(A,M,M')}$ commutes with sifted colimits, and the non-abelian derived functors $\LL_A^{\lambda/\mu}(M,M')_k$ and $\dSchur_A^{\lambda/\mu}(M, M'[1])_{(N-k,k)}$ commute with sifted colimits, it suffices to prove the assertion in the case where $A=R$ is a polynomial ring and $M, M'$ are finite free $R$-modules. 
%Therefore, we may assume that $R$ is a commutative ring and $M, M$ are finite free $R$-modules. 
In this case, on the one hand, we know each $\LL_R^{\lambda/\mu}(M,M')_k = \bSchur_R^{\lambda/\mu}(M,M')_k$ is a (discrete) finite free $R$-module. On the other hand, for each fixed $k$, by virtue of Theorem \ref{thm:fib:dSchur_oplus}, we obtain that $\dSchur_R^{\lambda/\mu}(M, M'[1])_{(N-k,k)}$ can be constructed through iterated extensions of objects of the form 
	$$\dSchur_A^{\gamma/\mu}(M) \otimes \dSchur_A^{\lambda/\gamma}(M'[1]) [-k], \quad \text{where $|\lambda| - |\gamma| = k$}.$$ 
By virtue of Theorem \ref{thm:Illusie--Lurie} and Proposition \ref{prop:dSchur:free}, we  have canonical equivalences 
	$$\dSchur_R^{\gamma/\mu}(M) \otimes \dSchur_R^{\lambda/\gamma}(M'[1]) [-k] \simeq \dSchur_R^{\gamma/\mu}(M) \otimes  \dWeyl_R^{\lambda^t/\gamma^t}(M') \simeq \Schur_R^{\gamma/\mu}(M) \otimes \Weyl_R^{\lambda^t/\gamma^t}(M')$$
	%$$\dSchur_R^{\gamma/\mu}(M'[1]) \otimes \dSchur_R^{\lambda/\gamma}(M) [-k] \simeq \dWeyl_R^{\gamma^t/\mu^t}(M') \otimes  \dSchur_R^{\lambda/\gamma}(M) \simeq \Weyl_R^{\gamma^t/\mu^t}(M') \otimes \Schur_R^{\lambda/\gamma}(M)$$
for all $\gamma$ such that $|\lambda| - |\gamma| = k$. Hence $\dSchur_R^{\lambda/\mu}(M, M'[1])_{(N-k,k)}[-k]$ is equivalent to a finite free $R$-module. Therefore, by taking $k$th homology of both sides of the morphism $\Phi_{(R, M, M')}$, we obtain the desired equivalence $\bSchur_R^{\lambda}(M,M')_k \simeq \dSchur_R^{\lambda/\mu}(M, M'[1])_{(N-k,k)} [-k]$. 
\end{proof}

\begin{remark}[Alternative Characterization of Schur Complexes] 
\label{rem:bSchur_k=Schur_oplus_k}
Combining Proposition \ref{prop:bSchur_k=Schur_oplus_k} with Remark \ref{rmk:fib:dSchur_oplus}, we see that the universal sequence  	
	$$F_{N}(A, \rho') \to \cdots \to F_{1}(A,\rho') \to F_{0}(A, \rho')$$
of Theorem \ref{thm:fib:dSchur_oplus} \eqref{thm:fib:dSchur_oplus-2i}, where $\rho'$ denotes the canonical morphism $M \to \cofib(\rho \colon M' \to M)$, gives rise to a {\em complex} of objects of the form
	$$0 \to \LL_A^{\lambda/\mu}(M, M')_{N} \to \LL_A^{\lambda/\mu}(M, M')_{N-1} \to \cdots \to \LL_A^{\lambda/\mu}(M, M')_{1} \to  \LL_A^{\lambda/\mu}(M, M')_{0} \to 0 $$
in the homotopy category $\D_{\ge 0}(A)={\rm Ho}(\Mod_A^\cn)$.

If $A=R$ is an ordinary commutative ring and $M, M'$ are finite projective modules over $R$, then the above complex is equivalent to a complex of finite projective modules of the form
	$$0 \to \bSchur_R^{\lambda/\mu}(M, M')_{N} \to \bSchur_R^{\lambda/\mu}(M, M')_{N-1} \to \cdots \to \bSchur_R^{\lambda/\mu}(M, M')_{1} \to  \bSchur_R^{\lambda/\mu}(M, M')_{0} \to 0,$$	
which has the same terms as the Schur complex $\bSchur_R^{\lambda/\mu}(\rho)$ (Definition \ref{def:bSchur}). Moreover, by virtue of Proposition \ref{prop:dSchur_vs_bSchur} and the equivalence $F_{0}(A, \rho') \simeq \dSchur^{\lambda/\mu}(\cofib(\rho))$ of Theorem \ref{thm:fib:dSchur_oplus} \eqref{thm:fib:dSchur_oplus-2i}, the above complex is also quasi-isomorphic to the Schur complex $\bSchur_R^{\lambda/\mu}(\rho)$. 
\end{remark}

% subsec: Schur over Prestacks
\subsection{Derived Schur and Weyl Functors over Prestacks}
\label{sec:dSchurdWeyl.prestacks}
Using Lurie's machinery developed in \cite[\S 6.2.1]{SAG}, we could ``globalize" all the constructions and results of the preceding subsections from affine cases to the cases of general prestacks. 

We refer readers to \cite[\S 3]{J22a} for a more detailed discussion in the situation of derived symmetric, exterior and divided powers; the same argument applies to the derived Schur and Weyl functors considered in this paper. We will only sketch the key concepts and steps here. 

Prestacks are arguably the most general type of spaces in algebraic geometry (see \cite[Chapter 2]{GR}). Special cases of prestacks include the classical algebro-geometric objects such as classical schemes, algebraic spaces, Deligne--Mumford stacks, Artin stacks, and higher algebraic stacks, as well as the derived algebro-geometric objects such as derived schemes, derived algebraic spaces, derived Deligne--Mumford, Artin, and higher algebraic stacks.

A {\em prestack} $X \in \Fun(\CAlgDelta,\shS)$ is a functor from the $\infty$-category of simplicial commutative rings $\CAlgDelta$ to the $\infty$-category $\shS$ of spaces (equivalently, $\infty$-groupoids). The representable ones, denoted by $X = \Spec A$, for $A \in \CAlgDelta$, are called the derived affine schemes. If $R \in \CAlg^\heartsuit$ is an ordinary commutative ring, then $X = \Spec R$ is the classical affine scheme in the sense of Grothendieck. 

\subsubsection{The functors $\underline{\shE^{[n]}}$} 
Using the construction of \cite[\S 6.2.1]{SAG} (see also \cite[\S 3.1]{J22a}), we can globalize the constructions of Notation \ref{notation:SCRMod} to the case of all prestacks and obtain functors
	$$\QCoh^\cn, \underline{\shE^{[n]}} \colon  \Fun(\CAlgDelta, \shS)^\op \to \widehat{\Cat}_{\infty}.$$
For any prestack $X$, the $\infty$-categories $\QCoh^\cn(X)$ and $\underline{\shE^{[n]}}(X)$ are defined by the formula:
	$$\QCoh^\cn(X) = \Fun_{/ \CAlgDelta}^{\rm CCart}(\int_{\CAlgDelta} X, \SCRMod^\cn), \quad 
	\underline{\shE^{[n]}}(X) = \Fun_{/ \CAlgDelta}^{\rm CCart}(\int_{\CAlgDelta} X, \shE^{[n]}).$$
Here, $\SCRMod^\cn \to \CAlgDelta$ and $\shE^{[n]} \to \CAlgDelta$ are the coCartesian fibrations defined in Notation \ref{notation:SCRMod}, 
$\int_{\CAlgDelta} X \to \CAlgDelta$ denotes the left fibration classified by the prestack $X \colon \CAlgDelta \to \shS$ (\cite[Proposition 3.3.2.5]{HTT}), and $\Fun_{/\CAlgDelta}^{\rm CCart}(\int_{\CAlgDelta} X, \shC)$ denotes the $\infty$-category of functors from $\int_{\CAlgDelta} X$ to $\shC$ which commute with their projections to $\CAlgDelta$ and carry each edge of $\int_{\CAlgDelta} X$ to a coCartesian edge of $q \colon \shC \to \CAlgDelta$ (see \cite[Definition 6.2.1.1]{SAG}), where $\shC = \SCRMod^\cn, \shE^{[n]}$. Notice that $\QCohcn$ could be regarded as the special case of $\underline{\shE^{[n]}}$ where $n=0$, i.e., $\QCohcn = \underline{\shE^{[0]}}$. In what follows, the case of $\QCohcn$ will be listed separately as it is more intuitive and may be helpful for understanding the general case of $\underline{\shE^{[n]}}$.

\begin{remark} Unwinding the definitions, we have the following more informal descriptions  
(see \cite[Remark 6.2.1.8]{SAG}, \cite[Remark 3.1]{J22a}):

\begin{enumerate}
	\item We can describe $\QCoh^\cn(X)$ as the limit of the $\infty$-categories of connective quasi-coherent complexes: $\QCoh^\cn(X) \simeq \varprojlim_{A \in \CAlgDelta, \eta \in X(A)} \Mod^\cn_A$, and $\underline{\shE^{[n]}}(X)$ as the limit of the $\infty$-categories of sequences of morphisms of length $(n+1)$: $\underline{\shE^{[n]}}(X) \simeq \varprojlim_{A \in \CAlgDelta, \eta \in X(A)} \shE^{[n]}_{A}$, where $\shE^{[n]}_A$ are the fibers of $\shE^{[n]}$ over $A \in \CAlgDelta$ defined in Notation \ref{notation:SCRMod} \eqref{notation:SCRMod-4}.

	\item We can think of a connective quasi-coherent complex $\sF \in \QCoh^\cn(X)$ as a functorial assignment $(A \in \CAlgDelta, \eta \in X(A)) \mapsto (\sF(\eta) \in \Mod^\cn_A)$ which commutes with base change of $A$, that is, if $\phi \colon A \to B$ is a map in $\CAlgDelta$ and $\eta' \in X(B)$ is the image of $\eta$ under $\phi$, then there is a canonical equivalence $B \otimes_A \sF(\eta) \xrightarrow{\sim} \sF(\eta')$ in $\Modcn_B$.
	
	\item Similarly, we can think of an object $\sM \in \underline{\shE^{[n]}}(X)$ as a functorial assignment 
		$$(A \in \CAlgDelta, \eta \in X(A)) \mapsto (\sM(\eta) = (M^0 \to M^1 \to \cdots \to M^n) \in \shE^{[n]}_A)$$ 
	where $M^i \in \Modcn_A$, which commutes with base change of $A$: if $\phi \colon A \to B$ is a map of simplicial commutative rings, and we let $\eta' \in X(B)$ denote the image of $\eta$ under $\phi$ and let $\sM(\eta') = (M'^0 \to M'^1 \to \cdots \to M'^n) \in \shE^{[n]}_B$, then there is a canonical equivalence $B \otimes_A \sM(\eta) \xrightarrow{\sim} \sM(\eta')$ in $\shE^{[n]}_{B}$, that is, the natural map $B \otimes_A M^* \to M'^*$ induces equivalences $B \otimes_A M^i \to M'^i$ for all $0 \le i \le n$. 
\end{enumerate}
\end{remark}

\begin{remark}
 \label{rmk:En.prestacks:properties} 
The above constructions have the following properties (\cite[Remark 3.2]{J22a}):
\begin{enumerate}
	\item 
	 \label{rmk:En.prestacks:properties-1} 
	 Let $f \colon X \to Y$ be a morphism of prestacks, then there is a canonical pullback functor $f^* \colon \QCoh(Y)^\cn \to \QCoh(X)^\cn$ (resp. $f^* \colon \underline{\shE^{[n]}}(Y) \to \underline{\shE^{[n]}}(X)$). 
	\item 
	 \label{rmk:En.prestacks:properties-2} 
	 Let $F \colon \SCRModcn \to \SCRModcn$ (resp. $F \colon 
	\shE^{[m]} \to \shE^{[n]}$ for some $m,n \ge 0$) be a functor which commutes base change of simplicial commutative rings (Remark \ref{rem:base-change:CAlgDelta}), then $F$ induces a natural transformation $\sF \colon \QCohcn \to \QCohcn$ (resp. $\sF \colon \underline{\shE^{[m]}}  \to \underline{\shE^{[n]}}$) which determines a functor $\sF_X \colon \QCoh(X)^\cn \to \QCoh(X)^\cn$ (resp. $\sF_X \colon \underline{\shE^{[m]}}(X) \to \underline{\shE^{[n]}}(X)$) for each prestack $X$. Moreover, the formation $X \mapsto \sF_X$ commutes with base change of prestacks, that is, for any morphism of prestacks $f \colon X \to Y$, there is a canonical equivalence of functors $f^* \circ \sF_X \xrightarrow{\sim} \sF_Y \circ f^* $. Furthermore, in the case where $X=\Spec A$ is an affine derived scheme, $A \in \CAlgDelta$, there is a canonical equivalence $\sF_{X} \simeq F_{A} \colon \Modcn_A \to \Modcn_A$ (resp. $\sF_{X} \simeq F_{A} \colon \shE^{[m]}_A \to \shE^{[n]}_A$.)
	\end{enumerate}
\end{remark}

\subsubsection{Derived Schur and Weyl Functors over Prestacks}
By applying Remark \ref{rmk:En.prestacks:properties}  \eqref{rmk:En.prestacks:properties-2} to the derived Schur and Weyl functors (of Definition \ref{def:dSchurWeyl}; they commute with base change of simplicial commutative rings by Proposition \ref{prop:dSchur:basechange}), we obtain, for any skew partition $\lambda/\mu$, natural transformations
	$$\dSchur^{\lambda/\mu} \colon \QCohcn \to \QCohcn \quad \text{resp.} \quad \dWeyl^{\lambda/\mu} \colon \QCohcn \to \QCohcn,$$
which determine, for each prestack $X$, functors
	$$\dSchur^{\lambda/\mu}_X \colon \QCoh(X)^\cn \to \QCoh(X)^\cn \quad \text{resp.} \quad  \dWeyl^{\lambda/\mu}_X \colon \QCoh(X)^\cn \to \QCoh(X)^\cn.$$
We will refer to $\dSchur_X^{\lambda/\mu}$ and $\dWeyl_X^{\lambda/\mu}$ as the {\em derived Schur functor} and {\em derived Weyl functor} over $X$ (associated with the skew partition $\lambda/\mu$), respectively. For any $\sE \in \QCoh(X)^\cn$, we will refer to $\dSchur_X^{\lambda/\mu}(\sE)$ and $\dWeyl_X^{\lambda/\mu}(\sE)$ as the {\em derived Schur power} and {\em derived Weyl power} {\em of $\sE$} (over $X$, associated with the skew partition $\lambda/\mu$), respectively. If the prestack $X$ is clear from the context, we will generally drop the subscript $X$ and write $\dSchur^{\lambda/\mu}(\sE)$ and $\dWeyl^{\lambda/\mu}(\sE)$ instead.

By virtue of Remark \ref{rmk:En.prestacks:properties}  \eqref{rmk:En.prestacks:properties-2}, the formations of the functors $\dSchur_X^{\lambda/\mu}$ and $\dWeyl_X^{\lambda/\mu}$ commute with base change of prestacks, that is, for any morphism $f \colon X \to Y$ of prestacks and any connective quasi-coherent complex $\sE$ on $X$, there are canonical functorial equivalences 
 	$$f^* \, \dSchur^{\lambda/\mu}_X(\sE) \xrightarrow{\sim} \dSchur^{\lambda/\mu}_Y (f^*\sE)  \quad \text{resp.} \quad  f^* \, \dWeyl^{\lambda/\mu}_X(\sE) \xrightarrow{\sim} \dWeyl^{\lambda/\mu}_Y (f^*\sE).$$
Furthermore, in the case where $X=\Spec A$ is an affine derived scheme, $A \in \CAlgDelta$, there are canonical equivalence $\dSchur^{\lambda/\mu}_{X} \simeq \dSchur^{\lambda/\mu}_{A} \colon \Modcn_A \to \Modcn_A$ and $\dWeyl^{\lambda/\mu}_{X} \simeq \dWeyl^{\lambda/\mu}_{A} \colon \Modcn_A \to \Modcn_A$, where $\dSchur^{\lambda/\mu}_{A}$ and $\dWeyl^{\lambda/\mu}_{A}$ are defined in Definition \ref{def:dSchurWeyl}.

If $\sE =\sV$ is a vector bundle over $X$, then Proposition \ref{prop:dSchur:free} implies that $\dSchur_X^{\lambda/\mu}(\sV)$ and $\dWeyl_X^{\lambda/\mu}(\sV)$ are both vector bundles over $X$. In this case, we will often use classical notations $\Schur_X^{\lambda/\mu}(\sV)=\dSchur_X^{\lambda/\mu}(\sV)$ and $\Weyl_X^{\lambda/\mu}(\sV)=\dWeyl_X^{\lambda/\mu}(\sV)$
 to emphasize that they are vector bundles and that if $X$ is classical, they are precisely the classical Schur and Weyl functors, respectively, obtained by globalizing the classical construction Definition \ref{def:SchurWeyl} (Proposition \ref{prop:dSchur:flat}).

All the properties of derived Schur and Weyl functors that we obtained in \S \ref{sec:dSchurdWeyl} and \S \ref{sec:dSchurdWeyl.properties} have direct generalizations to the case of prestacks, including the results on classical truncations (Proposition \ref{prop:dSchur:classical}), the d{\'e}calage isomorphisms (Theorem \ref{thm:Illusie--Lurie}, Corollary \ref{cor:dSchur.decalage}), connectivity  (Corollary \ref{cor:dSchur:connective}), pseudo-coherence (Proposition \ref{prop:dSchur:pc}), and perfectness (Proposition \ref{prop:dSchur:Tor-amp}). We leave the details of these statements in the prestack case to the readers.

\subsubsection{Universal sequences associated with derived Schur and Weyl functors over prestacks}
We could also apply Remark \ref{rmk:En.prestacks:properties}  \eqref{rmk:En.prestacks:properties-2} to the functors considered in \S \ref{sec:univ.fib:dSchur} and obtain the global version of all the results obtained in \S \ref{sec:univ.fib:dSchur}. 

For example, in the situation of Theorem \ref{thm:fil:dSchur_LR}, we could apply Remark \ref{rmk:En.prestacks:properties} \eqref{rmk:En.prestacks:properties-2} to the functor $\SCRModcn \to \shE^{[\ell-1]}$ of Theorem \ref{thm:fil:dSchur_LR}  \eqref{thm:fil:dSchur_LR-1} and obtain: 
\begin{itemize}
	\item For any skew partition $\lambda/\mu$ such that $\mu \neq \lambda$, there is a canonical natural transformation 
	$$\QCohcn \to \underline{\shE^{[\ell-1]}}$$
which determines, for each prestack $X$, a canonical functor
	$$(\sE \in \QCoh(X)^\cn) \mapsto  (\sF^{0}(X, \sE) \to  \sF^{1}(X, \sE) \to \cdots \to \sF^{\ell-1}(X, \sE) \in \underline{\shE^{[\ell-1]}}(X))$$
for which there are canonical equivalences
		$$\sF^{\ell-1}(X, \sE) = \dSchur_X^{\lambda/\mu}(\sE) \qquad \sF^{0}(X, \sE) \simeq \dSchur^{\tau^0}_X(\sE)$$
		$$\cofib\big(\sF^{i-1}(X, \sE) \to \sF^{i}(X, \sE)\big) \simeq \dSchur_X^{\tau^i}(\sE) \quad \text{for} \quad 1 \le i \le \ell-1,$$
	where $\tau^i \subseteq \lambda$ are the partitions such that $c_{\mu, \tau^i}^{\lambda} \neq 0$ considered in Theorem \ref{thm:fil:dSchur_LR}.
\end{itemize}

The same argument works for all the other situations in \S \ref{sec:univ.fib:dSchur}. Consequently, we obtain global versions of derived Cauchy decomposition formula (Theorem \ref{thm:fil:dsym_otimes}), decomposition formula for direct sums (Theorem \ref{thm:fib:dSchur_oplus}), and Littlewood--Richardson rules for derived Schur and Weyl functors (Theorem \ref{thm:fil:dSchur_LR} and Corolloary \ref{cor:fil:dSchur_LR}), and Koszul type sequences (Theorem \ref{thm:fib:dbSchur}). The details of these statements in the prestack case are left to readers.

%%% sec: derived Grassmannians
\section{Derived Grassmannians and Derived Flag Schemes}
\label{sec:dGrass.dFlag}
This section defines derived Grassmannians and derived flag schemes of connective complexes and study their fundamental properties. The classical Grassmannian functors parametrize locally free quotients of (discrete) quasi-coherent sheaves. More precisely:

\begin{definition}[Classical Grassmannian Functors]
\label{def:Grass:classical}
Let $X$ be a classical scheme, $\sE$ a discrete quasi-coherent sheaf on $X$, and let $d>0$ be an integer.
The {\em classical Grassmannian functor}
	$$\Grasscl_{X,d}(\sE) =\Grasscl_{d}(\sE) \colon (\shS\mathrm{ch}/X)^{\rm op} \to \shS\mathrm{et}$$ 
is defined as follows: 
	\begin{itemize}
		\item For any $X$-scheme $\eta \colon T \to X$, $\Grasscl_{d}(\sE)(\eta)$ is the set of equivalence classes of quotients $q \colon \pi_0(\eta^* \sE) \twoheadrightarrow \sP$ in $\QCoh(T)^\heartsuit$, where $\sP$ is a locally free rank $d$ quasi-coherent complex on $T$, and two quotients $q \colon  \pi_0(\eta^* \sE)   \twoheadrightarrow \sP$ and $q' \colon  \pi_0(\eta^* \sE)  \twoheadrightarrow \sP'$ are said to be equivalent if $\ker(q) = \ker(q')$ as subsheaves of $\pi_0(\eta^* \sE)$ in $\QCoh(T)^\heartsuit$.
		\item For a $g \colon T' \to T$ between $X$-schemes $\eta \colon T \to X$ and $\eta' \colon T' \to X$, 
			$$\Grasscl_{d}(\sE)(g) \colon \Grasscl_{d}(\sE)(T) \to \Grasscl_{d}(\sE)(T')$$
			 is the pullback morphism which carries an epimorphism $\pi_0(\eta^*\sE) \twoheadrightarrow \sP$ to the classical pullback $\pi_0(\eta'^* \sE)\simeq \pi_0(g^* \pi_0(\eta^*\sE))  \twoheadrightarrow g^*\sP$. (This is well defined since the classical pullback $\pi_0 \circ g^* \colon \QCoh(T)^\heartsuit \to \QCoh(T')^\heartsuit$ is right-exact.)
	\end{itemize}
\end{definition}

In \cite[\S 9]{EGAI}, Grothendieck showed that the Grassmannian functors are representable by (classical) schemes over the base scheme $X$ and systematically studied their properties.

This section extends this theory to the context of derived algebraic geometry. 

\S \ref{def:Grass:classical} defines derived Grassmannians and explores their basic properties. We show that the derived Grassmannian functors are representable by relative derived schemes (Proposition \ref{prop:Grass:represent}) and their classical truncations are the classical Grassmannian functors (Proposition \ref{prop:Grass-classical}). We study their functoriality (Proposition \ref{prop:Grass-4,5}) and finiteness properties (Proposition \ref{prop:Grass:finite}), and completely describe their relative cotangent complexes (Theorem \ref{thm:Grass:cotangent}) and the closed immersions induced by surjective morphisms of complexes (Proposition \ref{prop:Grass:PB}).

\S \ref{sec:dGrass:Plucker.Segre} studies the derived generalizations of P\"ucker Morphisms and Segre Morphisms, relating derived Grassmannians to derived projectivizations studied in \cite{J22a}.

\S \ref{sec:flag} generalizes the above theory of derived Grassmannians to the theory of derived flag schemes. We study their representability (Proposition \ref{prop:Flag:rep}) and finiteness properties (Proposition \ref{prop:Flag:rep}) and describe their relative cotangent complexes (Theorem \ref{thm:dflag:cotangent}) and the behavior of the closed immersions induced by surjective maps of complexes (Proposition \ref{prop:dflag:immersion}). Furthermore, we  systematically investigate the natural morphisms among various derived flag schemes, with a focus on the forgetful functors (\S \ref{sec:dflag:forget}) and the closed immersions of derived flag schemes into products of Grassmannians (\S \ref{sec:closed.dflag.to.dGrass}).

\subsection{Derived Grassmannians} 
\label{sec:dGrass}

We will generally fix a prestack $X$ as base space. Recall that, a prestack is a functor $X \colon \CAlgDelta \to \shS$. It is arguably the most general notation of space that one could study in algebraic geometry. For any morphism of prestacks $\eta \colon T \to X$ and any complex $\sE \in \QCoh(X)$, we will let $\sE_{T} = \eta^*(\sE)$ denote the base change of $\sE$ along $\eta$. By abuse of notations, we will not distinguish the $\infty$-category $\CAlgDelta$ and the opposite $\infty$-category of affined derived schemes. Similarly, if $X$ is a prestack and $A \in \CAlgDelta$, we will generally not distinguish the space of $A$-points $X(A)$ and the space of maps $T = \Spec A \to X$.

\subsubsection{Definition of derived Grassmannians}
\label{sec:defn:dGrass}
Let $X$ be a prestack, let $\sE$ be a connective quasi-coherent complex over $X$, and let $d \ge 1$ be an integer. For any $A \in \CAlgDelta$ and any morphism $\eta \colon T =\Spec A \to X$, we let $G_{\sE}(\eta)$ denote the full subcategory of $\Fun(\Delta^1, \QCoh(T)^\cn)^{\simeq}$ spanned by those morphisms $u \colon \sE_T:=\eta^*\sE \to \sP$ 
 satisfying the following two conditions: 
	\begin{itemize}
		\item $\sP$ is a vector bundle of rank $d$ on $T$.
		\item The map $u$ is surjective on $\pi_0$ (that is, the induced map $\pi_0(u) \colon \pi_0(\sE_T) \to \pi_0(\sP)$ is an epimorphism in $\QCoh(T)^\heartsuit$). 
	\end{itemize}
We will refer to the elements $\sE_T  \to \sP$ of $G_{\sE}(\eta)$ as {\em rank-$d$ locally free quotients of $\sE_T$}. We define the prestack $G_{\sE} \in \Fun(\CAlgDelta, \shS)_{/X}$ over $X$ by the formula:
		$$(A \in \CAlgDelta, \eta \in X(A)) \mapsto (G_{\sE}(\eta) \in \shS).$$

\begin{proposition}[Representability]
\label{prop:Grass:represent}
For any prestack $X$ and any connective quasi-coherent complex $\sE$ on $X$, the above-defined prestack $G_{\sE}$ is a relative derived scheme over $X$.\end{proposition}

\begin{definition}[Derived Grassmannians] Let $X$ be a prestack, $d \ge 1$ an integer, and let $\sE$ be a connective quasi-coherent complex on $X$. We let 
	$$\pr = \pr_{\Grass_d(\sE)} \colon  \Grass_{X,d}(\sE) = \Grass_{X}(\sE;d) \to X$$
denote the relative derived scheme over $X$ which represents the prestack $G_{\sF}$, and refer to it as the {\em rank $d$ derived Grassmannian of $\sE$ (over $X$)}, or {\em derived Grassmannian of rank $d$ (locally free) quotients of of $\sE$ (over $X$)}. We will simply write $\Grass_{d}(\sE) = \Grass(\sE;d) = \Grass_{X,d}(\sE)$ if the space $X$ is clear from the context. We let $\sQ$, $\sQ_d$, $\sQ(\sE)$, or $\sQ_{\Grass_d(\sE)}$ denote the universal vector bundle of rank $d$ on $\Grass_d(\sE)$, and let $\rho = \rho_{
\Grass_d(\sE)} \colon \pr^* \sE \to \sQ$ denote the tautological map that is surjective on $\pi_0$. We will refer $\rho$ as the {\em tautological quotient (map)}. We will let $\sR$, $\sR(\sE)$, or $\sR_{\Grass_d(\sE)}$ denote the fiber of the tautological quotient $\rho$, and refer to 
	$$\sR_{\Grass_d(\sE)} \to \pr^*(\sE) \xrightarrow{\rho} \sQ_{\Grass_d(\sE)}  \quad (\text{or simply} \quad \sR \to \pr^*(\sE) \xrightarrow{\rho} \sQ )$$ 
as the tautological fiber sequence on $\Grass_d(\sE)$. By convention, we will set $\Grass_d(\sE) = X$ if $d=0$, and $\Grass_d(\sE)=\emptyset$ if $d<0$.
\end{definition}

\begin{example} If $d=1$, then $\Grass_1(\sE) = \PP(\sE)$ is the {\em derived projectivization} of $\sE$ over $X$ that we studied in \cite[\S 4.2]{J22a}. 
\end{example}

\begin{example}
\label{eg:dGrass:Z[1]oplusZ}
Let $X = \Spec \ZZ$ and $\sE = [\ZZ^m \xrightarrow{0} \ZZ^d]$, where $m, d \ge 1$ are integers. Then $\Grass_{d}(\sE)$ is the derived zero locus in $\Spec \ZZ$ cut out by the  zero section of the vector bundle $\ZZ^{md}$ over $\Spec \ZZ$ . In particular, the derived Grassmannian scheme $\Grass_{d}(\sE)$ has underlying classical scheme $\Spec \ZZ$, but is equipped with a nontrivial derived structure. More concretely, as a special case of the later Example \ref{eg:prop:Grass:PB}, we have canonical identifications
	$$\Grass_{d}(\sE) \simeq \VV_{\Spec \ZZ}(\ZZ^{md}[1]) = \Spec ( \Sym_\ZZ^*(\ZZ^{md}[1])) \simeq \underbrace{\Spec (\ZZ[\varepsilon]) \times_{\ZZ} \cdots \times_{\ZZ} \Spec (\ZZ[\varepsilon])}_{\text{$md$-terms}}.$$
 Here, $\ZZ[\varepsilon] : = \Sym_\ZZ^*(\ZZ[1])$ denotes the simplicial commutative ring of derived dual numbers. 
 \end{example}

\begin{proof}[Proof of Proposition \ref{prop:Grass:represent}]
The proof is a combination of Grothendieck's strategy in \cite{EGAI} and the higher-rank version of the proof of the projectivization case \cite[Proposition 4.17]{J22a}.
It suffices to prove the proposition in the case where $X = \Spec R$, $R \in \CAlgDelta$, and $\sE$ corresponds to a connective $R$-module $M$. We let $(t_i)_{i \in I}$ be a family of (possibly infinite many) elements that generates the $\pi_0(R)$-module $\Ext_R^0(R, M)$. Then the corresponding family of morphisms $\{t_i \colon \sO_X \to \sE\}_{i \in I}$ induces a map $\sO_X^{\oplus I} \to \sE$ of quasi-coherent sheaves on $X$ that is surjective on $\pi_0$. For every {\em finite subset $J \subseteq I$ of size $d$}, we let $U_{J}(\eta)$ denote the full subcategory of $G_{\sE}(\eta)$ spanned by those surjections $u \colon \eta^* \sE \to \sP$ for which the composition map
 	$$\sO_T^{\oplus d} \xrightarrow{(\eta^* t_i)_{i \in J}} \eta^* \sE \xrightarrow{u} \sP$$
is surjective on $\pi_0$. Since the formation of $\eta^* t_i$ is functorial on $\eta$, and the condition of ``being surjective on $\pi_0$" is stable under base change, the assignment $\eta \mapsto U_{J}(\eta) \subseteq G_{\sE}(\eta)$ defines a sub-prestack $U_{J}$ of $G_\sE$. It suffices to show the following assertions:
	\begin{enumerate}[label=$(\roman*)$, ref=$\roman*$]
		\item \label{proof:prop:Grass:represent-i}
		For each finite subset $J \subseteq I$ of size $d$, the sub-prestack $U_{J}$ is representable by a derived affine scheme.
		\item \label{proof:prop:Grass:represent-ii} 
		For each finite subset $J \subseteq I$ of size $d$, the inclusion $U_{J} \subseteq G_{\sE}$ is an open immersion.
		\item \label{proof:prop:Grass:represent-iii}
		 The family of maps $\{U_{J} \to G_{\sE} \}_{J \subseteq I}$, where $J \subseteq I$ runs through all finite subsets of $I$ of size $d$, is jointly surjective.
	\end{enumerate}

We first prove assertion \eqref{proof:prop:Grass:represent-i}. Let $U_J'$ denote the subfunctor of $U_J$ such that for each $\eta \colon T = \Spec R \to X$, $U_J'(\eta)$ is the space spanned by surjections of the form $u \colon \eta^* \sE \to \sO_T^{\oplus d}$ for which the composition $\sO_T^{\oplus d} \xrightarrow{\eta^* t_i} \eta^* \sE \xrightarrow{u} \sO_T^{\oplus d}$ is an isomorphism. On the other hand, for each $(u \colon \eta^* \sE \to \sP) \in U_J(\eta)$, the epimorphism $f_J \colon \sO_T^{\oplus d} \xrightarrow{(\eta^* t_i)_{i \in J}} \eta^* \sE \xrightarrow{u} \sP$ between vector bundles of same rank $d$ is necessarily an isomorphism. Since $\Vect_d(T)^{\simeq}_{\sO_T^{\oplus d}/}$ is a contractible Kan complex, we obtain that the inclusion of spaces $U_J'(\eta) \subseteq U_J(\eta)$ is a homotopy equivalence, and $U_J(\eta)$ is canonically homotopy equivalent to the fiber space
	$$\mathrm{fib}\left( \Map_{\QCoh(T)} (\eta^* \sE, \sO_T^{\oplus d}) \xrightarrow{\circ (\eta^* t_i)_{i\in J}} \Map_{\QCoh(T)}(\sO_T^{\oplus d}, \sO_T^{\oplus d})\right)$$
over the point $\id_{\sO_T^{\oplus d}}$. Consequently, we obtain a canonical equivalence between the prestack $U_J$ and the fiber of the morphism
	$$ \vert \sHom_X(\sE ,  \sO_X^{\oplus d}) \vert \xrightarrow{\circ (t_i)_{i \in J}} \vert \sHom_X(\sO_X^{\oplus d}, \sO_X^{\oplus d}) \vert \simeq \AA_X^{d^2}$$ 
over the section $X \to \AA_X^{d^2}$ which classifies the identity map $\id_{\sO_X^{\oplus d}}$. Therefore, $U_J$ is equivalent to an affine closed derived subscheme of $\vert \sHom_X(\sE ,  \sO_X^{\oplus d}) \vert \simeq \VV(\sE \otimes  (\sO_X^{\oplus d})^\vee)$.

Next, in order to prove assertion \eqref{proof:prop:Grass:represent-ii}, it suffices to show that for every affine derived scheme $Z = \Spec B$ and every map $g \colon Z \to G_{\sE}$, the functor
	\begin{equation} \label{eqn:Grass:rep:UJ-Z}
	(\eta \colon T = \Spec A \to X) \mapsto (U_J(\eta) \times_{G_{\sE}(\eta)} Z(\eta) \in \shS)
	\end{equation}
is representable by a derived open subscheme of $Z$. The morphism $g \colon Z \to G_{\sE}$ classifies a surjection $u_Z \colon \sE_Z \to \sP$, where $\sE_Z = g^* \sE$ and $\sP \in \Vect_d(Z)$. For each finite subset $J \subseteq I$ of size $d$, we let  $f_{J}$ denote the composite map $\sO_Z^{\oplus d} \xrightarrow{(g^* t_{i})_{i\in J}} \sE_Z \xrightarrow{u_{Z}} \sP$. By working Zariski locally over $Z$, we may assume $\sP \simeq \sO_Z^{\oplus d}$ and regard $f_J \colon \sO_Z^{\oplus d} \to \sO_Z^{\oplus d}$ as a matrix $(f_{ij})_{1 \le i, j \le d}$ with entries in $\pi_0(B)$. We let $\det (f_J) = \det((f_{ij})_{1 \le i,j \le d}) \in \pi_0(B)$ denote the determinant of the matrix $(f_{ij})_{1 \le i,j \le d})$, and let $U_{Z, J} = \Spec B[\frac{1}{\det (f_J)}]$ be the derived open subscheme of $Z = \Spec B$, where $B[\frac{1}{\det (f_J)}] \to B$ is the localization map of the simplicial commutative ring $B$ with respect to the element $\det (f_J) \in \pi_0(B)$ defined in \cite[Proposition 4.1.18]{DAGV}. To show that $U_{J, Z} \subseteq Z$ represents the functor \eqref{eqn:Grass:rep:UJ-Z}, it suffices to show the following assertion:
	\begin{enumerate}[label=$(*)$, ref=$*$]
		\item  \label{proof:prop:Grass:represent-ii-*}
		Let $h \colon Y=\Spec C \to Z = \Spec B$ be any map of derived affine schemes. Then $h^*(f_J)$ is an epimorphism on $\pi_0$ if and only if $h$ factorizes through $U_{J, Z} \subseteq Z$.
	\end{enumerate}
The assertion \eqref{proof:prop:Grass:represent-ii-*} is a consequence of the following two assertions:
\begin{itemize}
	\item By virtue of \cite[Proposition 4.1.18 (1)]{DAGV}, there is a canonical homotopy equivalence 
	$$\Map_{\CAlgDelta}(B[\frac{1}{\det(f_J)}], C) \to \Map_{\CAlgDelta}^0(B, C)$$
where $\Map_{\CAlgDelta}^0(B, C)$ is the union of summands of $\Map_{\CAlgDelta}(B, C)$ spanned by those maps $B \to C$ which carry $\det (f_J) \in \pi_0(B)$ to an invertible element of $\pi_0(C)$. 
	\item	The map between vector bundles $h^*(f_J) \colon \sO_Y^{\oplus d} \to \sO_Y^{\oplus d}$ is surjective on $\pi_0$ if and only if $\det(h^* (f_J)) \simeq h^* \det(f_J)$ is an invertible element of $\pi_0(C) \simeq \pi_0(\sO_Y)$.
\end{itemize}

Finally, in order to prove assertion \eqref{proof:prop:Grass:represent-iii}, it will suffice to show that, for each map of the form  $g \colon Z = \Spec B \to G_{\sE}$, the open subschemes $\{U_{J,Z} \subseteq Z\}_{J \subseteq I}$ forms a Zariski open cover of $Z$, where $J \subseteq I$ runs through all finite subsets of $I$ of size $d$, and $U_{J,Z}$ is the open subscheme of $Z$ which represents the functor \eqref{eqn:Grass:rep:UJ-Z} as in assertion \eqref{proof:prop:Grass:represent-ii}. To prove this, let $z \in |\Spec B| = |\Spec \pi_0(B)|$ be any point, and let $\kappa$ denote the residue field of $\pi_0(B)$ at $z$. Since $(f_{i, Z})_{i \in I} \colon \sO_Z^{\oplus I} \to \sP$ is surjective on $\pi_0$, the map $(f_{i,Z} \otimes_{\pi_0(B)} \kappa)_{i \in I} \colon \kappa^{\oplus I} \to \sP \otimes \kappa$ is a surjection of vector spaces. Since $\sP \otimes \kappa \simeq \kappa^{\oplus d}$, there exists a $J \subseteq I$ of size $d$ such that the induced map $(f_J)_Z = (g^* f_{i})_{i\in J} \colon \kappa^{\oplus d} \to \sP \otimes \kappa$ consists a basis of $\sP\otimes \kappa$, and thus surjective. By \eqref{proof:prop:Grass:represent-ii-*}, we deduce that $z \in U_{J, Z}$. 
\end{proof}
 
 \begin{remark}[Affine open covers] \label{remark:Grass:UJ} 
Let $X$ be a prestack and $\sE$ a connective quasi-coherent complex on $X$. Let $\{t_i \colon \sO_X \to \sE\}_{i \in I}$ be a family of sections that is jointly surjective on $\pi_0$, then the proof of Proposition \ref{prop:Grass:represent} shows that there is a Zariski open cover $\{U_J \subseteq \Grass_d(\sE)\}_{J \subseteq I}$ of $\Grass_d(\sE)$, where $J \subseteq I$ runs through all finite subsets of $I$ of size $d$, and $U_J \subseteq \Grass_d(\sE)$ is the open subfunctor characterized by the following universal property:
 	\begin{enumerate}[label=$(*')$, ref=$*'$]
		\item  \label{remark:prop:Grass:represent-ii-*'}
		Let $h \colon Y \to \Grass_d(\sE)$ be any map of prestacks. Then $h^* (\rho \circ (\pr^* t_i)_{i \in J}) \colon \sO_Y^{\oplus d} \to h^* \sQ$ is an epimorphism on $\pi_0$ if and only if $h$ factorizes through $U_{J} \subseteq \Grass_d(\sE)$.
	\end{enumerate}
In particular, over each $U_J$, the composite map of the restrictions
	$$f_J \colon \sO_{U_J}^{\oplus d} \xrightarrow{(\pr^*t_i)_{i \in J}|_{U_J}} \pr_{U_J}^* (\sE) \xrightarrow{\rho|_{U_J}} \sQ|_{U_J}$$
is surjective on $\pi_0$, where $\pr_{U_J}$ is the composite map $U_J \subseteq \Grass_d(\sF) \xrightarrow{\pr} X$. Therefore, $f_J$ is an isomorphism, and we let $f_J^{-1}$ denote an inverse of $f_J$. By virtue of the proof of assertion \eqref{proof:prop:Grass:represent-ii} of Proposition \ref{prop:Grass:represent}, the composite map
	$$\pr_{U_J}^* (\sE) \xrightarrow{\rho|_{U_J}} \sQ|_{U_J} \xrightarrow{f_J^{-1}}   \sO_X^{\oplus d}$$
classifies a closed immersion $s_J \colon U_J \to \vert \sHom_{X}(\sE, \sO_X^{\oplus d}) \vert$ which fits into a pullback diagram:
 	$$
	\begin{tikzcd}
	U_J \ar{r}{s_J}  \ar{d}[swap]{\pr_{U_J}} & \vert \sHom_X(\sE, \sO_X^{\oplus d}) \vert \ar{d}{\circ (t_i^*)_{i \in J} }\\
	X \ar{r}{s_X} & \vert \sHom_X(\sO_X^{\oplus d}, \sO_X^{\oplus d}) \vert .
	\end{tikzcd}
	$$
where $s_X \colon X \to \vert\sHom_X(\sO_X^{\oplus d}, \sO_X^{\oplus d})\vert$ is the section which classifies the identity map. Here, for a connective complex $\sE$, $|\sHom_X(\sE, \sO_X^{\oplus d})|  =  \VV_X(\sE \otimes \sO_X^{\oplus d}) = \Spec \Sym_X^*(\sE \otimes \sO_X^{\oplus d})$ is the affine cone over $X$ that classifies all maps from $\sE$ to $\sO_X^{\oplus d}$; see \cite[Example 4.8]{J22a} for details.
\end{remark}

\subsubsection{Classical truncations and functorial properties}
\begin{proposition}[Classical truncations]
\label{prop:Grass-classical}
 Let $X \colon \CAlgDelta \to \shS$ be a prestack, let $\sE$ be a connective quasi-coherent complex on $X$, and let $d \ge 1$ be an integer. Then the restriction 
 	$$\Grass_d(\sE)|_{\CAlg^\heartsuit} \colon \CAlg^\heartsuit \subseteq \CAlgDelta \xrightarrow{\Grass_d(\sE)} \shS$$
% functor $\Grass_d(\sE)_\cl$ of the derived Grassmannian functor $\Grass_d(\sE)$ 
is canonically equivalent to the classical Grassmannian functor 
	$$\Grass_{d}^\cl(\pi_0(\sE)) \colon \CAlg^\heartsuit  \to \shS{\rm et} \subseteq \shS$$
which carries each pair $(R \in \CAlg^\heartsuit, \eta \colon \Spec R \to X)$ to the set of isomorphism classes %$\{\eta_\cl^*(\pi_0 \sE) \twoheadrightarrow \sP\}/\simeq$ 
of rank-$d$ locally free quotients of $\eta_\cl^*(\pi_0 \sE):=\pi_0 (\eta^* (\pi_0 \sE))$ in $\QCoh(\Spec R)^\heartsuit$.
Consequently, if $X$ is a derived scheme, then $\Grass_d(\sE)$ is a derived scheme, and its underlying classical scheme is canonically equivalent to the classical Grassmannian scheme (Definition \ref{def:Grass:classical}) of rank-$d$ locally free quotients of the discrete sheaf $\pi_0(\sE)$ over the classical scheme $X_\cl$.
\end{proposition}
 
 \begin{proof} 
 For any ordinary commutative ring $R \in \CAlg^\heartsuit$, any map $\eta \colon T=\Spec R \to X$ and any rank-$d$ vector bundle $\sP$ on $T$, there are canonical homotopy equivalences
	$$\Map_{\QCoh(T)}(\eta^* \sE, \sP) \simeq \Map_{\QCoh(T)^\heartsuit}(\pi_0(\eta^* \sE), \sP) \simeq \Map_{\QCoh(T)^\heartsuit}(\eta_{\cl}^* (\pi_0 \sE), \sP).$$
(Here, $\eta_\cl^* = \pi_0 \circ \eta^*|_{\QCoh(X_\cl)^\heartsuit} \colon \QCoh(X_\cl)^\heartsuit \to \QCoh(T)^\heartsuit$ denotes the classical pullback of discrete sheaves.) The last mapping space is discrete because, for any $i > 0$, we have
	$$\pi_i (\Map_{\QCoh(T)^\heartsuit}(\eta_{\cl}^* (\pi_0 \sE), \sP)) \simeq \Ext^{-i}_R(\eta_{\cl}^* (\pi_0 \sE), \sP) =0.$$
Therefore, we obtain functorial homotopy equivalences  $\Grass_d(\sE)(\eta) \to \pi_0(\Grass_d(\sE)(\eta)) \simeq \Grass_{d}^\cl(\pi_0(\sE))(\eta)$ for all $(R \in \CAlg^\heartsuit, \eta \colon \Spec R \to X)$. In the case where $X$ is a derived scheme, any morphism of the form $\eta \colon \Spec R \to X$, where $R \in \CAlg^\heartsuit$, factorizes uniquely through $X_{\cl} \subseteq X$ (see \cite[Proposition 1.1.8.1]{SAG}). Hence the proposition follows. %Hence the desired equivalence follows. 
\end{proof}

\begin{example}
\label{eg:Grass:points}
Let $\kappa$ be any field and consider a morphism $\eta \colon \Spec \kappa \to X$. Then the fiber product $\Spec \kappa \times_X \Grass_d(\sE)$ is canonically equivalent to the classical Grassmannian variety $\Gr(\eta_\cl^*(\pi_0 \sE) ;d):  = \Grass_{\Spec \kappa}(\eta_\cl^*(\pi_0 \sE);d)$ over the field $\kappa$ which parametrizes rank-$d$ locally free quotients of the $\kappa$-vector space $\eta_\cl^*(\pi_0 \sE) : = \pi_0(\eta^* (\pi_0 \sE)) \simeq \pi_0(\sE|_\kappa)  = \pi_0(\eta^*(\sE))$.
\end{example}
 
 \begin{remark}
 \label{rmk:Grass:nonempty}
Let $X$ be a derived scheme, $\sE$ a connective complex such that $\pi_0(\sE)$ is of finite type, and $d \ge 1$ an integer. Then it follows from Proposition \ref{prop:Grass-classical} and Example \ref{eg:Grass:points} that $\Grass_d(\sE)$ is nonempty if and only if the $d$th Fitting locus $|X^{\ge d}(\pi_0(\sE))| := \{x \in |X| \mid \rank_{\kappa(x)} \pi_0(\sE \otimes \kappa(x)) \ge d \} \subseteq |X|=|X_\cl|$ of $\pi_0(\sE)$ over $X_{\cl}$ is nonempty (here, $|X^{\ge d}(\pi_0(\sE))|$ is also called the degeneracy locus of $\pi_0(\sE)$ of rank $\ge d$, as it is the set of points where the discrete sheaf $\pi_0(\sE)$ has rank $\ge d$; see \cite[\S 2.2]{J21}). 
\end{remark}
  
\begin{proposition} \label{prop:Grass-4,5} Let $X$ be a prestack, $\sE \in \QCoh(X)^\cn$ and $d \ge 1$ an integer.
\begin{enumerate}[leftmargin=*]
	\item \label{prop:Grass-4} (The formation of derived Grassmannians commutes with base change) 
	Let $f \colon X' \to X$ be a morphism of prestacks, then there is a natural equivalence $\Grass_{X',d}(f^* \sE) \xrightarrow{\sim} \Grass_{X,d}(\sE) \times_X X'$ such that the pullback of the universal quotient bundle (resp. the tautological quotient map) on $\Grass_{X,d}(\sE)$ is canonically equivalent to the universal quotient bundle (resp. the tautological quotient map) on $\Grass_{X',d}(f^*\sE)$.
	\item \label{prop:Grass-5} (Tensoring with line bundles) 
	Let $\sL$ be a line bundle on $X$. Then there is a canonical equivalence $g \colon \Grass_{X,d}(\sE) \xrightarrow{\sim} \Grass_{X,d}(\sE \otimes \sL)$ for which there are canonical equivalences 
		$$g^* (\sQ(\sE \otimes \sL)) \simeq \sQ(\sE) \otimes \pr_{\Grass_d(\sE)}^* \sL \qquad g^* (\sR(\sE \otimes \sL)) \simeq \sR(\sE) \otimes \pr_{\Grass_d(\sE)}^* \sL,$$ 
		$$g^*\big(\pr_{\Grass_d(\sE \otimes \sL)}^* (\sE\otimes \sL) \xrightarrow{\rho_{\Grass_d(\sE \otimes \sL)}} \sQ (\sE\otimes \sL)  \big) \simeq \big(\pr_{\Grass_d(\sE)}^* \sE \xrightarrow{\rho_{\Grass_d(\sE)}} \sQ(\sE)\big)  \otimes \pr_{\Grass_d(\sE)}^* \sL.$$
\end{enumerate}
\end{proposition}
\begin{proof}
The proof is identical to the proof in the projectivization case \cite[Proposition 4.25]{J22a}.
\end{proof}

\subsubsection{Properness and finiteness properties}
We first review some related concepts:

\begin{definition}[Pseudo-coherence] \label{def:pseudo-coherent}
Let $A$ be a simplicial commutative ring, $n \ge 0$ an integer, and $M \in \Mod_A$ an $A$-complex.
\begin{enumerate}[leftmargin=*]
	\item We say that $M$ is {\em pseudo-coherent to order $n$} (an a complex over $A$) if every filtered diagram $N_\alpha$ in $\Mod_R$ such that each $N_\alpha$ is $n$-truncated and each transition map $\pi_{0} N_\alpha \to \pi_{0} N_\beta$ is a monomorphism, the canonical map $\varinjlim_{\alpha} \Map_{\Mod_A}(M, N) \to \Map_{\Mod_A}(M, \varinjlim_{\alpha} N_\alpha)$ is a homotopy equivalence. The $A$-complex $M$ is called {\em of finite type} if $M$ is pseudo-coherent to order $0$ over $A$.
	\item We say that $M$ is {\em almost perfect} if it is pseudo-coherent to order $n$ for all $n \ge 0$.
	\item We say that $M$ is {\em perfect} if $\Map_{\Mod_A}(M, \blank)$ preserve filtered colimits.
\end{enumerate}
\end{definition}

\begin{remark} Notice our condition that`` $M$ is pseudo-coherent to order $n$" is the same as the condition that ``$M$ is perfect to order $n$" in the sense of \cite[Definition 2.7.0.1]{SAG}, but is {\em different from} (and slightly weaker than) the condition that ``$M$ is perfect to order $n$" in the sense of  \cite[Proposition 2.5.7]{DAG}. Hence, we choose the above terminology to avoid confusions. If $A$ is a commutative ring, then our condition that ``$M$ is pseudo-coherent to order $n$" is equivalent to the condition that ``$M$ $(-n)$-pseudo-coherent" in the sense of \cite[\href{https://stacks.math.columbia.edu/tag/064Q}{Tag 064Q}]{stacks-project}.
\end{remark}

\begin{definition}[Finite generation and finite presentation]
\label{def:finite.generation}
Let $\phi \colon A \to B$ be a map of simplicial commutative rings and let $n \ge 0$ be an integer.

\begin{enumerate}[leftmargin=*]
		\item We say that $\phi$ is {\em of ﬁnite generation to order $n$} (cf. \cite[Definition 4.1.1.1]{SAG}) if for every filtered diagram $\{C_{\alpha}\}$ in $\CAlgDelta_A$ such that each $C_\alpha$ is $n$-truncated and each transition map $\pi_n C_\alpha \to \pi_n C_\beta$ is a monomorphism, the canonical map $\varinjlim_{\alpha} \Map_{\CAlgDelta_A}(B, C_\alpha) \to  \Map_{\CAlgDelta_A}(B, \varinjlim_{\alpha} C_\alpha)$ is a homotopy equivalence.  We say $\phi$ is {\em of ﬁnite type} if it is of ﬁnite generation to order $0$.
		\item We say $\phi$ is {\em almost of finite presentation} if $\phi$ is of ﬁnite generation to order $n$ for all $n\ge 0$.
		\item We say $\phi$ is {\em locally of finite presentation} if $B$ is compact object of $\CAlgDelta_A$. 
%		\item (Finite presentation) We say an $A$-algebra $B$ is {\em of finite presentation} if it is contained in the smallest full subcategory of $\CAlgDelta_A$ which contains free objects $A[X_1, \ldots, X_m]$ (or simply $A[X]$) and is stable under finite colimits.% (That is, $B$ can be built from $A[X]$ using only finite colimits.)
	\end{enumerate}
\end{definition}

We need the following variant of \cite[Proposition 2.16 (2)]{J22a}:

\begin{lemma} 
\label{lem:affine:finite.generation}
Let $A$ be a simplicial commutative ring, let $M \in \Modcn_A$ be a connective complex, and let $n \ge 0$ be an integer. Then $M$ is pseudo-coherent to order $n$ (resp. almost perfect, resp. perfect) over $A$ if and only if the map $A \to \Sym_A^*(M)$ is of ﬁnite generation to order $n$ (resp. almost of finite presentation, resp. locally of finite presentation).
\end{lemma}

\begin{proof}
If $A \to \Sym_A^*(M)$ is of finite generation to order $n$, then it follows from  \cite[Proposition 4.1.2.1]{SAG} that $ L_{\Sym_A^*(M)/A}$ is pseudo-coherent to order $n$ as a $\Sym_A^*(M)$-complex. Consequently, $M \simeq L_{\Sym_A^*(M)/A} \otimes_{\Sym_A^*(M)} A$ is pseudo-coherent to order $n$ as an $A$-complex. For the converse direction, if $M$ is pseudo-coherent to order $n$, let $\{C_{\alpha}\}$ be any filtered diagram in $\CAlgDelta_A$ such that each $C_\alpha$ is $n$-truncated and each transition map $\pi_n C_\alpha \to \pi_n C_\beta$ is a monomorphism. Consider the following commutative diagram
		$$
	\begin{tikzcd} 
		\varinjlim_{\alpha} \Map_{\CAlgDelta_A}(\Sym_A^*(M), C_\alpha) \ar{d}{\simeq} \ar{r} & \Map_{\CAlgDelta_A}(\Sym_A^*(M), \varinjlim_{\alpha} C_\alpha) \ar{d}{\simeq} \\
		\varinjlim_{\alpha} \Map_{\Modcn_A}(M, C_\alpha)   \ar{r} & \Map_{\Modcn_A}(M, \varinjlim_{\alpha} C_\alpha)
	\end{tikzcd}
	$$
where the vertical arrows are the canonical equivalences from the adjuction between $\Sym_A^*$ and the forgetful functor from $\CAlgDelta_A$ to $\Modcn_A$ and the bottom horizontal arrow is an equivalence by our assumption on $M$. Hence, the top horizontal arrow is an equivalence. Therefore, $A \to \Sym_A^*(M)$ is of ﬁnite generation to order $n$. The rest of the statements are already proved in \cite[Proposition 2.16 (2)]{J22a}.
\end{proof}

We now define these concepts in the geometric setting (cf.  \cite[Definition 4.9]{J22a}).

\begin{definition}\label{def:perfectfg:prestacks}
Let $X$ be a prestack, $\sF \in \QCoh(X)$, and let $f \colon X \to Y$ be a morphisms between prestacks which is a relative derived scheme. Let $n \ge 0$ be an integer.   
\begin{enumerate}[leftmargin=*]
	\item We say that $\sF$ is {\em pseudo-coherent to order $n$ (resp. 
 almost perfect, resp. perfect)} if for every simplicial commutative ring $R$ and every point $\eta \in X(\eta)$, then $R$-module $\sF(\eta)$ has the same property in the sense of Definition \ref{def:pseudo-coherent}. % (this is well-defined since the above properties are stable under base change \cite[Proposition 3.5.2]{DAG}). 
 We say $\sF$ is {\em locally of finite type} if $\sF$ is pseudo-coherent to order $0$.
 	\item If $f \colon X \to Y$ is a morphism of derived schemes, then we say that $f$ is locally of finite generation to order $n$ (resp. locally almost of finite presentation, resp. locally of ﬁnite presentation), if for every pair of open immersions $\Spec B \subseteq X$, $\Spec A \subseteq Y$, where $A, B \in \CAlgDelta$, such that $f$ restricts to a morphism $\Spec B \to \Spec A$, $B$ is of finite generation to order $n$ (resp. almost of finite presentation, locally of ﬁnite presentation) over $A$ in the sense of Definition \ref{def:finite.generation}. In general, if  $f \colon X \to Y$ be a morphism of prestacks which is a relative derived scheme, then we say that $f$ is {\em locally of finite generation to order $n$ (resp. locally almost of finite presentation, resp. locally of ﬁnite presentation)} if for every map $\eta \colon \Spec R \to Y$, where $R \in \CAlgDelta$, the induced morphism of derived schemes $X_R = X \times_{\Spec R} Y \to \Spec R$ has the same property in the above sense for derived schemes. % (this is well defined as these properties are stable under base change and local on the target with respect to {\'e}tale topology; see \cite[Propositions 3.5.2, 3.5.3]{DAG}).
	We say $f$ is {\em locally of finite type} if $f$ is  locally of finite generation to order $0$. 
\end{enumerate}
\end{definition}

\begin{definition}[Properness]
A morphism $f \colon X \to Y$ between derived schemes are called {\em proper} if it is quasi-compact, separated, locally of finite type, and universally closed (Definition \cite[Definition 5.1.2.1]{SAG}); This is equivalent to the condition that the underlying morphism of classical schemes $f_\cl \colon X_\cl \to Y_\cl$ is proper in the sense of classical algebraic geometry (\cite[Remark 5.1.2.2]{SAG}). In general, a morphism $f \colon X \to Y$ between prestacks which is a relative derived scheme is called {\em proper} if the induced morphism $X_A \to \Spec A$ along every base change $\Spec A \to Y$, where $A \in \CAlgDelta$, is a proper morphism between derived schemes.
\end{definition}

\begin{proposition}[Finiteness properties] 
\label{prop:Grass:finite}
 Let $X$ be a prestack, $\sE$ a connective quasi-coherent complex on $X$, $d \ge 1$ an integer. Then:
 \begin{enumerate}
	\item 
	\label{prop:Grass:finite-1}
	If $\sE$ is locally of finite type, then the canonical projection $\pr \colon \Grass_d(\sE) \to X$ is proper (hence quasi-compact, separated and locally of finite type).
	\item 
	\label{prop:Grass:finite-2}
	If $\sE$ is pseudo-coherent to order $n$ for some $n \ge 0$, then the projection $\pr \colon \Grass_d(\sE) \to X$ is proper and locally of finite generation to order $n$.
	\item 
	\label{prop:Grass:finite-3}
	If $\sE$ is almost perfect, then the projection $\pr \colon \Grass_d(\sE) \to X$  is proper and locally almost of finite presentation.
	\item
	\label{prop:Grass:finite-4}
	 If $\sE$ is perfect, then the projection $\pr \colon \Grass_d(\sE) \to X$ is proper and locally of finite presentation.
\end{enumerate}
\end{proposition}

\begin{proof}
Since the question is local with respect to Zariski topology of $X$, we may reduce to the case where $X = \Spec A$ for some $A \in \CAlgDelta$. In order to prove assertion \eqref{prop:Grass:finite-1}, by virtue of \cite[Remark 5.0.0.1]{SAG}, Proposition \ref{prop:Grass-classical}, and Proposition \ref{prop:Grass:represent}, it suffices to show that the projection morphism $\pr_\cl \colon \Grass_d^\cl(\pi_0 \sE) \to X_\cl$ of the classical Grassmannian is proper in the sense of classical algebraic geometry. 
By virtue of \cite[Proposition 9.8.4]{EGAI}, we know that the classical Pl\"ucker morphism $\varpi_\cl \colon \Grass_d^\cl(\pi_0 \sE) \to \PP_\cl(\bigwedge_\cl^d (\pi_0\sE))$ is a closed immersion, hence proper. The projection $\PP_\cl(\bigwedge_\cl^d (\pi_0\sE)) \to X_\cl$ is locally projective as $\bigwedge_\cl^d( \pi_0\sE)$ is of finite type, hence proper (\cite[\href{https://stacks.math.columbia.edu/tag/01WC}{Tag 01WC}]{stacks-project}). Therefore, $\pr_\cl \colon \Grass_d^\cl(\pi_0 \sE) \to X_\cl$ is proper. Hence assertion \eqref{prop:Grass:finite-1} is proved, as are the claims about properness of assertions \eqref{prop:Grass:finite-2} through \eqref{prop:Grass:finite-4}. 
 
To prove the rest of the assertions, we let $\{U_J \subseteq \Grass_d(\sE)\}_{J \subseteq I}$ of $\Grass_d(\sE)$ be a Zariski open cover of $\Grass_d(\sE)$ (associated with a family of sections $\{t_i \colon \sO_X \to \sE\}_{i \in I}$ that is jointly surjective on $\pi_0$) constructed in Remark \ref{remark:Grass:UJ}. We assume $\sE$ is pseudo-coherent to order $n$ (resp. almost perfect, perfect), and we wish to show that $\Grass_d(\sE) \to X = \Spec A$ is locally of finite generation to order $n$ (resp. locally almost of finite presentation, locally of finite presentation). Since the desired properties for $\Grass_d(\sE) \to X$ are local on the source with respect to Zariski topology, it suffices to prove the assertions for the projections $\pr_{U_J} \colon U_J \to X$ from each open chart $U_J$. From Remark \ref{remark:Grass:UJ}, we have a pullback diagram:
 	$$
	\begin{tikzcd}
	U_J \ar{r}{s_J}  \ar{d}[swap]{\pr_{U_J}} & \vert \sHom_X(\sE, \sO_X^{\oplus d}) \vert = \VV(\sE)^{\times d} \ar{d}{\pr_\VV}\\
	X \ar{r}{s_X} & \vert \sHom_X(\sO_X^{\oplus d}, \sO_X^{\oplus d}) \vert  = \AA_X^{d^2}.
	\end{tikzcd}
	$$	
By virtue of Lemma \ref{lem:affine:finite.generation} and \cite[Corollary 4.1.3.3]{SAG}, the canonical projection $\VV(\sE)^{\times d} \to X$ is locally of finite generation to order $n$ (resp. locally almost of finite presentation, locally of finite presentation). Since the projection $\AA_X^{d^2} \to X$ is locally of finite presentation, we obtain from \cite[Proposition 4.1.3.1]{SAG} that the projection $\pr_{\VV} \colon \VV(\sE)^{\times d} \to \AA_X^{d^2}$ is locally of finite generation to order $n$ (resp. locally almost of finite presentation, locally of finite presentation), hence $\pr_{U_J} \colon U_J \to X$ has the same property by base change (\cite[Proposition 4.2.1.6]{SAG}).
\end{proof}

\subsubsection{Relative cotangent complexes}
The framework of derived algebraic geometry allows us to effectively describe cotangent complexes for all derived Grassmannians, in the same manner as it does for derived projectivizations \cite[Theorem 4.27]{J22a}:

\begin{theorem}[Relative cotangent complexes] 
\label{thm:Grass:cotangent} 
Let $X$ be a prestack, $\sE$ a connective quasi-coherent complex on $X$, and $d \ge 1$ an integer. Then the canonical projection map of the derived Grassmannian $\pr \colon \Grass_d(\sE) \to X$ admits a connective relative cotangent complex $\LL_{\Grass_d(\sE)/X}$ which can be described by the canonical equivalence 
	$$\LL_{\Grass_d(\sE)/X} \simeq \sR \otimes \sQ^\vee = \fib(\pr^*(\sE) \xrightarrow{\rho} \sQ) \otimes \sQ^\vee,$$
where $\sQ$ is the tautological rank $d$ vector bundle and $\rho \colon \pr^*(\sE) \xrightarrow{\rho} \sQ$ is the tautological quotient map. Consequently, if $\sE$ is pseudo-coherent to order $n$, for some integer $n \ge 0$ (resp. almost perfect, perfect, of Tor-amplitude $\le n$), then $\LL_{\Grass_d(\sE)/X}$ is pseudo-coherent to order $n$ (resp. almost perfect, perfect, of Tor-amplitude $\le n$). 
\end{theorem}

\begin{proof}
For any $A \in \CAlgDelta$ and $\eta \in \Grass_d(\sE)(A)$, we consider the functor $F_\eta \colon \Mod_A^{\rm cn} \to \shS$,
	$$F_\eta(M)  = {\rm fib} (\Grass_d(\sE)(A \oplus M)  \to \Grass_d(\sE)(A) \times_{X(A)} X(A \oplus M))$$
(where the fiber is taken over the point of $\Grass_d(\sE)(A) \times_{X(A)} X(A \oplus M))$ determined by $\eta$). By virtue of \cite[Remark 17.2.4.3]{SAG} (see also \cite[Remark 4.29]{J22a} in our setting), to prove the desired assertion of the theorem, it suffices to show:
\begin{enumerate}
	\item  For any $A \in \CAlgDelta$ and $\eta \colon T:= \Spec A  \to \Grass_d(\sE)$ which corresponds to $\overline{\eta} \colon T= \Spec A \to X$ and quotient $\rho_A \colon \sQ_A \to \sP_A$, the above defined functor $F_\eta \colon \Mod_A^{\rm cn} \to \shS$ is corepresented the connective quasi-coherent complex $\sM_\eta = \fib(\sQ_A \xrightarrow{\rho_A} \sP_A) \otimes \sQ_A^\vee$. This follows from \cite[Proposition 4.30 (3)]{J22a}.

	\item The quasi-coherent complexes $\sM_\eta$ that corepresent $F_\eta$ depend functorially on $A$ in the following strong sense: for each map $A \to B$ in $\CAlgDelta$ and $\eta \in \Grass_d(\sE)(A)$, let $\eta' \in \Grass_d(\sE)(B)$ denote the image of $\eta$, then the functor $F_{\eta'}$ is corepresented by $B \otimes_A \sM_\eta$. The quasi-coherent complexes $\sM_\eta = \fib(\sQ_T \xrightarrow{\rho_A} \sP_T) \otimes \sQ_A^\vee$ in $(1)$ evidently satisfy this property. 	
\end{enumerate} 
\end{proof}
 
\subsubsection{Closed immersions induced by surjective morphisms of complexes}
 
 \begin{proposition}[Closed immersions] \label{prop:Grass:PB}  Let $X$ be a prestack, $d \ge 1$ an integer, and $\sE' \xrightarrow{\varphi'} \sE \xrightarrow{\varphi''} \sE''$ a fiber sequence of connective quasi-coherent complexes on $X$. Let $\iota_{\varphi''} \colon \Grass_d(\sE'') \to \Grass_d(\sE)$ be the canonical map which, for any morphism $\eta \colon T=\Spec A \to X$, $A \in \CAlgDelta$, carries the element $(\eta^*(\sE'') \twoheadrightarrow \sP) \in \Grass_d(\sE'')(\eta)$ of $\Grass_d(\sE'')$ to the composite map $(\eta^*(\sE) \xrightarrow{\eta^*(\varphi'')} \eta^*(\sE'') \twoheadrightarrow \sP) \in \Grass_d(\sE)(\eta)$. Then:
 	\begin{enumerate}
		\item 
		 \label{prop:Grass:PB-1}
		The map $\iota_{\varphi''} \colon \Grass_d(\sE'') \to \Grass_d(\sE)$ is a closed immersion. 
		\item
		 \label{prop:Grass:PB-2}
		  The map $\iota_{\varphi''}$ admits a relative cotangent complex which can be described by the formula 
			$$\LL_{\Grass_d(\sE'')/\Grass_d(\sE)} \simeq \pr_{\Grass_d(\sE'')}^*(\sE') \otimes \sQ(\sE'')^\vee [1].$$
		\item 
		 \label{prop:Grass:PB-3}
		The map $\iota_{\varphi''}$ fits into a canonical pullback diagram of close immersions:
	\begin{equation} \label{eqn:Grass:PB:V}
	\begin{tikzcd} 
		\Grass_d(\sE'') \ar{d}[swap]{\iota_{\varphi''}} \ar{r}{\iota_{\varphi''}} & \Grass_d(\sE) \ar{d}{i_{\varphi'}} \\
		\Grass_d(\sE)  \ar{r}{i_{\mathbf{0}}} & \VV_{\Grass_d(\sE)}(\pr_{\Grass_d(\sE)}^* (\sE') \otimes \sQ(\sE)^\vee),
	\end{tikzcd}
	\end{equation}
where $i_{\mathbf{0}}$ and $i_{\varphi'}$ are sections of the projection map $\VV_{\Grass_d(\sE)}(\pr_{\Grass_d(\sE)}^* (\sE') \otimes \sQ(\sE)^\vee) \to \Grass_d(\sE)$ that classify the zero cosection 
	$$\pr_{\Grass_d(\sE)}^* (\sE') \otimes \sQ(\sE)^\vee \xrightarrow{0} \sO_{\Grass_d(\sE)}$$
and the cosection $\rho_{\varphi'} \colon \pr_{\Grass_d(\sE)}^* (\sE') \otimes \sQ(\sE)^\vee \to \sO_{\Grass_d(\sE)}$ corresponding to the composition
	$$\pr_{\Grass_d(\sE)}^* (\sE') \xrightarrow{\pr_{\Grass(\sE)}^* (\varphi')} \pr_{\Grass_d(\sE)}^* (\sE) \xrightarrow{\rho_{\Grass_d(\sE)}} \sQ(\sE),$$ 
respectively. In other words, $\iota_{\varphi''}$ identifies $\Grass_d(\sE'')$ as the derived zero locus of the canonical cosection  $\rho_{\varphi'} \colon \pr_{\Grass_d(\sE)}^* (\sE') \otimes \sQ(\sE)^\vee \to \sO_{\Grass_d(\sE)}$ determined by $\varphi'$.
	\end{enumerate}
\end{proposition}

\begin{proof}
The strategy is the same as the case of projectivizations \cite[Corollary 4.32, Proposition 4.33]{J22a}. More concretely, to prove assertion \eqref{prop:Grass:PB-1} and \eqref{prop:Grass:PB-3}, we can reduce to the case where $X = \Spec R$, $R \in \CAlgDelta$. Let $\{t_i \colon \sO_X \to \sE\}_{I \in I}$ be a family of sections that is jointly surjective on $\pi_0$, then $\{\varphi'' \circ t_i \colon \sO_X \to \sE''\}_{i \in I}$ is also jointly surjective on $\pi_0$. Let $\{U_J \subseteq \Grass_d(\sE)\}_{J \subseteq I, |J|=d}$ and $\{V_J \subseteq \Grass_d(\sE'')\}_{J \subseteq I, |J|=d}$ denote the respective Zariski open covers as in Remark \ref {remark:Grass:UJ}. By construction, $V_J$ is the inverse image of $U_J$ under the map $\iota_{\varphi''}$, and it suffices to show each restriction $\iota_{\varphi''}|_{V_J} \colon V_J \to U_J$ is a closed immersion. By virtue of Remark \ref {remark:Grass:UJ}, then map $\iota_{\varphi''}|_{V_J} $ is the base change of the map $\VV_X(\sE'' \otimes\sO_X^{\oplus d}) \to \VV_X(\sE \otimes \sO_X^{\oplus d})$ along the section map $s_X \colon X \to  |\sHom_X(\sO_X^{\oplus d}, \sO_X^{\oplus d})|$ which classifies the identity map. By virtue of \cite[Proposition 4.10 (6)]{J22a},  the map $\VV_X(\sE'' \otimes\sO_X^{\oplus d}) \to \VV_X(\sE \otimes \sO_X^{\oplus d})$ is a closed immersion since the map $\sE \otimes\sO_X^{\oplus d} \to \sE'' \otimes\sO_X^{\oplus d}$ is surjective on $\pi_0$. This proves assertion \eqref{prop:Grass:PB-1}.

To prove assertion \eqref{prop:Grass:PB-3}, it suffices to show that for every $J \subseteq I$, $|J| = d$, the restriction of the diagram \eqref{eqn:Grass:PB:V} to each $U_J$,
	\begin{equation*}
	\begin{tikzcd} 
		V_J \ar{d}[swap]{\iota_{\varphi''}|V_J} \ar{r}{\iota_{\varphi''}|V_J} & U_J\ar{d}{i_{\varphi'}|U_J} \\
		U_J \ar{r}{i_{\mathbf{0}}|U_J} & \VV_{U_J}(\pr_{U_J}^* (\sE') \otimes \sQ(\sF)^\vee|_{U_J}),
	\end{tikzcd}
	\end{equation*}
is a pullback square. By virtue of Remark \ref {remark:Grass:UJ}, $\sQ(\sF)^\vee|_{U_J} \simeq \sO_{U_J}^{\oplus d}$ is trivial, and the above diagram is a base change of the commutative square
	$$
	\begin{tikzcd} 
		\VV(\sE'' \otimes\sO_X^{\oplus d}) \ar{d}[swap]{\VV(\rho''^*)} \ar{r}{\VV(\rho''^*)}& \VV(\sE \otimes\sO_X^{\oplus d}) \ar{d}{i_{\varphi'}} \\
		\VV(\sE \otimes\sO_X^{\oplus d})  \ar{r}{i_{\mathbf{0}}} & \VV(\sE' \otimes\sO_X^{\oplus d}) \times_X \VV(\sE \otimes\sO_X^{\oplus d}),
	\end{tikzcd}
	$$
which is a pullback square by virtue of \cite[Remark 4.12]{J22a}. This proves assertion \eqref{prop:Grass:PB-3}.

The proof of assertion \eqref{prop:Grass:PB-2} is similar to the proof of   \cite[Corollary 4.32 (2)]{J22a}; Alternatively, it follows from assertion  \eqref{prop:Grass:PB-3} and \cite[Proposition 4.10 (3)]{J22a}.
\end{proof}

\begin{example} Let $X$ be a prestack, $\sV$ a vector bundle of rank $d$ on $X$, and $\sE$ a connective quasi-coherent complex on $X$, where $d \ge 1$ is an integer. Let $\varphi \colon \sE \to \sV$ be any map in $\QCoh(X)$. Then Proposition \ref{prop:Grass:PB} implies that there is a canonical closed immersion $\iota \colon \Grass_{X, d}(\cofib(\varphi)) \hookrightarrow \Grass_{X,d}(\sV) \simeq X$ which identifies $\Grass_{X,d}(\cofib(\varphi))$ as the derived zero locus the cosection $\sE \otimes \sV^\vee \to \sO_X$ induced by the map $\varphi \colon \sE \to \sV$. Moreover, the closed immersion $\iota$ admits a relative cotangent complex given by
	$$\LL_{\Grass_{X, d}(\cofib(\varphi))/X} \simeq (\sE \otimes \sV^\vee)|_{\Grass_{X, d}(\cofib(\varphi))}[1].$$
\end{example}

\begin{example}\label{eg:prop:Grass:PB} Let $X$ be a prestack, $\sW$ a vector bundle on $X$, and $\sE$ a connective quasi-coherent complex on $X$. Let $d \ge 1$ be an integer. By virtue of Proposition \ref{prop:Grass:PB}, the natural inclusion map $\sE \to \sE \oplus \sW[1]$ induces a canonical closed immersion
	$\iota \colon \Grass_{X,d}(\sE \oplus \sW[1]) \hookrightarrow \Grass_{X,d}(\sE)$ which fits into a pullback diagram of prestacks:
	$$
	\begin{tikzcd} 
		\Grass_{X,d}( \sE \oplus \sW[1]) \ar[hook]{d}[swap]{\iota} \ar[hook]{r}{\iota}& \Grass_{X,d}(\sE) \ar[hook]{d}{i_{\mathbf{0}}} \\
		\Grass_{X,d}(\sE) \ar[hook]{r}{i_{\mathbf{0}}} & \VV_{\Grass_{X,d}(\sE)}(\pr_{\Grass_{d}(\sE)}^*(\sW) \otimes \sQ_{\Grass_d(\sE)}^\vee).
	\end{tikzcd}
	$$
Consequently (see \cite[Proposition 4.11]{J22a}), we obtain a canonical equivalence 
	$$\Grass_{X,d}(\sE \oplus \sW[1]) \simeq \VV_{\Grass_{X,d}(\sE)}(\pr_{\Grass_{d}(\sE)}^*(\sW) \otimes  \sQ_{\Grass_d(\sE)}^\vee [1]).$$
In other words, $\Grass_{X,d}(\sE \oplus \sW[1]) \hookrightarrow \Grass_{X,d}(\sE)$ is the derived zero locus of the zero cosection of the vector bundle $\pr_{\Grass_{d}(\sE)}^*(\sW) \otimes \sQ_{\Grass_d(\sE)}^\vee$ over $\Grass_{X,d}(\sE)$ (or equivalently, the derived zero locus of the zero section of the vector bundle $\pr_{\Grass_{d}(\sE)}^*(\sW^\vee) \otimes  \sQ_{\Grass_d(\sE)}$ over $\Grass_{X,d}(\sE)$).
\end{example}

\subsection{P\"ucker Morphisms and Segre Morphisms}
\label{sec:dGrass:Plucker.Segre}
\subsubsection{P\"ucker morphisms}
\label{sec:dGrass:Plucker}
Let $X$ be a prestack, $\sE$ a connective quasi-coherent complex on $X$, and $d \ge 1$ an integer. There is a canonical morphism
		$$\varpi_\sE \colon \Grass_{d}(\sE) \to \PP(\bigwedge\nolimits^d \sE),$$
between the derived Grassmannian $\pr_{\Grass_d(\sE)} \colon \Grass_d(\sE) \to X$ and the derived projectivization $\pr_{\PP(\bigwedge\nolimits^d \sE)} \colon \PP(\bigwedge\nolimits^d \sE) = \Grass_1(\bigwedge\nolimits^d \sE) \to X$, defined as follows: for any $A \in \CAlgDelta$ and any morphism $\eta \colon T =\Spec A \to X$, the morphism $\varpi_\sE(\eta)$ is the map of spaces
	$$\big( (u \colon \sE_T \to \sP) \in \Grass_d(\sE)(\eta) \big) \mapsto \big( (\wedge^d u \colon \bigwedge\nolimits^e \sE_T \to \bigwedge\nolimits^d \sP) \in \PP(\bigwedge\nolimits^d \sE)(\eta) \big).$$ 
We will refer to the morphism $\varpi_\sE$ as the {\em Pl\"ucker morphism for $\Grass_d(\sE)$ over $X$}. By construction, there is a canonical equivalence 
	$\varpi_\sE^* (\sO_{\PP(\bigwedge\nolimits^d \sE)}(1)) \simeq \bigwedge\nolimits^d (\sQ_{\Grass_d(\sE)}).$

\begin{lemma}[{Compare with \cite[Proposition (9.8.3)]{EGAI}}]
The formation of Pl\"ucker morphisms commutes with the closed immersions studied in Proposition \ref{prop:Grass:PB}, that is, if $\varphi \colon \sE \to \sE''$ is a morphism of connective quasi-coherent complexes on $X$ which is surjective on $\pi_0$, then the following diagram commutes:
	$$
	\begin{tikzcd} 
		\Grass_d(\sE'') \ar{d}{\iota_{\varphi}} \ar{r}{\varpi_{\sE''}}& \PP(\bigwedge\nolimits^d \sE'') \ar{d}{\iota_{(\bigwedge^d \varphi)}} \\
		\Grass_d(\sE)  \ar{r}{\varpi_{\sE}} & \PP(\bigwedge\nolimits^d \sE).
	\end{tikzcd}
	$$
\end{lemma}
\begin{proof}
This follows easily from construction. 
\end{proof}

\begin{proposition}[Finiteness properties for Pl\"ucker morphisms; compare with {\cite[Proposition (9.8.4)]{EGAI}}]
\label{prop:Plucker}
The Pl\"ucker morphism $\varpi_\sE$ is a closed immersion. If $\sE$ is pseudo-coherent to order $n$ for some $n \ge 0$ (resp. almost perfect, resp. perfect), then  the closed immersion $\varpi_{\sE}$ is locally of finite generation to order $n$ (resp. locally almost of finite presentation, locally of finite presentation).
\end{proposition}

\begin{proof}
The problem being local, we may reduce to the case where $X = \Spec A$, $A \in \CAlgDelta$. It follows from the same argument of Proposition \ref{prop:Grass-classical} that the underlying classical map of $\omega_\sE$ is canonically equivalent to the classical P\"ucker morphism $\Grass^\cl_{X_\cl}(\pi_0(\sE)) \to \PP_{X_\cl}^\cl(\pi_0(\bigwedge\nolimits_\cl \pi_0 \sE))$ for the underlying classical Grassmannian.
Then the first assertion follows from the corresponding assertion for the classical P\"ucker morphism \cite[Proposition (9.8.4)]{EGAI}. For the rest of the assertions, consider the composite map $\pr_\GG \colon \Grass_d(\sE) \xrightarrow{\varpi_{\sE}} \PP(\bigwedge\nolimits^d \sE) \xrightarrow{\pr_{\PP}} X$. By virtue of Proposition \ref{prop:dSchur:pc} and Proposition \ref{prop:dSchur:Tor-amp}, if $\sE$ is is pseudo-coherent to order $n$ for some $n \ge 0$ (resp. almost perfect, resp. perfect), then $\bigwedge^d \sE$ has the same property. As a result, by Proposition \ref{prop:Grass:finite}, the projections $\pr_{\GG} \colon \Grass_d(\sE) \to X$ and $\pr_{\PP} \colon \PP(\bigwedge\nolimits^d \sE) \to X$ are locally of finite generation to order $n$ (resp. locally almost of finite presentation, locally of finite presentation). Therefore, the morphism $\omega_{\sE}$ is locally of finite generation to order $n$ (resp. locally almost of finite presentation, locally of finite presentation), by virtue of \cite[Proposition 4.1.3.1]{SAG} (resp. \cite[Corollary 7.4.3.19]{HA} and its analogous statement regarding locally finite presentations; the latter follows easily from \cite[Theorem 7.4.3.18 (1)]{HA}).
\end{proof}

\begin{remark}[Pl\"ucker morphisms in affines charts]
In the situation of Proposition \ref{prop:Plucker}, let us furthermore assume that there exists a family $\{t_i \colon \sO_X \to \sE\}_{i \in I}$ of sections of $\sE$ which is jointly surjective on $\pi_0$ as in Remark \ref{remark:Grass:UJ}, the family of sections $\{\wedge^d (t_J) \colon \sO_X \simeq \bigwedge^d (\sO_X^{\oplus d}) \to \bigwedge^d (\sE)\}_{J \subseteq I, |J|=d}$ is also jointly surjective on $\pi_0$, where $t_J := \prod\nolimits_{j\in J} t_j \colon \sO_X^J \to \sE$. We thus obtain from Remark \ref{remark:Grass:UJ} Zariski open covers $\{U_J \subseteq \Grass_d(\sE)\}_{J \subseteq I, |J|=d}$ and $\{V_J \subseteq \PP(\bigwedge^d \sE)\}_{J \subseteq I, |J|=d}$, where $U_J$ and $V_J$ are the open subschemes which represent the respective open subfunctors of $\Grass_d(\sE)$ and $\PP(\sE)$ characterized by properties that for all $T \to X$, the composite maps
	$$\sO_{T}^{\oplus d} \xrightarrow{(t_J)_T} \sE_T \to \sQ_T \quad \text{and} \quad \sO_{T} \simeq \bigwedge\nolimits^d (\sO_{T}^{\oplus d})  \xrightarrow{\wedge^d (t_J)_T} \bigwedge\nolimits^d (\sE_T) \to \sL_T$$
are isomorphisms, where $\sF_T$ denote the base change of a complex $\sF$ on $X$ along $T \to X$, and $\sE_T \to \sQ_T$ and $\sE_T \to \sL_T$ are the rank $d$ and rank $1$ locally free quotients of $\sE_T$, respecitvely. It follows that $\varpi_{\sE}^{-1}(V_J) = U_J$ for each $J \subseteq I$, $|J|=d$.
\end{remark}

\subsubsection{Segre morphisms}
\label{sec:dGrass:Segre}
Let $X$ be a prestack and,let $\sE$ and $\sF$ be two connective quasi-coherent complexes on $X$. Consider the derived projectivizations $\PP(\sE)=\Grass_1(\sE)$ and $\PP( \sF) = \Grass_1(\sF)$ over $X$, with tautological quotients $\pr_{\PP(\sE)}^*(\sE) \to \sO_{\PP( \sE)}(1)$ and  $\pr_{\PP(\sF)}^*(\sF) \to \sO_{\PP( \sF)}(1)$, respectively. Then there is a canonical morphism between prestacks
		$$\varsigma_{\sE,\sF} \colon \PP(\sE) \times_X \PP(\sF) \to \PP(\sE \otimes_{\sO_X} \sF)$$
which, for any $\eta \colon T = \Spec A \to X$, $A \in \CAlgDelta$, carries the pair of line bundle quotients
	$$(u \colon \sE_T \to \sL_1, v \colon \sF_T \to \sL_2) \in \PP(\sE) \times_X \PP(\sF) (\eta)$$
to the line bundle quotient of the tensor products
	$$\big(u \otimes v \colon (\sE \otimes_{\sO_X} \sF)_T \simeq \sE_T \otimes_{\sO_T} \sF_T  \to \sL_1 \otimes_{\sO_T} \sL_2
	\big) \in \PP(\sE \otimes_{\sO_X} \sF)(\eta).$$
Then there is a canonical equivalence  
	$\varsigma_{\sE,\sF}^*(\sO_{\PP(\sE \otimes \sF)}(1)) \simeq  \sO_{\PP( \sE)}(1)\boxtimes_X  \sO_{\PP( \sF)}(1).$ 
We will refer to the morphism $\varsigma_{\sE,\sF}$ as the {\em Segre morphism} for $\sE$ and $\sF$.

\begin{lemma}[{Compare with \cite[Proposition (9.8.7)]{EGAI}}]
The formation of Segre morphisms commutes with the closed immersions studied in Proposition \ref{prop:Grass:PB}, that is, if $\varphi \colon \sE \to \sE''$ and $\psi \colon \sF \to \sF''$ are morphism of connective quasi-coherent complexes on $X$ which are surjective on $\pi_0$, then the following diagram commutes:
	$$
	\begin{tikzcd} 
		\PP(\sE'') \times_X \PP(\sF'')  \ar{d}{\iota_{\varphi} \times_X \iota_{\psi}} \ar{r}{\varsigma_{\sE'',\sF''}}& \PP(\sE'' \otimes_{\sO_X} \sF'') \ar{d}{\iota_{\varphi \otimes \psi}} \\
		\PP(\sE) \times_X \PP(\sF)   \ar{r}{\varsigma_{\sE,\sF}} & \PP(\sE \otimes_{\sO_X} \sF).
	\end{tikzcd}
	$$
\end{lemma}
\begin{proof}
This follows easily from construction. 
\end{proof}

\begin{proposition}[Compare with {\cite[Proposition (9.8.8)]{EGAI}}]
\label{prop:Segre}
The Segre morphism $\varsigma_{\sE,\sF}$ is a closed immersion. If $\sE$ and $\sF$ are almost perfect (resp. perfect), then $\varsigma_{\sE,\sF}$ is locally almost of finite presentation (resp. locally of finite presentation).
\end{proposition}

\begin{proof}
Similar to Proposition \ref{prop:Plucker}, the first assertion follows from the classical case \cite[Proposition (9.8.8)]{EGAI}. If $\sE$ and $\sF$ are almost perfect (resp. perfect), then $\sE \otimes_{\sO_X} \sF$ is also almost perfect (resp. perfect) by the following Lemma \ref{lem:otimes_pc}. Therefore, by virtue of Proposition \ref{prop:Grass:finite}, the projections $\PP(\sE) \to X$, $\PP(\sF) \to X$, and $\PP(\sE \otimes_{\sO_X} \sF) \to X$ are locally almost of finite presentation (resp. locally of finite presentation). It follows from \cite[Proposition 4.2.1.6]{SAG} that $\PP(\sE) \times_X \PP(\sF) \to X$ is locally almost of finite presentation (resp. locally of finite presentation). Therefore, by virtue \cite[Theorem 7.4.3.18 \& Corollary 7.4.3.19]{HA}, we obtain that $\varsigma_{\sE,\sF}$ is locally almost of finite presentation (resp. locally of finite presentation).
\end{proof}

\begin{lemma} 
\label{lem:otimes_pc}
Let $X$ be a prestack, let $m_i \le n_i$ be integers, and let $\sE_i \in \QCoh(X)$ be quasi-coherent complexes on $X$, where $i=1,2$. If $\sE_i$ is pseudo-coherent to order $n_i$ (Definition \ref{def:perfectfg:prestacks}) and $m_i$-connective for each $i \in \{1,2\}$, then $\sE_1 \otimes_{\sO_X} \sE_2$ is pseudo-coherent to order $n$ and $m$-connective, where $n = \min \{n_1 + m_2, n_2 + m_1\}$ and $m = m_1 + m_2$. In particular, if $\sE_1$ and $\sE_2$ are almost perfect, then $\sE_1 \otimes_{\sO_X} \sE_2$ is almost perfect. 
\end{lemma}
\begin{proof}
The problem being local, we may assume that $X = \Spec A$ and $\sE_i$ is represented by a complex $M_i \in \Mod_A$, $i=1,2$. By virtue of \cite[Corollary 7.2.1.23]{HA}, $M_1 \otimes_A M_2$ is $m$-connective. For each $i \in \{1,2\}$, by virtue of \cite[Corollary 2.7.2.2 (2)]{SAG}, there exists a perfect complex $P_i$ of Tor-amplitude $\le n_i$ and a morphism $\phi_i \colon P_i \to M_i$ such that $\fib(\phi_i)$ is $n_i$-connective; In particular, $P_i$ is $m_i$-connective. From the fiber sequence
	$$\fib(\phi_1) \otimes_A \fib(\phi_2) \to P_1 \otimes_A \fib(\phi_2) \to M_1 \otimes_A \fib(\phi_2),$$
we obtain that $M_1 \otimes_A \fib(\phi_2)$ is $(m_1 + n_2)$-connective. The  morphism $\phi_1 \otimes_{A} \phi_2 \colon P_1 \otimes_A P_2 \to M_1 \otimes_A M_2$ factorizes as a composite map $P_1 \otimes_A P_2 \to M_1 \otimes P_2 \to M_1 \otimes M_2$, using octahedral axiom (see $(TR4)$ in the proof of  \cite[Theorem 1.1.2.14]{HA}) we obtain a fiber sequence 
	$$\fib(\phi_1) \otimes_A P_2 \to \fib(\phi_1 \otimes_A \phi_2 \colon P_1 \otimes_A P_2 \to M_1 \otimes_A M_2) \to M_1 \otimes \fib(\phi_2).$$
Therefore, $\fib(\phi_1 \otimes_A \phi_2)$ is $n = \min\{m_1 + n_2, m_2 + n_1\}$-connective. By virtue of \cite[Corollary 2.7.2.2 (3)]{SAG}, $M_1 \otimes_A M_2$ is pseudo-coherent to order $n$.
\end{proof}

% sec: Derived Flag Schemes
\subsection{Derived Flag Schemes} 
\label{sec:flag} 
This subsection studies the derived extension of the classical theory of flag schemes (\cite[\S 9.9]{EGAI}). In this subsection, we will fix a base prestack $X \colon \CAlgDelta \to \shS$, and all fiber products are taken as the fiber products over $X$.

% defn: derived Flags
\subsubsection{Definition of Derived Flag Schemes}
\begin{definition}[Derived Flag Functors] 
\label{def:dflag}
Let $X$ be a prestack, $\sE$ a connective quasi-coherent complex on $X$, and $\bdd = (d_1, \ldots, d_k)$ an increasing sequence of positive integers, where $k$ is a positive integer. For any morphism of prestacks $\eta \colon T \to X$, we let $\sE_T = \eta^* \sE \in \QCohcn(T)$ denote the base-change of $\sE$ along $\eta$.   
\begin{itemize}[leftmargin=*]
	\item We define the {\em space of flags of type $\bdd$ of $\sE_T$}, denoted by $\Flag_{\bdd}(\sE)(\eta)$, to be the full subspace of $\Fun(\Delta^k, \QCoh(T)^{\cn} )^{\simeq}$ spanned by those elements 
	\begin{equation}\label{eqn:dflag:zeta_T}
		\zeta_T = (\sE_T \xrightarrow{\varphi_{k, T} = \phi_{k,k+1}} \sE_{k} \xrightarrow{\phi_{k-1,k}} \sE_{k-1} \to \cdots \xrightarrow{\phi_{1,2}} \sE_{1})
	\end{equation}
such that each $\sE_{i}$ is a vector bundle on $T$ of rank $d_i$ and each morphism $\phi_{i,i+1}$ is surjective on $\pi_0$, where $1 \le i \le k$. We will refer to an element $\zeta_T \in \Flag_{\bdd}(\sE)(\eta)$ of the form \eqref{eqn:dflag:zeta_T} as a {\em flag of $\sE_T$ of type $\bdd$}. It is usually convenient to set $\sE_{k+1} = \sE_T$, $\sE_{0} = 0$, $\phi_{0,1} = 0$, and represent $\zeta_T$ by an extended sequence of morphisms $(\sE_{k+1} \xrightarrow{\phi_{k,k+1}} \sE_{k} \to \cdots \to \sE_{1} \xrightarrow{\phi_{0,1}} \sE_{0})$. It is also convenient to introduce the following notations: we define morphisms $\varphi_{i,T} \colon \sE_T \to \sE_i$, $1 \le i \le k$, inductively by the formula $\varphi_{k, T} = \phi_{k,k+1}$ and $\varphi_{i-1, T} = (\phi_{i-1,i}) \circ \varphi_{i,T}$. For any $0 \le i < j \le k+1$, we define $\phi_{i,j} \colon \sE_j \to \sE_i$ to be the composite morphism $\phi_{j-1,j} \circ \cdots \circ \phi_{i+1,i+2} \circ \phi_{i,i+1}$. 
	\item  We let $\Flag_{X, \bdd}(\sE) = \Flag_X(\sE;\bdd)$ denote the functor which carries each morphism of the form $\eta \colon T = \Spec A \to X$, $A \in \CAlgDelta$, to the space flags of type $\bdd$ of $\sE_T$ over $T$, that is,
	$$\Flag_{X, \bdd}(\sE)  \colon (\eta \colon T  = \Spec A \to X, A \in \CAlgDelta) \mapsto \Flag_{\bdd}(\sE)(\eta) \in \shS.$$
If the base prestack $X$ is clear from the context, we will generally omit the suffix $X$ from the expressions. We denote the canonical projection by 
	$$\pr = \pr_{\Flag_\bdd(\sE)} \colon \Flag_{\bdd}(\sE) = \Flag(\sE;\bdd) \to X$$
and refer to $\Flag_\bdd(\sE)/X$ as the
{\em derived flag scheme of $\sE$ of type $\bdd$ over $X$}. 
\end{itemize}
\end{definition}

Proposition \ref{prop:Flag:rep} below shows that $\Flag_\bdd(\sE)$ is a relative derived scheme over $X$. It would be more accurate to refer to $\Flag_\bdd(\sE)$ as the {\em relative derived flag scheme} of $\sE$ over $X$. However, we will omit the word ``relative" in order to avoid awkward terminology. 

\begin{notation} In the situation of Definition \ref{def:dflag}, the following notations are introduced: 
\label{not:dflag}
\begin{enumerate}[leftmargin=*]
	\item  \label{not:dflag-1}
	We denote the (spine of the) universal element of $\Flag_\bdd(\sE)$ by 
	$$\zeta_{\Flag_\bdd (\sE)} = (\pr^* (\sE) \xrightarrow{\phi_{k,k+1}} \sQ_{k} \xrightarrow{\phi_{k-1,k}} \sQ_{k-1} \xrightarrow{\phi_{k-1,k-2}}  \cdots \xrightarrow{\phi_{1,2}} \sQ_{1}) \in \Flag_\bdd(\sE)(\pr)$$
and refer to it as the {\em universal flag of $\sE$ of type $\bdd$} (or the {\em universal quotient sequence}). It is usually convenient to set $\sQ_{k+1} = \pr^*(\sE)$, $\sQ_{0} = 0$, $d_0 = 0$, $\phi_{0,1} = 0 \colon \sQ_1 \to \sQ_0$, and represent $\zeta_{\Flag_\bdd(\sE)}$ by the extended universal flag 
	$(\sQ_{k+1} \to \sQ_{k} \to \cdots \to \sQ_{1} \to \sQ_{0}=0).$
	\item  \label{not:dflag-2}
	For each $1 \le i \le k$, we refer to $\sQ_i = \sQ_{d_i}(\sE)$ as the {universal quotient bundle of $\sE$ of rank $d_i$}. For any $0 \le i < j \le k+1$, we define $\phi_{i, j}$ to be the composite map 
			$\phi_{i,j}: = \phi_{j-1,j} \circ \cdots \circ \phi_{i+1,i+2} \circ \phi_{i,i+1} \colon \sQ_j \to \sQ_i.$
			For $1 \le i \le k$, we let $\varphi_{i}  = \phi_{i,k+1} \colon \pr^*(\sE) \to \sQ_i$ and refer to it as the universal quotient morphisms. We let $\sR_i = \fib(\varphi_i)$ and refer to the fiber sequences
	$\sR_i \xrightarrow{\iota_i} \pr^*(\sE) \xrightarrow{\varphi_i} \sQ_i$
			as the universal fiber sequences.
	\item \label{not:dflag-3}
	For each $1 \le i \le k$, we set 
	$\sV_i  :=\fib(\phi_{i-1,i} \colon \sQ_{i} \to \sQ_{i-1}),$
then $\sV_i$ is a vector bundle of rank $(d_{i} - d_{i-1})$ over $\Flag_{\bdd}(\sE)$ (since $\phi_{i-1,i}$ is a surjective map between vector bundles). The composition $\varphi_{i-1} = (\phi_{i-1,i}) \circ \varphi_{i}$ induces a canonical fiber sequence
	$\sR_{i-1} \to \sR_{i} \to \sV_i$
(by octahedral axiom; see $(TR4)$ in the proof of \cite[Theorem 1.1.2.14]{HA}).
	\item \label{not:dflag-4}
	We will also use the notations $\sQ_i = \sQ_{i}(\sE) =\sQ_{i}^{(\bdd)}$, $\sR_i = \sR(\sE)=\sR_{i}^{(\bdd)}$, $\sV_i =\sV_i(\sE)= \sV_{i}^{(\bdd)}$, $\phi_{i,j} =\phi_{i,j}^{\sE} = \phi_{i,j}^{(\bdd)}$, etc, to emphasize the dependence of these constructions on $\sE$ or $\bdd$.
\end{enumerate}
\end{notation}

\begin{example}[Flag Bundles]
\label{eg:flag.bundle}
In the situation of Definition \ref{def:dflag}, let $\sE=\sV$ be a vector bundle of rank $n$, and assume that $d_k < n$. By convention, we set $d_0 = 0$ and $d_{k+1} = n$. Then for any field $\kappa$ and any morphism $\Spec \kappa \to X$, the fiber of $\Flag_{\bdd}(\sV) \to X$ over $\Spec \kappa \to X$ is the usual partial flag manifold $\Flag_{\bdd}(\kappa^{\oplus n})$, which is a nonsingular projective $\kappa$-variety of dimension $\sum_{i=1}^{k} d_i(d_{i+1} - d_i)$ (\cite[Proposition 3.2.9]{Wey}). Combined with Propositions \ref{prop:Flag:finite} and \ref{prop:Flag:finite} below, we obtain that the projection $\Flag_{\bdd}(\sV) \to X$ is smooth, proper, a relative derived scheme of relative dimension $\sum_{i=1}^{k} d_i(d_{i+1} - d_i)$.
\end{example}

\begin{example}[Derived Complete Flag Schemes]
\label{eg:dflag}
Let $n \ge 1$ be a positive integer, let $\bdd = \underline{n} :=(1,2, 3, \cdots, n)$ be the sequence of all integers in $[1,n]$, and consider
	$$\pr \colon \Flag_{\underline{n}}(X;\sE)  = \Flag_{X}(\sE; \underline{n}) \to X.$$
Then the extended universal quotient sequence has the form 
	$$\pr^*(\sE) \xrightarrow{\varphi_{n}} \sQ_{n} \xrightarrow{\phi_{n-1,n}} \sQ_{n-1} \to \cdots \xrightarrow{\phi_{1,2}} \sQ_{1}  \xrightarrow{\phi_{0,1}=0} \sQ_{0} = 0.$$
For each $1 \le i \le n$, we define
	$\sL_i : =  \Ker(\phi_{i-1,i} \colon \sQ_{i} \twoheadrightarrow \sQ_{i-1})$. Then 
$\sL_i$'s are line bundles on $\Flag(\sE; \underline{n})$; we will refer to the $\sL_i$'s as the {\em universal line bundles}. Consequently, for each $1 \le i \le n$, there is a canonical equivalence $\sL_1 \otimes \sL_2 \otimes \cdots \otimes \sL_i \simeq \det \sQ_i$.

In the case where $\sE = \sV$ is a vector bundle of rank $n$, the morphism $\varphi_{n} \colon \pr^{*}(\sV) \to \sQ_n$ is an isomorphism and the forgetful map  $\Flag(\sV; \underline{n}) \to \Flag(\sV;\underline{n-1})$ is an equivalence of functors. The relative scheme $\Flag(\sV): = \Flag(\sV; \underline{n}) \simeq \Flag(\sV;\underline{n-1}) \to X$ is precisely the classical {\em complete flag bundle} of $\sV$ over $X$ (compare with \cite[\S 3.1, Definition before (3.1.6)]{Wey}). Therefore, even when $\sE$ is a perfect complex rather than a vector bundle, we will sometimes abuse terminology by referring to $\Flag(\sE; \underline{n})$ as the {\em derived complete flag scheme of $\sE$}.
 \end{example}

The derived flag schemes have same functorial properties as the derived Grassmannians Proposition \ref{prop:Grass-4,5} with the same proof (see also \cite[Proposition 4.25]{J22a}):

 \begin{proposition} \label{prop:flag:functorial} 
 In the situation of Definition \ref{def:dflag} and Notation \ref{not:dflag}. 
\begin{enumerate}[leftmargin=*]
	\item \label{prop:dflag-1} 
	(Base change)
	Let $f \colon X' \to X$ be a morphism of prestacks, then there is a canonical equivalence $\Flag_{X',\bdd}(f^* \sE) \xrightarrow{\sim} \Flag_{X,\bdd}(\sE) \times_X X'$ such that the pullback of the universal quotient sequence over $\Flag_{X,\bdd}(\sE)$ along the map $\Flag_{X',\bdd}(f^* \sE) \xrightarrow{\sim} \Flag_{X,\bdd}(\sE) \times_X X' \to \Flag_{X,\bdd}(\sE)$ is canonically equivalent to the universal quotient sequence over $\Flag_{X',\bdd}(f^* \sE)$.
	\item \label{prop:dflag-2} (Tensoring with line bundles) 
	Let $\sL$ be a line bundle on $X$. Then there is a canonical equivalence $g \colon \Flag_{X,\bdd}(\sE) \xrightarrow{\sim} \Flag_{X,d}(\sE \otimes \sL)$ such that there is a canonical equivalence 
		$$g^*(\zeta_{\Flag_{X,\bdd}(\sE \otimes \sL)}) \xrightarrow{\sim} \zeta_{\Flag_{X,\bdd}(\sE)} \otimes \sL.$$
\end{enumerate}
\end{proposition}

% sec: Forgetful morphisms
\subsubsection{Forgetful morphisms}
\label{sec:dflag:forget}
In the situation of Definition \ref{def:dflag} and Notation \ref{not:dflag}, we let $\bdd' = (d_{i_j})_{1 \le j \le \ell}$ be a subsequence of $\bdd = (d_i)_{1 \le i \le k}$, where $(i_1 < i_2 < \ldots < i_\ell)$ is a subsequence of $\underline{n}: = (1 < 2 <\ldots < n)$. Then there is a natural forgetful morphism
	$$\pi_{\bdd',\bdd} \colon \Flag(\sE; \bdd) \to \Flag(\sE; \bdd')$$
which, for each $\eta \colon T = \Spec A \to X$, carries the flag $\zeta_T \in \Flag_{\bdd}(\sE)(\eta)$ of \eqref{eqn:dflag:zeta_T} to the flag
	$$\zeta_T' = (\sE_T \xrightarrow{\varphi_{i_\ell}} \sE_{i_\ell} \xrightarrow{\phi_{i_{\ell-1},i_{\ell}}} \sE_{i_{\ell-1}}  \xrightarrow{\phi_{i_{\ell-2},i_{\ell-1}}} \cdots \xrightarrow{\phi_{i_{1},i_{2}}} \sE_{i_1}) \in \Flag_{\bdd'}(\sE)(\eta).
	$$
of $\sE_T$ of type $\bdd'$. It follows that the pullback of the universal quotient sequence, $\pi_{\bdd',\bdd}^*(\zeta_{\Flag_{\bdd'}(\sE)})$,	
 is canonically equivalent to the quotient sequence
	$$(\pr_{\Flag_{\bdd}(\sE)}^*(\sE) \xrightarrow{\varphi_{i_\ell}^{(\bdd)}} \sQ_{i_\ell}^{(\bdd)} \xrightarrow{\phi_{i_{\ell-1},i_{\ell}}^{(\bdd)}} \sQ_{i_{\ell-1}}^{(\bdd)}  \xrightarrow{\phi_{i_{\ell-2},i_{\ell-1}}^{(\bdd)}} \cdots \xrightarrow{\phi_{i_{1},i_{2}}^{(\bdd)}} \sQ_{i_1}^{(\bdd)})
	$$
on $\Flag_{\bdd}(\sE)$, where $\sQ_{i}^{(\bdd)}$ and $\phi_{i,j}^{(\bdd)}$ are as defined in Notation \ref{not:dflag} \eqref{not:dflag-4}.

We will agree on convention that $\pi_{\bdd, \bdd} = \id$ denotes the identity morphism,  $\Flag(\sE; \underline{0}) = X$, and the forgetful morphism $\pi_{\underline{0}, \bdd} \colon \Flag(\sE; \bdd) \to X$ denotes the natural projection morphism.

% Lem: forget
\begin{lemma}
\label{lem:dflag:forget}
 In the situation of Notation \ref{not:dflag}, let $\bdd = (d_i)_{1 \le i \le k}$ and assume that $k \ge 2$.
\begin{enumerate}[leftmargin=*]
	\item \label{lem:dflag:forget-1}
	(The case $\bdd' = \bdd \backslash \{d_1\}$.)
	The forgetful morphism $\pi_{\bdd \backslash \{d_1\},\bdd} \colon \Flag_X(\sE; \bdd) \to \Flag_X(\sE; \bdd \backslash \{d_1\})$ canonically identifies $\Flag_X(\sE; \bdd)$ as the derived Grassmannian $\Grass \big(\sQ_{2}^{(\bdd \backslash \{d_1\})} ;d_1 \big)$ over $\Flag_X(\sE; \bdd \backslash \{d_1\})$, where $\sQ_{2}^{(\bdd \backslash \{d_1\})}$ is the universal quotient vector bundle of rank $d_2$. 
	\item \label{lem:dflag:forget-2}
	(The case $\bdd' = \bdd \backslash \{d_k\}$.)
	The forgetful morphism  $\pi_{\bdd \backslash \{d_k\},\bdd} \colon  \Flag_X(\sE; \bdd) \to \Flag_X(\sE; \bdd \backslash \{d_k\})$ canonically identifies $\Flag_X(\sE; \bdd)$ as the derived Grassmannian $\Grass\big(\sR_{k-1}^{(\bdd \backslash \{d_k\})} ; d_{k} - d_{k-1}\big)$ over $\Flag_X(\sE; \bdd \backslash \{d_k\})$, where $\sR_{k-1}^{(\bdd \backslash \{d_k\})} = \fib(\varphi_{k-1}^{(\bdd \backslash \{d_k\})})$ (see Notation \ref{not:dflag} \eqref{not:dflag-2}).
	\item \label{lem:dflag:forget-3}
	(The case $\bdd' = \bdd \backslash \{d_i\}$, $1 < i < k$.)
	If $k \ge 3$, and let $i$ be an integer such that $1 < i < k$. Then the forgetful morphism $\pi_{\bdd \backslash \{d_i\},\bdd} \colon  \Flag_X(\sE; \bdd) \to \Flag_X(\sE; \bdd \backslash \{d_i\})$ canonically identifies $\Flag_X(\sE; \bdd)$ as the derived Grassmannian $\Grass\big(\sV_{i-1, i+1}; d_{i} - d_{i-1}  \big)$ 
	over $\Flag_X(\sE; \bdd \backslash \{d_i\})$, where $\sV_{i-1, i+1} :=\fib(\phi_{i-1,i+1}^{\bdd \backslash \{d_i\}} \colon \sQ_{i+1}^{\bdd \backslash \{d_i\}} \to \sQ_{i-1}^{\bdd \backslash \{d_i\}})$ is a vector bundle on $\Flag_X(\sE; \bdd \backslash \{d_i\})$ of rank $(d_{i+1}-d_{i-1})$, and $\sQ_{\ell}^{\bdd \backslash \{d_i\}}$ denote the universal quotient bundles on $\Flag_X(\sE; \bdd \backslash \{d_i\})$ of rank $\ell$ where $\ell=i-1,i+1$.
\end{enumerate}
\end{lemma}

\begin{proof} In order to prove assertion \eqref{lem:dflag:forget-1}, it suffices to observe that, for each $\eta \colon T = \Spec A \to X$, where $A \in \CAlgDelta$, the morphism which carries each flag of the form $\zeta_T$ of \eqref{eqn:dflag:zeta_T} to the rank-$d_1$ locally free quotient $\sE_{2} \to \sE_{1}$ defines a homotopy equivalence from the space of liftings $\zeta_T \in \Flag_X(\sE; \bdd)(\eta)$ of $\zeta_T' = (\sE_T \to \sE_k \to \cdots \to \sE_2) \in\Flag_X(\sE; \bdd \backslash \{d_1\})(\eta)$ to the space $\Grass_{d_1}(\sQ_{2}^{(\bdd \backslash \{d_1\})})(\zeta_T')$ of rank-$d_1$ locally free quotients $\sE_2 \to \sE_1$ over $\zeta_T'$. 

Next, we prove assertion \eqref{lem:dflag:forget-2}. Given any $\eta \colon T = \Spec A \to X$, $A \in \CAlgDelta$, each element $\zeta_T$ of the form \eqref{eqn:dflag:zeta_T} canonically determines a fiber sequence $\sR_{k} \to \sR_{k-1} \to \sV_{k}$ as in  Notation \ref{not:dflag} \eqref{not:dflag-4}, where $\sR_i=\fib(\varphi_{i,T} \colon \sE_T \to \sE_i)$ as usual, $i=k,k-1$, and $\sV_k = \fib(\sE_{k} \to \sE_{k-1})$ is a vector bundle of rank $(d_k - d_{k-1})$. Hence, $(\eta \colon T \to X, \zeta_T)$ canonically determines a point $(\sR_{k-1} \to \sV_k)$ on the Grassmannian $\Grass_{d_{k} - d_{k-1}}\big(\sR_{k-1}^{(\bdd \backslash \{d_k\})} \big)(\zeta_T')$ over the point $\zeta_T':=\pi_{\bdd \backslash \{d_k\},\bdd}(\zeta_T)$. Conversely, given any quotient sequence $\zeta_T' = (\sE_T \to \sE_{k-1} \to \cdots \to \sE_1) \in \Flag_X(\sE; \bdd \backslash \{d_k\})(\eta)$ and any rank-$(d_{k}-d_{k-1})$ locally free quotient $\sR_{k-1} \to \sV_{k}$ of $\sR_{k-1} = \fib(\sE_T \to \sE_{k-1})$, we define
	$$\sE_{k} = \cofib(\psi_k \colon \fib(\sR_{k-1} \to \sV_{k}) \to \sE_T),$$
where the morphism $\psi_k$ is the composition of canonical morphisms $\fib(\sR_{k-1} \to \sV_{k}) \to \sR_{k-1} \to \sE_T$.
Therefore, the edge $(\sE_T \to \sE_{k-1})$ in $\zeta_T'$ factorize through a sequence of quotients $\sE_T \to \sE_{k} \to \sE_{k-1}$. Moreover, from the fiber sequence $\sV_{k} \to \sE_{k} \to \sE_{k-1}$, we obtain that $\sE_{k}$ is a vector bundle of rank $d_{k}$. In particular, the sequence $(\sE_T \to \sE_{k} \to \sE_{k-1} \to \cdots \to \sE_1)$ determines an element in $\Flag_X(\sE; \bdd)(\eta)$ up to contractible choices. The above argument shows that there is a homotopy equivalence between the space of liftings $\zeta_T \in \Flag_X(\sE; \bdd)(\eta)$ of  $\zeta_T'$ and the space $\Grass_{d_{k} - d_{k-1}}\big(\sR_{k-1}^{(\bdd \backslash \{d_k\})} \big)(\zeta_T')$ of locally free quotients $\sR_{k-1} \to \sV_k$ over $\zeta_T'$. This proves \eqref{lem:dflag:forget-2}. 

The proof of assertion \eqref{lem:dflag:forget-3} is similar to that of \eqref{lem:dflag:forget-2}: given any $\eta \colon T = \Spec A \to X$, 
and any $\zeta_T'  = (\sE_T \to \sE_{k} \to \cdots \to \sE_{i+1} \xrightarrow{\phi_{i-1,i+1}} \sE_{i-1} \to \cdots \to \sE_{1}) \in \Flag_X(\sE; \bdd \backslash \{d_i\})(\eta)$, the space of  liftings $\zeta_T$ of $\zeta_T'$ of the form \eqref{eqn:dflag:zeta_T} is homotopy equivalent to the space of sequences of morphisms 
	$$\sE_{i+1} \xrightarrow{\phi_{i,i+1}} \sE_i \xrightarrow{\phi_{i-1,i}} \sE_{i-1}$$
such that $\phi_{i-1,i+1} \simeq \phi_{i,i+1} \circ \phi_{i-1,i}$ and $\sE_i$ is a vector bundle of rank $d_i$. In particular, by virtue of octahedral axiom, any lifting $\zeta_T$ determines a fiber sequence 
	$$\fib(\phi_{i,i+1})=\sV_{i+1}  \to \fib(\phi_{i-1,i+1}) \to \fib(\phi_{i-1,i}) = \sV_i$$
 where the morphism $\fib(\phi_{i-1,i+1}) \to \sV_i$ is a rank-$(d_{i}-d_{i-1})$ locally free quotient of $\fib(\phi_{i-1,i+1})$. Conversely, given any rank-$(d_{i}-d_{i-1})$ locally free quotient $\fib(\phi_{i-1,i+1}) \to \sV_{i}$, we set 
	$$\sE_{i} = \cofib\big( \psi_i' \colon \fib(\fib(\phi_{i,i+1}) \to \sV_{i}) \to \sE_{i+1} \big)$$
where $\psi_i'$ is the composition of canonical maps $\fib(\fib(\phi_{i,i+1}) \to \sV_{i})  \to \fib(\phi_{i,i+1})  \to \sE_{i+1}$, then we obtain a sequence of morphisms $\sE_{i+1} \xrightarrow{\phi_{i,i+1}} \sE_i \xrightarrow{\phi_{i-1,i}} \sE_{i-1}$ which satisfies $\phi_{i-1,i+1} \simeq \phi_{i,i+1} \circ \phi_{i-1,i}$ and that $\sE_i$ is a vector bundle of rank $d_i$. Therefore, we obtain a homotopy equivalence between the space of liftings $\zeta_T \in \Flag_X(\sE; \bdd)(\eta)$ of $\zeta_T'$ and the space $\Grass_{d_{i} - d_{i-1}}(\fib(\phi_{i-1,i+1})(\zeta_T')$ of locally free quotient $\fib(\phi_{i-1,i+1}) \to \sV_{i}$ over $\zeta_T'$. This proves assertion \eqref{lem:dflag:forget-3}.
\end{proof}

\begin{corollary}
\label{cor:dflag:forget}
In the situation of Notation \ref{not:dflag}, let $\bdd = (d_i)_{1 \le i \le k}$ be an increasing sequence of integers, and assume that $k \ge 2$. Let $i$ be an integer such that $1 \le i \le k-1$, and let $\bdd' = (d_1, \cdots, d_{i})$ and $\bdd''=(d_{i+1}, \cdots, d_{k})$ so that $\bdd = (\bdd', \bdd'')$.
\begin{enumerate}[leftmargin=*]
	\item \label{cor:dflag:forget-1}
	The forgetful map $\pi_{\bdd'',\bdd} \colon \Flag_X(\sE; \bdd) \to \Flag_X(\sE;  \bdd'')$ canonically identifies $\Flag_X(\sE; \bdd)$ as the derived flag scheme $\Flag\big(\sQ_{i+1}^{(\bdd'')}; \bdd'\big)$ over $\Flag_X(\sE;  \bdd'')$, where $\sQ_{i+1}^{(\bdd'')}$ is the universal quotient bundle on $\Flag_X(\sE;  \bdd'')$ of rank $d_{i+1}$.
	\item \label{cor:dflag:forget-2}
	The forgetful map $\pi_{\bdd',\bdd}  \colon \Flag_X(\sE; \bdd) \to \Flag_X(\sE; \bdd')$ canonically identifies $\Flag_X(\sE; \bdd)$ as the derived flag scheme $\Flag\big(\fib(\varphi_i^{(\bdd')}); d_{i+1}-d_{i}, \ldots, d_{k} - d_{i}\big)$ over $\Flag_X(\sE; \bdd')$.
	\item \label{cor:dflag:forget-3}
	Assume that $k \ge 3$ and let $i,j$ be integers such that $1 \le i < j < k$. We write $\bdd$ as:
		$$\bdd = (\underbrace{d_1, \cdots, d_{i}}_{\bdd^{(1)}}; \underbrace{d_{i+1}, \cdots, d_{j}}_{\bdd^{(2)}}; \underbrace{d_{j+1}, \cdots, d_{k}}_{\bdd^{(3)}}).$$
		Then the forgetful map 
			$\Flag_X(\sE; \bdd) \to \Flag_X(\sE; \bdd^{(1)}, \bdd^{(3)})$
			 canonically identifies $\Flag_X(\sE; \bdd)$ as the derived flag scheme $\Flag\big(\fib(\phi_{i,j+1}^{(\bdd^{(1)}, \bdd^{(3)})}) ; d_{i+1}-d_{i}, \ldots, d_{j} - d_{i}\big)$ over $\Flag_X(\sE; \bdd^{(1)}, \bdd^{(3)})$, where
			 $\phi_{i,j+1}^{(\bdd^{(1)}, \bdd^{(3)})} \colon \sQ_{j+1}^{(\bdd^{(1)}, \bdd^{(3)})} \to \sQ_{i}^{(\bdd^{(1)}, \bdd^{(3)})}$ is the canonical quotient morphism between the universal quotient bundles of ranks $d_{j+1}$ and respectively $d_{i}$ on $\Flag_X(\sE; \bdd^{(1)}, \bdd^{(3)})$. 
\end{enumerate}
\end{corollary}

\begin{proof}
This follows from Lemma \ref{lem:dflag:forget} and induction. 
\end{proof}

Using Corollary \ref{cor:dflag:forget}, we can completely describe all forgetful morphisms $\pi_{\bdd', \bdd}$ among derived flag schemes of different types. We spell out one special case in detail:

\begin{corollary}
\label{prop:dflag:forget}
Assume we are in the situation in Notation \ref{not:dflag}. 
Let $n \ge 1$ be an integer and $\bdd=(d_1, \ldots, d_k)$ a subsequence of $\underline{n} = (1, 2, \ldots, n)$. Then the forgetful map
	$$\pi_{\bdd, \underline{n}} \colon \Flag_X(\sE; \underline{n}) \to \Flag_X(\sE; \bdd)$$
is canonically identified with the fiber products of derived flag schemes over $\Flag_X(\sE;\bdd)$,
	$$\Flag(\sR; \underline{n-d_k}) \times_{\Flag(\sE;\bdd)}  \Flag(\sV_k; \underline{d_k - d_{k-1}}) \times_{\Flag(\sE;\bdd)} \cdots \times_{\Flag(\sE;\bdd)} \Flag(\sV_1; \underline{d_1}),$$
where $\sR = \fib(\pr_{\Flag_{\bdd}(\sE)}^*(\sE) \xrightarrow{\varphi_{k}^{(\bdd)}} \sQ_{k}^{\bdd})$ and $\sV_j = \Ker(\sQ_{j}^{(\bdd)} \xrightarrow{\phi_{j-1,j}^{(\bdd)}} \sQ_{j-1}^{(\bdd)})$. Moreover, we let 
	\begin{align*}
	& \pi_{\bdd, \underline{n}}^*(\sR) \to \sQ_{n-d_k}(\sR) \twoheadrightarrow \sQ_{n - d_{k}-1}(\sR) \twoheadrightarrow \cdots \twoheadrightarrow \sQ_{1}(\sR) \twoheadrightarrow \sQ_{0}(\sR)=0 \\
	& \pi_{\bdd, \underline{n}}^* (\sV_j) \xrightarrow{\sim} \sQ_{d_j - d_{j-1}}(\sV_j) \twoheadrightarrow \sQ_{d_j - d_{j-1}-1}(\sV_j) \twoheadrightarrow \cdots \twoheadrightarrow \sQ_{1}(\sV_j) \twoheadrightarrow \sQ_{0}(\sV_j)=0
	\end{align*}
denote the sequences of quotients on $\Flag(\sE;\underline{n})$ which are the pullbacks of the extended universal quotient sequences on the derived (complete) flag schemes $\Flag(\sR; \underline{n-d_k})$, and respectively $\Flag(\sV_j; \underline{d_j - d_{j-1}})$, where $j=1,\ldots,k$, then there are canonical equivalences
	$$\sL_s \simeq 
	\begin{cases}
	\Ker(\sQ_{s-d_{j-1}}(\sV_j) \twoheadrightarrow \sQ_{s-d_{j-1}-1}(\sV_j)) & \text{if} \quad d_{j-1}+1 \le s \le d_j ~\text{for $j=1,\ldots, k$}; \\
	\Ker(\sQ_{s-d_{k}}(\sR) \twoheadrightarrow \sQ_{s-d_{k}-1}(\sR)) & \text{if} \quad d_{k}+1 \le s \le n,
	\end{cases}
	$$
where $\{\sL_s\}_{1 \le s \le n}$ are the line bundles on $\Flag(\sE; \underline{n})$ defined in Example \ref{eg:dflag}.
\end{corollary}

\begin{example}
\label{eg:forget:flag.to.Grass}
Let $1 \le d \le n$ be an integer, let $\pr \colon \Grass(\sE;d) \to X$ denote the derived Grassmannian, and let 
	$\sR \to \pr^*(\sE) \to \sQ$
denote the tautological fiber sequence on $\Grass(\sE;d)$. Then the natural forgetful map
	$\pi_{(d), \underline{n}} \colon \Flag_X(\sE; \underline{n}) \to \Grass_X(\sE; d)$
is canonically identifies $\Flag_X(\sE; \underline{n})$ with the fiber product 
	$\Flag(\sR; \underline{n-d}) \times_{\Grass(\sE;d)} \Flag(\sQ; \underline{d}).$
\end{example}

\begin{proposition}[Classical truncations]
\label{prop:Flag:classical}
In the situation of Definition \ref{def:dflag}, the restriction 
 	$\Flag_{\bdd}(\sE)|_{\CAlg^\heartsuit} \colon \CAlg^\heartsuit \subseteq \CAlgDelta \xrightarrow{\Flag_{\bdd}(\sE)} \shS$
is canonically equivalent to the classical flag functor $\Flag^\cl_{\bdd}(\pi_0 \sE)$ (\cite[(9.9.2)]{EGAI})
which carries each pair $(R \in \CAlg^\heartsuit, \eta \colon \Spec R \to X)$ to the set of isomorphism classes of flags $\eta_\cl^*(\pi_0 \sE) \to \sP_{k} \to  \sP_{k-1} \to \cdots \to \sP_{1}$ of vector bundle quotients of the discrete sheaf $\eta_\cl^*(\pi_0 \sE):=\pi_0 (\eta^* (\pi_0 \sE))$ of type $\bdd$.
\end{proposition}
\begin{proof}
This can be proved similarly as Proposition \ref{prop:Grass-classical}. Alternatively, we could deduce it from Proposition \ref{prop:Grass-classical} by considering the forgetful morphism $\pi_{(d_k), \bdd}$ as with Proposition \ref{prop:Flag:rep}.
\end{proof}

\begin{proposition}[Representability]
\label{prop:Flag:rep}
In the situation of Definition \ref{def:dflag}, the canonical projection $\pr \colon \Flag_{\bdd}(\sE) \to X$ is a relative derived scheme. 
\end{proposition}
\begin{proof}
The natural forgetful morphism 
	$$\pi_{(d_k), \bdd} \colon \Flag_{\bdd}(\sE) \to \Grass_{d_k}(\sE)$$
factorizes through a sequence of forgetful morphisms
	$$\Flag_{\bdd}(\sE) \to \Flag_{\bdd \backslash \{d_1\}}(\sE) \to \Flag_{\bdd \backslash \{d_1, d_2\}}(\sE) \to \cdots \to \Flag_{(d_{k-1}, d_k)}(\sE) \to \Grass_{d_k}(\sE)$$
where each one of the consecutive morphisms is a relative derived Grassmannian of a vector bundle, by virtue of Lemma \ref{lem:dflag:forget} \eqref{lem:dflag:forget-1}, hence in particular a relative derived scheme. Therefore, the desired assertion follows from Propositions \ref{prop:Grass:represent}.
\end{proof}

\begin{proposition}[Finiteness Properties for Derived Flag Schemes] 
\label{prop:Flag:finite}
In the situation of Definition \ref{def:dflag},
if $\sE$ is locally of finite type (Definition \ref{def:perfectfg:prestacks}), then the canonical projection $\pr \colon \Flag_{\bdd}(\sE) \to X$ is proper. % (hence quasi-compact, separated and locally of finite type). 
If $\sE$ is pseudo-coherent to order $n$ for some $n \ge 0$ (resp. almost perfect, resp. perfect), then  $\pr \colon \Flag_{\bdd}(\sE) \to X$ is proper and locally of finite generation to order $n$ (resp. locally almost of finite presentation, resp. locally of finite presentation).
\end{proposition}

\begin{proof}
The forgetful morphism $\pi_{(d_k), \bdd}$ considered in the proof of Proposition \ref{prop:Flag:rep} is smooth and proper. Consequently, the desired results follow from Proposition \ref{prop:Grass:finite}.
\end{proof}

\begin{remark}
Using the above proof and Remark \ref{rmk:Grass:nonempty}, we obtain that $\Flag_{\bdd}(\sE)$ is nonempty if and only if $\Grass_{d_k}(\sE)$ is nonempty, if and only if there exists a morphism $\eta \colon \Spec \kappa \to X$, where $\kappa$ is a field, such that $\rank_{\kappa} (\pi_0(\eta^*\sE)) \ge d_k$. If $X$ is a derived scheme and $\pi_0(\sE)$ is of finite type over $X_\cl$, then above conditions are equivalent to that the
$d_k$th Fitting locus $|X^{\ge d_k}(\pi_0(\sE))| \subseteq |X|$ of $\pi_0(\sE)$ over $X_{\cl}$ is nonempty. 
\end{remark}

% sec: Forgetful morphisms
\subsubsection{Closed immersions into products of Grassmannians}
\label{sec:closed.dflag.to.dGrass}
Assume we are in the same situation as in Definition \ref{def:dflag}, and assume that the increasing sequence of positive integers $\bdd = (d_1, \ldots, d_k)$ has length $k \ge 2$. Then there is a canonical morphism of prestacks
	$$i_{\sE, \bdd} \colon \Flag_{\bdd}(\sE) = \Flag_{X}(\sE; \bdd) \to \GG_{\bdd}: = \prod_{i=1}^k \Grass_{d_i}(\sE)$$
(recall that all products in this section are taken in the category of prestacks over $X$) which, for any $\eta \colon T = \Spec A \to X$, $A \in \CAlgDelta$, carries each flag $\zeta_T \in \Flag_{\bdd}(\sE)(\eta)$ of the form \eqref{eqn:dflag:zeta_T} to the collection of locally free quotients
	$$( \varphi_{i, T} \colon \sE_T \to \sE_{i})_{1 \le i \le k} \in \GG_{\bdd}(\eta) = \prod\nolimits_{i=1}^k \Grass_{d_i}(\sE)(\eta).$$
For each $1 \le i \le k$, we let $(\sQ_{i})_{\GG_{\bdd}}$ and $(\sR_{i})_{\GG_{\bdd}}$ denote the pullback of the bundle $\sQ_{i}$ and complex $\sR_{i} = \fib(\pr_{\Grass_{d_i}}^*(\sE) \to \sQ_{i})$ along the natural projection $\GG_{\bdd} \to \Grass_{d_i}(\sE)$, so that there are canonical fiber sequences 
	$(\sR_{i})_{\GG_{\bdd}} \xrightarrow{\iota_i} \pr_{\GG_{\bdd}}^* (\sE) \xrightarrow{\rho_i} (\sQ_{i})_{\GG_{\bdd}}.$
For each $1 \le i \le k-1$, let
	$$\vert \sHom_{\GG_{\bdd}}(\sR_{i+1})_{\GG_{\bdd}}, (\sQ_i)_{\GG_{\bdd}}) \vert : %= \VV((\sR_{i+1})_{\GG_{\bdd}} \otimes_{\sO_{\GG_{\bdd}}} (\sQ_i)_{\GG_{\bdd}}^\vee)
	 = \Spec_{\GG_{\bdd}} \Sym_{\sO_{\GG_{\bdd}}}^*((\sR_{i+1})_{\GG_{\bdd}} \otimes_{\sO_{\GG_{\bdd}}} (\sQ_i)_{\GG_{\bdd}}^\vee) \to \GG_{\bdd}$$
 be the relative derived affine scheme which classifies morphisms from $(\sR_{i+1})_{\GG_{\bdd}} $ to $(\sQ_i)_{\GG_{\bdd}}$ over $\GG_{\bdd}$ (\cite[Example 4.8]{J22a}). 
Then the composite morphisms $(\rho_i \circ \iota_{i+1})_{1 \le i \le k-1}$ and the zero morphisms $(0)_{1 \le i \le k-1}$ are classified by the respective section maps 
 	$$i_{\rm can} \colon \GG_{\bdd} \to \prod_{i=1}^{k-1} \vert \sHom_{\GG_{\bdd}}((\sR_{i+1})_{\GG_{\bdd}}, (\sQ_i)_{\GG_{\bdd}}) \vert \qquad 
	 i_{\mathbf{0}}  \colon \GG_{\bdd} \to \prod_{i=1}^{k-1} \vert \sHom_{\GG_{\bdd}}((\sR_{i+1})_{\GG_{\bdd}}, (\sQ_i)_{\GG_{\bdd}}) \vert.$$

\begin{proposition}
\label{prop:dflag:into.Grass}
In the above situation, the following diagram is a pullback diagram:
	\begin{equation*} 
	\begin{tikzcd}
		 \Flag_{\bdd}(\sE) \ar{d}[swap]{i_{\sE, \bdd}} \ar{r}{i_{\sE, \bdd}} & \GG_{\bdd} = \prod_{i=1}^k \Grass_{d_i}(\sE)  \ar{d}{i_{\rm can}} \\
		\GG_{\bdd} \ar{r}{i_{\mathbf{0}}} &  \prod_{i=1}^{k-1} \vert \sHom_{\GG_{\bdd}}((\sR_{i+1})_{\GG_{\bdd}}, (\sQ_i)_{\GG_{\bdd}}) \vert.
	\end{tikzcd}
	\end{equation*}
In particular, the morphism $i_{\sE, \bdd}$ is a closed immersion. Moreover, if $\sE$ is pseudo-coherent to order $n$ for some $n \ge 0$ (resp. almost perfect, resp. perfect; see Definition \ref{def:perfectfg:prestacks}), then the closed immersion $i_{\sE, \bdd}$ is locally of finite generation to order $n$ (resp. locally almost of finite presentation, resp. locally of finite presentation).
\end{proposition}

\begin{proof}
We first prove that the diagram of Proposition \ref{prop:dflag:into.Grass} is a pullback square in the case where $k=2$, $\bdd = (d_1, d_2)$. 
%Let $\pr_{\FF} \colon \Flag(\sE; \bdd) \to X$ and $\pr_{\GG} \colon \GG_{\bdd} \to X$ denote the natural projections. 
Let $Z$ denote the fiber product of $\GG_{\bdd}$ with itself over $\vert \sHom_{\GG_d}((\sR_2)_{\GG_{\bdd}}, (\sQ_1)_{\GG_{\bdd}}) \vert$ along the section maps $i_{\mathbf{0}}$ and $i_{\rm can}$. For any $\eta \colon T = \Spec A \to X$, where $A \in \CAlgDelta$, and any flag 
	$\zeta_T = (\sE_T \to \sE_2 \to \sE_1) \in \Flag_{\bdd}(\sE)(\eta),$
we let $\sR_2:= \fib(\sE_T \to \sE_2)$, then the canonical map $\sE_2 \to \sE_1$ induces a path between the composite map 
	$$\sR_2 \to \sE_T \to \sE_2 \to \sE_1$$
and the zero map in the mapping space $\Map_{\QCoh(T)}(\sR_2, \sE_1)$. Therefore, the morphism $i_{\sE,\bdd}$ factorizes through an essentially unique map $\Flag_{\bdd}(\sE) \to Z$. Conversely, 
the space $Z(\eta)$ is homotopy equivalent to the space of triples $(\zeta_1,\zeta_2, \theta)$, where $\zeta_1$ and $\zeta_2$ are cofiber sequences
	$$\zeta_1 \colon \sR_1 \xrightarrow{\iota_{1,T}} \sE_T \xrightarrow{\varphi_{1,T}} \sE_1 \qquad \zeta_2 \colon \sR_2 \xrightarrow{\iota_{2,T}} \sE_T \xrightarrow{\varphi_{2,T}} \sE_2$$
in $\QCoh(T)$, essentially uniquely determined by a pair $(\varphi_{1,T}, \varphi_{2,T}) \in \GG_\bdd(\eta)$, and $\theta$ is an element of the path space 
%$\Map_{|\Hom_{T}(\sR_2, \sQ_1)|}(\varphi_{1,T} \circ \iota_{2,T}, 0)$ 
between the composite map $\sR_2 \xrightarrow{\iota_{2,T}} \sE_T \xrightarrow{\varphi_{1,T}} \sE_1$ and the zero map in the mapping space $\Map_{\QCoh(T)}(\sR_2, \sE_1)$. By the universal property of pushout squares, the above path space is homotopy equivalent to the space of factorizations $\sE_2 \simeq \cofib(\sR_2 \to \sE_T) \to \sE_1$ of $\sE_T \xrightarrow{\varphi_{1,T}} \sE_1$. Hence the map $\Flag_{\bdd}(\sE)(\eta) \to Z(\eta)$ is a homotopy equivalence. Consequently, the morphism $\Flag_{\bdd}(\sE) \to Z$ is an isomorphism of prestacks. The case for general $k \ge 2$ is similar. 

The assertions about finiteness properties can be similarly proved as in Proposition \ref{prop:Plucker}, using Propositions \ref{prop:Flag:finite} and \ref{prop:Grass:finite} and \cite[Proposition 4.1.3.1]{SAG} (and \cite[Theorem 7.4.3.18]{HA}).
\end{proof}

\begin{corollary} 
\label{cor:Plucker:flag:finite}
In the above situation, let $\varpi_{\sE, \bdd}$ denote the composite map
	$$
	\begin{tikzcd}[column sep = 5 em, row sep = 2.5 em]
		\Flag(\sE; \bdd) \ar{r}{i_{\sE, \bdd}}& \prod_{i=1}^k \Grass_{d_i}(\sE) \ar{r}{\prod_{i=1}^k \varpi_{\sE, d_i}} & \prod_{i=1}^k \PP(\bigwedge\nolimits^{d_i} \sE),
	\end{tikzcd}
	$$
where $\varpi_{\sE, d_i}$ are the Pl\"ucker morphisms of derived Grassmannians (\S \ref{sec:dGrass:Plucker}). Then $\varpi_{\sE, \bdd}$ is a closed immersion. Moreover, if $\sE$ is pseudo-coherent to order $n$ for some $n \ge 0$ (resp. almost perfect, resp. perfect), then $\varpi_{\sE, \bdd}$ is of locally of finite generation to order $n$ (resp. locally almost of finite presentation, resp. locally of finite presentation).
\end{corollary}
\begin{proof} 
Combine Propositions \ref{prop:Plucker} and \ref{prop:dflag:into.Grass}.
\end{proof}

The following results shows that the compatibility of the closed immersions $i_{\sE,\bdd}$ and forgetful morphisms considered in \S \ref{sec:dflag:forget}.

\begin{proposition} 
\label{prop:Plucker:flag}
In the above situation, if $\bdd' = (d_{i_j})_{1 \le j \le \ell}$ is a subsequence of $\bdd = (d_i)_{1 \le i \le k}$. There are commutative diagrams of prestacks over $X$:
	$$
	\begin{tikzcd}[column sep = 5 em, row sep = 2.5 em]
		\Flag(\sE; \bdd) \ar{d}{\pi_{\bdd',\bdd}} \ar{r}{i_{\sE, \bdd}}& \prod_{i=1}^k \Grass_{d_i}(\sE) \ar{d}{p_{\bdd', \bdd}} \ar{r}{\prod_{i=1}^k \varpi_{\sE, d_i}} & \prod_{i=1}^k \PP(\bigwedge\nolimits^{d_i} \sE) \ar{d}{q_{\bdd',\bdd}} \\
		\Flag(\sE;\bdd')  \ar{r}{i_{\sE, \bdd'}} & \prod_{j=1}^{\ell} \Grass_{d_{i_j}}(\sE) \ar{r}{\prod_{j=1}^\ell \varpi_{\sE, d_{i_j}}} & \prod_{j=1}^\ell \PP(\bigwedge\nolimits^{d_{i_j}} \sE).
	\end{tikzcd}
	$$
Here, the first vertical arrow is the forgetful morphism (\S \ref{sec:dflag:forget}), and the vertical arrows of the middle and right-most column are the canonical projection morphisms. 
\end{proposition}
\begin{proof}
This follows easily from the definitions of these functors.
\end{proof}

\subsubsection{Closed immersion induced by surjective morphisms of complexes}
Similar to the situation of derived Grassmannians, if $\varphi \colon \sE \to \sE''$ is a morphism of connective complexes on $X$ which is surjective on $\pi_0$, then there is a canonical induce morphism  $\iota_{\Flag,\varphi} \colon \Flag_{\bdd}(\sE'') \to \Flag_{\bdd}(\sE)$, which, for any $\eta \colon T = \Spec A \to X$, $A \in \CAlgDelta$, carries each flag
	$$\zeta_T'' = (\sE_T'' \xrightarrow{\varphi_{k, T}} \sE_{k} \xrightarrow{\phi_{k-1,k}} \sQ_{d_{k-1}} \to \cdots \xrightarrow{\phi_{1,2}} \sE_{1}) \in \Flag_{\bdd}(\sE'')(\eta)
$$
of $\sE_T''$ of type $\bdd$ to the flag
	$$\iota_{\Flag,\varphi}(\eta) (\zeta_T'') = (\sE_T \xrightarrow{\varphi_{T} \circ \varphi_{k, T}} \sE_{k} \xrightarrow{\phi_{k-1,k}} \sQ_{d_{k-1}} \to \cdots \xrightarrow{\phi_{1,2}} \sE_{1}) \in \Flag_{\bdd}(\sE)(\eta)
$$
of $\sE_T$ of type $\bdd$.
We have the following generalization of Proposition \ref{prop:Grass:PB}:

\begin{proposition}[Closed immersions] 
\label{prop:dflag:immersion}
Let $X$ be a prestack and $\bdd = (d_1, \ldots, d_k)$ an increasing sequence of positive integers.
If $\varphi \colon \sE \to \sE''$ is a morphism of connective quasi-coherent complexes on $X$ which is surjective on $\pi_0$, then the canonical map $\iota_{\Flag,\varphi} \colon \Flag_{\bdd}(\sE'') \to \Flag_{\bdd}(\sE)$ is closed immersion which admits a relative cotangent complex described by the formula
	$$\LL_{\Flag_{\bdd}(\sE'')/\Flag_{\bdd}(\sE)} \simeq \pr_{\Flag_{\bdd}(\sE'')}^*(\fib(\varphi)) \otimes \sQ_{k}(\sE'')^\vee [1].$$
Moreover, $\iota_{\Flag,\varphi}$ fits into a canonical pullback diagram of close immersions:
	\begin{equation*}
	\begin{tikzcd} 
		\Flag_{\bdd}(\sE'') \ar{d}[swap]{\iota_{\Flag, \varphi}} \ar{r}{\iota_{\Flag, \varphi}} & \Flag_{\bdd}(\sE) \ar{d}{i_{\varphi}} \\
		\Flag_{\bdd}(\sE)  \ar{r}{i_{\mathbf{0}}} & \VV_{\Flag_{\bdd}(\sE)}(\pr_{\Flag_{\bdd}(\sE)}^* (\fib(\varphi)) \otimes \sQ_{k}(\sE)^\vee)
	\end{tikzcd}
	\end{equation*}
where $i_{\mathbf{0}}$ denote the zero section, and $i_{\varphi}$ is the section map which classifies the composition $\pr_{\Flag_{\bdd}(\sE)}^* (\fib(\varphi)) \to \pr_{\Flag_{\bdd}(\sE)}^* (\sE) \to \sQ_{k}(\sE)$.
\end{proposition}

\begin{proof} By virtue of Corollary \ref{cor:dflag:forget} \eqref{cor:dflag:forget-1} and Proposition \ref{prop:Grass:PB}, the following commutative diagram
	$$
	\begin{tikzcd}[column sep = 4 em]
		\Flag_{\bdd}(\sE'') \ar{d}[swap]{\iota_{\Flag, \varphi}} \ar{r}{\pi_{(d_k),\bdd}^{\sE''}}& \Grass_{d_k}(\sE'') \ar{d}{\iota_{\varphi}} \\
		\Flag_{\bdd}(\sE)   \ar{r}{\pi_{(d_k),\bdd}^{\sE}} & \Grass_{d_k}(\sE),
	\end{tikzcd}
	$$
is a pullback diagram. Therefore, the desired assertions follow from Proposition \ref{prop:Grass:PB}.
\end{proof}

\begin{proposition}
If $\varphi \colon \sE \to \sE''$ is a morphism of connective quasi-coherent complexes on $X$ which is surjective on $\pi_0$, then there is are commutative diagrams
	$$
	\begin{tikzcd}[column sep = 5 em, row sep = 2.5 em]
		\Flag_{\bdd}(\sE'') \ar{d}{\iota_{\Flag, \varphi}} \ar{r}{i_{\sE''}}& \prod_{i=1}^k \Grass_{d_i}(\sE'') \ar{d}{\prod \iota_{\Grass_{d_i}, \varphi}} \ar{r}{\prod_{i=1}^k \varpi_{\sE'', d_i}} & \prod_{i=1}^k \PP(\bigwedge\nolimits^{d_i} \sE'') \ar{d}{\prod \iota_{\Grass_{1}, \wedge^{d_i}(\varphi)}} \\
		\Flag_{\bdd}(\sE)  \ar{r}{i_{\sE}} & \prod_{i=1}^k \Grass_{d_i}(\sE) \ar{r}{\prod_{i=1}^k \varpi_{\sE, d_{i}}} & \prod_{i=1}^k  \PP(\bigwedge\nolimits^{d_{i} }\sE).
	\end{tikzcd}
	$$
%(Here, the products are taken in the category of prestacks over $X$.) 
\end{proposition}
\begin{proof}
This follows easily from construction. (Beware that the squares in the above diagram are generally {\em not} pullback squares.)
\end{proof}

\subsubsection{Relative cotangent complexes}
One of the benefits of considering the derived flag schemes is that, we could completely describe their relative cotangent complexes, generalizing the cases of derived Grassmannians (Theorem \ref{thm:Grass:cotangent}). 

\begin{theorem}[Relative Cotangent Complexes] 
\label{thm:dflag:cotangent}
Let everything be defined as in Notation \ref{not:dflag}; in particular, $X$ is prestack, $\sE \in \QCoh(X)^\cn$ is any connective complex , $\bdd = (d_1, \ldots, d_k)$ is an increasing sequence of positive integers, $\pr \colon \Flag_{\bdd}(\sE) \to X$ is the derived flag scheme, $\varphi_i \colon \pr^*(\sE) \to \sQ_{i}$ denotes the canonical quotient morphisms, $1 \le i \le k$, and 
	$$\pr^* (\sE) \xrightarrow{\varphi_{k}} \sQ_{k} \xrightarrow{\phi_{k-1,k}} \sQ_{k-1} \to \cdots \xrightarrow{\phi_{1,2}} \sQ_{1}$$
denotes the universal quotient sequence. Let $\LL_{\Flag_{\bdd}(\sE)/X}$ denote the relative cotangent complex.
\begin{enumerate}[leftmargin=*]
	\item 
	\label{thm:dflag:cotangent-1}
	There exists a canonical sequence of morphisms in $\QCohcn(\Flag_{\bdd}(\sE))$,
			$$\sF_k \to \sF_{k-1} \to \cdots \to \sF_{2} \to \sF_{1} = \LL_{\Flag_{\bdd}(\sE)/X},$$
		such that $\sF_k \simeq \fib(\varphi_k) \otimes \sQ_{k}^\vee$, and 
			$$\cofib(\sF_{i+1} \to \sF_{i}) \simeq \fib( \phi_{i,i+1} \colon \sQ_{i+1} \to \sQ_{i}) \otimes \sQ_{i}^\vee \quad \text{for} \quad 1 \le i \le k-1.$$
	\item
	\label{thm:dflag:cotangent-2}
	There is a canonical equivalence 
			$$\LL_{\Flag_{\bdd}(\sE)/X}  \simeq \cofib\Big(\psi \colon \bigoplus_{i=1}^{k-1} \fib(\varphi_{i+1} ) \otimes \sQ_{i}^\vee \to \bigoplus_{i=1}^{k} \fib(\varphi_{i}) \otimes \sQ_{i}^\vee\Big),$$
		where $\psi = (\psi_i)_{1 \le i \le k-1}$ is defined as follows: for each $1 \le i \le k-1$, the morphism $\psi_i = (\psi_{i, j})_{1 \le j \le k} \colon \fib(\varphi_{i+1} ) \otimes \sQ_{i}^\vee \to \bigoplus_{i=1}^{k} \fib(\varphi_{i}) \otimes \sQ_{i}^\vee$ is given by the formula: $\psi_{i,i} = (\fib(\varphi_{i+1}) \to \fib(\varphi_i)) \otimes \id_{\sQ_{i}^\vee}$, $\psi_{i,i+1} = \id_{\fib(\varphi_{i+1})} \otimes (\sQ_{i}^\vee \to \sQ_{i+1}^\vee)$, and $\psi_{i,j} = 0$ for $j \neq i, i+1$.
\end{enumerate}
Consequently, if $\sE$ is pseudo-coherent to order $n$ for some integer $n \ge 0$ (resp. almost perfect, perfect, of Tor-amplitude $\le n$), then $\LL_{\Flag_{\bdd}(\sE)/X}$ is pseudo-coherent to order $n$ (resp. almost perfect, perfect, of Tor-amplitude $\le n$). 
\end{theorem}

\begin{proof}
In order to prove assertion \eqref{thm:dflag:cotangent-1}, we consider the sequence of forgetful morphisms
	$$\Flag_{\bdd}(\sE) \to \Flag_{\bdd \backslash \{d_1\}}(\sE) \to \Flag_{\bdd \backslash \{d_1, d_2\}}(\sE) \to \cdots \to \Flag_{(d_{k-1}, d_k)}(\sE) \to \Grass_{d_k}(\sE).$$
For each $1 \le i \le k$, we let $\sF_i$ be the pullback of relative cotangent complex:
	$$\sF_i = \pi_{(d_i, d_{i+1}, \ldots, d_{k}), \bdd}^* \big(\LL_{\Flag(\sE; d_i, d_{i+1}, \ldots, d_k)/X} \big) \in \QCoh(\Flag(\sE;
\bdd))^\cn.$$ 
In particular, $\sF_1 = \LL_{\Flag(\sE;\bdd)/X}$ and $\sF_k = \pi_{(d_k), \bdd}^*(\LL_{\Grass_{d_k}(\sE)/X}) \simeq \fib(\varphi_k) \otimes \sQ_{k}^\vee$. By virtue of Lemma \ref{lem:dflag:forget} \eqref{lem:dflag:forget-1}, for each $1 \le i \le k-1$, the forgetful morphism
	$$\pi_{(d_i, d_{i+1}, \ldots, d_k), (d_{i+1}, \ldots, d_k)} \colon \Flag(\sE; d_i, d_{i+1}, \ldots, d_k) \to \Flag(\sE; d_{i+1}, \ldots, d_k)$$
is the relative derived Grassmannian parametrizing rank-$d_{i}$ locally free quotients of the universal quotient bundle $\sQ_{i+1}$ of rank $d_{i+1}$. From the canonical fiber sequence of relative cotangent complexes (\cite[Proposition 3.2.12]{DAG}) associated with the composite map 
	$$ \Flag(\sE; d_i, d_{i+1}, \ldots, d_k) \to \Flag(\sE; d_{i+1}, \ldots, d_k) \to X$$
and  Theorem \ref{thm:Grass:cotangent}, we obtain a canonical fiber sequence
	$$\sF_{i+1} \to \sF_{i} \to \fib(\phi_{i,i+1}) \otimes \sQ_{i}^\vee.$$
This proves assertion \eqref{thm:dflag:cotangent-1}. Assertion \eqref{thm:dflag:cotangent-2} follows from the combination of Proposition \ref{prop:dflag:immersion}, Theorem \ref{thm:Grass:cotangent} and \cite[Proposition 4.10]{J22a}.
\end{proof}

In the special case where $\sE$ is a perfect complex of Tor-amplitude in $[0,1]$, we obtain:

\begin{corollary}
\label{cor:dpflag:perfect.amp<=1}
Let $X$ be a prestack, $\sE$ a connective perfect complex on $X$ of Tor-amplitude $\le 1$, $\bdd = (d_1, \ldots, d_k)$ is an increasing sequence of positive integers, and let $\pr \colon \Flag_{\bdd}(\sE) \to X$ be the derived flag scheme (Definition \ref{def:dflag}). Then the following statements are trure:
\begin{enumerate}[leftmargin=*]
	\item 
	\label{cor:dpflag:perfect.amp<=1-1}
	The projection $\pr$ is a relative derived scheme, proper, locally of finite presentation, locally of finite Tor-amplitude, admits a connective perfect relative cotangent complex $\LL_{\Flag_{\bdd}(\sE)/X}$ of Tor-amplitude $\le 1$ (in particular, $\pr$ is quasi-smooth), and admits an invertible relative dualizing complex $\omega_{\pr}$ (that is, $\omega_{\pr}$ is a degree shift of an line bundle). 
	\item
	\label{cor:dpflag:perfect.amp<=1-2}
	The pushforward functor $\pr_*$ admits a right adjoint given by the formula $\pr^! (\blank) = \pr^*(\blank) \otimes \omega_{\pr}$ and the pullback functor $\pr^*$ admits a left adjoint given by the formula $\pr_!(\blank) = \pr_*(\blank \otimes \omega_{\pr})$. In addition, each member in the adjunction sequence ${\pr}_! \dashv \pr^* \dashv {\pr}_*\dashv {}\pr^!$ preserves perfect complexes (respectively, almost perfect complexes, and respectively, locally bounded almost perfect complexes). 
\end{enumerate}
\end{corollary}

\begin{proof}
Assertion \eqref{cor:dpflag:perfect.amp<=1-1} is a consequence of Proposition \ref{prop:Flag:finite} and Theorem \ref{thm:dflag:cotangent}. Assertion \eqref{cor:dpflag:perfect.amp<=1-2} follows from \eqref{cor:dpflag:perfect.amp<=1-1} by virtue of  \cite[Theorem 3.7]{J22a}.
\end{proof}

%%% sec: Bott's theorem
\section{Borel--Weil--Bott Theorem for Derived Flag Schemes}
This section establishes derived versions of the Borel--Weil--Bott theorem which relate the geometry of derived Grassmannians and flag schemes studied in \S \ref{sec:dGrass.dFlag} to the algebra of derived Schur functors studied in \S \ref{sec:dSchur}. Our presentation of derived versions of  Borel--Weil--Bott  theorem is divided into two cases: dominant weight case (\S \ref{sec:Bott.dominant}) and non-dominant weight case (\S \ref{sec:Bott.non-dominant}).

\S \ref{sec:Bott.dominant} investigates the dominant weight case, in which the derived versions of Borel--Weil--Bott theorem (Theorem \ref{thm:Bott:dflag}, Corollary \ref{cor:Bott:dpflag}, and Corollary \ref{cor:Bott:dGrass:dominant}) generalize both the classical Borel--Weil--Bott theorem \cite{Bott} and Kempf's vanishing theorem \cite{Kempf} from the case of vector bundles to the case of perfect complexes of Tor-amplitude $\le 1$. These results are characteristic-free. 

\S \ref{sec:Bott.non-dominant} establishes derived versions of Bott's theorem in the non-dominant weight case (Theorem \ref{thm:BBW:dflag}, Variant \ref{variant:BBW:dflag}, Corollary \ref{cor:BBW:dpflag}, Corollary \ref{cor:BBW:dGrass}), generalizing the classical Borel--Weil--Bott theorem (\cite{Dem, Wey, Lurie}) for flag varieties. Except for the vanishing results in Proposition \ref{prop:dflag:vanishing}, most results in this subsection require characteristic-zero conditions.

The derived versions of Borel--Weil--Bott theorem, like the classical ones, compute (derived) pushforwards of tautological complexes on the derived flag schemes; these tautological complexes are defined in  \S \ref{sec:tautological.complexes}.
Our proof is based on the classical Borel--Weil--Bott theorem for flag bundles over prestacks, which we study in \S \ref{sec:Bott.Grass.Flag.Bundles}, and a preliminary version of the Borel--Weil--Bott theorem for derived Grassmannians, whose proof depends on computations of Koszul-type complexes in the universal local situation, which we investigate in \S \ref{sec:dBott.Pre.Grass}.

\subsection{Tautological Complexes on Derived Flag Schemes}
\label{sec:tautological.complexes}
Classically, if $X = \Spec R$ is an affine scheme, where $R$ is an ordinary commutative ring, and $\sE = R^n$ is a trivial vector bundle of rank $n \ge 1$, then complete flag scheme $\Flag_R(R^{n}; \underline{n})$ is the homogeneous scheme for the Chevalley group scheme ${\rm GL}(n; R)$ over $R$: for each $1 \le i \le n$, let $\sR_i = \fib(\pr^*(R^n) \to  \sQ_i)$ as in Notation \ref{not:dflag-2}, then $0 = \sR_n \subseteq \sR_{n-1} \subseteq \cdots \subseteq \sR_1 \subseteq \pr^*(R^n)$ is a sequence of locally free submodules of $\pr^*(R^n)$ with $\rank \sR_i = n -i$, and there is a natural free action
	$${\rm GL}(n; R) \times \Flag_R(R^{n}; \underline{n}) \to \Flag_R(R^{n}; \underline{n})$$
which carries the above sequence of locally free submodules to their images under the action of ${\rm GL}(n; R)$. This action identifies $\Flag_R(R^{n}; \underline{n})$ with the quotient space ${\rm GL}(n; R)/ B_R$ (which {\em a priori} is only a fppf $R$-sheaf), where $B_R \subseteq {\rm GL}(n; R)$ is the Borel subgroup scheme consisting of lower triangular matrices. We let $T_R \subseteq B_R$ be the maximal $R$-torus consisting of diagonal matrices, then we can identify the character group $X(T_R)$ of $({\rm GL}(n; R), T_R)$ with $\ZZ^n$ and represent each $\lambda \in X(T_R)$ by a sequence of integers $(\lambda_1, \ldots, \lambda_n) \in \ZZ^n$. For each $\lambda \in X(T_R)$, there is a canonical associated line bundle $\sL(\lambda)$ on $ \Flag_R(R^{n};\underline{n})$ (\cite[\S I.5.8]{Jan}). Calculating the cohomologies of the line bundles $\sL(\lambda)$ is the core content of the classical Borel--Weil--Bott theorem.

To extend the Borel--Weil--Bott theorem to the derived setting, in \S \ref{sec:linebundle} we consider the generalization of the above construction $\lambda \mapsto \sL(\lambda)$ for all derived complete flag schemes $\Flag(\sE; \underline{n})$. Additionally, we discuss its variant for derived (partial) flag schemes in \S \ref{sec:V.lambda}. 

\subsubsection{Line bundles on Derived Complete Flag Schemes}
\label{sec:linebundle}

\begin{construction}[Line bundles on Derived Complete Flag Schemes]
\label{constr:dflag}
Let $X$ be a prestack, $\sE \in \QCoh(X)^\cn$, and for $n \ge 1$ an integer we let $\underline{n}  =(1,2,\ldots,n)$. Let $\pr \colon \Flag(\sE; \underline{n}) \to X$ be  the derived flag scheme of $\sE$ of type $\underline{n}$ with extended universal quotient sequence
	$$\pr^*(\sE) \xrightarrow{\varphi_{n}} \sQ_{n} \xrightarrow{\phi_{n-1,n}} \sQ_{n-1} \to \cdots \xrightarrow{\phi_{1,2}} \sQ_{1}  \xrightarrow{\phi_{0,1}=0} \sQ_{0} = 0,$$ 
and let $\sL_i = \Ker(\phi_{i-1,i} \colon \sQ_{i} \to \sQ_{i-1})$ be associated line bundle as in Example \ref{eg:dflag}. For any sequence of integers $\lambda = (\lambda_1, \lambda_2, \ldots, \lambda_n) \in \ZZ^n$, we define a line bundle $\sL(\lambda)$ by the formula:
\begin{equation}\label{eqn:flag:linebundle}
	\sL(\lambda) : = \sL_1^{\otimes \lambda_1} \otimes \sL_2^{\otimes \lambda_2} \otimes \cdots \otimes \sL_n^{\otimes \lambda_n} \in \Pic(\Flag(\sE;\underline{n})).
\end{equation}
\end{construction}

\begin{remark}[Relationships with Forgetful Morphisms]
\label{rem:forget:linebundle}
 In the situation of Construction \ref{constr:dflag}, we consider the sequence of forgetful morphisms (\S \ref{sec:dflag:forget})
	$$\Flag(\sE; \underline{n}) \xrightarrow{\pi_{\underline{n-1}, \underline{n}}} \Flag(\sE; \underline{n-1}) \xrightarrow{\pi_{\underline{n-2}, \underline{n-1}}}  \cdots \to \Flag(\sE; 1,2) \xrightarrow{\pi_{\underline{1}, \underline{2}}} \PP(\sE) \xrightarrow{\pi_{\underline{0}, \underline{1}}} X.$$
(Here, we agree by convention that $\Flag(\sE; \underline{0}) = X$ and the forgetful morphism $\pi_{\underline{0}, \underline{1}} \colon \PP(\sE)\to X$ denotes the natural projection morphism.) 
Then it follows from Lemma \ref{lem:dflag:forget} \eqref{lem:dflag:forget-2} that, for each $1 \le i \le n$, the forgetful morphism 
	$$\pi_{\underline{i-1}, \underline{i}} \colon \Flag(\sE; \underline{i}) \to \Flag(\sE; \underline{i-1})$$
 is a derived projectivization. We let $\sO_{\pi_{\underline{i-1}, \underline{i}}}(1) \in \Pic (\Flag(\sE; \underline{i}))$ denote the universal line bundle of the derived projectivization, then there is a canonical identification for each $1 \le i \le n$:
 	$$\sL_i \simeq \pi_{\underline{i}, \underline{n}}^* \left( \sO_{\pi_{\underline{i-1}, \underline{i}}}(1)  \right) \in \Pic (\Flag(\sE; \underline{n})).$$
Consequently, we have a canonical identification
\begin{align*}
	\sL(\lambda) \simeq \bigotimes_{i=1}^{n} \pi_{\underline{i}, \underline{n}}^*  \left( \sO_{\pi_{\underline{i-1}, \underline{i}}}(\lambda_i)  \right) \in \Pic(\Flag(\sE;\underline{n})).
\end{align*}
\end{remark}

\begin{remark}[Relationships with Pl\"ucker Morphisms]
\label{rem:plucker.map:linebundle}
In the situation of Construction \ref{constr:dflag}, we consider the Pl\"ucker morphism (Corollary \ref{cor:Plucker:flag:finite}) defined as the composite morphism
	$$	
		\varpi_{\sE, \underline{n}} \colon \Flag(\sE; \underline{n}) \xrightarrow{~i_{\sE, \underline{n}}~}  \prod_{i=1}^n \Grass_{i}(\sE) \xrightarrow{~\prod_{i=1}^n \varpi_{\sE, i}~}  \prod_{i=1}^n \PP(\bigwedge\nolimits^{i} \sE),
	$$
where $\{\varpi_{\sE, i} \}_{1 \le i \le n}$ are the Pl\"ucker morphisms of derived Grassmannians (\S \ref{sec:dGrass:Plucker}). Then for any sequence of integers $(s_1, s_2, \ldots, s_{n}) \in \ZZ^{n}$, we have a canonical identification
	\begin{align*}
	\varpi_{\sE, \underline{n}}^*\left(\bigotimes_{i=1}^n \sO_{\PP(\bigwedge^{i} \sE)} (s_i) \right)   \simeq  \sL_1^{\otimes (s_1 + \cdots + s_{n})} \otimes \sL_2^{\otimes (s_2 + \cdots + s_{n})} \otimes \cdots \otimes \sL_{n}^{\otimes s_{n}} \in \Pic(\Flag(\sE;\underline{n})).
	\end{align*}
By setting $\lambda_i = s_i + \cdots s_n$, $1 \le i \le n$, and $\lambda_{n+1}=0$, we obtain a canonical identification
\begin{align*}
	\sL(\lambda) \simeq \varpi_{\sE, \underline{n}}^*\left(\bigotimes_{i=1}^n \sO_{\PP(\bigwedge^{i} \sE)} (\lambda_i - \lambda_{i+1}) \right) \in \Pic(\Flag(\sE;\underline{n})).
\end{align*}
\end{remark}

\subsubsection{Tautological Complexes on Derived Partial Flag Schemes}
\label{sec:V.lambda}
One can extend the construction $\lambda \mapsto \sL(\lambda)$ of the \S \ref{sec:linebundle} to any derived (partial) flag schemes $\Flag(\sE; \bdd)$ under certain assumption on the ``weight" $\lambda = (\lambda_1, \ldots, \lambda_n) \in \ZZ^n$.

\begin{construction}[Tautological Complexes on Partial Flag Schemes]
\label{constr:V.lambda:dpflag}
Let $X$ be any prestack, 
let $n \ge 1, k \ge 1$ be integers, and let $\bdd = (d_1, \ldots, d_k)$ be an increasing sequence of integers in $[1,n]$. Let 
	$\pr \colon \Flag_X(\sE; \bdd) \to X$
denote the derived (partial) flag scheme. For any sequence of integers $\lambda = (\lambda_1, \ldots , \lambda_n)$, we define sequences $\lambda^{(j)}$ by the formula $\lambda^{(j)}= (\lambda_{d_{j-1}+1}, \ldots, \lambda_{d_{j}})$ for all $j=1, \ldots, k, k+1$. Here, by convention, we set $d_0 = 0$ and $d_{k+1} =n$; and if $d_k=n$, we set $\lambda^{(k+1)}=(0)$. More intuitively, we group the entries of $\lambda$ in the following form
	$$\lambda = \big( \underbrace{\lambda_1, \ldots \lambda_{d_1}}_{\lambda^{(1)}\,\text{terms}}; \underbrace{\lambda_{d_1+1} , \ldots, \lambda_{d_2}}_{\lambda^{(2)}\,\text{terms}};
	\ldots \ldots; 
	 \underbrace{\lambda_{d_k+1}  ,\ldots \lambda_{n}}_{\lambda^{(k+1)}\,\text{terms}} \big).$$
For each $j=1, 2, \ldots, k$, we define a vector bundle of rank $(d_{j} - d_{j-1})$ by the formula
	$$\sV_j = \Ker(\phi_{j-1,j} \colon \sQ_{j} \twoheadrightarrow \sQ_{j-1}) \in \Vect(\Flag_X(\sE;\bdd));$$
see Notation \ref{not:dflag} \eqref{not:dflag-3}. Assume that the following condition is satisfied by $\lambda$ and $\bdd$:
	\begin{itemize}
	\item[(*)] The sequences $\lambda^{(j)}$ are {\em partitions} for all $j=1,2,\ldots, k+1$.
\end{itemize}
We can define a connective quasi-coherent complex $\sV(\lambda)$ associated with $\lambda$ by the formula
	\begin{equation}\label{eqn:flag:Vlambda}
	\sV(\lambda) :=  \dSchur^{\lambda^{(k+1)}}\left(\fib(\pr^*\sE \xrightarrow{\varphi_k} \sQ_{d_k})\right) \otimes \bigotimes_{j=1}^{k}  \dSchur^{\lambda^{(j)}} (\sV_j) \in \QCoh(\Flag_X(\sE;\bdd))^\cn.
	\end{equation}
If $\lambda^{(k+1)} = (0)$, by convention we have $\dSchur^{(0)} (\fib( \varphi_k)) = \sO_X$, and in this case $\sV(\lambda) = \bigotimes_{j=1}^{k}  \dSchur^{\lambda^{(j)}} (\sV_j)$ is a vector bundle on $\Flag_X(\sE;\bdd)$. For example, if we let $\bdd = \underline{n} = (1, 2, \ldots, n)$, then $\lambda^{(k+1)} = (0)$ and $\sV(\lambda) = \sL(\lambda)$ is precisely the line bundle defined in Construction \ref{constr:dflag}.
If $\lambda^{(k+1)} \neq (0)$, then $\sV(\lambda)$ is generally {\em not} a vector bundle but only a connective quasi-coherent complex. 
\end{construction}

\begin{remark}[Finiteness Properties of Tautological Complexes]
In the situation of Construction \ref{constr:V.lambda:dpflag}, it follows from Propositions \ref{prop:dSchur:pc} and \ref{prop:dSchur:Tor-amp} that if $\sE$ is pseudo-coherent to order $n$ for some integer $n \ge 0$ (resp. almost perfect, perfect) over $X$, then the complex $\sV(\lambda)$ defined by \eqref{eqn:flag:Vlambda} is pseudo-coherent to order $n$ (resp. almost perfect, perfect) over $\Flag_X(\sE;\bdd)$.
\end{remark}

\begin{example}[Tautological complexes on Derived Grassmannians] 
\label{eg:notation:V:dGrass}
Let $X$ be any prestack, let $n \ge 1$ be a pair of integers, and let $\bdd = (d)$ for some integer $1 \le d \le n$. Let $\alpha = (\alpha_1, \ldots , \alpha_{d}) \in \ZZ_{\ge 0}^d$ and $\beta = (\beta_1, \ldots, \beta_{n-d}) \in \ZZ_{\ge 0}^{n-d}$ be partions, and let $\lambda = (\alpha,\beta) \in \ZZ_{\ge 0}^{n}$ be the concatenation of partions. We let
	$\pr \colon \Grass_X(\sE; d) \to X$
denote the derived Grassmannian and let 
	$\sR \to \pr^*(\sE) \to \sQ$
denote the universal fiber sequence. Then the canonical connective quasi-coherent complex $\sV(\lambda)$ associated with $\lambda = (\alpha, \beta)$ is described by the formula
	\begin{equation}\label{eqn:Grass:Vlambda}
		 \sV(\alpha, \beta) = \dSchur^{\alpha}(\sQ) \otimes \dSchur^{\beta}(\sR) \in \QCoh(\Grass(\sE;d))^\cn.
	\end{equation}
\end{example}

% sec: bundles
\subsection{Borel--Weil--Bott Theorem for Grassmannian Bundles and Flag Bundles}
\label{sec:Bott.Grass.Flag.Bundles}
This subsection considers the Borel--Weil--Bott theorem in the case of dominant weights and where $\sE=\sV$ is a vector bundle over a prestack $X$ (Theorem \ref{thm:Bott:flagbundle}). 
The results of this subsection are essentially classical, with the only new content being their presentation in a more functional form. The functoriality is, however, crucial to their generalizations over prestacks.

A key insight behind Theorem \ref{thm:Bott:flagbundle} (\cite{ABW, Wey, Tow1, Tow2}) is the intimate relationship between Pl\"ucker relations (that is, the relations satisfied by the homogeneous coordinate rings for Grassmannian and flag schemes) and the defining relations for Schur modules (Notation \ref{notation:square}). 
We review the relevant results (and present in a functorial way) in \S \ref{subsec:plucker.relations:Grass} and \S \ref{subsec:plucker.relations:flag}; our exposition here closely follows Weyman's book \cite[\S 3.1]{Wey}.

\subsubsection{Pl{\"u}cker Relations for Grassmannians}
\label{subsec:plucker.relations:Grass}
We first briefly review the classical theory of homogeneous prime spectra (\cite[II \S 2]{EGA}, \cite[\href{https://stacks.math.columbia.edu/tag/01M3}{Tag 01M3}]{stacks-project}). Let $S_* = \bigoplus_{m \ge 0} S_m$ be any (ordinary) non-negatively graded ring with homogeneous component $S_m$. Then $S_0 =: R$ the subring. We let $S_+ = \bigoplus_{m \ge 1} S_1$ denote the irrelevant ideal. We let $\Proj S_*$ denote the (classical) homogenous prime spectrum of $S_*$ (\cite[II \S 2.3]{EGA}, \cite[\href{https://stacks.math.columbia.edu/tag/01M6}{Tag 01M6}]{stacks-project}). Consider the following pair of functors:
	$$
	\begin{tikzcd}
	& (\blank)^{\sim} \colon \{\text{graded $S_*$-modules}\}	 \ar[r, shift left] & \ar[l, shift left] 	\QCoh(\Proj S_*)^\heartsuit \colon \Gamma_*
	\end{tikzcd}
	$$
where $ (\blank)^{\sim}$ is the functor which carries a graded $S_*$-module $M$ to its associated sheaf $M^{\sim}$ (\cite[II (2.5.3)]{EGA}, \cite[\href{https://stacks.math.columbia.edu/tag/01M6}{Tag 01M6}]{stacks-project}), and $\Gamma_*$ is the functor which carries each (discrete) quasi-coherent $\sF$ on $\Proj S_*$ to the graded $S_*$-module
	$\Gamma_*(\sF) = \bigoplus_{m \in \ZZ} \H^0(\Proj S_*; \sF (m))$
(\cite[II (2.6.1)]{EGA}). The canonical map (\cite[II, (2.6.2.3)]{EGA})
	$$S_* \to \Gamma_*(\sO_{\Proj S_*}) = \bigoplus_{m \ge 0} \H^0\left(\Proj S_*, \sO_{\Proj S_*}(m)\right)$$ 
is a homomorphism of graded $R$-algebras. For each graded $S_*$-module, the canonical map
	$$M \to \Gamma_*(M^{\sim}) =\bigoplus_{m \in \ZZ} \H^0\left(\Proj S_*; M^{\sim} (m)\right)$$
is a homomorphism of graded $S_*$-modules. If we assume that $S_+$ is generated in degree one and $S_1$ is finitely generated as an $S_0=R$-module, then for each $\sF \in \QCoh(\Proj S_*)^\heartsuit$, the canonical morphism $\Gamma_*(\sF)^{\sim} \to \sF$ is an isomorphism of quasi-coherent sheaves (\cite[II, Theorem (2.7.5)]{EGA}). Furthermore, for each closed subscheme $Z \subseteq \Proj S_*$ defined by a quasi-coherent ideal $\sI_Z$, we let $I(Z)_*$ denote the preimage of $\Gamma_*(\sI_Z)$ under the canonical morphism ${\rm sat} \colon S_* \to \Gamma_*(\sO_{\Proj S_*})$. Then $I(Z)_* \subseteq S_*$ is a homogeneous ideal. We will refer to $I(Z)_*$ as the {\em homogeneous ideal of $Z \subseteq \Proj S_*$}, and refer to the non-negatively graded quotient ring $A(Z)_*=S_*/I(Z)_*$ as the {\em homogeneous coordinate ring of $Z \subseteq \Proj S_*$}  (see \cite[II (2.9.2)]{EGA}).
 
Now we consider the situation of Pl\"ucker immersions. Let $R$ be an ordinary commutative ring, $X = \Spec R$, and $\sV= R^n$ a free $R$-module of rank $n$ over $X$. Let $d \ge 1$ be an integer, and we consider the Pl\"ucker morphism  for the Grassmannian $\Grass_X(V;d)$ (\S \ref{sec:dGrass:Plucker}):
 	$$\varpi_{\sV} \colon \GG:=\Grass(V;d) \to \PP:=\PP(\bigwedge\nolimits^d \sV).$$
The morphism $\varpi_{\sV}$ is a closed immersion and locally of finite presentation, by virtue of Proposition \ref{prop:Plucker}.
 We now specify the preceding discussions to the situation of Pl\"ucker immerions:
\begin{itemize}
	\item Let $S_* $ be the non-negatively graded $R$-algebra $ \Sym_R^*(\bigwedge\nolimits^d \sV) = S_R^*(\bigwedge\nolimits^d \sV)$ (\S \ref{sec:Schur} (2)). Then $S_0 = R$, $S_+$ is generated by $S_1 = \Sym_R^1(\sV) = \sV$ which is a finite free $R$-module. We have $\PP = \Proj S_*$ and $\sO_{\PP}(1) = \sO_{\Proj S_*}(1)$. In this case, the canonical morphism 
		$S_* \to \Gamma_*(\sO_{\PP}) = \bigoplus_{m \ge 0} \H^0(\PP; \sO_{\PP}(m))$
	 is an isomorphism of graded $R$-algebras.
	\item Let $I(\GG)_*$ denote the homogeneous ideal of the closed subscheme which is the image of the Pl\"ucker morphism $\varpi_{\sV, d} \colon \GG \to \PP$, and let $A(\GG)_* = S_*/ I(\GG)_*$ denote the homogeneous coordinate ring for $\GG$. Then there is a canonical morphism of graded $R$-algebras
			$$A(\GG)_* \to\Gamma_*(\sO_{\GG}) =  \bigoplus\nolimits_{m \ge 0}  \H^0(\GG; \sO_{\GG}(m)).$$
\end{itemize}

The next result describes the homogeneous ideal $(I_{\GG})_*$ and the  homogeneous coordinate ring $A(\GG)_*$ for $\GG$ and establish their connections with Schur functors.
 
 \begin{proposition}[Pl\"ucker Relations for Grassmannians; {see \cite[\S 3.1 \& \S 4.2]{Wey}.}]
 \label{prop:Plucker.relation:Grass}
 In the above situation:
 \begin{enumerate}[leftmargin=*]
	\item  \label{prop:Plucker.relation:Grass-1}
	%(Pl\"ucker relations).
	Let $\square_{(2^d)}$ denote the morphism of finite free $R$-modules
		$$\square_{(2^d)} \colon \bigoplus_{u,v \ge 0, u+v<d} \left(\bigwedge\nolimits^u \sV \otimes \bigwedge\nolimits^{2d - u - v} \sV \otimes \bigwedge\nolimits^v \sV \right)\to \bigwedge\nolimits^d \sV \otimes \bigwedge\nolimits^d \sV$$
		defined as in Notation \ref{notation:square}; that is, it is the sum of the composite maps 
		$$\bigwedge^u \sV \otimes \bigwedge^{2d - u - v} \sV \otimes \bigwedge^v \sV \xrightarrow{1 \otimes \Delta' \otimes 1} \bigwedge^u \sV \otimes \bigwedge^{d-u} \sV \otimes \bigwedge^{d - v} \sV \otimes \bigwedge^v \sV \xrightarrow{m' \otimes m'} \bigwedge^d \sV  \otimes \bigwedge^d \sV,$$
	where $\Delta'$ and $m'$ are the comultiplication and multiplication maps for exterior products (\S \ref{sec:classical_sym}). Then the homogeneous ideal $I(\GG)_* \subseteq S_*$ of $\GG$ is generated by the image of ${\rm Im}(\square_{(2^d)}) \subseteq \bigwedge\nolimits^d \sV \otimes \bigwedge\nolimits^d \sV$ under the canonical projection $\bigwedge\nolimits^d \sV \otimes \bigwedge\nolimits^d \sV  \twoheadrightarrow S_2 = \Sym^2(\bigwedge\nolimits^{d} \sV)$.
	\item  \label{prop:Plucker.relation:Grass-2}
	The canonical morphism
				$$ A(\GG)_* \to\Gamma_*(\sO_{\GG}) =  \bigoplus\nolimits_{m \ge 0}  \H^0(\GG; \sO_{\GG}(m))$$
	is an isomorphism of graded $R$-algebras. 
	\item   \label{prop:Plucker.relation:Grass-3}
	For each integer $m \ge 0$, the canonical map 
		$$S_m=\Sym_R^m\left(\bigwedge\nolimits^d \sV\right) \xrightarrow{\sim} \H^0\left(\PP(\bigwedge\nolimits^d \sV); \sO_\PP(m)\right) \to \H^0(\GG; \sO_{\GG}(m))$$
	is a surjection of $R$-modules and factorizes through canonical isomorphisms
		$$\Schur_R^{(m^d)}(\sV)  \xrightarrow{\sim}  A(\GG)_m  
		\xrightarrow{\sim}  \H^0(\GG; \sO_{\GG}(m))$$
	of finite free $R$-modules, where $\Schur_R^{(m^d)}(\blank)$ is the Schur functor associated with the partition $(m^d) =(\underbrace{m,\ldots,m}_{\text{$d$ terms}})$ (Definition \ref{def:SchurWeyl}). 
\end{enumerate}
\end{proposition}
 \begin{proof} 
In the case where $R$ is a field, assertions \eqref{prop:Plucker.relation:Grass-1} through \eqref{prop:Plucker.relation:Grass-3} are proved in \cite[Proposition 3.1.2, Remark 3.1.3, Proposition 3.1.4, and proof of Theorem 4.1.4]{Wey}. In general, observe that the formulations of assertions \eqref{prop:Plucker.relation:Grass-1}-\eqref{prop:Plucker.relation:Grass-3} commute with base change of commutative rings. Moreover, let $I_*' \subseteq S_*$ denote the homogeneous ideal generated by all the ``P\"ucker relations", that is, generated by the image of $\square_{(2^d)}$. Then it is direct to verify that $A(\GG)_*$ vanishes on the ideal $I_*'$ (see \cite[Proposition (3.1.2)]{Wey}, whose proof does not require $R=\KK$ to be a field). Hence there is a graded morphism $I_*' \to I(\GG)_*$ whose induced maps on each homogeneous component is a morphism of finite type $R$-modules. To show that $I_*' \to I(\GG)_*$ is an isomorphism, by Nakayama's lemma, it suffices to prove that it induces an isomorphism over each residue field of $R$. Hence we are reduced to the case where $R$ is a field and \eqref{prop:Plucker.relation:Grass-1} is proved. For assertions \eqref{prop:Plucker.relation:Grass-2} and \eqref{prop:Plucker.relation:Grass-3}, 
since $\sO_{\GG}(m)$ is $R$-ample on $\GG$ for any $m>0$, and $\H^0(\GG; \sO_\GG) = R$, the $R$-module $\H^0(\GG; \sO_\GG(m))$ is flat over $R$ and its  formation commutes with base change of $R$ for all $m \ge 0$. % By construction, the formation of $A(\GG)_*$ commutes with base change of commutative rings. 
Hence to show the natural map $A(\GG)_m \to \H^0(\GG; \sO_\GG(m))$ is an isomorphism, it suffices to prove in the special case where $R$ is field. Similarly, since the natural morphism $\Sym^m(\bigwedge\nolimits^d \sV)  \to \H^0(\GG; \sO_{\GG}(m))$ is map of flat $R$-modules and $\Schur_R^{(m^d)}(\sV)$ is $R$-flat, in order to prove assertion \eqref{prop:Plucker.relation:Grass-3} it suffices to verify it over each residue field of $R$. Consequently, the assertion \eqref{prop:Plucker.relation:Grass-3} also follows from the special case where $R$ is a field. 
\end{proof}

\begin{remark} In the statements of Proposition \ref{prop:Plucker.relation:Grass}, we implicitly use the fact that for any integer $m \ge 0$, the canonical projection map
	$\left(\bigwedge\nolimits^d \sV \right)^{\otimes m} \twoheadrightarrow \Schur_R^{(m^d)}(\sV)$ 
factorizes uniquely through surjective morphisms of $R$-modules
	$\left(\bigwedge\nolimits^d \sV \right)^{\otimes m} \twoheadrightarrow \Sym^m_R \left(\bigwedge\nolimits^m \sV \right) \twoheadrightarrow \Schur_R^{(m^d)}(\sV).$
To prove this assertion, we observe that the image of an element $(v_{1}, \ldots, v_{m}) \in  (\bigwedge\nolimits^d \sV)^{\otimes m}$ under the composite map
	$\left(\bigwedge\nolimits^d \sV\right)^{\otimes m} \xrightarrow{\otimes_{j=1}^m \Delta_j'}  
	\sV^{\otimes md}
	\xrightarrow{\otimes_{i=1}^{d} m_i} \left(\Sym_R^m \sV \right)^{\otimes d}$
is invariant under permutations of the $m$ entries $v_1, \ldots, v_m \in \bigwedge\nolimits^d(\sV)$ (here, $\Delta_j'$ and $m_i$ are the comultiplication and multiplication maps for exterior and symmetric products, respectively; see \S \ref{sec:classical_sym}). Therefore, the desired factorization follows from the universal properties of symmetric powers of modules (see \cite[Chapter III, \S 6.3, Proposition 6]{Bou}). 
\end{remark}
 
 \subsubsection{Pl{\"u}cker Relations for Flag Schemes}
 \label{subsec:plucker.relations:flag}
 Next, we consider the Pl{\"u}cker morphism for (partial) flag schemes. Let $R$ be an ordinary commutative ring, $X = \Spec R$, and $\sV= R^n$ a free $R$-module of rank $n \ge 1$ over $X$. Let $\bdd = (d_1, \ldots, d_k)$ be an increasing sequence inside $[1,n]$, and consider the Pl\"ucker morphism for the flag scheme $\FF:= \Flag_X(\sV; \bdd)$,
 	$$\varpi_{\sV, \bdd} \colon \FF= \Flag(\sV; \bdd)  \xrightarrow{~i_{\sV, \bdd}~} \prod_{i=1}^k \Grass_{d_i}(\sV) \xrightarrow{~\prod_{i=1}^k \varpi_{\sV, d_i}~} \PP: = \prod_{i=1}^k \PP(\bigwedge\nolimits^{d_i} \sV),$$
which is a closed immersion and locally of finite presentation, by virtue of Corollary \ref{cor:Plucker:flag:finite} (here, the products denote fiber products over $X =\Spec R$ as usual).

We consider the following generalization of homogeneous coordinate rings of Grassmannians:
\begin{itemize}
	\item Let $S_{*, \ldots, *} = \Sym_R^{*}(\bigwedge\nolimits^{d_1} \sV) \otimes \cdots \otimes \Sym_R^{*}(\bigwedge\nolimits^{d_k} \sV)$ denote the $\ZZ_{\ge 0}^k$-graded homogeneous coordinate ring of $\PP$, with homogeneous component, for each $(m_1, \ldots, m_k) \in \ZZ_{\ge 0}^k$,
	$$S_{m_1, \ldots, m_k} = \Sym_R^{m_1}(\bigwedge\nolimits^{d_1} \sV) \otimes \cdots \otimes \Sym_R^{m_k}(\bigwedge\nolimits^{d_k} \sV).$$
	We define $\sO_{\PP} (m_1, \ldots, m_k): = \sO_{\PP(\bigwedge\nolimits^{d_1} \sV)}(m_1) \boxtimes \cdots \boxtimes  \sO_{\PP(\bigwedge\nolimits^{d_k} \sV)}(m_k) \in \Pic(\PP)$, and for each discrete quasi-coherent sheaf $\sF$ on $\PP$, we can define a $\ZZ^k$-graded $S_{*,\ldots,*}$-module by
		$$\Gamma_{*, \ldots, *}(\sF) = \bigoplus_{(m_1, \ldots,m_k) \in \ZZ^k} \H^0(\PP; \sF \otimes \sO_{\PP}(m_1, \ldots, m_k)).$$
	Then the canonical map
		$$S_{*, \ldots, *} \longrightarrow \Gamma_{*, \ldots, *}(\sO_{\PP}) =\bigoplus_{(m_1, \ldots,m_k) \in \ZZ_{\ge 0}^k} \H^0(\PP; \sO_{\PP}(m_1, \ldots, m_k))$$
	is an isomorphism of $\ZZ_{\ge 0}^{k}$-graded $R$-algebras. 
	\item Let $\sI_\FF \subseteq \sO_{\PP}$ denote the quasi-coherent ideal for the closed immersion $\varpi_{\sV, \bdd} \colon \FF \to \PP$, and let $I(\FF)_{*,\ldots,*}$ denote the multigraded homogeneous ideal defined as the image of the morphism $\Gamma_{*, \ldots, *}(\sI_\FF) \to \Gamma_{*, \ldots, *}(\sO_\PP) = S_{*, \ldots, *}$. We will refer to $I(\FF)_{*, \ldots, *} \subseteq S_{*, \ldots, *}$ as the {\em multigraded homogeneous ideal of $\FF$}. We let $A(\FF)_{*, \ldots, *}  = S_{*, \ldots, *}/I(\FF)_{*, \ldots, *}$ and refer to it as the {\em multigraded homogeneous coordinate ring of $\FF \to \PP$}. The canonical morphism
		$$A(\FF)_{*, \ldots, *} \longrightarrow \bigoplus_{(m_1, \ldots, m_k) \in \ZZ_{\ge 0}^k} \H^0(\FF; \sO_{\FF}(m_1, \ldots, m_k))$$
	is a homomorphism of $\ZZ_{\ge 0}^k$-graded rings, where $\sO_{\FF}(m_1, \ldots, m_k) := \varpi_{\sV, \bdd}^{*} (\sO_{\PP} (m_1, \ldots, m_k))$.
	\end{itemize}

 \begin{proposition}[Pl\"ucker Relations for Flag Schemes; {see \cite[\S 3.1 \& \S 4.2]{Wey}}]
 \label{prop:Plucker.relation:flag}
 In the above situation:
 \begin{enumerate}[leftmargin=*]
	\item  \label{prop:Plucker.relation:flag-1}
	For each pair of integers $a_1 \ge a_2$, where $a_1, a_2 \in \{d_1, \ldots, d_k\}$, we let $\square_{(a_1, a_2)^t}$ denote the morphism between finite free $R$-modules
		$$\square_{(a_1,a_2)^t} \colon \bigoplus_{u,v \ge 0, u+v< a_2} \left(\bigwedge\nolimits^u \sV \oplus \bigwedge\nolimits^{a_1 + a_2 - u - v} \sV \oplus \bigwedge\nolimits^v \sV \right)\to \bigwedge\nolimits^{a_1} \sV \otimes \bigwedge\nolimits^{a_2} \sV$$
		defined as in Notation \ref{notation:square} (that is, it is the sum of the composite maps 
	$$\bigwedge^u \sV \otimes \bigwedge^{a_1 + a_2 - u - v} \sV \otimes \bigwedge^v \sV \xrightarrow{1 \otimes \Delta' \otimes 1} \bigwedge^u \sV \otimes \bigwedge^{a_1-u} \sV \otimes \bigwedge^{a_2 - v} \sV \otimes \bigwedge^v \sV \xrightarrow{m' \otimes m'} \bigwedge^{a_1} \sV  \otimes \bigwedge^{a_2} \sV,$$
	where $\Delta'$ and $m'$ are the comultiplication and multiplication maps for exterior products defined in \S \ref{sec:classical_sym}). Then the multigraded homogeneous ideal $I(\FF)_{*, \ldots, *} \subseteq S_{*, \ldots, *}$ is  generated by the images of ${\rm Im}(\square_{(a_1,a_2)^t})$ inside $S_{*, \ldots, *}$ under the canonical map $\bigwedge^{a_1} \sV \otimes \bigwedge^{a_2} \sV \to S_{*, \ldots, *}$, where $(a_1, a_2)$ runs through all pair of integers $a_1, a_2 \in \{d_1, \ldots, d_k\}$ such that $a_1 \le a_2$.
	
	\item  \label{prop:Plucker.relation:flag-2}
	The canonical morphism
		$${\rm sat} \colon A(\FF)_{*, \ldots, *} \longrightarrow \bigoplus_{(m_1, \ldots, m_k) \in \ZZ_{\ge 0}^k} \H^0\left(\FF; \sO_{\FF}(m_1, \ldots, m_k)\right)$$
	is an isomorphism of $\ZZ_{\ge 0}^k$-graded $R$-algebras.
		
	\item \label{prop:Plucker.relation:flag-3}
	For all $(m_1, \ldots, m_k) \in \ZZ_{\ge 0}^k$, the canonical map 
		\begin{align*}
		\Sym_R^{m_1}(\bigwedge\nolimits^{d_1} \sV) \otimes \cdots \otimes \Sym_R^{m_k}(\bigwedge\nolimits^{d_k} \sV)   \xrightarrow{\sim} \H^0\left(\PP; \sO_\PP(m_1, \ldots, m_k)\right) \to \H^0(\FF; \sO_{\FF}(m_1, \ldots, m_k))
		\end{align*}
	is a surjective map of free $R$-modules, and factorizes through an isomorphism
		$$\Schur^{(m_k^{d_k}, m_{k-1}^{d_{k-1}}, \ldots, m_{1}^{d_1})}(\sV) \xrightarrow{\sim} A(\FF)_{m_1, \ldots, m_k} \xrightarrow{\sim}  \H^0(\Flag(\sV;\bdd); \sO_{\FF}(m_1, \ldots, m_k)),$$
	where $\Schur^{(m_k^{d_k}, m_{k-1}^{d_{k-1}}, \ldots, m_{1}^{d_1})}(\blank)$ is the Schur functor 
	associated with the partition:
	
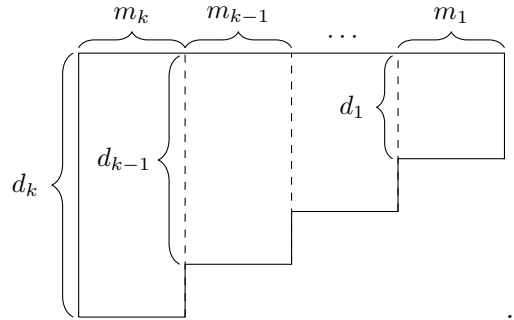
\begin{center}
\begin{tikzpicture}[scale=0.7]
    \draw
     (0,0) -- (0,5) -- (8,5) -- (8,3) -- (6,3) -- (6,2) --  (4,2) -- (4,1) -- (2,1) -- (2,0) -- (0,0);
     \draw[dashed]  
      (2,0) -- (2,5);
     \draw[dashed]
       (4,1) -- (4,5); 
      \draw[dashed]
       (6,2) -- (6,5);
     
     \draw [decorate,decoration={brace,amplitude=8pt, raise=2pt},yshift=0pt]
(0,0) -- (0,5) node [black,midway,xshift=-0.7cm] {\footnotesize
$d_k$};
   
   \draw [decorate,decoration={brace,amplitude=6pt, raise=2pt},yshift=0pt] (6,3) -- (6,4.95) node [black,midway,xshift=-0.6cm] {\footnotesize $d_1$};

     \draw [decorate,decoration={brace,amplitude=8pt, raise=2pt},yshift=0pt]
(2,1) -- (2,4.95) node [black,midway,xshift=-0.8cm] {\footnotesize
$d_{k-1}$};
     
       \draw [decorate,decoration={brace,amplitude=6pt, raise=2pt},yshift=0pt]
(0,5) -- (2,5) node [black,midway,yshift=0.5cm] {\footnotesize
$m_k$};
       \draw [decorate,decoration={brace,amplitude=6pt, raise=2pt},yshift=0pt]
(2,5) -- (4,5) node [black,midway,yshift=0.5cm] {\footnotesize
$m_{k-1}$};

      \node at (5, 5.3) {\footnotesize $\cdots$};
     
            \draw [decorate,decoration={brace,amplitude=6pt, raise=2pt},yshift=0pt]
(6,5) -- (8,5) node [black,midway,yshift=0.5cm] {\footnotesize
$m_{1}$};
     
 \end{tikzpicture}.
\end{center}
\end{enumerate}
 \end{proposition}

\begin{proof}
The same argument as in the proof Proposition \ref{prop:Plucker.relation:Grass} reduces us to the case where $R$ is a field. In this case, the assertion \eqref{prop:Plucker.relation:flag-1} is proved in \cite[Proposition 3.1.9]{Wey}; the cited result also implies that the canonical map in \eqref{prop:Plucker.relation:flag-3} factorizes through
	$$f_\bdd \colon \Schur^{(m_k^{d_k}, m_{k-1}^{d_{k-1}}, \ldots, m_{1}^{d_1})}(\sV) \xrightarrow{\sim}  A(\Flag(\sV;\bdd))_{m_1, \ldots, m_k} \xrightarrow{{\rm sat}}  \H^0(\Flag(\sV;\bdd); \sO_{\FF}(m_1, \ldots, m_k)).$$
To prove assertions \eqref{prop:Plucker.relation:flag-2} and \eqref{prop:Plucker.relation:flag-3}, it suffices to show that $f_\bdd$ is an isomorphism. 
In the case where each $m_i >0$, the assertion that $f_\bdd$ is an isomorphism is proved in \cite[proof of Theorem 4.1.4, page 122]{Wey}. In general, let $m_{i_1}, \ldots, m_{i_\ell} >0$ be {\em all} the {non-zero elements} of $\{m_1, \ldots, m_k\}$ where $i_1 < \ldots < i_\ell$, and let $\bdd' = (d_{i_1}, \ldots, d_{i_\ell})$ be the subsequence of $\bdd$. By virtue of Proposition \ref{prop:Plucker:flag}, there is a commutative diagram
		$$
	\begin{tikzcd} 
		\Schur^{(m_{i_\ell}^{d_\ell},  \ldots, m_{i_1}^{d_{i_1}})}(\sV)  \ar{d}{\simeq} \ar{r}{f_{\bdd'}} & \H^0(\Flag(\sV;\bdd'); \sO_{\FF'}(m_{i_1}, \ldots, m_{i_\ell})) \ar{d}{\simeq} \\
		\Schur^{(m_k^{d_k}, \ldots, m_{1}^{d_1})}(\sV)  \ar{r}{f_\bdd} &  \H^0(\Flag(\sV;\bdd); \sO_{\FF}(m_1, \ldots, m_k))
	\end{tikzcd}
	$$
Here, the first vertical arrow is an isomorphism by our definition of the $m_{i_j}$,  second vertical arrow is an isomorphism by projection formula and the fact that
	$$\sO_{\FF}(m_1, \ldots, m_k) \simeq \pi_{\bdd', \bdd}^*(\sO_{\FF'}(m_{i_1}, \ldots, m_{i_\ell})).$$
The top horizontal map $f_{\bdd'}$ is an isomorphism, by virtue of \cite[page 122, proof of Theorem 4.1.4]{Wey} since each $m_{i_j}>0$. Hence $f_\bdd$ is an isomorphism, and the proposition is proved. 
\end{proof}

\subsubsection{Borel--Weil--Bott Theorem for Grassmannians Bundles and Flag Bundles}
Based on preceding subsections' results, we can present Borel--Weil--Bott theorem in the case of vector bundles over prestacks:

\begin{theorem}[{Borel--Weil--Bott Theorem for Complete Flag Bundles; see \cite[Theorem 4.1.10]{Wey}}]
\label{thm:Bott:flagbundle}
Let $X$ be a prestack, and $\sV$ a vector bundle of rank $n \ge 1$ over $X$. Let 
	$$\pr \colon \Flag_X(\sV; \underline{n}) \to X$$
 denote the derived flag scheme of $\sV$ of type $\underline{n}$ over $X$ (Example \ref{eg:dflag}). Let $\lambda = (\lambda_1 \ge \cdots \ge \lambda_n \ge 0)$ be a partition, let $\sL(\lambda)$ be the associated line bundle  \eqref{eqn:flag:linebundle} and let $\Schur_X^\lambda(\sV)$ be the Schur power of $\sV$ over $X$. Then there exists a canonical map
	$\pr^*(\bigwedge\nolimits^{\lambda^t}(\sV)) \to \sL(\lambda)$
inducing a canonical surjective maps of vector bundles
	$\bigwedge\nolimits^{\lambda^t}(\sV) \twoheadrightarrow \pr_*(\sL(\lambda))$
which factorizes through a canonical isomorphism of vector bundles over $X$:
	\begin{equation}\label{eqn:Bott:flagbundle}
		\Schur_X^\lambda(\sV) \xrightarrow{\simeq} \pr_*(\sL(\lambda))
\end{equation}
\end{theorem}

\begin{remark}[Kempf's vanishing theorem]
\label{rmk:Kempf}
As mentioned in the introduction, our above formulation of Borel--Weil--Bott theorem \ref{thm:Bott:flagbundle} includes Kempf's vanishing theorem \cite{Kempf}   since the pushforward functors are {\em derived} in this paper:
the isomorphism \eqref{eqn:Bott:flagbundle} implies that
	$$\RR^0 \pr_*(\sL(\lambda)) \simeq \Schur^\lambda(\sV) \quad \text{and} \quad \RR^i \pr_*(\sL(\lambda)) =0 \quad \text{for} \quad i\ne0 .$$
\end{remark}

\begin{proof}[Proof of Theorem \ref{thm:Bott:flagbundle}]
We first prove that $\pr_*(\sL(\lambda))$ is a vector bundle on $X$. Since this question is local, we may assume that $X = \Spec A$ and $\sV =A^n$, where $A \in \CAlgDelta$ is a simplicial commutative ring. Since $\sV = A \otimes_{\ZZ} \ZZ^n$, $\Flag_{A}(A^n; \underline{n})$ is the base change of $\Flag_{\ZZ}(\ZZ^n; 
\underline{n})$ along $\ZZ \to A$ (Proposition \ref{prop:flag:functorial}). Since the base-change morphism is an equivalence in the derived context (\cite[Theorem 3.7 (1b)]{J22a}, \cite[Proposition 6.3.4.1]{SAG}), there is a canonical equivalence
	$$(\pr_{\Flag_R(R^n;\underline{n})})_* (\sL_A(\lambda)) \simeq A \otimes_\ZZ  \H^*\big(\Flag_\ZZ(\ZZ^n; \underline{n}) ;  \sL_{\ZZ}(\lambda)\big).$$
(Here, $\sL_A(\lambda)$ and $\sL_{\ZZ}(\lambda)$ denote the constructions of the line bundle \eqref{eqn:flag:linebundle} over the base spaces $\Spec A$ and $\Spec \ZZ$, respectively.) By Kempf's vanishing theorem \cite{Kempf} over $\Spec \ZZ$, 
	$$\H^*\big(\Flag_\ZZ(\ZZ^n; \underline{n}) ; \sL_{\ZZ}(\lambda)\big) \simeq  \H^{0}\big(\Flag_\ZZ(\ZZ^n; \underline{n}) ; \sL_{\ZZ}(\lambda)\big) = :\H^0_\ZZ(\lambda)$$
 is a finite free $\ZZ$-module placed in degree $0$. Hence $\pr_*(\sL(\lambda))$ is a vector bundle.

Next, for any prestack $X$, we construct a canonical and functorial morphism between vector bundles $\bigwedge\nolimits^{\lambda^t} \sV \to \pr_*(\sL(\lambda))$, where $\bigwedge\nolimits^{\lambda^t} \sV$ denote the vector bundle 
	$$\bigwedge\nolimits^{\lambda^t} \sV: = \bigotimes_{j=1}^{\lambda_1} \big(\bigwedge\nolimits^{\lambda_j^t} \sV\big) = \bigotimes_{i=1}^{n} \big(\bigwedge\nolimits^i \sV \big)^{\otimes (\lambda_i - \lambda_{i+1})}$$
as in Remark \ref{rem:Schur.Weyl.as.images} (we set by convention $\lambda_{n+1}=0$). Consider the Pl\"ucker morphism
	$$	
	\varpi_{\sV, \underline{n}} \colon \Flag(\sV; \underline{n}) \xrightarrow{~i_{\sV, \underline{n}}~}  \prod_{i=1}^n \Grass_{i}(\sV) \xrightarrow{~\prod_{i=1}^n \varpi_{\sE, i}~}  \prod_{i=1}^n \PP(\bigwedge\nolimits^{i} \sV)
	$$
of Corollary \ref{cor:Plucker:flag:finite} (where the products denote fiber products over $X$ as in Corollary \ref{cor:Plucker:flag:finite}). By virtue of Remark \ref{rem:plucker.map:linebundle}, we have a canonical equivalence
\begin{align*}
	\sL(\lambda) \simeq \varpi_{\sV, \underline{n}}^*\left(\bigotimes_{i=1}^n \sO_{\PP(\bigwedge^{i} \sV)} (\lambda_i - \lambda_{i+1}) \right).
\end{align*}
For each $1 \le i \le n$, let $\varphi_i \colon \pr^*(\sV) \to \sQ_i$ denote the universal quotient morphism, where $\sQ_i$ is the universal quotient bundle of rank $i$ (Notation \ref{not:dflag}). Then the canonical morphisms
	$$\pr^*\big(\bigwedge\nolimits^i \sV\big) \simeq \bigwedge\nolimits^i \pr^*( \sV) \xrightarrow{\wedge^i \varphi_i} \bigwedge\nolimits^i \sQ_i \to \varpi_{\sV, \underline{n}}^*(\sO_{\PP(\bigwedge^i \sV)}(i)) \simeq \sL_1 \otimes \cdots \otimes \sL_i$$
for all $i=1, 2, \ldots, n$ induce a canonical morphism 
	\begin{align*}
		\pr^* \big(\bigwedge\nolimits^{\lambda^t} \sV \big) \simeq  \bigotimes_{i=1}^n  \big(\bigwedge\nolimits^i \pr^* (\sV) \big)^{\otimes (\lambda_i - \lambda_{i+1})}  \xrightarrow{\otimes_i \wedge^i(\varphi_i)} \bigotimes_{i=1}^n \big(\bigwedge\nolimits^i \sQ \big)^{\otimes (\lambda_i - \lambda_{i+1})}  \to \sL(\lambda).
	\end{align*}
It is clear from the above construction that the formation of the above morphism commutes with base change of prestacks. By adjunction, we obtain the desired morphism $\bigwedge\nolimits^{\lambda^t} \sV \to \pr_*(\sL(\lambda))$ over $X$ whose formation commutes with base change of prestacks. 

It remains to show that the canonical morphism $\bigwedge\nolimits^{\lambda^t} \sV \to \pr_*(\sL(\lambda))$ between vector bundles is surjective and factorizes through an isomorphism $\Schur^{\lambda}(\sV) \simeq  \pr_*(\sL(\lambda))$. Since this  assertion is local with respect to Zariski topology on $X$, we may again reduce to the case where $X = \Spec A$ and $\sV =A \otimes_\ZZ \ZZ^n$, where $A \in \CAlgDelta$ is a simplicial commutative ring. Since the pullback functor $f^* \colon \QCoh(Y) \to \QCoh(X)$ for the morphism $f \colon X= \Spec R \to Y=\Spec \ZZ$  restricts to an exact functor $f^*\colon \Vect(Y) \to \Vect(X)$ between the full subcategories spanned by vector bundles. Therefore, it suffices to prove the desired assertion in the special case where $X = \Spec \ZZ$. In this case, let $m_i= \lambda_i - \lambda_{i+1}$ for $i=1, \ldots, n$ (where we set $\lambda_{n+1} = 0$) and $\bdd = \underline{n}$, then the partition $(m_{n}^n, m_{n-1}^{n-1}, \ldots, m_1^{1})$ of Proposition \ref{prop:Plucker.relation:flag} \eqref{prop:Plucker.relation:flag-3} is precisely the partition $\lambda$. Hence the desired assertion follows from Proposition \ref{prop:Plucker.relation:flag} \eqref{prop:Plucker.relation:flag-3}.
\end{proof}

\begin{remark}[Alternative Proof via Reduction to $B {\rm GL}_{n}(\ZZ)$] As we have seen in the above proof of Theorem \ref{thm:Bott:flagbundle}, the formation of the Theorem \ref{thm:Bott:flagbundle} commutes with base change of prestacks. Therefore, we may reduce to the cases where $X = B{\rm GL}_{n}(\ZZ)$.  
In this case, we may deduce the desired assertion from Donkin's work \cite[\S 2.7]{Do}. (Notice that, although \cite[\S 2.7 (5)]{Do} only asserts the existence of an isomorphism $\Schur^{\lambda}_\ZZ (\ZZ^n) \simeq Y_{\ZZ}(\lambda)$ without specifying the morphism, we expect that its argument could be used to show that the canonical morphism $\bigwedge\nolimits^{\lambda^t} (\ZZ^n) \to \H^0(\Flag_\ZZ(\ZZ^n; \underline{n})$ induces via adjunction a canonical isomorphism \eqref{eqn:Bott:flagbundle}.)
\end{remark}

\begin{corollary}[{Borel--Weil Theorem for Grassmannian Bundles; see \cite[Corollary 4.1.12]{Wey}}]
\label{cor:Bott:Grassbundle}
Let $X$ be a prestack, and $\sV$ a vector bundle of rank $n \ge 1$ over $X$, and let $1 \le d \le n$ be an integer. Let 
	$\pr \colon \Grass_X(\sV; d) \to X$
 denote the derived Grassmannian of $\sV$ of rank $d$ over $X$, with tautological quotient morphism $\rho \colon \pr^*(\sV) \to \sQ$, where $\sQ$ is the universal quotient bundle of rank $d$. Let $\lambda = (\lambda_1 \ge \cdots \ge \lambda_d \ge 0)$ be any partition, then the canonical morphism
 	$$\pr^*(\Schur_X^{\lambda}(\sV)) \simeq \Schur^{\lambda}_{\Grass(\sV;d)}(\pr^*(\sV)) \xrightarrow{\Schur^{\lambda}(\rho)} \Schur_{\Grass(\sV;d)}^{\lambda}(\sQ)$$
induces a canonical isomorphism of vector bundles over $X$:
	$$\Schur_{X}^{\lambda}(\sV) \xrightarrow{\simeq} \pr_*(\Schur_{\Grass(\sV;d)}^{\lambda}(\sQ)).$$
\end{corollary}
\begin{proof}
By virtue of Example \ref{eg:forget:flag.to.Grass}, the forgetful map
	$\pi_{(d), \underline{n}} \colon \Flag_X(\sV; \underline{n}) \to \Grass_X(\sV; d)$
identifies $\Flag_X(\sV; \underline{n})$ with the fiber product 
	$\Flag(\fib(\rho); \underline{n-d}) \times_{\Grass(\sV;d)} \Flag(\sQ; \underline{d}).$
Therefore, the desired results follow from Theorem \ref{thm:Bott:flagbundle} and the projection formula. 
\end{proof}

\begin{corollary}[Borel--Weil Theorem for Grassmannian Bundles: the Case of Skew Partitions]
\label{cor:Bott:Grassbundle.skew}
Assume we are in the same situation as Corollary \ref{cor:Bott:Grassbundle}. Let $\lambda = (\lambda_1 \ge \cdots \ge \lambda_d \ge 0)$ and $\mu = (\mu_1 \ge \cdots \ge \mu_d \ge 0)$ be a pair of partitions such that $\mu \subseteq \lambda$ (that is, $\mu_i \le \lambda_i$ for all $i$). Then the canonical morphism
 	$$\pr^*(\Schur_X^{\lambda/\mu}(\sV)) \simeq \Schur_{\Grass(\sV;d)}^{\lambda/\mu}(\pr^*(\sV)) \xrightarrow{\Schur^{\lambda/\mu}(\rho)} \Schur_{\Grass(\sV;d)}^{\lambda/\mu}(\sQ)$$
induces a canonical isomorphism of vector bundles
	$$\Schur_X^{\lambda/\mu}(\sV) \xrightarrow{\simeq} \pr_*(\Schur_{\Grass(\sV;d)}^{\lambda/\mu}(\sQ)).$$
\end{corollary}
\begin{proof}
Let $\GG = \Grass(\sV;d)$. By virtue of Theorem \ref{thm:fil:dSchur_LR}  \eqref{thm:fil:dSchur_LR-1}, there are canonical sequences of morphisms of complexes on $X$ and  $\GG$, respectively:
		$$\sF^{0}(X, \sV) \to  \sF^{1}(X, \sV) \to \cdots \to \sF^{\ell-1}(X, \sV)$$
		$$\sF^{0}(\GG; \sQ) \to  \sF^{1}(\GG, \sQ) \to \cdots \to \sF^{\ell-1}(\GG, \sQ)$$
such that there are canonical equivalences
		$$\sF^{0}(X, \sV) \simeq \dSchur^{\tau^0}_X(\sV)  \qquad \sF^{\ell-1}(X, \sV) = \dSchur_X^{\lambda/\mu}(\sV)$$ 
		$$\sF^{0}(\GG, \sQ) \simeq \dSchur^{\tau^0}_{\GG}(\sQ) \qquad \sF^{\ell-1}(\GG, \sQ) = \dSchur_{\GG}^{\lambda/\mu}(\sQ)$$
			$$\cofib\big(\sF^{i-1}(X,\sV) \to \sF^{i}(X, \sV)\big) \simeq \dSchur_X^{\tau^i}(\sV)  \quad \text{for} \quad 1 \le i \le \ell-1$$
			$$\cofib\big(\sF^{i-1}(\GG,\sQ) \to \sF^{i}(\GG, \sQ)\big) \simeq \dSchur_{\GG}^{\tau^i}(\sQ)  \quad \text{for} \quad 1 \le i \le \ell-1.$$
Here, $\tau^i \subseteq \lambda$ are the partitions defined in Theorem \ref{thm:fil:dSchur_LR}  \eqref{thm:fil:dSchur_LR-1}.

We prove by induction that the morphism $\rho \colon \pr^*(\sV) \to \sQ$ induces a canonical equivalence $\sF^i(X, \sV) \to \pr_*\big( \sF^i(\GG, \sV)\big)$ for all $0 \le i \le \ell-1$. The case $i=0$ follows from Corollary \ref{cor:Bott:Grassbundle}. If $1 \le i \le \ell-1$, then $\rho \colon \pr^*(\sV) \to \sQ$ induces a diagram of cofiber sequences:
	$$
	\begin{tikzcd} 
		\sF^{i-1}(X, \sV) \ar{d} \ar{r}& \sF^i(X, \sV) \ar{d} \ar{r} & \dSchur_X^{\tau^i}(\sV) \ar{d} \\
		\pr_*\big(\sF^{i-1}(\GG, \sV)\big) \ar{r}& \pr_*\big( \sF^i(\GG, \sV)\big)  \ar{r} & \pr_* \big( \dSchur_{\GG}^{\tau^i}(\sQ)),
	\end{tikzcd}
	$$
in which the first vertical arrow is an equivalence by induction hypothesis and the last vertical arrow is an equivalence by Corollary \ref{cor:Bott:Grassbundle}. Therefore, the middle vertical arrow is an equivalence, completing the induction step. The case $i=\ell-1$  implies the desired assertion. 
\end{proof}

We could also obtain from Theorem \ref{thm:Bott:flagbundle} the corresponding statements for partial flag bundles (see \cite[Theorem 4.1.8]{Wey}) as well as refinements of Corollary \ref{cor:Bott:Grassbundle}. As the general theorems in the following subsections will subsume these results, we do not include them here.

% sec: BWB dGrass
\subsection{Borel--Weil--Bott Theorem for Derived Grassmannians: a Preliminary Version}
\label{sec:dBott.Pre.Grass}
The goal of this subsection is to establish the following preliminary version of derived  Borel--Weil--Bott theorem:

\begin{theorem}
\label{thm:Bott:dGrass}
Let $X$ be a prestack, let $\sE$ be a perfect complex of constant $\rank \sE \ge 1$ and Tor-amplitude in $[0,1]$ over a prestack $X$, and let $d$ be an integer satisfying $1 \le d \le \rank \sE$. We let $\pr \colon \Grass(\sE;d) \to X$ denote the derived Grassmannian and let $\rho \colon \pr^*(\sE) \to \sQ$ denote tautological morphism, where $\sQ$ is the universal vector bundle of rank $d$. Then for any partition $\lambda = (\lambda_1 \ge \ldots \ge \lambda_d \ge 0)$, the canonical morphism
	$$ \pr^*(\dSchur_X^{\lambda} (\sE)) \simeq \dSchur_{\Grass(\sE;d)}^{\lambda} (\pr^*(\sE)) \xrightarrow{\dSchur^{\lambda}(\rho)} \Schur_{\Grass(\sE;d)}^{\lambda} (\sQ)$$
 in $\QCoh(\Grass(\sE;d))$ induces a canonical equivalence in $\Perf(X)$
	$$\dSchur_X^{\lambda} (\sE) \xrightarrow{\simeq} \pr_*(\Schur_{\Grass(\sE;d)}^{\lambda} (\sQ)).$$
\end{theorem}

\begin{corollary}
\label{cor:Bott:dGrass}
Assume we are in the same situation as Theorem \ref{thm:Bott:dGrass}. Let $\lambda = (\lambda_1 \ge \cdots \ge \lambda_d \ge 0)$ and $\mu = (\mu_1 \ge \cdots \ge \mu_d \ge 0)$ be a pair of partitions such that $\mu \subseteq \lambda$. Then the canonical morphism in  $\QCoh(\Grass(\sE;d))$,
 	$$\pr^*(\dSchur_X^{\lambda/\mu}(\sE)) \simeq \dSchur_{\Grass(\sE;d)}^{\lambda/\mu}(\pr^*(\sE)) \xrightarrow{\Schur^{\lambda/\mu}(\rho)} \Schur_{\Grass(\sE;d)}^{\lambda/\mu}(\sQ),$$
 induces a canonical equivalence of perfect complexes:
 	$$\dSchur_X^{\lambda/\mu}(\sE) \xrightarrow{\simeq} \pr_*(\Schur_{\Grass(\sE;d)}^{\lambda/\mu}(\sQ)).$$
\end{corollary}
\begin{proof}
Similar to the proof of Corollary \ref{cor:Bott:Grassbundle.skew}, we can derive the desired result from Theorem \ref{thm:Bott:dGrass} using the universal cofiber sequences established by Theorem \ref{thm:fil:dSchur_LR} \eqref{thm:fil:dSchur_LR-1}.
\end{proof}

\begin{proof}[Proof of Theorem \ref{thm:Bott:dGrass}]
Since the formation of the results of Theorem \ref{thm:Bott:dGrass} is local with respect to Zariski topology on $X$ and is stable under base change, we can reduce the proof of Theorem \ref{thm:Bott:dGrass} to the universal local case where $X = \vert \sHom_{\ZZ}(\ZZ^m, \ZZ^n )\vert$ and $\sE$ is represented by the tautological morphism $[\rho \colon \sO_X \otimes \ZZ^m \to \sO_X \otimes \ZZ^n]$. If $m=0$, then $\sE$ is a vector bundle and the desired result follows from  Corollary \ref{cor:Bott:Grassbundle}. Hence we may assume $m \ge 1$; in this case the desired result follows from the canonical equivalences of Corollary \ref{cor:Koszul_vs_bSchur}. 
\end{proof}

The rest of this subsection is devoted to the proof of Theorem \ref{thm:Bott:dGrass} in the universal local case (Corollary \ref{cor:Koszul_vs_bSchur}); its proof relies on the comparison between certain Koszul-type complexes $\Kos_*^{\lambda}(\rho)$ associated with $\lambda$ (\S \ref{subsec:Koszul}) and the Schur complexes $\bSchur^{\lambda}(\rho)$  (Definition \ref{def:bSchur}).

\subsubsection{Koszul-type complexes associated with partitions}
\label{subsec:Koszul}
Let $R$ be any ordinary commutative ring, let $W = R^m$ and $V = R^n$ be finite free $R$-modules of ranks $m \ge 1$ and $n\ge1$, respectively, let $d \ge 1$ be an integer and assume that $n \ge m + d$. 
Consider the universal $\Hom$ space 
	$$X = \vert \Hom_R(W,V) \vert = \Spec_R (\Sym_R^*(W \otimes V^\vee)) \simeq \Spec (R[ \{X_{ij} \}_{1 \le i \le m, 1 \le j \le n}]).$$ 
Let $\sW = W \otimes_R \sO_X$, $\sV = V \otimes_R \sO_X$, and let $\rho \colon \sW \to \sV$ denote the tautological morphism which represents the tautological matrix $(X_{ij})_{1 \le i \le m, 1 \le j \le n} \colon R[X_{ij}]^m \to R[X_{ij}] ^n$. We let $\sE$ denote the cokernel of $\rho$; then $\sE$ is a coherent sheaf on $X$ of generic rank $n-m \ge d$.

Let $\GG= \Grass_R(V;d)$ denote the Grassmannian scheme of rank $d$ locally free quotients of $V= R^n$ over $\Spec R$, and let  
	$0 \to \sR_\GG \to V \otimes \sO_\GG \to \sQ_\GG \to 0$
denote the tautological short exact sequence on $\GG$ (so that $\sQ_\GG$ and $\sR_\GG$ are vector bundles of ranks $d$ and $n-d$, respectively). Let $q \colon \Grass_X(\sV; d) \to X$ denote the Grassmannian bundle of $\sV$, then there is a canonical identification $\Grass_X(\sV; d) = X \times_R \GG$, by virtue of Proposition \ref{prop:Grass-4,5} \eqref{prop:Grass-4}. By abuse of notations, we will use $\sR_\GG$ and $\sQ_\GG$ to denote the universal subbundle and quotient bundles on $\Grass_X(\sV; d)$, respectively. Consider the derived Grassmannian $\pr \colon Z := \Grass_X(\sE; d) \to X$ (we will later see that $Z$ is a classical smooth $R$-scheme), with tautological fiber sequence $\sR \to \pr^*(\sE) \to \sQ$. By virtue of Proposition \ref{prop:Grass:PB}, there is a commutative diagram
		$$
	\begin{tikzcd} 
		Z:=\Grass(\sE; d) \ar{d}{\pr} \ar[hook]{r}{\iota}& \Grass_X(\sV;d) = X \times_R \GG \ar{ld}{q} \\
		X
	\end{tikzcd}
	$$
where $\iota \colon Z \to X \times_R \GG$ is a closed immersion, such that $\sQ \simeq \iota^*(\sQ_\GG)$. Furthermore, the closed immersion $\iota \colon Z \hookrightarrow X \times_R \GG$ is canonically identified with the inclusion of the zero locus of a canonical regular section $s_\rho$ of the vector bundle $\sW^\vee \boxtimes_R \sQ_\GG$ on $X \times_R \GG$. Here, $s_{\rho}$ is the section corresponding to composite map $q^*\sW \xrightarrow{q^*(\rho)} q^*\sV \twoheadrightarrow \sQ_\GG$ under the canonical identification 
	$$\H^0(X \times_R \GG; \sW^\vee \boxtimes_{R} \sQ_\GG) \simeq \Ext^0_{X \times_R \GG}(q^*\sW, \sQ_\GG).$$
Consequently, $Z$ is a classical scheme, and $\iota_* \sO_Z$ is resolved by the Koszul complex 
	$$\Kos_*(s_\rho) \colon  0 \to \bigwedge\nolimits^{md} (\sW \boxtimes_R \sQ_\GG^\vee) \xrightarrow{\partial_{md}} \cdots \to \bigwedge\nolimits^{i} (\sW \boxtimes_R \sQ_\GG^\vee) \xrightarrow{\partial_i} \cdots \to \sW \boxtimes_R \sQ_\GG^\vee  \xrightarrow{\partial_1} \sO_{X \times_R \GG},$$
where the differentials $\partial_i$ are given by contractions with the section $s_{\rho}$.

The goal of this subsection is to study the complexes $\Kos^{\lambda}_*(X; \rho)$ which canonically represent the pushforwards $q_* (\Kos_*(s_{\rho}) \otimes \Schur^{\lambda}(\sQ_{\GG}))$, where $\lambda = (\lambda_1, \ldots, \lambda_d)$ are partitions. Generally, such complexes are constructed using right resolutions of $\Kos_*(s_{\rho}) \otimes \Schur^{\lambda}(\sQ_{\GG})$ by  double complexes (see \cite[\S 5.2]{Wey}). This subsection, however, considers a  situation which is special enough that right resolutions are not necessary.

We consider the (second quadrant) hypercohomology spectral sequence which computes the hyper-direct image of the complex $\Kos_*(s_{\rho}) \otimes \Schur^{\lambda}(\sQ_{\GG})$ (see \cite[Lemma B.1.5]{Laz}):
	\begin{equation}\label{eqn:Koszul:spectral}
		E_1^{s,t} = \RR^{t} q_* \left(\bigwedge\nolimits^{-s}(\sW \boxtimes_{R} \sQ_\GG^\vee) \otimes \Schur^{\lambda}(\sQ_\GG) \right)  \implies \RR^{s+t} q_*(\Kos_*(s_{\rho}) \otimes \Schur^{\lambda}(\sQ_{\GG})).
	\end{equation}
Since $\sW = W \otimes_R \sO_X$ and $X$ is affine, we have canonical identifications
	$$\RR^{t} q_* \left(\bigwedge\nolimits^{-s}(\sW \boxtimes_{R} \sQ_\GG^\vee) \otimes \Schur^{\lambda}(\sQ_\GG) \right)  =  \H^{t} \left(\GG; \bigwedge\nolimits^{-s}(W \otimes_{\sO_\GG} \sQ_\GG^\vee) \otimes \Schur^{\lambda}(\sQ_\GG) \right) \otimes_R \sO_X.$$

\begin{lemma}
\label{lem:Koszul:lambda:complex}
 In the above situation, the spectral sequence \eqref{eqn:Koszul:spectral} gives rise to a single complex $\Kos^{\lambda}_*(X; \rho)$ of vector bundles on $X$ which satisfies $\Kos^{\lambda}_i(X; \rho)=0$ for $i \notin [0, md]$, and 
	$$\Kos^{\lambda}_i(X; \rho)  = \H^0\left(\GG; \bigwedge\nolimits^{i}(W \otimes_{\sO_\GG} \sQ_\GG^\vee) \otimes \Schur^{\lambda}(\sQ_\GG) \right) \otimes_R \sO_X \quad \text{for} \quad 0 \le i \le md.$$
The differentials $d_i \colon \Kos^{\lambda}_i(X; \rho)  \to \Kos^{\lambda}_{i-1}(X; \rho)$
are induced by taking zeroth cohomology of the differentials 
	$\partial_i \colon \bigwedge\nolimits^i(W \otimes_{\sO_\GG} \sQ_\GG^\vee) \to \bigwedge\nolimits^{i-1}(W \otimes_{\sO_\GG} \sQ_\GG^\vee)$
of the Koszul complex $\Kos_*(s_{\rho})$
\end{lemma}

\begin{proof}
This results from the following Lemma \ref{lem:vanish:G:koszul} and the degeneracy of the spectral sequence whose $E_1$-page satisfies $E_1^{s,t}=0$ for all $t \neq 0$ (see \cite[Example B.1.6]{Laz}).  
\end{proof}

\begin{lemma} 
\label{lem:vanish:G:koszul}
Let $R$ be an ordinary commutative ring, let $m,n,d$ be integers such that $1 \le d \le n-m$, let $\GG = \Grass_R(R^n;d)$ denote the Grassmannian scheme for rank $d$ locally free quotients of $R^n$, and let $0 \to \sR_\GG \to \sO_{\GG}^n \to \sQ_\GG \to 0$ denote the tautological short exact sequence on $\GG$. Then for all partitions $\lambda = (\lambda_1 \ge \ldots \ge \lambda_d \ge 0)$ and integers $j,k$, we have 
\begin{align*}
	\H^k \left(\GG; \bigwedge\nolimits^{j}(\sO_{\GG}^m \otimes_{\sO_\GG} \sQ_\GG^\vee) \otimes \Schur^{\lambda}(\sQ_\GG) \right) =0 \quad \text{for all} \quad k>0, j \ge 0.
\end{align*}
Moreover, $\H^0 \big(\GG; \bigwedge\nolimits^{j}(\sO_{\GG}^m \otimes_{\sO_\GG} \sQ_\GG^\vee) \otimes \Schur^{\lambda}(\sQ_\GG) \big)$ are finite free $R$-modules for all $j \ge 0$.
\end{lemma}

\begin{proof}
By virtue of Theorem \ref{thm:fil:sym_otimes} \eqref{thm:fil:sym_otimes-1}, the vector bundle $\bigwedge\nolimits^{j}(\sO_{\GG}^m \otimes_{\sO_\GG} \sQ_\GG^\vee)$ admits a filtration with subquotients of the form $\Weyl^{\mu^t}(\sO_{\GG}^m) \otimes \Schur^{\mu}(\sQ_\GG^\vee)$, where $\mu = (\mu_1, \ldots \mu_d)$ are partitions of $j$ such that $\mu_1 \le m$. On the other hand, since $m \le n-d$, there is an isomorphism
	$$ \Schur^{\mu}(\sQ_\GG^\vee) \simeq \Schur^{(n-d) - \mu}(\sQ_\GG) \otimes  (\bigwedge\nolimits^{n-d} \sR_\GG)^{\otimes (n-d)},$$
where $(n-d) -\mu$ denotes the partition $(n-d -\mu_d, \ldots, n-d-\mu_1)$ and we use the fact that $(\bigwedge\nolimits^{d} \sQ_\GG)^\vee \simeq (\bigwedge\nolimits^{n-d} \sR_\GG) \otimes (\bigwedge\nolimits^{n} \sO_{\GG}^n)^\vee \simeq (\bigwedge\nolimits^{n-d} \sR_\GG)$. From the characteristic-free version of the Littlewood--Richardson's rule (Corollary \ref{cor:fil:dSchur_LR}), $\Schur^{(n-d)-\mu}(\sQ_\GG) \otimes \Schur^{\lambda}(\sQ_\GG)$ has a filtration with subquotients of the form $\Schur^{\gamma}(\sQ_\GG)$, where $\gamma = (\gamma_1, \ldots, \gamma_d)$ are partitions satisfying $\gamma \supseteq \lambda$ and $|\gamma|=|\lambda| - |\mu| + d(n-d)$.  Therefore, to prove the desired result, it suffices to prove the following two assertions for {\em any} partition $\gamma = (\gamma_1, \ldots, \gamma_d)$:
\begin{enumerate}
	\item[$(i)$] $\H^k \left(\GG; \Schur^{\gamma}(\sQ_\GG) \otimes (\bigwedge\nolimits^{n-d} \sR_\GG)^{\otimes (n-d)} \right) = 0$ for any integer $k> 0$.  	
	\item[$(ii)$] $\H^0 \left(\GG; \Schur^{\gamma}(\sQ_\GG) \otimes (\bigwedge\nolimits^{n-d} \sR_\GG)^{\otimes (n-d)} \right)$ is (either zero or) finite  free over $R$.
\end{enumerate}
In order to prove these assertions, we consider the forgetful morphism $\pi_{(d), \underline{n}} \colon \Flag_R(R^n; \underline{n}) \to \GG$ (\S \ref{sec:dflag:forget}). By virtue of Theorem \ref{thm:Bott:flagbundle} and Example \ref{eg:forget:flag.to.Grass}, we have	
	$$\Schur^{\gamma}(\sQ_\GG) \otimes (\bigwedge\nolimits^{n-d} \sR_\GG)^{\otimes (n-d)} \simeq (\pi_{(d), \underline{n}} )_* (\sL(\alpha)),$$
where $\sL(\alpha)$ is the line bundle on $\Flag_R(R^n; \underline{n})$ associated with the sequence
	$$\alpha = (\gamma_1, \ldots , \gamma_d; \underbrace{n-d, \ldots, n-d}_{(n-d) \,\text{terms}}).$$
If $\gamma_d \ge n-d$, then $\alpha$ is dominant and assertions $(i)$ and $(ii)$ in this case are both consequences of Kempf's vanishing theorem (see \cite{Kempf} or Remark \ref{rmk:Kempf}). If $0 \le \gamma_d \le n-d-1$, then $\alpha$ satisfies the condition \eqref{prop:dflag:vanishing-a} of Proposition \ref{prop:dflag:vanishing}, and consequently, we have:
	$$\H^k \left(\GG; \Schur^{\gamma}(\sQ_\GG) \otimes (\bigwedge\nolimits^{n-d} \sR_\GG)^{\otimes (n-d)} \right) = 0 \quad \text{for all} \quad k \in \ZZ.$$
Hence assertions $(i)$ and $(ii)$ are proved, so is the lemma.
\end{proof}

\begin{lemma} 
\label{lem:Koszul:lambda:acyclic}
The complex $\Kos^{\lambda}_*(X; \rho)$ of Lemma \ref{lem:Koszul:lambda:complex}
 is a finite free resolution of the (discrete) $\sO_X$-module $\RR^0 \pr_* (\Schur_{\Grass(\sE;d)}^\lambda(\sQ)) \simeq \RR^0 q_*(\iota_*(\sO_Z) \otimes \Schur_{\GG}^{\lambda}(\sQ_\GG))$. 
\end{lemma}

\begin{proof} 
By virtue of \cite[Theorem 5.1.3]{Wey}, we only need to show that 
	$$\RR^i \pr_* (\Schur_{\Grass(\sE;d)}^\lambda(\sQ)) (\simeq \RR^i q_*(\iota_* \sO_Z \otimes \Schur^\lambda(\sQ_\GG))) \simeq  0 \quad \text{for all} \quad i>0.$$
From the short exact sequence of vector bundles on $X \times_R \GG$,
	$$0 \to (W^\vee \otimes \sO_X) \boxtimes \sR_{\GG} \to W^\vee \otimes V \otimes \sO_{\GG}  \to (W^\vee \otimes \sO_X) \boxtimes \sQ_{\GG} \to 0,$$
we see that the composite map 
	$Z \xrightarrow{\iota} X \times_R \GG \xrightarrow{\pr_2} \GG$
identifies $Z$ as the vector bundle scheme
	$$ Z \simeq \VV_\GG(W \otimes_{\sO_\GG} \sR_\GG^\vee) = \Spec_{\GG}(\Sym_{\sO_\GG}^*(W\otimes_{\sO_\GG} \sR_\GG^\vee))$$
over $\GG$. Consequently, we have
	$$\RR^i \pr_* (\Schur_{\Grass(\sE;d)}^\lambda(\sQ)) = \H^i(\GG; \Sym_{\sO_\GG}^*(W\otimes_{\sO_\GG} \sR_\GG^\vee) \otimes \Schur^\lambda(\sQ_\GG)) \otimes_R \sO_X.$$
By Cauchy's decomposition formula (Theorem \ref{thm:fil:sym_otimes} \eqref{thm:fil:sym_otimes-1}), there is a canonical filtration on $\Sym_{\sO_\GG}^*(W\otimes_{\sO_\GG} \sR_\GG^\vee)$ whose associated graded module is 
	$\bigoplus_{\mu} \Schur^\mu(W) \otimes_{\sO_\GG} \Schur^\mu(\sR_\GG^\vee)$,
where $\mu$ runs through all partitions. Therefore, we only need to prove:
	\begin{itemize}
		\item[(*)] For any partition $\mu$, $\H^i(\GG; \Schur^\mu(\sR_\GG^\vee) \otimes \Schur^\lambda(\sQ_\GG)) = 0$ for all $i > 0$.
	\end{itemize}
For any given partition $\mu$, consider an integer $N \gg 0$ (for example, let $N > \mu_1$). If $\mu_{n-d+1} \ne 0$, we have $\Schur^\mu(\sR_\GG^\vee)=0$. If $\mu_{n-d+1} = 0$, we have canonical isomorphisms
	$$\Schur^\mu(\sR_\GG^\vee) \simeq \Schur^{\gamma}(\sR_\GG) \otimes (\det \sR_\GG^\vee)^{\otimes N} \simeq \Schur^{\gamma}(\sR_\GG) \otimes (\bigwedge\nolimits^{d} \sQ_\GG)^{\otimes N},$$
where $\gamma$ denotes the partition $(N - \mu_{n-d}, \ldots, N- \mu_1)$. Similar to the proof of Lemma \ref{lem:vanish:G:koszul}, if we consider the forgetful map $\pi_{(d), \underline{n}} \colon \Flag_R(R^n; \underline{n}) \to \GG$, then $\Schur^\mu(\sR_\GG^\vee) \otimes \Schur^\lambda(\sQ_\GG) \simeq (\pi_{(d), \underline{n}})_{*} \sL(\alpha)$, where $\sL(\alpha)$ is the line bundle on $\Flag_R(R^n; \underline{n})$ associated with the partition $\alpha = (N+\lambda_1, \ldots, N + \lambda_{d}, N - \mu_{n-d}, \ldots, N - \mu_{1})$. Consequently, we have
	$$\H^i(\GG; \Schur^\mu(\sR_\GG^\vee) \otimes \Schur^\lambda(\sQ_\GG)) = \H^i(\Flag_R(R^n; \underline{n}); \sL(\alpha)) = 0 \quad \text{for all} \quad i>0,$$ 
by Kempf's vanishing theorem (see \cite{Kempf} or Remark \ref{rmk:Kempf}). Hence assertion $(*)$ is proved. 
\end{proof}

The next result compares the complex $\Kos_*^{\lambda}(X; \rho)$ and the Schur complex $\bSchur^{\lambda}(\rho)$. 

\begin{lemma}
\label{lem:Koszul_vs_bSchur}
In the situation of Lemma \ref{lem:Koszul:lambda:complex}, we consider the  commutative diagram:
	\begin{equation}\label{diag:prop:Koszul_vs_bSchur}
	\begin{tikzcd}
		\sW \otimes \Schur_X^{\lambda/1}(\sV) \ar{d}{f_1} \ar{r}{d_1}& 
	\Schur_X^{\lambda}(\sV) \ar{d}{f_0} \\
		\H^0(\GG; \sW \otimes \sQ_\GG^\vee \otimes \Schur^{\lambda}(\sQ_\GG)) \otimes_R \sO_X \ar{r}{d_1'} & \H^0(\GG; \Schur^{\lambda}(\sQ_\GG)) \otimes_R \sO_X,
	\end{tikzcd}
	\end{equation}
where $d_1$ and $d_1'$ are the differentials between degree-one and degree-zero components of the complexes $\Kos_*^{\lambda}(X; \rho)$ and $\bSchur^{\lambda}(\rho)$, respectively, the map $f_0$ is induced from the pushforward of tautological map $q^*(\Schur_X^\lambda(\sV)) \xrightarrow{\Schur^{\lambda}(\rho)} \Schur_{\GG}^{\lambda}(\sQ_\GG)$, and the map $f_1$ is induced from the composition
	$$\sW \otimes q^*\Schur^{\lambda/1}(\sV) \xrightarrow{\id \otimes \Schur^{\lambda/1}(\rho)} \sW \otimes \Schur^{\lambda/1}(\sQ_\GG) \xrightarrow{\id \otimes g} \sW \otimes \sQ_\GG^\vee \otimes \Schur^{\lambda}(\sQ_\GG).$$
(Here, the map $g$ is defined as in Lemma \ref{lem:lambda/1tolambda}.) Then the maps $f_1$ and $f_0$ are isomorphisms. 
\end{lemma}

\begin{proof}
It follows easily from the definitions of the differentials of $\Kos_*^{\lambda}(X; \rho)$ and $\bSchur^{\lambda}(\rho)$ that diagram \eqref{diag:prop:Koszul_vs_bSchur} is commutative. The fact that $f_0$ is an isomorphism follows from the classical Borel--Weil--Bott theorem for Grassmannians (Corollary \ref{cor:Bott:Grassbundle}). We will now show that $f_1$ is an isomorphism. Lemma \ref{lem:lambda/1tolambda} implies that there is a canonical map $g \colon \Schur^{\lambda/1}(\sQ_\GG) \to \sQ_\GG^\vee \otimes \Schur^{\lambda}(\sQ_\GG)$ such that the cokernel of map
	$\sW \otimes \Schur^{\lambda/1}(\sQ_\GG) \xrightarrow{\id \otimes g} \sW \otimes \sQ_\GG^\vee \otimes \Schur^{\lambda}(\sQ_\GG)$
is either zero (if $\lambda_1 = d$) or the sheaf 
	$\sW \otimes (\bigwedge\nolimits^d \sQ_\GG)^\vee \otimes \Schur^{(d-1,\lambda^t)^t}(\sQ_\GG)$
 (if $\lambda_1 \le d-1$). Assume that $\lambda_1 \le d-1$, and  consider the forgetful map $\pi_{(d), \underline{n}} \colon \Flag_R(R^n; \underline{n}) \to \GG$ as in the proof of Lemma \ref{lem:vanish:G:koszul}. Then we have
	$$ (\bigwedge\nolimits^d \sQ_\GG)^\vee \otimes \Schur^{(d-1,\lambda^t)^t}(\sQ_\GG) \simeq \Schur^{(d-1,\lambda^t)^t}(\sQ_\GG)  \otimes (\bigwedge\nolimits^{n-d} \sR_\GG) \simeq  (\pi_{(d), \underline{n}})_{*} \sL(\alpha),$$
where $\sL(\alpha)$ is the line bundle on $\Flag_R(R^n; \underline{n})$ associated with the partition 
	$$\alpha = \big((d-1,\lambda^t)^t, 1^{n-d} \big) = (1+\lambda_1, \ldots, 1 + \lambda_{d-1}, 0, \underbrace{1, 1, \ldots, 1}_{\text{$(n-d)$ terms}}).$$
Then $\alpha$ satisfies the condition \eqref{prop:dflag:vanishing-a} of Proposition \ref{prop:dflag:vanishing}, and consequently, we have
	$$\H^*\big(\GG; \sW \otimes (\bigwedge\nolimits^d \sQ_\GG)^\vee \otimes \Schur^{(d-1,\lambda^t)^t}(\sQ_\GG)\big) = W \otimes_R \H^*\big(\Flag(R^n;\underline{n}); \sL(\alpha)\big)= 0.$$	
Hence, the map $\id \otimes g$ induces an isomorphism
	$$\H^0\big(\GG; \sW \otimes \Schur^{\lambda/1}(\sQ_\GG)\big) \xrightarrow{\sim} \H^0\big(\GG; \sW \otimes \sQ_\GG^\vee \otimes \Schur^{\lambda}(\sQ_\GG)\big).$$
By the Borel--Weil theorem for skew partitions (Corollary 
\ref{cor:Bott:Grassbundle.skew}), the map
	$\id \otimes \Schur^{\lambda/1}(\rho) \colon \sW \otimes q^*\Schur^{\lambda/1}(\sV) \to \sW \otimes \Schur^{\lambda/1}(\sQ_\GG)$
induces an isomorphism
	$$\sW \otimes \Schur^{\lambda/1}(\sV) \xrightarrow{\sim} \H^0\big(\GG; \sW \otimes \Schur^{\lambda/1}(\sQ_\GG)\big) \otimes_R \sO_X.$$
All combined, the morphism $f_1$ is a composition of isomorphisms
	$$ \sW \otimes \Schur^{\lambda/1}(\sV) \xrightarrow{\sim} \H^0\big(\GG; \sW \otimes \Schur^{\lambda/1}(\sQ_\GG)\big) \otimes_R \sO_X \xrightarrow{\sim}  \H^0\big(\GG; \sW \otimes \sQ_\GG^\vee \otimes \Schur^{\lambda}(\sQ_\GG)\big) \otimes_R \sO_X.$$
This shows that $f_i$ in the diagram \eqref{diag:prop:Koszul_vs_bSchur} are isomorphisms for $i=0,1$. 
\end{proof}

\begin{corollary}
\label{cor:Koszul_vs_bSchur}
In the situation of Lemma \ref{lem:Koszul:lambda:complex}, we have canonical equivalences
	\begin{equation*}
		[\Kos_*^{\lambda}(X; \rho)] \simeq \pr_* (\Schur_{\Grass(\sE;d)}^\lambda(\sQ)) \simeq \RR^0 \pr_* (\Schur_{\Grass(\sE;d)}^\lambda(\sQ))  \simeq [\bSchur_A^{\lambda}(\rho)] \simeq \dSchur_X^{\lambda}(\sE) \simeq \Schur_A^{\lambda}(\sE)
	\end{equation*}
in $\QCoh(X)$. Here: $A=\Sym_R^*(W \otimes V^\vee)$; $[\Kos_*^{\lambda}(X; \rho)]$ and $[\bSchur_A^{\lambda}(\rho)]$ denote the classes of the complexes  $\Kos_*^{\lambda}(X; \rho)$ and $\bSchur_A^{\lambda}(\rho)$ in $\QCoh(X)$, respectively; $\Schur_A^{\lambda}(\sE)$ is the classical Schur module (Definition \ref{def:SchurWeyl}); $\dSchur_X^{\lambda}(\sE)$ is the derived Schur powers of $\sE$ (Definition \ref{def:dSchurWeyl}).
\end{corollary}
\begin{proof}
The first equivalences follows from the fact that $\Kos^{\lambda}_*(X; \rho)$ canonically represents the derived pushforwards $q_* (\iota_* \sO_Z \otimes \Schur^{\lambda}(\sQ_{\GG})) \simeq \pr_*(\Schur^{\lambda}_{\Grass(\sE;d)} (\sQ))$. The second and third equivalences follow from the isomorphisms $f_i$ of diagram \eqref{diag:prop:Koszul_vs_bSchur} and the acyclic properties of $\Kos_*^{\lambda}(X; \rho)$ (Lemma \ref{lem:Koszul:lambda:acyclic}) and $\bSchur_A^{\lambda}(\rho)$ (Theorem \ref{thm:acyclic} \eqref{thm:acyclic-2}). The last two equivalences follow from the canonical equivalences $[\bSchur_A^{\lambda}(\rho)] \simeq \dSchur_X^{\lambda}(\cofib(\rho))$ (Proposition \ref{prop:dSchur_vs_bSchur}) and $\pi_0(\dSchur_X^{\lambda}(\sE)) \simeq \Schur_A^{\lambda}(\sE)$ (Proposition \ref{prop:dSchur:classical}) and the acyclic property $\pi_i(\dSchur_X^{\lambda}(\sE) \simeq \H_i(\bSchur_A^{\lambda}(\rho)) \simeq 0$ for $i \ne 0$ (Proposition \ref{prop:dSchur_vs_bSchur} and Theorem \ref{thm:acyclic} \eqref{thm:acyclic-2}).
\end{proof}

\begin{remark}[Koszul-type Complexes $\Kos_*^{\lambda}(X;\rho)$ and Schur Complexes $\bSchur_X^{\lambda}(\rho)$]
\label{rmk:Koszul.v.s.Schur.complexes}
The construction of the complexes $\Kos_*^{\lambda}(X;\rho)$ extends to the case of any scheme $X$ as follows. Let $X$ be a scheme, let $\lambda = (\lambda_1, \ldots, \lambda_d)$ be a partition, and let $\rho \colon \sW \to \sV$ be a morphism between vector bundles over $X$ such that $\rank \sV - \rank \sW \ge d$. We let $q \colon \GG=\Grass_{X}(\sV; d) \to X$ denote the projection from the Grassmannian bundle, and let $0 \to \sR_{\GG} \to q^*(\sV) \to \sQ_{\GG} \to 0$ denote the tautological short exact sequence over $\Grass_{X}(\sV;d)$. We consider the Kosozul $\Kos_*(s_\rho) = \{\bigwedge\nolimits^*(q^*(\sW) \otimes \sQ_{\GG}^\vee)\}$ for the section $s_{\rho} \colon q^*(\sW)^\vee \otimes \sQ_{\GG} \to \sO_{\GG}$ determined by $\rho$. (In general, $\Kos_*(s_\rho)$ represents the structure sheaf of the {\em derived} closed subscheme $\Grass_{X}(\cofib(\rho); d) \subseteq \Grass_X(\sV; d)$ determined by Proposition \ref{prop:Grass:PB}.) We define $\Kos_*^{\lambda}(X;\rho)$ to be the complex which represents the pushforward $q_*(\Kos_*(s_\rho) \otimes \Schur^{\lambda}(\sQ_{\GG}))$ as in Lemma \ref{lem:Koszul:lambda:complex}. More concretely, let $m = \rank \sW$ and $n = \rank \sV$, then Lemma \ref{lem:vanish:G:koszul} implies that
	\begin{enumerate}
		\item [(1)] $\Kos^{\lambda}_*(X; \rho)$ is a complex of vector bundles whose components are given by 
	$$\Kos^{\lambda}_{k}(X; \rho)  = \RR^0 q_* \left(\bigwedge\nolimits^{k}(q^*(\sW) \otimes \sQ_\GG^\vee) \otimes \Schur^{\lambda}(\sQ_\GG) \right) \in {\rm Vect}(X),$$
	and whose differentials are induced by the pushforwards of the differentials of the Koszul complex $\Kos_*(s_{\rho})$. In particular, $\Kos^{\lambda}_k(X; \rho)=0$ for $k \notin [0, md]$.
		\item[(2)] The formation of the complex $\Kos^{\lambda}_*(X; \rho)$ commutes with base change of schemes; that is, let $f \colon Y \to X$ be a morphism of schemes, then there is a canonical isomorphism
			$$f^* \Kos_*^{\lambda}(X; \rho \colon \sW \to \sV) \simeq \Kos_*^{\lambda}(Y; f^*\rho \colon f^* \sW \to f^* \sV).$$
	\end{enumerate}
In the universal local case, we deduce from Corollary \ref{cor:Koszul_vs_bSchur} (and \cite[Proposition 5.2.5]{Wey}) that:
	\begin{enumerate}
		\item[(3)] If $X = \vert \sHom_{\ZZ}(\ZZ^m, \ZZ^n) \vert$ and $\rho \colon \sO_X^m \to \sO_X^n$ is the tautological morphism, $\Kos_*^{\lambda}(X; \rho)$ is a minimal finite free resolution of the discrete coherent sheaf $\Schur^{\lambda}({\rm Coker}(\rho))$.
	\end{enumerate}
Here, ``minimal" means that the differentials have positive degrees with respect to the natural grading of the graded coordinate ring $\Sym^*(\ZZ^m \otimes_{\ZZ}  (\ZZ^n)^\vee) \simeq \ZZ[\{X_{ij}\}_{1 \le i \le m, 1 \le j \le n}]$ of $X$; see \cite[\S 5.2]{Wey} for more details. 
Regarding Schur complexes, by virtue of the universal freeness (Theorem \ref{thm:bSchur:free}) we can generalize the construction and obtain a complex $\bSchur_X^{\lambda}(\rho)$ of vector bundles for any scheme $X$ and any morphism of $\rho \colon \sW \to \sV$ between vector bundles on $X$. These Schur complexes $\bSchur_X^{\lambda}(\rho)$ also have the properties $(2)$ and $(3)$ listed above.
Consequently, we deduce from the uniqueness of minimal resolutions and the base change property that
	\begin{enumerate}
		\item[(4)]  $\Kos^{\lambda}_*(X; \rho)$  is isomorphic to the Schur complex $\bSchur_X^{\lambda}(\rho)$. In particular, there is an isomorphism for each $0 \le k \le md$:
		$$\RR^0 q_* \left(\bigwedge\nolimits^{k}(q^*(\sW) \otimes \sQ_\GG^\vee) \otimes \Schur^{\lambda}(\sQ_\GG) \right)\simeq \bSchur_X^{\lambda}(\rho \colon \sW \to \sV)_k.$$
	\end{enumerate}
The strategy for proving $(4)$ described above is abstract. Providing an explicit isomorphism between these two complexes would be much desirable. 
Here is a direct proof of the isomorphisms of components in (4) in the case of characteristic zero:

Assume that $X$ is defined over the field $\QQ$ of rational numbers. On the one hand, \cite[Theorem 2.4.10]{Wey} (see also Theorem \ref{thm:fil:bSchur_k}) implies that there is an isomorphism 
	$$ \bSchur_X^{\lambda}(\rho \colon \sW \to \sV)_k \simeq \bigoplus_{(\mu, \nu) \, \colon \, |\mu|=k} (\Schur^{\mu^t}(\sW) \otimes \Schur^{\nu}(\sV))^{\oplus c_{\mu,\nu}^{\lambda}},$$
where $c_{\mu,\nu}^{\lambda}$ is the Littlewood--Richardson number. On the other hand, we have
	\begin{align*}
		\Kos^{\lambda}_k(X; \rho) 
		& \simeq \bigoplus_{\mu \colon |\mu|=k} \RR^0 q_* \big( q^* \Schur^{\mu^t}(\sW) \otimes \Schur^{\mu}(\sQ_\GG^\vee) \otimes \Schur^{\lambda}(\sQ_\GG) \big) \\
		& \simeq \bigoplus_{\mu \colon |\mu|=k} \Schur^{\mu^t}(\sW)  \otimes \RR^0 q_* \big( \Schur^{(n-d)- \mu}(\sQ_\GG) \otimes \Schur^{\lambda}(\sQ_\GG) \otimes \det(\sQ_\GG)^{\otimes (n-d)}\big) \\
		& \simeq \bigoplus_{(\mu, \gamma) \colon |\mu|=k} 
		\Schur^{\mu^t}(\sW)  \otimes \RR^0 q_* \big(\Schur^{\gamma}(\sQ_{\GG}) \otimes \det(\sQ_\GG)^{\otimes (n-d)}\big)^{\oplus c_{\lambda, (n-d)-\mu}^{\gamma}}
	\end{align*} 
where $(n-d) - \mu$ denotes $(n-d -\mu_d, \ldots, n-d-\mu_1)$. The proof of Lemma \ref{lem:vanish:G:koszul} shows that the above pushforward is nonzero only if $\gamma = (\gamma_1, \ldots, \gamma_d)$ satisfies $\gamma_d \ge n-d$, that is, only if there is a partition $\nu$ such that $\gamma = (n-d) + \nu = (n-d+\nu_1, \ldots, n-d+\nu_d)$. Hence we have
	\begin{align*}
		\Kos^{\lambda}_k(X; \rho) \simeq  \bigoplus_{(\mu, \nu) \colon |\mu|=k}(\Schur^{\mu^t}(\sW) \otimes \Schur^{\nu}(\sV) )^{\oplus c_{\lambda, (n-d)-\mu}^{(n+d)+\nu}}.
	\end{align*}
To complete the proof of the isomorphism $\Kos^{\lambda}_k(X; \rho) \simeq \bSchur_X^{\lambda}(\rho)_k$ in characteristic zero, it only remains to prove the equality  $c_{\mu, \nu}^{\lambda} = c_{\lambda, (n-d)-\mu}^{(n+d)+ \nu}$. This equality is a consequence of the symmetries of Littlewood--Richardson numbers and can be obtained, for example, from \cite{BZ}.
\end{remark}

\subsubsection{A short exact sequence for skew Schur modules}

For the proof of Lemma \ref{lem:Koszul_vs_bSchur}, we used a general short exact sequence (Lemma \ref{lem:lambda/1tolambda}) for skew Schur modules. The goal of this subsection is to prove Lemma \ref{lem:lambda/1tolambda}; its proof is purely classical and combinatorial, so readers might skip it on first reading since it is independent of other parts of the paper.

Let $R$ be a commutative ring and $E$ a free $R$-module of rank $d \ge 1$. Let $\lambda = (\lambda_1, \ldots, \lambda_d) \ne (0)$ be a partition, with transpose $\lambda^t = (p_1, \ldots, p_s)$ (in particular, we have $d \ge p_1 \ge p_2 \ge \cdots \ge p_s$), where $s = \lambda_1^t$ is a positive integer. Consider the $R$-module homomorphism 
	$$\widetilde{f} \colon \bigwedge\nolimits^d (E)  \otimes \bigwedge\nolimits^{p_1-1}(E)\otimes  %\bigwedge\nolimits^{p_2}(E) \ 
	\cdots \otimes \bigwedge\nolimits^{p_s}(E) \to \bigwedge\nolimits^{d-1} (E)  \otimes \bigwedge\nolimits^{p_1}(E) \otimes \cdots \otimes \bigwedge\nolimits^{p_s}(E)$$
induced by the composite map
	$$\bigwedge\nolimits^d (E) \otimes \bigwedge\nolimits^{p_1-1}(E)  \xrightarrow{\Delta' \otimes 1}  \bigwedge\nolimits^{d-1} (E)  \otimes E \otimes \bigwedge\nolimits^{p_1-1}(E)   \xrightarrow{1 \otimes m'} \bigwedge\nolimits^{d-1} (E)  \otimes \bigwedge\nolimits^{p_1}(E).$$
We wish to show that $\widetilde{f}$ descends to a morphism for the Schur modules. 
We observe that:  
	\begin{enumerate}
		\item The shuffle relations defining $\Schur^{\lambda/1}(E)$ and $\Schur^{\lambda}(E)$ (see Notation \ref{notation:square}) are the same except those from shuffling among the first two columns of $\lambda$ and $\lambda/1$, respectively. Consequently, 
		the morphism $\widetilde{f}$ carries the submodule of relations $\bigwedge\nolimits^d(E) \otimes \Im (\square_{\lambda/1; j})$ to the submodule of relations $\bigwedge\nolimits^{d-1}(E) \otimes \Im (\square_{\lambda; j})$ for all $j \ge 2$.
		\item Regarding the shuffle relations for the first two columns of $\lambda/1$ and $\lambda$, we note that, for each $0 \le u \le p_1-1$, $0 \le v \le p_2$ such that $u+v < p_2 -1$, the morphism $\widetilde{f}$ carries the submodule of relations $\bigwedge\nolimits^d(E) \otimes \Im (\square_{\lambda/1; u, v; 1})$ onto the submodule of relations $\bigwedge\nolimits^{d-1}(E) \otimes \Im (\square_{\lambda; u+1, v; 1})$, where $0 \le u+1 \le p_1$ and $(u+1) + v \le p_1 + p_2$.
	\end{enumerate}
Consequently, the morphism $\widetilde{f}$ descends to an $R$-module homomorphism for Schur modules
	\begin{equation}
		\label{eqn:lem:lambda/1tolambda}
	f \colon \bigwedge\nolimits^d (E) \otimes \Schur^{\lambda/1} (E)   \to  \bigwedge\nolimits^{d-1}(E) \otimes \Schur^{\lambda}(E).
	\end{equation}

\begin{remark}
The morphism $f$ admits the following equivalent characterization based on the isomorphism $\bigwedge^{d}(E) \otimes E^\vee \simeq \bigwedge^{d-1}(E)$. Using the same argument as above, we obtain that multiplication map $E \otimes \bigwedge^{\lambda/1}(E) \to \bigwedge^{\lambda}(E)$ descends to a map
	$E \otimes \Schur^{\lambda/1} (E)  \to \Schur^{\lambda}(E)$; we let $g \colon \Schur^{\lambda/1} (E) \to E^\vee \otimes \Schur^{\lambda}(E)$ denote its induced map. Then the morphism $f$ \eqref{eqn:lem:lambda/1tolambda} is canonically isomorphic to the composition
	$$\bigwedge\nolimits^{d} (E) \otimes \Schur^{\lambda/1} (E) \xrightarrow{\id \otimes g} \bigwedge\nolimits^{d} (E) \otimes   E^\vee \otimes \Schur^{\lambda}(E) \simeq  \bigwedge\nolimits^{d-1}(E)  \otimes \Schur^{\lambda}(E).$$ 
\end{remark}

\begin{lemma}
\label{lem:lambda/1tolambda}
In the above situation, the following statements are true:
\begin{enumerate}
	\item \label{lem:lambda/1tolambda-1}
	If $p_1 = d$, then the map \eqref{eqn:lem:lambda/1tolambda} induces an isomorphism
			$$f\colon \bigwedge\nolimits^d (E) \otimes \Schur^{\lambda/1} (E)   \xrightarrow{\simeq} \bigwedge\nolimits^{d-1}(E) \otimes \Schur^{\lambda}(E).$$
	\item \label{lem:lambda/1tolambda-2}
	If $p_1 \le d-1$, then the map $f$ defined by \eqref{eqn:lem:lambda/1tolambda} is injective, and fits into a short exact sequence of free $R$-modules:
			$$0 \to \bigwedge\nolimits^d (E) \otimes \Schur^{\lambda/1} (E)   \xrightarrow{f} \bigwedge\nolimits^{d-1}(E) \otimes \Schur^{\lambda}(E) \to \Schur^{(d-1,\lambda^t)^t}(E) \to 0.$$
\end{enumerate}
\end{lemma}

\begin{example}
\begin{enumerate}[label=(\roman*)]
	\item 
	If $\lambda = (k)$, $k \ge 2$, then Lemma \ref{lem:lambda/1tolambda} implies:
		\begin{itemize}
			\item	 If $\rank E= 1$, there is a canonical isomorphism $E \otimes \Sym^{k-1}(E) \simeq \Sym^{k}(E)$.
			\item If $\rank E= d \ge 2$, there is a canonical short exact sequence
				$$0 \to \bigwedge\nolimits^{d} (E) \otimes \Sym^{k-1}(E) \to \bigwedge\nolimits^{d-1} (E) \otimes \Sym^{k}(E) \to \Schur^{(k+1,1^{d-1})}(E) \to 0.$$
		\end{itemize}

	\item If $\lambda = (k)$, $1 \le k \le d$, then Lemma \ref{lem:lambda/1tolambda} implies:
		\begin{itemize}
			\item If $k=d$, there is a canonical isomorphism 
				$\bigwedge\nolimits^d (E) \otimes \bigwedge\nolimits^{d-1} (E)   \simeq \bigwedge\nolimits^{d-1}(E) \otimes \bigwedge\nolimits^{d}(E).$	
			\item  If $k \le d-1$, there is a canonical short exact sequence
				$$0 \to \bigwedge\nolimits^{d} (E) \otimes \bigwedge\nolimits^{k-1}(E) \to \bigwedge\nolimits^{d-1} (E) \otimes \bigwedge\nolimits^{k} (E) \to \Schur^{(d-1,k)^t}(E) \to 0.$$
		\end{itemize}
	\item  If $\lambda = (2,1)$,  then Lemma \ref{lem:lambda/1tolambda} implies:
		\begin{itemize}
			\item If $\rank E =2$, there is a canonical isomorphism 
				$\bigwedge\nolimits^2 (E) \otimes E \otimes E   \simeq \bigwedge\nolimits^{1}(E) \otimes \Schur^{(2,1)}(E).$
			\item If  $\rank E = d \ge 3$, there is a canonical short exact sequence 
	$$0 \to \bigwedge\nolimits^{d} (E) \otimes E \otimes E   \to \bigwedge\nolimits^{d-1}(E) \otimes \Schur^{(2,1)}(E) \to \Schur^{(d-1,2,1)^t}(E) \to 0.$$
		\end{itemize}
\end{enumerate}
\end{example}

\begin{proof}
If $\rank E = d =1$, the statements are trivial. We now assume $\rank E = d \ge 2$.
In case \eqref{lem:lambda/1tolambda-1}, the map $\widetilde{f}$ is an isomorphism and sends the submodule of defining relations $\Im (\square_{\lambda^t/1})$ of $\Schur^{\lambda/1}$ identically to the submodule of defining relations $\Im (\square_{\lambda^t})$ of $\Schur^{\lambda}$ (see Notation \ref{notation:square}) except possibly the shuffling relations for the first and second columns. Therefore, we only need to consider the case where $\lambda^t= (d, p_2)$, $p_2 \le d$, and $f$ is induced from the isomorphism
	$$\widetilde{f} \colon \bigwedge\nolimits^{d} E \otimes \bigwedge\nolimits^{d-1} E \otimes \bigwedge\nolimits^{p_2} E \to \bigwedge\nolimits^{d-1} E \otimes \bigwedge\nolimits^{d} E \otimes \bigwedge\nolimits^{p_2} E.$$
In this case, we have for all $u, v \ge 0$, $u+v < p_2-1$,
	$$\bigwedge\nolimits^{d} E \otimes \left(\sum_{u+v< p_2 -1} \bigwedge\nolimits^{u} E \otimes \bigwedge\nolimits^{p_1+p_2-1-u-v} E \otimes \bigwedge\nolimits^{v} E\right) \otimes \bigwedge\nolimits^{p_2} E = 0 $$
and for all $u, v \ge 0$,  $u+v < p_2$,
	$$\bigwedge\nolimits^{d} E \otimes \left(\sum_{u+v< p_2 } \bigwedge\nolimits^{u} E \otimes \bigwedge\nolimits^{p_1+p_2-u-v} E \otimes \bigwedge\nolimits^{v} \right) \otimes \bigwedge\nolimits^{p_2} E = 0.$$
Therefore, $\widetilde{\bigwedge}^{\lambda^t/1}(E) =0$ and $\widetilde{\bigwedge}^{\lambda^t}(E)=0$ (see Notation \ref{notation:square}). Hence $\widetilde{f} =f$ in the case where $\lambda^t= (d, p_2)$. Then the desired result for general $\lambda$ follows.

Now we consider case \eqref{lem:lambda/1tolambda-2}. We have a commutative diagram of morphisms
	$$
	\begin{tikzcd} 
		\bigwedge\nolimits^d (E) \otimes \bigwedge\nolimits^{\lambda^t/1} (E) \ar[two heads]{d} \ar{r}{\widetilde{f}} &  \bigwedge\nolimits^{d-1}(E) \otimes \bigwedge\nolimits^{\lambda^t}(E)  \ar[two heads]{d}  \ar[equal]{r}  & \bigwedge\nolimits^{d-1}(E) \otimes \bigwedge\nolimits^{\lambda^t}(E) \ar[two heads]{d} \\
		\bigwedge\nolimits^d (E) \otimes \Schur^{\lambda/1} (E)    \ar{r}{f} & \bigwedge\nolimits^{d-1}(E) \otimes \Schur^{\lambda}(E) \ar{r} & \Schur^{(d-1,\lambda^t)^t}(E).
	\end{tikzcd}
	$$
We observe that the submodule $\Im (\square_{(d-1,\lambda^t)})  \subseteq \bigwedge \nolimits^{d-1}(E) \otimes \bigwedge\nolimits^{\lambda^t}(E)$ of defining relations of $\Schur^{(d-1,\lambda^t)^t}(E)$ (Notation \ref{notation:square}) is precisely the sum of the submodule $ \bigwedge\nolimits^{d-1}(E)  \otimes \Im (\square_{\lambda^t}) \subseteq \bigwedge \nolimits^{d-1}(E) \otimes \bigwedge\nolimits^{\lambda^t}(E)$ (of defining relation of $\bigwedge \nolimits^{d-1} E \otimes \Schur^{\lambda}(E)$) and the submodule $\Im(\square_{(d-1,\lambda^t); 1})$ of shuffle relations for the first two columns.
Here, by virtue of \cite[(2.1.3) (f)]{Wey}, the image $\Im(\square_{(d-1,\lambda^t); 1})$ is the sum of the images of the morphisms 
	$$\square_{(d-1,\lambda^t); u=0, v; 1} \colon \bigwedge\nolimits^{d-1+p_{1}-v} (E) \otimes \bigwedge\nolimits^{v} (E) \otimes \bigotimes\nolimits _{j \ge 2}\bigwedge\nolimits^{p_j}(E)  \to  \bigwedge\nolimits^{d-1}(E) \otimes \bigotimes\nolimits _{j \ge 1}\bigwedge\nolimits^{p_j}(E),$$ 
where $v$ runs through $0 \le v \le p_1-1$. However, since $\rank E = d$, the source of the morphism $\square_{(d-1,\lambda^t); u=0, v; 1}$ is nonzero only when $v = p_1-1$; in this case, the morphism $\square_{(d-1,\lambda^t); u=0, v=p_1-1; 1}$ coincides with the morphism $\widetilde{f}$ by definition. Hence we have $\Im(\square_{(d-1,\lambda^t); 1}) = \Im(\widetilde{f})$. Consequently, the canonical surjection
	$\bigwedge\nolimits^{d-1} E \otimes \bigwedge\nolimits^{\lambda^t}(E) \to \Schur^{(d-1,\lambda^t)^t}(E)$
descends to a surjection
	$\bigwedge \nolimits^{d-1} E \otimes \Schur^{\lambda}(E) \to  \Schur^{(d-1,\lambda^t)^t}(E)$ 
whose kernel is precisely given by the image of $f$. 

In order to prove assertion \eqref{lem:lambda/1tolambda-2}, it only remains to show the injectivity of $f$ \eqref{eqn:lem:lambda/1tolambda}; we present an elementary proof of this fact by showing that the induced map
	$$g \colon  \Schur^{\lambda/1}(E) \to E^\vee \otimes \Schur^{\lambda/1}(E)$$
sends a basis of $\Schur^{\lambda/1}(E)$ to a set of linearly independent elements of $E^\vee \otimes \Schur^{\lambda}(E)$. 

We choose a basis $e_1 < e_2 < \ldots < e_d$ of $E$ and let $e_1^\vee < e_2^\vee < \ldots < e_d^\vee$ of $E^\vee$ denote the dual basis. Then the basis of $\Schur^{\lambda/1}(E)$ and $\Schur^{\lambda}(E)$ given by $e_{\lambda/1}(T)$ and $e_{\lambda}(U)$, where $e_{\lambda/1}(T)$ and $e_{\lambda}(U)$ are the elements corresponding to $\tau$-standard tableaux $T \in {\rm ST}^{\tau}(\lambda/1, \{e_i\})$ and $U \in {\rm ST}^{\tau}(\lambda, \{e_i\})$ of shapes $\lambda/1$ and $\lambda$, respectively (\cite[Proposition (2.1.4)]{Wey}). Beware that we are using the {\em opposite} convention as \cite{Wey} such that our $\Schur^{\lambda}(E)$ is the $L_{\lambda^t}(E)$ of {\em loc. cit.}; we say a tableau $T$ of shape $\gamma$ and content $\{e_i\}$ is {\em $\tau$-standard}, and denote $T \in {\rm ST}^{\tau}(\gamma, \{e_i\})$, if its transpose $T^t \in {\rm ST}(\gamma^t, \{e_i\})$ is standard in the sense of \cite[(1.1.12)]{Wey}). 

We order the basis in lexicographical order (note that this order corresponds to the order $\preceq$ of row standard tableaux of \cite[(1.1.14)]{Wey}).
We also label the basis of $E^\vee \otimes \Schur^{\lambda}(E)$ lexicographically, that is, $e_i^\vee \otimes U_1 < e_j ^\vee \otimes U_2$ if $i < j$ or if $i=j$ and $U_1 < U_2$ in ${\rm ST}^{\tau}(\lambda, \{e_i\})$. The map $g$ sends the basis $e_{\lambda/1}(T)$, where $T \in {\rm ST}^{\tau}(\lambda/1, \{e_i\})$, 
to a $\ZZ$-linear combinations of basis 
	$$g(e_{\lambda/1}(T)) = \sum_{i,j} c_{i,j} e_i^\vee \otimes e_{\lambda}(U_j),$$
where $e_{\lambda}(U_j)$ are basis corresponding to $\tau$-standard tableaux $U_j$ of shape $\lambda$. Then the smallest base element $e_{\lambda}(U_j)$ (with respect to the lexicographical order) that appears in the summand of $g(e_{\lambda/1}(T))$ has coefficients $\pm 1$; we will refer to it as the {\em leading} term of $g(e_{\lambda/1}(T))$. 

We will prove the following claim:
\begin{itemize}
	\item[(*)] The leading terms of $g(e_{\lambda/1}(T))$, as $T$ runs through all $\tau$-standard tableaux of shape $\lambda/1$, are pairwise distinct. 
\end{itemize}
The claim $(*)$ implies that $\{g(e_{\lambda/1}(T)) \mid T \in {\rm ST}^{\tau}(\lambda/1, \{e_i\}) \}$ are linearly independent over $R$, and the desired result of the lemma will then follow.

To prove the above claim $(*)$, we observe that, if the first column $\{T(1,2), \ldots, T(1, p_1)\}$ of $T$ does not contain $e_1$, then the leading term of $g(e_{\lambda/1}(T))$ is simply given by the image of $e_1^\vee \otimes e_{\lambda}(U_T)$, where $U_T = e_1 \wedge T$ denotes the tableau of shape $\lambda$ such that $U(1,1) = e_1$ and $U(i,j) = T(i,j)$ for $(i,j) \ne (1,1)$. Since the map $T \mapsto U_T$ preserves the orderings of $\tau$-standard tableaux, the leading terms of these $g(e_{\lambda/1}(T))$ are pairwise distinct. 

More generally, if $i = i(T) \ge 2$ is the smallest number such that $\{e_i\}$ does not appear in the first column $\{T(1,2), \ldots, T(1, p_1)\}$, then $\{T(1,2), \ldots, T(1,p_1)\}$ contains $\{e_1, \ldots, e_{i-1}\}$. Therefore, the leading term of $g(e_{\lambda/1}(T))$ has the form 
	$$(-1)^{i-1} e_i^\vee \otimes (e_1 \wedge \cdots \wedge e_{i} \wedge e_{T(1,i+1)} \wedge \cdots \wedge e_{T(1,p_1)}) \otimes e_{T(2,*)} \otimes e_{T(3,*)}\cdots \otimes e_{T(s,*)}.$$
Consequently, the tableau $U_{i,T}$ of shape $\lambda$ corresponding to the above element (i.e., the tableaux which has first column $\{1, 2, \ldots, i-1, i, T(1,i+1), \ldots, T(1,i-1)\}$ and all other columns same as those of $T$) is $\tau$-standard. Apparently, for any fixed $i$, the map $T \mapsto U_{i,T}$ preserves the ordering on $\tau$-standard tableaux, and hence the corresponding tableaux $U_{i,T}$ are pairwise distinct. Since $e_i^\vee \otimes U_{i,T} < e_j^\vee \otimes U_{j,T'}$ whenever $i<j$, the claim $(*)$ is proved and so is the lemma. \end{proof}

% sec: BWB: dFlag Dominant Weights
\subsection{Borel--Weil--Bott Theorem for Derived Flag Schemes: Dominant Weights}
\label{sec:Bott.dominant}
Building on the results from earlier subsections, this section establishes  Borel--Weil--Bott theorem for derived Grassmannians and derived flag schemes in the case of  ``dominant weights" (i.e., for partitions). These derived versions of Borel--Weil--Bott theorem  simultaneously generalize the classical Borel--Weil--Bott theorem \cite{Bott} and Kempf's vanishing theorem \cite{Kempf} for dominant weights. The results presented in this subsection are characteristic-free. 

\begin{theorem}[{Borel--Weil--Bott Theorem for Derived Complete Flag Schemes: dominant weights}]  
\label{thm:Bott:dflag}
Let $X$ be a prestack, $\sE$ a perfect complex of rank $\ge n \ge 1$ and Tor-amplitude in $[0,1]$ over $X$, and let $\pr \colon \Flag_X(\sE; \underline{n}) \to X$ denote the derived flag scheme of $\sE$ of type $\underline{n}$ over $X$.
Let $\lambda = (\lambda_1 , \ldots, \lambda_n)$ be a partition and let $\sL(\lambda)$ be the associated line bundle \eqref{eqn:flag:linebundle}. Then there exits a canonical morphism of perfect complexes: 
	$\pr^*(\dSchur^{\lambda}(\sE)) \to \sL(\lambda)$
on $\Flag(\sE;\underline{n})$ 
which induces a canonical equivalence of perfect complexes on $X$:
	$$\dSchur^{\lambda}(\sE) \xrightarrow{\simeq} \pr_*(\sL(\lambda)).$$
\end{theorem}

\begin{proof}
Let $\pr_{\Grass} \colon \Grass(\sE; n) \to X$ denote the derived Grassmannian of $\sE$ over $X$ with tautological fiber sequence	 
	$\sR_{\Grass} \to \pr_{\Grass}^* (\sE) \xrightarrow{\rho_{\Grass}}\sQ_{\Grass}$. By virtue of Corollary \ref{cor:dflag:forget} \eqref{cor:dflag:forget-1}, the natural forgetful map
	$\pi_{(n), \underline{n}} \colon \Flag(\sE; \underline{n}) \to \Grass_X(\sE; n)$
canonically identifies $\Flag(\sE; \underline{n})$ as the complete flag bundle of the vector bundle $\sQ_{\Grass}$ over $\Grass(\sE;n)$. Theorem \ref{thm:Bott:flagbundle} implies that there is a canonical morphism $\alpha \colon \pi_{(n), \underline{n}}^*(\dSchur^{\lambda}(\sQ_{\Grass})) \to \sL(\lambda)$ inducing a canonical equivalence 
	$$\beta \colon \dSchur^{\lambda}(\sQ_{\Grass}) \simeq (\pi_{(n), \underline{n}})_* (\sL(\lambda)).$$

Meanwhile, Proposition \ref{thm:Bott:dGrass} implies that the canonical morphism 
	$\Schur^{\lambda}(\rho_{\Grass}) \colon \pr_{\Grass}^*(\dSchur^{\lambda}(\sE))\simeq \dSchur^{\lambda}(\pr_{\Grass}^* \sE) \to \dSchur^{\lambda}(\sQ_{\Grass})$
induces a canonical equivalence 
	$$\gamma \colon \dSchur^{\lambda}(\sE) \xrightarrow{\simeq} (\pr_{\Grass})_* (\dSchur^{\lambda}(\sQ_{\Grass})).$$
Combined, we obtain that the composite morphism
	$$\pr^*(\dSchur^{\lambda}(\sE)) 
	\simeq \pi_{(n), \underline{n}}^*(\dSchur^{\lambda}(\pr_{\Grass}^* \sE)) \xrightarrow{\pi_{(n),\underline{n}}^*(\Schur^{\lambda}(\rho_{\Grass}))} \pi_{(n), \underline{n}}^*(\dSchur^{\lambda}(\sQ_{\Grass})) \xrightarrow{\alpha} \sL(\lambda)$$
induces a sequence of canonical equivalences
	$$\dSchur^{\lambda}(\sE) \xrightarrow{\alpha} (\pr_{\Grass})_* (\dSchur^{\lambda}(\sQ_{\Grass})) \xrightarrow{(\pr_{\Grass})_*(\beta)} (\pr_{\Grass})_* \, (\pi_{(n),\underline{n}})_*(\sL(\lambda)) \simeq \pr_*(\sL(\lambda)).$$
\end{proof}

\begin{remark}
Assume we are in the situation of Theorem \ref{thm:Bott:dflag}. The assignment $(\lambda, m) \mapsto \lambda +m  := (\lambda_1 +m , \ldots, \lambda_n + m)$, where $m \in \ZZ_{\ge 0}$, $\lambda \in \ZZ_{\ge 0}^n$, defines an action of $\ZZ_{\ge 0}$ on all partitions of $n$ entries. The corresponding action on line bundles, $(\sL(\lambda), m) \mapsto \sL(\lambda +m) \simeq \sL(\lambda) \otimes (\sL_1 \otimes \cdots \otimes \sL_{n})^{\otimes m}$, induces canonical morphisms of perfect complexes on $X$:
	$$\dSchur^{\lambda}(\sE) \otimes \big(\bigwedge\nolimits^n \sE \big)^{\otimes m} \to \dSchur^{\lambda}(\sE) \otimes \dSchur^{(m^n)}(\sE) \to \dSchur^{\lambda+m}(\sE).$$
If $\sE$ is a vector bundle of rank $n$, then $\det(\sE) =  \bigwedge\nolimits^n(\sE)$ is a line bundle and both above morphisms are isomorphisms of vector bundles. If $\sE$ is a perfect complex of rank $n$, then generally $\bigwedge\nolimits^{n}(\sE)$ is not a line bundle and the above morphisms are not isomorphisms. 
\end{remark}

\begin{remark}[Kempf's Vanishing Results] 
Similar to Remark \ref{rmk:Kempf}, Theorem \ref{thm:Bott:dflag} includes Kempf-type vanishing results \cite{Kempf} as part of its formulation:
	$$\RR^0 \pr_*(\sL(\lambda)) \simeq \pi_0(\dSchur^\lambda(\sE)) \quad \text{and} \quad \RR^i \pr_*(\sL(\lambda)) \simeq \pi_{-i}(\dSchur^\lambda(\sE)) \simeq 0 \quad \text{for} \quad i > 0 .$$
Be aware that, unlike in the classical case, the derived pushfoward  $\pr_*(\sL(\lambda))$ may have nonzero negative cohomologies as computed by the formula $\RR^{-i} \pr_*(\sL(\lambda)) \simeq \pi_i(\dSchur^\lambda(\sE))$ for $  i \ge 0$. However, in the universal local situation where $X = \vert \sHom_R(R^{m}, R^{m+n})\vert$ and $\rho \colon \sO_X^{m} - \sO_X^{m+n}$ is the tautological morphism, $m \ge 1$, Corollary \ref{cor:Koszul_vs_bSchur} implies that the negative cohomologies vanish and the following strong form of Kempf's vanishing results holds:
	$$\RR^0 \pr_*(\sL(\lambda)) \simeq \Schur^\lambda({\rm Coker}(\rho)) \quad \text{and} \quad \RR^i \pr_*(\sL(\lambda)) =0 \quad \text{for} \quad i \neq 0,$$
where $\Schur^\lambda({\rm Coker}(\rho))$ is the classical Schur module  (Definition \ref{def:SchurWeyl}).
\end{remark}

\begin{corollary}[{Borel--Weil--Bott Theorem for Derived Partial Flag Schemes}]
\label{cor:Bott:dpflag}
Let $X$ be a prestack, $\sE$ a perfect complex of rank $n \ge 1$ and Tor-amplitude in $[0,1]$ over $X$, and $\bdd = (d_1, \ldots, d_k)$ an increasing sequence of integers in $[1,n]$, where $k \ge 1$ is an integer. Consider the derived flag scheme $\pr \colon \Flag_X(\sE; \bdd) \to X$. Let $\lambda = (\lambda_1, \ldots, \lambda_n)$ be a partition, and let $\sV(\lambda)$ denote the associated perfect complex \eqref{eqn:flag:Vlambda} on $\Flag(\sE; \bdd)$. Then there is a canonical map
	 $\pr^*(\dSchur^{\lambda} (\sE)) \to \sV(\lambda)$
in $\Perf(\Flag_X(\sE;\bdd))$ inducing a canonical equivalence in $\Perf(X)$:
	$$\dSchur^{\lambda} (\sE) \xrightarrow{\simeq} \pr_*(\sV(\lambda)).$$
\end{corollary}

\begin{proof}
Let $\Flag(\sE; \underline{n})$ denote the derived flag scheme of $\sE$ of type $\underline{n}$ (Example \ref{eg:dflag}) and consider the natural forgetful morphism  $\pi_{\bdd, \underline{n}} \colon \Flag(\sE; \underline{n}) \to \Flag(\sE; \bdd)$ (\S \ref{sec:dflag:forget}). Let $\sL(\lambda)$ be the associated line bundle \eqref{eqn:flag:linebundle} on $\Flag(\sE; \underline{n})$, then Proposition \ref{prop:dflag:forget} and Theorem \ref{thm:Bott:dflag} implies that there is a canonical equivalence $(\pi_{\bdd, \underline{n}})_{*} (\sL(\lambda)) \simeq \sV(\lambda).$ Hence the desired result follows from Theorem \ref{thm:Bott:dflag}.
\end{proof}

\begin{corollary}[{Borel--Weil--Bott Theorem for Derived Grassmannians}]
\label{cor:Bott:dGrass:dominant}
Let $X$ be a prestack, $\sE$ a perfect complex of rank $n \ge 1$ and Tor-amplitude in $[0,1]$ over $X$, and $d$ be an integer such that $1 \le d \le n$. Let  $\pr \colon \Grass(\sE;d) \to X$ denote the derived Grassmannian, let $\alpha = (\alpha_1, \ldots , \alpha_{d})$ and $\beta = (\beta_1, \ldots, \beta_{n-d})$ be two partitions, and let $\sV(\alpha, \beta)$ denote the associated perfect complex \eqref{eqn:Grass:Vlambda} as in Example \ref{eg:notation:V:dGrass}. Then there is a canonical map 
	$\pr^*(\dSchur^{(\alpha,\beta)} (\sE)) \to \sV(\alpha, \beta)$
 in $\Perf(\Grass(\sE;d))$ which induces a canonical equivalence in $\Perf(X)$:
	$$\dSchur^{(\alpha,\beta)} (\sE) \xrightarrow{\simeq} \pr_*(\sV(\alpha, \beta)).$$
\end{corollary}
\begin{proof}
Apply Corollary \ref{cor:Bott:dpflag} to the case where $\bdd=(d)$.
\end{proof}

\subsection{Borel--Weil--Bott Theorem for Derived Flag Schemes: Non-dominant Weights}
\label{sec:Bott.non-dominant}
This subsection establishes derived versions of the Borel--Weil--Bott theorem theorem for non-dominant weights, generalizing classical Borel--Weil--Bott theorem (\cite{Dem, Wey, Lurie}) for flag varieties. 
Most results in this subsection, like in the classical situation, require the characteristic-zero assumption, with the exception of the vanishing results of Proposition \ref{prop:dflag:vanishing}, which are characteristic-free.

Our strategy is similar to the classical situation (\cite{Dem, Wey, Lurie}), using geometry of $\PP^1$-bundle to reduce the non-dominant weight cases to dominant weight cases. 
To implement this strategy, we need to extend the classical result for $\PP^1$-bundles (\cite[Theorem 1]{Dem}, \cite[Proposition 4.2.2]{Wey}, \cite[Theorem 3]{Lurie}) to the case of general prestacks:

\begin{proposition}
\label{prop:P1:isom}
Let $Y$ be a prestack and let $\sV$ be a vector bundle of rank two on $Y$. We let $\pi \colon \PP(\sV) \to Y$ denote the associated projective bundle and let $\sL$ be a line bundle on $\PP(\sV)$ which has degree $d$ over each fiber of $\pi$, where $d \ge -1$ is an integer. Let $\sK$ denote the relative canonical line bundle for $\pi$. If $d=  -1$, then there are canonical isomorphisms 
	$$\pi_* (\sL) \simeq 0 \simeq \pi_*(\sL \otimes \sK^{\otimes (d+1)}) [1].$$
If $d \ge 0$, then there are canonical morphisms in $\Perf(Y)$,
	$$\phi_{\sL} \colon \pi_* (\sL) \to \pi_*(\sL \otimes \sK^{\otimes (d+1)}) [1] \qquad 
	\psi_{\sL} \colon \pi_*(\sL \otimes \sK^{\otimes (d+1)}) [1]  \to \pi_* (\sL), 
	$$
and canonical equivalences of functors
	$$\psi_{\sL} \circ \phi_{\sL} \simeq (d!)^2 \cdot \id_{ \pi_* (\sL) } 
	\qquad 
	 \phi_{\sL} \circ \psi_{\sL} \simeq (d!)^2 \cdot \id_{\pi_*(\sL \otimes \sK^{\otimes (d+1)}) [1]}.$$
In particular, $\phi_{\sL}$ and $\psi_{\sL}$ are isomorphisms if $Y$ is defined over a ring in which $d!$ is invertible. 
\end{proposition}

We remind the readers again that in this paper, the functor $\pi_*$ is always {\em derived}; that is, if $Y$ is a classical stack, then $\pi_*$ is the derived functor $\RR \pi_*$ of the classical convention.

\begin{proof} 
We first consider the case where $\sL = \sO_{\PP(\sV)}$. By Serre's theorem for projective bundles \cite[Theorem 5.2 (1)]{J22a}, there is a canonical isomorphism $\pi_* (\sO_{\PP(\sV)}) \simeq \sO_Y$. On the other hand, by virtue of Serre duality \cite[Theorem 5.2 (2)]{J22a},  we have canonical isomorphisms
	$$\sO_Y \simeq \pi_* (\sK) [1] \quad \text{and} \quad \sK \simeq \sO_{\PP(\sV)}(-2) \otimes \pi^*(\bigwedge\nolimits^2 \sV).$$
In this case, we let $\phi_\sO$ denote the composition of the above canonical isomorphisms
	$$\pi_* (\sO_{\PP(\sV)}) \simeq \sO_Y \simeq \pi_* (\sK [1])$$
and let $\psi_{\sO}$ denote an inverse fo $\phi_{\sO}$. If $d=0$, then the canonical map $\pi^*(\pi_* (\sL)) \to \sL$ is an isomorphism (since it induces an isomorphism over each fiber of $\pi$) and thus $\pi_*(\sL)$ is a line bundle. We then let $\phi_\sL$ denote the composition of the canonical isomorphisms:
	$$\pi_* (\sL) \xrightarrow{\phi_\sO \otimes \pi_*(\sL)} \pi_* (\sK [1]) \otimes \pi_*(\sL) \xrightarrow{\simeq} \pi_*(\pi^* (\pi_*(\sL)) \otimes \sK [1]) \xrightarrow{\simeq} \pi_*(\sL \otimes \sK)[1]$$
and let $\psi_{\sL}$ be an inverse of $\phi_{\sL}$. 

In general, we let 
	$\sL_Y = \pi_*(\sL \otimes \sO_{\PP(\sV)}(-d)),$
which is a line bundle by virtue of the above discussion. Then there are canonical isomorphisms
	\begin{align*}
		& \pi^*(\sL_Y) \otimes \sO_{\PP(\sV)}(d) \simeq \sL, \quad \\
		& \pi^*(\sL_Y) \otimes \sO_{\PP(\sV)}(-2-d) \otimes \pi^*((\bigwedge\nolimits^2 \sV)^{\otimes (d+1)}) \simeq \sL \otimes \sK^{\otimes (d+1)}.
\end{align*}
By virtue of Serre's theorem \cite[Theorem 5.2]{J22a}, we have canonical isomorphisms
	\begin{align*}
		&\pi_*(\sL) \simeq \Sym^d(\sV) \otimes \sL_Y , \quad \\ 
		&\pi_*(\sL \otimes \sK^{\otimes (d+1)})[1] \simeq \Sym^d(\sV)^\vee \otimes (\bigwedge\nolimits^2 \sV)^{\otimes d} \otimes \sL_Y.
\end{align*}
Here, $\Sym^d(\sV)$ is the $d$th symmetric product of $\sV$ and is a vector bundle, by virtue of Proposition \ref{prop:dSchur:free}, and $\Sym^d(\sV)^\vee = \sHom_{Y}(\Sym^d(\sV), \sO_Y)$ is the dual vector bundle. If $d=-1$, then both the right-hand sides of the above formulae are zero. 

It only remains to prove the cases where $d \ge 1$. We let $\Gamma^d(\sV)$ denote the $d$th divided power of $\sV$ (\cite[III \S1, p. 248]{Ro}, \cite[Construction 3.15]{J22a}). Then there is a canonical isomorphism $\Gamma^d(\sV) \simeq \Sym^d(\sV^\vee)^\vee$. 
Notice that, by virtue of Propositions \ref{prop:dSchur:basechange} and \ref{prop:dSchur:free} (see also \cite[Proposition 25.2.3.1, Corollary 25.2.3.2]{SAG}), $\Gamma^d(\sV)$ is a vector bundle and behaves exactly the same as the classical divided powers of finite free modules over an ordinary commutative ring (\S \ref{sec:classical_sym}). Therefore, we might assume $X = \Spec R$, $\sV = R^2$, where $R$ is a commutative ring in what follows;  these results could be extended canonically over any base via the standard procedure (see \cite[\S 3]{J22a}). 
We consider the following pairing, defined as the composite map 
	$$B \colon \Gamma^d(\sV) \otimes \Gamma^d(\sV) \xrightarrow{\theta} \Gamma^d(\sV \otimes \sV) \xrightarrow{\Gamma^d (can.)} \Gamma^d(\bigwedge\nolimits^2 \sV) \simeq (\bigwedge\nolimits^2 \sV)^{\otimes d},$$
where the first map $\theta$ is the dual map of the Cauchy embedding 
	$$\Sym^d(\sV^\vee \otimes \sV^\vee) \to \Sym^d(\sV^\vee) \otimes \Sym^d(\sV^\vee),$$
which carries a local element $(x_1 \otimes y_1) \cdots (x_d \otimes y_d) \in \Sym^d(\sV^\vee \otimes \sV^\vee)$ to the local element $(x_1 \cdots x_d) \otimes (y_1 \cdots y_d) \in \Sym^d(\sV^\vee) \otimes \Sym^d(\sV^\vee)$, and $can.$ is the canonical projection $\sV \otimes \sV \to \bigwedge^2 \sV$. (Notice that the morphisms $\theta$ and $\Gamma^d(can.)$ are defined for any vector bundle $\sV$, but the isomorphism $\Gamma^d(\bigwedge\nolimits^2 \sV) \simeq (\bigwedge\nolimits^2 \sV)^{\otimes d}$ is true only when $\sV$ has rank two.) In concrete terms, if $\{e_1, e_2\}$ is a local basis of $\sV$, then $\{e_1^{(h)} e_2^{(d-h)} \otimes e_1^{(\ell)} e_2^{(d-\ell)} \}_{0 \le h,\ell \le d}$ forms a local basis of $\Gamma^d(\sV) \otimes \Gamma^d(\sV)$, and  we have 
	$$\theta(e_1^{(h)} e_2^{(d-h)} \otimes e_1^{(\ell)} e_2^{(d-\ell)}) = \sum\nolimits_{i} (e_1 \otimes e_1)^{(i)} (e_1 \otimes e_2)^{(h-i)} (e_2 \otimes e_1)^{(\ell-i)} (e_2 \otimes e_2)^{(d-h-\ell+i)}$$
	$$B(e_1^{(h)} e_2^{(d-h)} \otimes e_1^{(\ell)} e_2^{(d-\ell)}) = \delta_{h,\ell} \cdot (-1)^{d-h} \cdot (e_1\wedge e_2)^{\otimes d} \in (\bigwedge\nolimits^2 \sV)^{\otimes d}.$$ 	
(Here, $\delta_{h,\ell}$ is the Kronecker delta function: $\delta_{h,\ell} = 1$ if $h=\ell$ and $\delta_{h,\ell} = 0$ if $h \neq \ell$.)	
Therefore, the above pairing $B$ is perfect and induces a canonical isomorphism 
	$$\gamma \colon \Gamma^d(\sV) \simeq \Gamma^d(\sV)^\vee \otimes (\bigwedge\nolimits^2 \sV)^{\otimes d}.$$
	
Next, by virtue of \cite[Proposition III.3, p. 256]{Ro}, there are canonical maps
	$$\alpha_d \colon \Sym^d(\sV) \to \Gamma^d(\sV) \quad \text{and} \quad \beta_d \colon \Gamma^d(\sV) \to \Sym^d(\sV)$$
(for a vector bundle $\sV$ of any rank) for which there are canonical isomorphisms
	\begin{equation}\label{eqn:P1:alpha.beta}
	\beta_d \circ \alpha_d \simeq d! \cdot \id_{\Sym^d(\sV)}  \quad \text{and} \quad  \alpha_d \circ \beta_d \simeq  d! \cdot \id_{\Gamma^d(\sV)}.
\end{equation}
(Here, our map $\beta_d$ is the composite map $q \circ \beta$ of \cite[Proposition III.2, p. 255]{Ro}). In concrete terms, if $\{e_i\}$ is a local basis of $\sV$, then $\alpha_{d}$ carries $e_{i_1} \cdots e_{i_d} \in \Sym^d(\sV)$ to $e_{i_1} \cdots e_{i_d} \in \Gamma^d(\sV)$, and $\beta_{d}$ carries $e_{j_1}^{(\nu_1)} \cdots e_{j_h}^{(\nu_h)} \in \Gamma^d(\sV)$ to $q(e_{j_1}^{\otimes \nu_1} * \cdots * e_{j_h}^{\otimes \nu_h}) \in \Sym^d(\sV)$, where $j_k$'s are pairwise distinct, $\nu_1 + \cdots + \nu_h = d$, the symbol ``$*$" denotes the shuffle product (\cite[III \S5, p. 252-523]{Ro}), and $q$ denotes the canonical projection $\sV^{\otimes d} \to \Sym^d(\sV)$.

We now let $\phi'$ and $\psi'$ denote the compositions of canonical morphisms
	\begin{align*}
	&\phi' \colon \Sym^d(\sV) \xrightarrow{\alpha_d} \Gamma^d(\sV) \xrightarrow{\gamma} \Gamma^d(\sV)^\vee \otimes (\bigwedge\nolimits^2 \sV)^{\otimes d} \xrightarrow{\alpha_{d}^\vee \otimes \id} \Sym^d(\sV)^\vee \otimes (\bigwedge\nolimits^2 \sV)^{\otimes d}\\
	&\psi' \colon \Sym^d(\sV)^\vee \otimes (\bigwedge\nolimits^2 \sV)^{\otimes d}
  \xrightarrow{\beta_d^\vee \otimes \id}  
  \Gamma^d(\sV)^\vee \otimes (\bigwedge\nolimits^2 \sV)^{\otimes d} 
  \xrightarrow{\gamma^{-1}}
   \Gamma^d(\sV) \xrightarrow{\beta_d}
   \Sym^d(\sV)
   \end{align*}
and let $\phi_{\sL}$ and $\psi_{\sL}$ denote the compositions of canonical morphisms
	\begin{align*}
		\phi_{\sL} \colon \pi_*(\sL) \simeq \Sym^d(\sV) \otimes \sL_Y \xrightarrow{\phi' \otimes \id_{\sL_Y}}
		\Sym^d(\sV)^\vee \otimes (\bigwedge\nolimits^2 \sV)^{\otimes d} \otimes \sL_Y \simeq \pi_*(\sL \otimes \sK^{\otimes (d+1)})[1]. \\
		\psi_{\sL} \colon \pi_*(\sL \otimes \sK^{\otimes (d+1)})[1]\simeq \Sym^d(\sV)^\vee \otimes (\bigwedge\nolimits^2 \sV)^{\otimes d} \otimes  \sL_Y 
		\xrightarrow{\psi' \otimes  \id_{\sL_Y}}  \Sym^d(\sV) \otimes \sL_Y \simeq \pi_*(\sL).
	\end{align*}
By virtue of \eqref{eqn:P1:alpha.beta}, there are canonical isomorphisms $\psi' \circ \phi' \simeq (d!)^2$ and $\phi' \circ \psi' \simeq (d!)^2$ which imply the desire canonical isomorphisms. \end{proof}

We recall the following notations from the representation theory of reductive groups. 

\begin{notation}[Roots and Reflections]
Let $\Lambda = \ZZ \varepsilon_1 + \ZZ \varepsilon_2 + \cdots + \ZZ \varepsilon_n \simeq \ZZ^n$ denote a rank $n$ lattice, and we regard a sequence $\lambda$ as an element of $\Lambda$: 
	$$\lambda = (\lambda_1, \ldots, \lambda_n) = \lambda_1 \varepsilon_1 + \cdots + \lambda_n \varepsilon_n \in \Lambda.$$
Let $\foS_n$ denote the permutation group of the entries of $\Lambda$, that is, for a permutation $w \in \foS_n$, 
	$$w(\lambda) = (\lambda_{w(1)}, \ldots, \lambda_{w(n)}) \in \Lambda.$$
We let $\Lambda^\vee = \Hom(\Lambda, \ZZ)$ denote the dual lattice, and let $\{\varepsilon_i^\vee \}$ denote the dual basis of $\{\varepsilon_i\}$. For $\lambda \in \Lambda$, $\alpha^\vee \in \Lambda^\vee$, we let $\langle \lambda, \alpha^\vee \rangle = \alpha^\vee(\lambda) \in \ZZ$ denote the natural pairing. We let $\Phi = \{\varepsilon_i - \varepsilon_j\}_{1 \le i,j \le n} \subseteq \Lambda$, $\Phi_+ =\{\varepsilon_i - \varepsilon_j\}_{1 \le i<j \le n} \subseteq \Phi$, and $\Phi^\vee = \{\varepsilon_i^\vee - \varepsilon_j^\vee\}_{1 \le i,j \le n} \subseteq \Lambda^\vee$. The elements of $\Phi$ and $\Phi^\vee$ are called {\em roots} and {\em coroots}, respectively, and the elements of $\Phi_+$ are called {\em positive roots}. The permutation group $\foS_n$ also acts on $\Lambda$. $\Phi \subseteq \Lambda$ and $\Phi^\vee \subseteq \Lambda^\vee$ are invariant under the action of $\foS_n$. Moreover, there is an isomorphism
	$$\Phi \to \Phi^\vee, \qquad \alpha = \varepsilon_i - \varepsilon_j \mapsto \alpha^\vee: = \varepsilon_i^\vee - \varepsilon_j^\vee$$
which commutes with the action of $\foS_n$. For any element $\alpha \in \Phi$, there is a reflection map 
	$$s_{\alpha} \colon \Lambda \to \Lambda \qquad s_{\alpha}(\lambda) = \lambda - \langle \lambda, \alpha^\vee\rangle \alpha.$$
Then we can identify $s_{\alpha}$ as an element of $\foS_n$. For each $1 \le i < n$, we let $\alpha_i = \varepsilon_i - \varepsilon_{i+1}$ and $\alpha_i^\vee =  \varepsilon_i - \varepsilon_{i+1}$, and let $s_i = s_{\alpha_i}$ denote the resulting reflection (in literatures, these $\alpha_i$ are called {\em simple roots}, and $s_i$ are called {\em simple reflections}). In concrete terms, 
	$$s_i((\lambda_1, \ldots, \lambda_i, \lambda_{i+1}, \ldots, \lambda_n))
	 = (\lambda_1, \ldots, \lambda_{i+1}, \lambda_i, \ldots, \lambda_n).$$
Finally, we let $\rho$ denote the element
	$$\rho = (n-1, n-2, \ldots, 1, 0) \in \Lambda$$
and introduce the ``dot action" of $\foS_n$ on $\Lambda$: for $w \in \foS_n$, $\lambda \in \Lambda$, we let
	$$w \bigdot \lambda = w(\lambda+\rho) - \rho \in \Lambda.$$
In concrete terms, the ``dot action" of simple roots are given by
	\begin{equation}\label{eqn:si.dot}
	s_i \bigdot (\lambda_1, \ldots,\lambda_{i-1}, \lambda_i, \lambda_{i+1},  \lambda_{i+2}, \ldots, \lambda_n)
	 = (\lambda_1, \ldots, \lambda_{i-1}, \lambda_{i+1} -1, \lambda_i +1, \lambda_{i+2}, \ldots, \lambda_n).
\end{equation}
Notice that our choice of $\rho$ is a convenient one and differs from the half-sum of positive roots
	$$\rho_{\rm can} =\frac{1}{2} \sum_{\alpha \in \Phi_+} \alpha= \Big(\frac{n-1}{2}, \frac{n-3}{2}, \ldots, \frac{3-n}{2}, \frac{1-n}{2} \Big)$$
by a translation: $\rho = \rho_{\rm can} + ( \frac{n-1}{2} , \frac{n-1}{2} , \ldots,  \frac{n-1}{2} )$. The translation does not affect the definition of the ``dot action".
\end{notation}

\begin{corollary}\label{cor:P1:flag}
Let $X$ be a prestack, $\sE$ a connective quasi-coherent complex on $X$ which is locally of finite type, and $\pr \colon \Flag_X(\sE; \underline{n}) \to X$ the derived flag scheme of $\sE$ of type $\underline{n}$ over $X$ (Example \ref{eg:dflag}). Assume that $n \ge 2$. Let $\lambda = (\lambda_1 , \ldots, \lambda_n) \in \ZZ^n$ be a sequence of integers, let $\sL(\lambda)$ be the associated line bundle  \eqref{eqn:flag:linebundle}, and let $1 \le i  \le n-1$ be an integer.

\begin{enumerate}[leftmargin=*]
	\item \label{cor:P1:flag-1}
	If $\langle \lambda, \alpha_i^\vee \rangle \equiv \lambda_{i} - \lambda_{i+1} = -1$, then there are canonical isomorphisms
		$$\pr_* (\sL(\lambda)) \simeq 0 \simeq \pr_*(\sL(s_i \bigdot \lambda)) [1] .$$
	\item \label{cor:P1:flag-2}
	If 
	$d := \langle \lambda, \alpha_i^\vee \rangle \equiv \lambda_{i} - \lambda_{i+1} \ge 0,$
then there are canonical morphisms 
	$$\phi_{\sL(\lambda)} \colon \pr_* (\sL(\lambda)) \to \pr_*(\sL(s_i \bigdot \lambda)) [1]  \qquad 
	\psi_{\sL(\lambda)} \colon  \pr_*(\sL(s_i \bigdot \lambda)) [1]  \to \pr_* (\sL(\lambda))
	$$
and canonical equivalences of functors
	$$\psi_{\sL(\lambda)} \circ \phi_{\sL(\lambda)} \simeq (d!)^2 \cdot \id_{ \pr_* (\sL(\lambda)) } \qquad  \phi_{\sL(\lambda)} \circ \psi_{\sL(\lambda)} \simeq (d!)^2 \cdot \id_{\pr_*(\sL(s_i \bigdot \lambda) [1]}).$$
	\item \label{cor:P1:flag-3}
	If 
	$ \langle \lambda, \alpha_i^\vee \rangle \equiv \lambda_{i} - \lambda_{i+1} \le -2$, we let 	
	$d' = \langle s_i \bigdot \lambda, \alpha_i^\vee \rangle \equiv \lambda_{i+1} - \lambda_i -2 \ge 0$, then there are canonical morphisms 
	$$\phi_{\sL(\lambda)}'  \colon \pr_* (\sL(\lambda)) \to \pr_*(\sL(s_i \bigdot \lambda)) [-1]  \qquad 
	\psi_{\sL(\lambda)}' \colon  \pr_*(\sL(s_i \bigdot \lambda)) [-1]  \to \pr_* (\sL(\lambda)), 
	$$
and canonical equivalences of functors
	$$\psi_{\sL(\lambda)}' \circ \phi_{\sL(\lambda)}' \simeq (d'!)^2 \cdot \id_{ \pr_* (\sL(\lambda)) } \qquad  
	 \phi_{\sL(\lambda)}' \circ \psi_{\sL(\lambda)}' \simeq (d'!)^2 \cdot \id_{\pr_*(\sL(s_i \bigdot \lambda) [-1]}).$$
\end{enumerate}	
Consequently, if $X$ is defined over $\QQ$, then there are canonical isomorphisms 
	\begin{align*}
	&\phi_{\sL(\lambda)} \colon \pr_* (\sL(\lambda)) \xrightarrow{\sim} \pr_*(\sL(s_i \bigdot \lambda)) [1]  &\text{if} \quad \lambda_i - \lambda_{i+1} \ge -1; \\
	&\phi'_{\sL(\lambda)} \colon \pr_* (\sL(\lambda)) \xrightarrow{\sim} \pr_*(\sL(s_i \bigdot \lambda)) [-1]  &\text{if} \quad \lambda_i - \lambda_{i+1} \le -1.
	\end{align*}
\end{corollary}

\begin{proof} 
Let $\pi = \pi_{\underline{n} \backslash \{i\}, \underline{n}} \colon \Flag(\sE; \underline{n}) \to \Flag(\sE; \underline{n} \backslash \{ i \})$ denote the natural forgetful morphism (\S \ref{sec:dflag:forget}), and let $\sV$ denote the rank two vector bundle $\Ker(\sQ_{i+1} \twoheadrightarrow \sQ_{i-1})$ over $\Flag(\sE; \underline{n} \backslash \{i\})$. Then by virtue of Lemma \ref{lem:dflag:forget} \eqref{lem:dflag:forget-3}, $\Flag(\sE; \underline{n})$ is canonically identified with the projective bundle $\PP(\sV)$ over $\Flag(\sE; \underline{n} \backslash \{i\})$ such that there is a canonical equivalence $\sO_{\PP(\sV)}(1) \simeq \sL_{i}$. Consequently, there is a canonical short exact sequence of vector bundles over $\Flag(\sE; \underline{n})$,
	$$0 \to \sL_{i+1} \to \pi^*(\sV) \to \sL_{i} \to 0,$$
and the relative canonical line bundle $\sK$ of $\pi \colon \PP(\sV) \to \Flag(\sE; \underline{n})$ is given by
	$$\sK = \sO_{\PP(\sV)}(-2) \otimes \bigwedge\nolimits^2 \sV \simeq \sL_{i+1} \otimes \sL_{i}^{\vee} \simeq \sL(\varepsilon_{i+1} - \varepsilon_{i}).$$
From the equivalences $\sL_{i} \simeq \sO_{\PP(\sV)}(1)$ and $\sL_{i+1} \simeq \pi^*(\bigwedge^2 \sV) \otimes \sL_i^\vee \simeq \pi^* \det(\sV) \otimes \sO_{\PP(\sV)}(-1)$,  we see that $\sL(\lambda)$ has degree $d = \lambda_i - \lambda_{i+1} \ge -1$ over fibers of $\pi$. Hence we have
	$$\sL(\lambda) \otimes \sK^{\otimes (d+1)} \simeq \sL(\lambda) \otimes (\sL_{i+1} \otimes \sL_{i}^{\vee})^{\otimes (\lambda_i - \lambda_{i+1}+1)} \simeq \sL(s_i \bigdot \lambda),$$
where $s_i \bigdot \lambda$ is the ``dot action", concretely given by the formula \eqref{eqn:si.dot}. Therefore assertions \eqref{cor:P1:flag-1} and \eqref{cor:P1:flag-2} follow from Proposition \ref{prop:P1:isom} by setting $\sL = \sL(\lambda)$. Finally, assertion \eqref{cor:P1:flag-3} follows from assertion \eqref{cor:P1:flag-2} by reversing the roles of $\sL(\lambda)$ and $\sL(s_i \bigdot \lambda)$, that is, by setting $\sL = \sL(s_i \bigdot \lambda)$ in Proposition \ref{prop:P1:isom} and letting $\phi_{\sL(\lambda)}'=\psi_{\sL(s_i \bigdot \lambda)}[-1] $ and $\psi_{\sL(\lambda)}' = \phi_{\sL(s_i \bigdot \lambda)}[-1]$.
\end{proof}

We observe that the vanishing result $\pr_* (\sL(\lambda)) \simeq 0$ of Corollary \ref{cor:P1:flag} \eqref{cor:P1:flag-1} does not require any condition on the characteristic of the ring (over which $X$ is defined) or the quasi-coherent complex $\sE$. It admits the following generalization:

\begin{proposition}[Vanishing Results for Derived Complete Flag Schemes]
\label{prop:dflag:vanishing}
Let $\sE$ be a connective quasi-coherent complex on a prestack $X$ which is locally of finite type, let $\pr \colon \Flag_X(\sE; \underline{n}) \to X$ denote the derived flag scheme of $\sE$ of type $\underline{n}$ (Example \ref{eg:dflag}), where $n \ge 2$ is an integer, and let $\sL(\lambda)$ be the line bundle \eqref{eqn:flag:linebundle} on $\Flag_X(\sE; \underline{n})$ associated with $\lambda = (\lambda_1 , \ldots, \lambda_n) \in \ZZ^n$.
Assume that either one of the following two conditions hold:
	\begin{enumerate}[label=$(\alph*)$, ref=$\alph*$, leftmargin=*]
	\item \label{prop:dflag:vanishing-a}
	There exists integers $k \in \ZZ$, $1 \le \ell \le m \le n-1$ and $1 \le i \le n$ such that $\lambda_i = k-\ell$, $\lambda_{i+1} = \ldots = \lambda_{i+m} = k$; that is, $\lambda$ contains $(m+1)$ consecutive entries of the form 
			$$(k-\ell, \underbrace{k, k, \ldots, k}_{\text{$m$ terms}}) \qquad \text{where} \qquad 1 \le \ell \le m.$$
	\item \label{prop:dflag:vanishing-b}
	There exists integers $k \in \ZZ$, $1 \le \ell \le m \le n-1$ and $1 \le i \le n$ such that $\lambda_{i}= \ldots = \lambda_{i+m-1} = k$ and $\lambda_{i+m} = k+\ell$; that is,  $\lambda$ contains $(m+1)$ consecutive entries of the form 
			$$(\underbrace{k, k, \ldots, k}_{\text{$m$ terms}}, k+ \ell) \qquad \text{where} \qquad 1 \le \ell \le m.$$
	\end{enumerate}
Then $\pr_*(\sL(\lambda)) \simeq 0$. 
\end{proposition}
\begin{proof}
Consider the derived partial flag scheme 
	$$\pr_{\FF} \colon \FF := \Flag_X(\sE; \underline{n} \backslash \{i,i+1,\ldots,i+m-1\}) \to X$$
and let $\phi_{i-1,i+m} \colon \sQ_{i+m} \to \sQ_{i-1}$ denote the tautological morphism between from tautological quotient bundles over $\FF$ of ranks $i+m$ and $i-1$. Then $\sV=\Ker (\phi_{i-1,i+1})$ is a vector bundle over $\FF$ of rank $m+1$. By virtue of Corollary \ref{cor:dflag:forget} \eqref{cor:dflag:forget-3}, the forgetful map
	$$\pi: = \pi_{\underline{n} \backslash \{i,i+1,\ldots,i+m-1\}, \underline{n}} \colon \Flag_X(\sE; \underline{n}) \to \FF =  \Flag_X(\sE; \underline{n} \backslash \{i,i+1,\ldots,i+m-1\})$$
canonically identifies $\Flag_X(\sE; \underline{n})$ as the complete flag bundle $\Flag_{\FF}(\sV; \underline{m})$ over $\FF$. 

Assume condition \eqref{prop:dflag:vanishing-a} holds. Let $\sL_\FF = \sL_{1}^{\lambda_1} \otimes \cdots \otimes \sL_{i-1}^{\lambda_{i-1}} \otimes \sL_{i+m+1}^{\lambda_{i+m+1}} \otimes \cdots \otimes \sL_{n}^{\lambda_n}$ be the line bundle on $\FF$ corresponding to the ``leftover entries" $(\lambda_1, \ldots, \lambda_{i-1}, \lambda_{i+m+1}, \ldots, \lambda_{n})$ of $\lambda$, then 
	$$\sL(\lambda) \simeq \pi^*(\sL_\FF) \otimes \sL_{i}^{k-\ell} \otimes (\sL_{i+1} \otimes \cdots \otimes \sL_{i+m})^{k}.$$ 
By virtue of Corollary \ref{cor:dflag:forget} \eqref{cor:dflag:forget-2}, there is a canonical factorization
	$$\pi \colon \Flag_X(\sE; \underline{n}) \simeq \Flag_{\FF}(\sV; \underline{m})  \xrightarrow{\pi'} \Flag_{\FF}(\sV; 1) \simeq \PP_{\FF}(\sV) \xrightarrow{\pr_{\PP(\sV)}} \FF,$$
where $\sL_i \simeq \pi'^{*} (\sO_{\PP(\sV)}(1))$ (Remark \ref{rem:forget:linebundle}) and the forgetful morphism $\pi'$ canonically identifies $\Flag_X(\sE; \underline{n})$ as the complete flag bundle $\Flag_{\PP_\FF(\sV)}(\Omega_{\PP(\sV)}^1(1); \underline{m-1})$ over $\PP_{\FF}(\sV)$ (where $\Omega_{\PP(\sV)}^1(1)$ is the rank $m$ vector bundle over $\PP_{\FF}(\sV)$ which fits into the short exact sequence $0 \to \Omega_{\PP(\sV)}^1(1) \to \pr_{\PP(\sV)}^* (\sV) \to \sO_{\PP(\sV)}(1) \to 0$). Consequently, we have 
	$$\sL_{i+1} \otimes \cdots \otimes \sL_{i+m} \simeq \pi'^* \det (\Omega_{\PP(\sV)}^1(1)) \simeq \pi'^*(\sO_{\PP(\sV)}(-1)) \otimes \pi^* (\det \sV).$$
Therefore, we obtain canonical equivalences
	\begin{align*}
	\sL(\lambda) & \simeq  \pi^*(\sL_{\FF})  \otimes \pi'^*\big(\sO_{\PP(\sV)}(k-\ell)\big)  \otimes \pi'^*\Big( \pr_{\PP(\sV)}^*(\det \sV)^{\otimes k} \otimes (\sO_{\PP(\sV)}(-k)\Big) \\
			& \simeq   \pi^*\big(\sL_{\FF} \otimes \det(\sV)^{\otimes k}\big) \otimes \pi'^*\big(\sO_{\PP(\sV)}(-\ell)\big).
\end{align*}
As a result, $\pi_*(\sL(\lambda)) \simeq (\sL_{\FF} \otimes \det(\sV)^{\otimes k}) \otimes \pr_{\PP(\sV)}(\sO_{\PP(\sV)}(-\ell)) \simeq 0$, where we use Serre's vanishing $\pr_{\PP(\sV)}(\sO_{\PP(\sV)}(-\ell)) \simeq 0$ ($1 \le \ell \le m = \rank \sV-1$) for the projective bundle $\PP_{\FF}(\sV)$ (\cite[III, 2.1.15]{EGA}, \cite[Theorem 5.2 (2i)]{J22a}). Hence $\pr_*(\sL(\lambda)) \simeq (\pr_{\FF})_* \, \pi_*(\sL(\lambda)) \simeq 0$. 

The proof of the result under condition \eqref{prop:dflag:vanishing-b} is similar: we consider another factorization,
	$$\pi \colon \Flag_X(\sE; \underline{n}) \simeq \Flag_{\FF}(\sV; \underline{m})  \xrightarrow{\pi''} \Flag_{\FF}(\sV; m) \simeq \PP_{\FF}(\sV^\vee) \xrightarrow{\pr_{\PP(\sV^\vee)}} \FF,$$
where $\sL_{i+m} \simeq \pi''^*(\sO_{\PP(\sV^\vee)}(-1))$ and $\sL_{i} \otimes \cdots \otimes \sL_{i+m-1} \simeq \pi''^*(\sO_{\PP(\sV^\vee)}(1)) \otimes \pi^*(\det \sV)$; the desired vanishing result follows from Serre's vanishing for the projective bundle $\PP_{\FF}(\sV^\vee)$.
\end{proof}

\begin{theorem}[Borel--Weil--Bott Theorem for Derived Complete Flag Schemes]
\label{thm:BBW:dflag}
Let $X$ be a prestack defined over $\QQ$, $\sE$ a perfect complex of rank $n \ge 1$ and Tor-amplitude in $[0,1]$ over $X$, and let $\pr \colon \Flag_X(\sE; \underline{n}) \to X$ denote the derived flag scheme of $\sE$ of type $\underline{n}$ over $X$. Let $\lambda = (\lambda_1 , \ldots, \lambda_n) \in \ZZ_{\ge 0}^n$ and let $\sL(\lambda)$ be the associated line bundle \eqref{eqn:flag:linebundle}. %Then: % one of the following two mutually exclusive cases occurs:
\begin{enumerate}
	\item \label{thm:BBW:dflag-1}
	If there exists a pair integers $1 \le i < j \le n-1$ such that $\lambda_i  - \lambda_j = i - j$, then 
		$$\pr_* (\sL(\lambda)) \simeq 0.$$
	\item \label{thm:BBW:dflag-2}
	If there exists a unique permutation $w \in \foS_n$ such that $w(\lambda+ \rho)$ is strictly decreasing (equivalently, $w \bigdot \lambda$ is non-increasing), then there is a canonical equivalence
		$$\pr_* (\sL(\lambda)) \simeq \dSchur^{w \bigdot \lambda}(\sE) [- \ell(w)].$$
\end{enumerate}
Moreover, for any $\lambda \in \ZZ_{\ge 0}^n$, one (and only one) of these two cases occurs.
\end{theorem}

\begin{proof} 
We follow the same approach as in the classical case.
There exists a permutation $w \in \foS_n$ such that $w(\lambda + \rho)$ is dominant (that is, non-increasing). Let $w = s_{i_1} s_{i_2} \cdots s_{i_{\ell(w)}}$ be a reduced expression, where $s_{i_j}$ are simple reflections around the simple roots $\alpha_{i_j} = \varepsilon_{i_j} - \varepsilon_{i_j+1}$, and $\ell(w)$ is the length of $w$. We claim that, for each $j = 1, 2, \ldots, \ell(w)$, 
	$$\langle s_{i_{j+1}} \cdots s_{i_{\ell(w)}} (\lambda + \rho), \alpha_{i_j}^\vee\rangle \le 0.$$
(Here, we set $s_{i_{j+1}} \cdots s_{i_{\ell(w)}} (\lambda + \rho) = \lambda + \rho$ if $j=\ell(w)$.) To prove the claim, we let $w_{j} = s_{i_1} \cdots s_{i_{j-1}}$ (where, we set $w_{1} = \id$ if $j=1$). Then from the identity 
	$$\ell(w_j s_{i_j}) = \ell(w_j) + 1$$
we obtain that $w_j(\alpha_{i_j})$ is a positive root (see, for example, \cite[Part II, 1.5(3)]{Jan}). 
Therefore,
	\begin{align*}
		&\langle s_{i_{j+1}} \cdots s_{i_{\ell(w)}} (\lambda + \rho), \alpha_{i_j}^\vee \rangle = \langle s_{i_j}^{-1} w_j^{-1} w (\lambda + \rho), \alpha_{i_j}^\vee \rangle \\ 
		&= - \langle w_j^{-1} (w (\lambda + \rho)), \alpha_{i_j}^\vee  \rangle = - \langle w (\lambda + \rho),  w_j(\alpha_{i_j}^\vee)\rangle  \le 0.
	\end{align*}
Hence the claim is proved, and consequently, we have
	$$\langle s_{i_{j+1}} \cdots s_{i_{\ell(w)}} \bigdot \lambda, \alpha_{i_j}^\vee \rangle = \langle s_{i_{j+1}} \cdots s_{i_{\ell(w)}} (\lambda + \rho), \alpha_{i_j}^\vee \rangle - 1 \le -1.$$
Therefore, we can apply Corollary \ref{cor:P1:flag} and obtain canonical isomorphisms
	$$\pr_*(\sL(s_{i_{j+1}} \cdots s_{i_{\ell(w)}} \bigdot \lambda)) \simeq \pr_*( \sL(s_{i_j} s_{i_{j+1}} \cdots s_{i_{\ell(w)}} \bigdot \lambda))[-1]$$
for all $1 \le j \le \ell(w)$. Using the above isomorphisms repeatedly, we obtain canonical isomorphisms
	$$\pr_*(\sL(\lambda)) \simeq \pr_*(\sL(s_{\ell(w)} \bigdot \lambda))[-1] \simeq \cdots \simeq \pr_*( \sL(w\bigdot \lambda))[- \ell(w)].$$
If $w (\lambda +\rho) $ is not strictly decreasing, then there exists a simple root $\alpha_i$ such that $\langle w (\lambda+\rho), \alpha_i^\vee \rangle = 0$, that is, $\langle w \bigdot \lambda, \alpha_i^\vee \rangle = -1$. Then we obtain from Corollary \ref{cor:P1:flag} \eqref{cor:P1:flag-1} that
	 $$\pr_*(\sL(\lambda))  \simeq \pr_*( \sL(w\bigdot \lambda))[-\ell(w)] \simeq 0.$$
If $w (\lambda +\rho) $ is strictly decreasing (that is, $w \bigdot \lambda$ is non-increasing), then we obtain from Theorem \ref{thm:Bott:dflag} canonical equivalences
	$$\pr_*(\sL(\lambda)) \simeq \pr_*( \sL(w\bigdot \lambda))[-\ell(w)] \simeq \dSchur^{w \bigdot \lambda}(\sE) [-\ell(w)].$$
Hence the theorem is proved.
\end{proof}

\begin{remark}[Negative Entries]
Theorem \ref{thm:BBW:dflag} can be extended to the case where $\lambda$ has negative entries under certain conditions. Assume that we are in the situation of Theorem \ref{thm:BBW:dflag} except we allow the entries of $\lambda$ to be negative. In this case, the proof of Theorem \ref{thm:BBW:dflag} shows that assertion \eqref{thm:BBW:dflag-1} remains true, while assertion \eqref{thm:BBW:dflag-2} remains true if all entries of $w \bigdot \lambda$ are nonnegative. 
The case where $w \bigdot \lambda$ has negative entries is not covered by Theorem \ref{thm:BBW:dflag}. 
\end{remark}

\begin{variant}[{Compare with \cite[Part II, \S 5, Corollary 5.5]{Jan}}] 
\label{variant:BBW:dflag}
Assume that we are in the situation of Theorem \ref{thm:BBW:dflag} except that we now assume $X$ is defined over a field $\kappa$ which is not necessarily of characteristic zero. Assume furthermore:
\begin{enumerate}
	\item[$(a)$] Either $\kappa$ has characteristic zero and 
	$$\langle \lambda, \alpha_i^\vee\rangle :=  \lambda_{i} - \lambda_{i+1} \ge -1 \quad \text{for all} \quad 1 \le i \le n-1;$$
	\item[$(b)$] Or $\kappa$ has positive characteristic ${\rm char}(\kappa)$ and 
	$$0 \le \lambda_{i} - \lambda_{j} - (i-j) \le {\rm char}(\kappa) \quad \text{for all} \quad 1 \le i < j \le n-1.$$
\end{enumerate}
Then one of the following two mutually exclusive cases occurs:
 \begin{enumerate}
	\item There exists $i$ such that $\lambda_{i} - \lambda_{i+1} = -1$. In this case, for all permutations $w \in \foS_n$,
		 $$\pr_* (\sL(w \bigdot \lambda)) \simeq 0.$$
	\item The sequence $\lambda$ is non-increasing (that is, $\lambda_1 \ge \ldots \ge \lambda_n$). In this case, there are canonical equivalences for all permutations $w \in \foS_n$:
		$$ \pr_* (\sL(w \bigdot \lambda)) \simeq \dSchur^{\lambda}(\sE) [- \ell(w)].$$
\end{enumerate}	
\end{variant}
\begin{proof}
The proof is similar to Theorem \ref{thm:BBW:dflag}, except that we use the equivalences of Corollary \ref{cor:P1:flag} in the reverse direction. We will only prove the case $(b)$; the proof for $(a)$ is similar. Let $w = s_{i_1} s_{i_2} \cdots s_{i_{\ell(w)}}$ be a reduced expression, where $s_{i_j}$ are simple reflections around the simple roots $\alpha_{i_j} = \varepsilon_{i_j} - \varepsilon_{i_j+1}$. We claim that, for each $j=1, \ldots, \ell(w)$, the following holds:
	$$-1 \le \langle s_{i_{j+1}} \cdots s_{i_{\ell(w)}} \bigdot \lambda, \alpha_{i_j}^\vee \rangle \le {\rm char}(\kappa)-1.$$
In fact, let $v_j = s_{i_{j+1}} \cdots s_{i_{\ell(w)}}$, then $\ell(s_{i_j} v_j) = \ell(v_j) +1$ implies that $v_j^{-1}(\alpha_{i_j})$ is a negative root.
\begin{align*}
		 \langle s_{i_{j+1}} \cdots s_{i_{\ell(w)}} \bigdot \lambda, \alpha_{i_j}^\vee \rangle = \langle s_{i_j} v_j  (\lambda + \rho) - \rho, \alpha_{i_j}^\vee \rangle = - \langle \lambda + \rho, v_j^{-1}(\alpha_{i_j}^\vee) \rangle -1 \in [-1, {\rm char}(\kappa)-1].
	\end{align*}	
Therefore, we could apply Corollary \ref{cor:P1:flag} \eqref{cor:P1:flag-1} \& \eqref{cor:P1:flag-2} and obtain canonical equivalences
	$$\pr_*( \sL(s_{i_j} s_{i_{j+1}} \cdots s_{i_{\ell(w)}} \bigdot \lambda)) \simeq \pr_*(\sL(s_{i_{j+1}} \cdots s_{i_{\ell(w)}} \bigdot \lambda)) [-1]$$
for $j=1, \ldots, \ell(w)$. Using the above equivalences repeatedly, we obtain canonical equivalences
	$$\pr_*(\sL(w \bigdot \lambda)) \simeq \pr_*(\sL(s_{i_2} \cdots s_{i_{\ell(w)}} \bigdot \lambda))[-1] \simeq \cdots \simeq \pr_*(\sL(\lambda))[-\ell(w)].$$
Therefore, the desired equivalences follow from Corollary \ref{cor:P1:flag} \eqref{cor:P1:flag-1} and Theorem \ref{thm:Bott:dflag}.
\end{proof}

\begin{corollary}[Borel--Weil--Bott Theorem for Derived Partial Flag Schemes]
\label{cor:BBW:dpflag}
Let $X$ be a prestack over $\QQ$, $\sE$ a perfect complex of rank $n \ge 1$ and Tor-amplitude in $[0,1]$ over $X$, and $\bdd = (d_1, \ldots, d_k)$ an increasing sequence of integers in $[1,n]$. Let 
	$\pr \colon \Flag_X(\sE; \bdd) \to X$
denote the derived flag scheme.
Let $\lambda = (\lambda_1, \ldots , \lambda_n)$ be a sequence of non-negative integers such that the subsequences $\lambda^{(j)}= (\lambda_{d_{j-1}+1}, \ldots, \lambda_{d_{j}})$ are partitions for all $j=1, \ldots, k, k+1$ (that is, the condition $(*)$ of Construction \ref{constr:V.lambda:dpflag} is satisfied) and let $\sV(\lambda)$ be the associated perfect complex \eqref{eqn:flag:Vlambda}. Then one of the following two mutually exclusive cases occurs:
\begin{enumerate}
	\item There exists a pair of integers $1 \le i < j \le n-1$ such that $\lambda_i  - \lambda_j = i - j$. In this case,  
		$$\pr_* (\sV(\lambda)) \simeq 0.$$
	\item There exists a unique permutation $w \in \foS_n$ such that $w \bigdot \lambda$ is non-increasing. In this case, there is a canonical equivalence
		$$\pr_* (\sV(\lambda)) \simeq \dSchur^{w \bigdot \lambda}(\sE) [- \ell(w)].$$
\end{enumerate}
\end{corollary}
\begin{proof}
Let $\sL(\lambda)$ be the associated line bundle \eqref{eqn:flag:linebundle} on $\Flag(\sE; \underline{n})$ and let $\pi_{\bdd, \underline{n}} \colon \Flag(\sE; \underline{n}) \to \Flag(\sE; \bdd)$ denote the forgetful morphism (\S \ref{sec:dflag:forget}). By virtue of Proposition \ref{prop:dflag:forget} and Theorem \ref{thm:Bott:dflag}, there is a canonical equivalence 
	$(\pi_{\bdd, \underline{n}})_{*} (\sL(\lambda)) \simeq \sV(\lambda).$
Hence there is a canonical equivalence $\pr_*(\sV(\lambda)) \simeq (\pr_{\Flag(\sE;\underline{n})})_*(\sL(\lambda))$ and the desired result follows from Theorem \ref{thm:BBW:dflag}. 
\end{proof}

\begin{corollary}[Borel--Weil--Bott theorem for derived Grassmannians]
\label{cor:BBW:dGrass}
Let $X$ be a prestack over $\QQ$, $\sE$ a perfect complex of rank $n \ge 1$ and Tor-amplitude in $[0,1]$ over $X$. Let $d$ be an integer such that $1 \le d \le n$ and let $\pr \colon \Grass(\sE;d) \to X$ denote the derived Grassmannian.  Let $\alpha = (\alpha_1, \ldots , \alpha_{d})$ and $\beta = (\beta_1, \ldots, \beta_{n-d})$ be two partitions and let $\lambda = (\alpha,\beta)$ be their concatenation. Let $\sV(\alpha, \beta)$ denote the associated perfect complex \eqref{eqn:Grass:Vlambda} (see Example \ref{eg:notation:V:dGrass}). Then one of the following two  mutually exclusive cases occurs:
\begin{enumerate}
	\item There exists $1 \le i < j \le n-1$ such that $\lambda_i  - \lambda_j = i - j$. Then 
		$\pr_* (\sV(\alpha,\beta)) \simeq 0.$
	\item There exists a unique permutation $w \in \foS_n$ such that $w \bigdot (\alpha,\beta)$ is non-increasing. In this case, there is a canonical equivalence
		$$\pr_* (\sV(\alpha,\beta)) \simeq \dSchur^{w \bigdot (\alpha,\beta)}(\sE) [- \ell(w)].$$
\end{enumerate}		
\end{corollary}
\begin{proof}
Let $\bdd = (d)$ in Corollary \ref{cor:BBW:dpflag}.
\end{proof}

%%% Reference
%\newpage	
\addtocontents{toc}{\vspace{\normalbaselineskip}}

\end{document}